\documentclass[12pt,leqno]{amsart}

\usepackage{amssymb}
\usepackage[all]{xy}
\usepackage{enumerate}

\usepackage{color}

\newenvironment{rouge}{\color{red}}{}
\newcommand{\btr}{\begin{rouge}}
\newcommand{\etr}{\end{rouge}}
\newenvironment{bleu}{\color{blue}}{}
\newcommand{\btb}{\begin{bleu}}
\newcommand{\etb}{\end{bleu}}

\newcommand{\mo}{\mathopen}
\newcommand{\mc}{\mathclose}

\newtheorem{theorem}{Theorem}[section]
\newtheorem{proposition}[theorem]{Proposition}
\newtheorem{lemma}[theorem]{Lemma}
\newtheorem{corollary}[theorem]{Corollary}

\theoremstyle{definition}

\newtheorem{definition}[theorem]{Definition}
\newtheorem{notation}[theorem]{Notation}
\newtheorem{example}[theorem]{Example}

\newtheorem{remark}[theorem]{Remark}

\newtheorem{assumption}[theorem]{Assumption}

\newcommand{\N}{{\mathbb{N}}}
\newcommand{\R}{{\mathbb{R}}}

\newcommand{\Z}{{\mathbb{Z}}}

\newcommand{\shf}{{\mathcal{F}}}
\newcommand{\shi}{{\mathcal{I}}}
\newcommand{\shk}{{\mathcal{K}}}
\newcommand{\shl}{{\mathcal{L}}}
\newcommand{\shm}{{\mathcal{M}}}
\newcommand{\shn}{{\mathcal{N}}}
\newcommand{\sht}{{\mathcal{T}}}

\newcommand{\Cinf}{C^\infty}

\newcommand{\cor}{{\bf k}}

\newcommand{\pt}{\{{\rm pt}\}}
\newcommand{\rpos}{\R_{\geq 0}}
\newcommand{\rspos}{\R_{>0}}

\renewcommand{\to}[1][]{\xrightarrow[]{#1}}
\newcommand{\from}[1][]{\xleftarrow[]{#1}}
\newcommand{\isoto}[1][]{\xrightarrow[#1]%
{{\raisebox{-.6ex}[0ex][-.6ex]{$\mspace{1mu}\sim\mspace{2mu}$}}}}
\newcommand{\isofrom}[1][]{\xleftarrow[#1]%
{{\raisebox{-.6ex}[0ex][-.6ex]{$\mspace{1mu}\sim\mspace{2mu}$}}}}

\newcommand{\xto}[2][]{\xrightarrow[#1]{#2}}
\newcommand{\xfrom}[2][]{\xleftarrow[#1]{#2}}

\newcommand{\RR}{\mathrm{R}}

\newcommand{\muhom}{\mu hom}
\newcommand{\Hom}{\mathrm{Hom}}
\newcommand{\RHom}{\RR\mathrm{Hom}}

\renewcommand{\hom}{{\mathcal{H}om}}
\newcommand{\rhom}{{\RR\hom}}
\newcommand{\homeps}{{\mathcal{H}om^\varepsilon}}
\newcommand{\rhomeps}{{\RR\hom^\varepsilon}}
\newcommand{\DD}{\mathrm{D}}
\newcommand{\tens}{\otimes}
\newcommand{\ltens}{\mathbin{\overset{\scriptscriptstyle\mathrm{L}}\tens}}
\newcommand{\epstens}{\otimes^\varepsilon}
\newcommand{\etens}{\mathbin{\boxtimes}}
\newcommand{\letens}{\overset{\mathrm{L}}{\etens}}

\newcommand{\Iso}{{\operatorname{Iso}}}

\newcommand{\rsect}{\mathrm{R}\Gamma}

\newcommand{\oim}[1]{{#1}_*}
\newcommand{\eim}[1]{{#1}_!}
\newcommand{\roim}[1]{\RR{#1}_*}
\newcommand{\reim}[1]{\RR{#1}_!}
\newcommand{\opb}[1]{#1^{-1}}
\newcommand{\epb}[1]{#1^{!}}

\newcommand{\eqdot}{\mathbin{:=}}
\newcommand{\cl}{\colon}
\newcommand{\pointdiag}{\makebox[0mm]{\;\;.}}

\newcommand{\scbul}{{\,\raise.4ex\hbox{$\scriptscriptstyle\bullet$}\,}}
\newcommand{\ol}{\overline}
\newcommand{\ul}{\underline}

\newcommand{\id}{\mathrm{id}}

\newcommand{\supp}{\operatorname{supp}}

\newcommand{\SSi}{\mathrm{SS}}
\newcommand{\Int}{\operatorname{Int}}

\newcommand{\Der}{\mathsf{D}}
\newcommand{\Derb}{\Der^{\mathrm{b}}}
\newcommand{\Derlb}{\Der^{\mathrm{lb}}}

\newcommand{\Mod}{\operatorname{Mod}}

\newcommand{\codim}{\operatorname{codim}}
\newcommand{\dT}{{\dot{T}}}
\newcommand{\Cor}{{\mathbb{K}}}

\numberwithin{equation}{section}
\usepackage{hyperref}

\newcommand{\gammaof}{\gamma}

\newcommand{\gammaf}{{\ol\gamma}}
\newcommand{\Derbtp}{\Derb_{\tau>0}}
\newcommand{\Derbtpn}{\Derb_{\tau\geq 0}}
\newcommand{\Derbra}{\Der^{\mathrm{b}, r_!\mathrm{a}}_{\tau\geq 0}}

\newcommand{\tU}{\widetilde U}
\newcommand{\loc}{\mathsf{Loc}}
\newcommand{\Dloc}{\mathsf{DL}}
\newcommand{\dltens}{\otimes}

\newcommand{\loceps}{\mathsf{Loc}^\varepsilon}
\newcommand{\Dloceps}{\mathsf{DL}^\varepsilon}
\newcommand{\kss}{\mathfrak{S}}
\newcommand{\kssfunc}{\mathfrak{s}}
\newcommand{\mks}{\mathfrak{m}}
\newcommand{\mucirc}{\mathbin{\overset{\scriptscriptstyle\mu}{\circ}}}
\newcommand{\tIso}{\operatorname{\widehat{Iso}}}

\newcommand{\tshi}{\widetilde{\mathcal{I}}}
\newcommand{\tshm}{\widetilde{\mathcal{M}}}
\newcommand{\nshm}{\mathcal{L}}
\newcommand{\lag}{\mathcal{L}}

\newcommand{\demi}{\frac{1}{2}}
\newcommand{\pdemi}{{\textstyle \frac{1}{2}}}

\newcommand{\qq}{\mathrm{q}}
\newcommand{\qqq}{\mathbf{q}}

\newcommand{\hplus}{\mathbin{\widehat +}}

\newcommand{\mgL}{\mathcal{L}}
\newcommand{\mgLp}{\mathbb{L}}
\newcommand{\catc}{\mathcal{C}}

\newcommand{\perf}{\mathsf{perf}}
\newcommand{\Orb}{{\mathsf{D}_{/[1]}^{\mathrm{b}}}}
\newcommand{\OrbL}[1]{{\mathsf{D}_{/[1],#1}^{\mathrm{b}}}}
\newcommand{\orb}{\mathrm{orb}}
\newcommand{\SSo}{\mathrm{SS}^\orb}

\newcommand{\supporb}{\mathrm{supp}^\orb}
\newcommand{\Oloc}{\mathsf{OL}}

\newcommand{\psh}{\mathsf{Psh}}
\newcommand{\rec}{\alpha}

\newcommand{\clms}{{ext}}

\newcommand{\bpmat}{\left( \begin{smallmatrix}}
\newcommand{\epmat}{\end{smallmatrix} \right)}

\begin{document}

\begin{abstract}
  Let $M$ be a manifold and $\Lambda$ a compact exact connected Lagrangian
  submanifold of $T^*M$. We can associate with $\Lambda$ a conic Lagrangian
  submanifold $\Lambda'$ of $T^*(M\times\R)$.  We prove that there exists a
  canonical sheaf $F$ on $M\times\R$ whose microsupport is $\Lambda'$ outside
  the zero section. We deduce the already known results that the Maslov class
  of $\Lambda$ is $0$ and that the projection from $\Lambda$ to $M$ induces
  isomorphisms between the homotopy groups.
\end{abstract}

\date{January 24, 2015}
\title[Quantization of Lagrangian submanifolds]
{Quantization of conic Lagrangian  submanifolds of cotangent bundles}
\author{St{\'e}phane Guillermou}

\maketitle


\setcounter{tocdepth}{1}
\tableofcontents

\section{Introduction}

Let $N$ be a $C^\infty$ manifold and $\Lambda$ a closed conic Lagrangian
submanifold of $\dT^*N$, the cotangent bundle of $N$ with the zero section
removed.  Motivated by the paper~\cite{T08} where D.~Tamarkin used the
microlocal theory of sheaves of M.~Kashiwara and P.~Schapira to obtain results
in symplectic geometry (see~\cite{T08} and the survey~\cite{GS11}), we consider
the problem of constructing a sheaf on $N$ whose microsupport coincides with
$\Lambda$ outside the zero section.  We call such a sheaf a ``quantization'' of
$\Lambda$.  We assume $N = M\times \R$ and $\Lambda$ is the ``conification'' of
a compact exact Lagrangian submanifold of $T^*M$.  As explained to me by
C. Viterbo it is possible to define a quantization of $\Lambda$ by means of
Floer homology (see~\cite{V11}) and, conversely, to recover some aspects of
Floer homology from a quantization (for example, analogs of the spectral
invariants are introduced in~\cite{Vic12} using a quantization).

The main purpose of this paper is to construct a quantization using only the
microlocal theory of sheaves. We also recover from the existence of the
quantization results of Kragh~\cite{Kr13}, saying that the Maslov class of
$\Lambda$ is zero, and Fukaya-Seidel-Smith~\cite{FSS08} and Abouzaid~\cite{A12},
saying that the projection $\Lambda \to M$ is a homotopy equivalence.  We also
obtain the vanishing of the relative Stiefel-Whitney class of $\Lambda$.  We
point out that the link between the microlocal theory of sheaves and the
symplectic geometry is studied in another way by D.~Nadler and E.~Zaslow
in~\cite{NZ09,N09}, where they give in particular another proof of the result
of~\cite{FSS08}.

\subsection*{Microsupport}
Let us first recall one of the main ingredients of the microlocal theory of
sheaves (see~\cite{KS82,KS85,KS90}), namely, the microsupport of sheaves. Let
$\cor$ be a commutative unital ring of finite global dimension. We denote by
$\Derb(\cor_N)$ the bounded derived category of sheaves of $\cor$-modules on
$N$.  In loc.\ cit. the authors attach to an object $F$ of $\Derb(\cor_N)$ its
singular support, or microsupport, $\SSi(F)$, a closed subset of $T^*N$.  We
recall its definition in Section~\ref{section:mts}. The microsupport is conic
for the action of $(\R^+,\times)$ on $T^*N$ and is coisotropic.  It gives some
information on how the cohomology groups $H^i(U;F)$ vary when the open subset
$U\subset N$ moves. The microsupport was introduced as a tool for the study of
sheaves.  In~\cite{T08} Tamarkin uses it in the other direction: if a given
conic Lagrangian submanifold $\Lambda$ of $\dT^*N$ admits a quantization $F\in
\Derb(\cor_N)$, then we may use the cohomology of $F$ or its extension groups to
obtain results on $\Lambda$.  Tamarkin constructs quantizations of Lagrangian
submanifolds of $T^*SU(n)$ associated with subsets of the complex projective
space, in particular the real projective space and the Clifford torus. He
deduces non-displaceability results for these subsets.  In~\cite{GKS10},
building on Tamarkin's ideas, the authors consider a Hamiltonian isotopy, say
$\Phi\cl \dT^*N\times \mo]-1,1[ \to \dT^*N$, homogeneous for the $\R^+$-action
on $T^*N$. We can see its graph as a conic Lagrangian submanifold of $\dT^*(N
\times N\times \mo]-1,1[)$. The authors prove that this graph admits a
quantization and they deduce a new proof of a non-displaceability conjecture of
Arnold and results on non-negative isotopies.

In this paper we construct a quantization for $\Lambda$ when $N = M \times \R$
and $\Lambda$ is the ``conification'' of a compact exact Lagrangian submanifold
of $T^*M$.  It is well known that quantizations locally exist, that is, for a
given $p \in \Lambda$ there exist $F \in \Derb(\cor_N)$ and a neighborhood
$\Omega$ of $p$ in $T^*N$ such that $\SSi(F) \cap \Omega = \Lambda \cap
\Omega$.  If the projection $\Lambda/\rspos \to N$ is finite we even have: for
any $x\in N$ there exist a neighborhood $U$ of $x$ and $F^x \in \Derb(\cor_U)$
such that $\dot\SSi(F^x) = \Lambda \cap T^*U$, where we set $\dot\SSi(F) =
\SSi(F) \cap \dT^*N$.  We will prove that when $N = M \times \R$ we can choose
the $F^x$'s so that we can glue them into a global object.

\smallskip

A first step is to glue these locally defined objects ``microlocally''.  To
give a meaning to this we recall the definition of some categories introduced
in~\cite{KS85,KS90}.  Let $S$ be a subset of $\dT^*N$.  Following~\cite{KS90}
we denote by $\Derb_S(\cor_N)$ the full triangulated subcategory of
$\Derb(\cor_N)$ formed by the $F$ such that $\dot\SSi(F) \subset S$ (this
differs slightly from~\cite{KS90} since we forget the zero section).  We denote
by $\Derb_{(S)}(\cor_N)$ the full triangulated subcategory of $\Derb(\cor_N)$
formed by the $F$ such that $\SSi(F)\cap \Omega \subset S$, for some
neighborhood $\Omega$ of $S$.  We let $\Derb(\cor_N;S)$ be the quotient
$\Derb(\cor_N) / \Derb_{\dT^*N\setminus S}(\cor_N)$.  In particular the $F \in
\Derb(\cor_N)$ satisfying $\dot\SSi(F) \cap S = \emptyset$ vanish in the
quotient $\Derb(\cor_N;S)$. A morphism $u \cl F \to F'$ in $\Derb(\cor_N)$
becomes an isomorphism in $\Derb(\cor_N;S)$ if the cone of $u$, that is, the
object $F''$ which fits in a distinguished triangle $F \to F' \to F'' \to[+1]$,
satisfies $\SSi(F'') \cap S = \emptyset$.

Let $\Lambda$ be a locally closed conic submanifold of $\dT^*N$.  We define the
Kashiwara-Schapira stack of $\Lambda$, denoted $\kss(\cor_\Lambda)$, as the
stack associated with the prestack $\kss^0_\Lambda$ given as follows.  For
$\Lambda_0$ open in $\Lambda$ the objects of $\kss^0_\Lambda(\Lambda_0)$ are
those of $\Derb_{(\Lambda_0)}(\cor_N)$.  The morphisms between two objects $F,G$
are
$$
\Hom_{\kss^0_\Lambda(\Lambda_0)}(F,G) \eqdot
\Hom_{\Derb(\cor_N;\Lambda_0)}(F,G) .
$$
Taking the associated stack means that the objects of $\kss(\cor_\Lambda)$ are
represented by local objects $F_i \in \kss^0_\Lambda(\Lambda_i)$, for a covering
$\{\Lambda_i\}_{i\in I}$ of $\Lambda$, and isomorphisms $u_{ji} \cl
F_i|_{\Lambda_{ij}} \isoto F_j|_{\Lambda_{ij}}$ in
$\kss^0_\Lambda(\Lambda_{ij})$ on the intersections $\Lambda_{ij} = \Lambda_i
\cap \Lambda_j$, such that $u_{kj} \circ u_{ji} = u_{ki}$.  (We remark that the
triangulated structure is lost when we take the associated stack.)  In fact the
morphisms in $\kss(\cor_\Lambda)$ also have a definition which does not involve
the microsupport and the quotient categories $\Derb(\cor_N;S)$.
In~\cite{KS85,KS90} the authors give a version of Sato's microlocalization as a
functor
$$
(\Derb(\cor_N))^{opp} \times \Derb(\cor_N) \to \Derb(\cor_{T^*N}),
\qquad
(F,G) \mapsto \mu hom(F,G)
$$
with the properties $\roim{\pi_N} \mu hom (F,G) \simeq \rhom (F,G)$ and, if $F$
is constructible, $\reim{\pi_N} \mu hom (F,G) \simeq \DD'(F) \ltens G$, where
$\pi_N \cl T^*N \to N$ is the projection and $\DD'(F)$ denotes the dual.  The
support of $\mu hom (F,G)$ is bounded by $\SSi(F)\cap \SSi(G)$. By~\cite{KS90}
the $\hom$ sheaf in $\kss(\cor_\Lambda)$ is given by $\mu hom$: for $F,G \in
\Derb_{(\Lambda)}(\cor_N)$ we have $\hom_{\kss(\cor_\Lambda)}(F,G) = H^0\mu
hom(F,G)$.

\smallskip

We prove in Theorem~\ref{thm:KSstack=faisc_tordus} and
Corollary~\ref{cor:objet_global_KSstack} that $\kss(\cor_\Lambda)$ is equivalent
to a twisted stack of twisted local systems on $\Lambda$.  Let us explain what
it means.  For $p\in \Lambda$ we have the Lagrangian subspaces of $T_pT^*N$
given by $\lambda_\Lambda(p) = T_p\Lambda$ and $\lambda_0(p) =
T_p\opb{\pi}\pi(p)$, where $\pi\cl T^*N \to N$ is the projection.  Let
$\shi_\Lambda$ be the fiber bundle over $\Lambda$ with fiber the space of
orientation preserving isomorphisms $\Iso^+(\lambda_0(p) \otimes
\Lambda^n\lambda_0(p), \lambda_\Lambda(p)\otimes \Lambda^n\lambda_\Lambda(p))$.
Then $\shi_{\Lambda,p} \simeq GL^+_n$ and we have a canonical morphism
$\varepsilon \cl \pi_1(\shi_{\Lambda,p}) \to \Z/2\Z$.  We let
$\loceps(\cor_\Lambda)$ be the stack whose objects over $\Lambda_0 \subset
\Lambda$ are the local systems on $\shi_{\Lambda_0}$ with monodromy
$\varepsilon$ in the fibers. If $\Lambda_0$ is contractible this is the same as
the local systems on $\Lambda_0$, say $\loc(\cor_{\Lambda_0})$.  The obstruction
to a global isomorphism $\loceps(\cor_\Lambda) \simeq \loc(\cor_{\Lambda})$ is a
relative Stiefel-Whitney class, say $rw_2(\Lambda) \in H^2(\Lambda; \Z/2\Z)$.
We define also $\Dloceps(\cor_\Lambda)$ as the stack associated with the
prestack formed by the $F \in \Derb(\cor_{\shi_{\Lambda_0}})$ with locally
constant cohomology sheaves with monodromy $\varepsilon$.  Then
$\Dloceps(\cor_\Lambda) \simeq \bigoplus_{i\in \Z} \loceps(\cor_\Lambda)[i]$.
We prove
$$
\kss(\cor_\Lambda) \simeq (\bigoplus_{i\in \Z} \loceps(\cor_\Lambda)[i])_m ,
$$
where $m = m(\Lambda) \in H^1(\Lambda;\Z)$ is the Maslov class of $\Lambda$ and
the twist $(\cdot)_m$ means that, representing $m$ by a \v Cech cocycle
$\{c_{ij}\}$, the gluing isomorphisms between local objects are $u_{ji} \cl
L_i|_{\Lambda_{ij}} \isoto L_j|_{\Lambda_{ij}}[c_{ij}]$ instead of $u_{ji} \cl
L_i|_{\Lambda_{ij}} \isoto L_j|_{\Lambda_{ij}}$. In particular if
$rw_2(\Lambda)$ and $m(\Lambda)$ vanish, then $\kss(\cor_\Lambda) \simeq
\bigoplus_{i\in \Z} \loc(\cor_\Lambda)[i]$ has globally defined objects.
Conversely, if $\kss(\cor_\Lambda)$ has a global object, then $m(\Lambda)
=0$. If it has a global ``simple'' object (see Section~\ref{sec:simpsheaves}),
then $rw_2(\Lambda) = 0$.

\subsection*{Convolution}
Now we assume to be given $\shf \in \kss(\cor_\Lambda)$ and we want to construct
a quantization of $\shf$, that is, $F \in \Derb_\Lambda(\cor_N)$ whose image in
$\kss(\cor_\Lambda)$ is $\shf$.  We first find $F \in \Derb_{(\Lambda)}(\cor_N)$
which represents $\shf$ (that is, $\dot\SSi(F)$ coincides with $\Lambda$ only in
a neighborhood of $\Lambda$).  As we already explained, $\shf$ is given by $F_i
\in \Derb_{(\Lambda_i)}(\cor_N)$, $i\in I$, for a covering $\{\Lambda_i\}_{i\in
  I}$ of $\Lambda$, and isomorphisms $u_{ji} \cl F_i|_{\Lambda_{ij}} \isoto
F_j|_{\Lambda_{ij}}$ in $\kss^0_\Lambda(\Lambda_{ij})$ on the intersections
$\Lambda_{ij} = \Lambda_i \cap \Lambda_j$, such that $u_{kj} \circ u_{ji} =
u_{ki}$.  We also recalled that $\hom$ in $\kss(\cor_\Lambda)$ is $\mu hom$.
Hence we can assume $u_{ji} \in H^0(\Lambda_{ij}; \mu hom (F_i,F_j))$.  We want
to glue the $F_i$ into a global object. The first task is of course to represent
the $u_{ji}$ by morphisms in $\Derb(\cor_N)$.  We explain this now.

We assume that our manifold is a product $N = M\times \R$.  We let $(t;\tau)$
be the coordinates on $T^*\R$ and, for an open subset $U\subset M\times \R$, we
set $T^*_{\tau >0}U = (T^*M \times \{(t;\tau);\;\tau>0\}) \cap T^*U$.  We
assume that $\Lambda \subset T^*_{\tau >0}(M\times \R)$.  We denote by
$\Derbtp(\cor_U)$ the full subcategory of $\Derb(\cor_U)$ formed by the $F$
such that $\dot\SSi(F) \subset T^*_{\tau >0}U$.  We will define a functor
$\Psi_U\cl \Derb(\cor_U) \to \Derb(\cor_{U\times \mo]0,+\infty[})$ such that,
setting $\Psi_U^\varepsilon(F) = \Psi_U(F)|_{U \times \mo]0,\varepsilon[}$, we
have
\begin{equation}\label{eq:intro_hompsimuhom}
\varinjlim_{\varepsilon>0} \Hom_{\Derb(\cor_{U\times \mo]0,\varepsilon[})}
(\Psi_U^\varepsilon(F), \Psi_U^\varepsilon(G) )
\simeq H^0(\dT^*U; \mu hom(F,G)) ,  
\end{equation}
for any $F,G \in \Derbtp(\cor_U)$.  Moreover we can describe
$\SSi(\Psi_U(F))$. It gives the following expression for the restrictions
$\Psi_U(F)|_{U \times \{\varepsilon\}}$ for $\varepsilon>0$.  Let
$T'_\varepsilon$ be the translation $T'_\varepsilon(x,t;\xi,\tau) =
(x,t+\varepsilon;\xi,\tau)$.  Then $\dot\SSi((\Psi_U(F))|_{U \times
  \{\varepsilon\}}) = \dot\SSi(F) \cup T'_\varepsilon(\dot\SSi(F))$.

Using $\Psi_U$ we can modify the local representatives of our $\shf \in
\kss(\cor_\Lambda)$.  We assume that each $\Lambda_i$ is of the form $\Lambda
\cap T^*U_i$ for some $U_i \subset M \times\R$.  We set $F_i^\varepsilon =
\Psi_{U_i}(F_i)|_{U_i \times \{\varepsilon\}}$.  Then $F_i^\varepsilon \in
\Derb_{(\Lambda_i)}(\cor_M)$ has the same image as $F_i$ in
$\kss(\cor_{\Lambda_i})$.  But now the formula~\eqref{eq:intro_hompsimuhom}
turns $u_{ji}$ into an isomorphism $v_{ji} \cl F_i^\varepsilon \isoto
F_j^\varepsilon$ on $U_{ij}$ for $\varepsilon>0$ small enough.  The same formula
also gives $\Hom(F_i^\varepsilon,F_i^\varepsilon [d]) = 0$ for all $i\in I$ and
all $d<0$.  It is then possible to glue the $F_i^\varepsilon$ into a global
object $F^\varepsilon \in \Derb(\cor_{M\times\R})$ representing $\shf$.  We have
$\dot\SSi(F^\varepsilon) = \Lambda \cup T'_\varepsilon(\Lambda)$.

\smallskip

Now we give some details on the functor $\Psi_U$. It is introduced to solve the
problem that two non isomorphic sheaves in $\Derb(\cor_{M\times\R})$ may
represent the same object in $\kss(\cor_\Lambda)$. For example we consider
$\Lambda = \{(x,0;0,\tau)$; $\tau >0\}$, $F = \cor_{M\times [0,+\infty[}$ and
$G = \cor_{M\times ]-\infty,0[}[1]$. Then $\dot\SSi(F) = \dot\SSi(G) =
\Lambda$, $F$ and $G$ are isomorphic in $\kss(\cor_\Lambda)$, but not in
$\Derb(\cor_{M\times\R})$.  We recall the convolution functor, which is a
variant of the composition defined in~\cite{KS90} and was already used by
Tamarkin to build canonical representatives of objects of
$\Derb(\cor_{M\times\R}; \{\tau>0\})$.  For $F\in \Derb(\cor_{M\times\R})$ and
$F'\in \Derb(\cor_\R)$ we set $F'\star F \eqdot \reim{s}(F\etens F')$, where
$s\cl M\times \R^2 \to M\times\R$ is the sum $(x,t,t') \mapsto (x,t+t')$.  For
our example we have $\cor_{[0,\varepsilon[} \star F \simeq
\cor_{[0,\varepsilon[} \star G \simeq \cor_{M\times [0,\varepsilon[}$.  This is
a general fact: for $F,G \in \Derbtp(\cor_{M\times\R})$, if $F$ and $G$ are
isomorphic in $\Derb(\cor_{M\times\R}; \dT^*(M\times\R) )$, then
$\cor_{[0,\varepsilon[}\star F \simeq \cor_{[0,\varepsilon[}\star G$.  If $F\in
\Derb(\cor_U)$ for $U\subset M\times\R$, then $\cor_{[0,\varepsilon[} \star F$
is only defined over $U\cap T_\varepsilon(U)$.  Hence we have to consider all
$\varepsilon>0$ at once.  We define $\Psi_U\cl \Derb(\cor_U) \to
\Derb(\cor_{U\times \mo]0,+\infty[})$ by $\Psi_U(F) = \cor_\gamma \star F$,
where $\gamma = \{(t,u)\in \R\times \mo]0,+\infty[$; $0\leq t<u\}$.

\subsection*{Quantization}
In the last two parts of the paper we build a quantization of a closed conic
Lagrangian $\Lambda \subset T^*_{\tau >0}(M\times\R) $ and deduce topological
consequences for $\Lambda$.  We assume that $\Lambda/\R_{>0}$ is compact and
that the map $T^*_{\tau >0}(M\times\R) \to T^*M$, $(x,t;\xi,\tau) \mapsto
(x;\xi/\tau)$ gives an embedding of $\Lambda/\R_{>0}$ in $T^*M$. Then $\Lambda$
can be recovered up to translation from its image in $T^*M$ which is a compact
exact Lagrangian submanifold.

In this situation, for a given $\shf \in \kss(\cor_\Lambda)$ we find $F \in
\Derb_\Lambda(\cor_M)$ whose image in $\kss(\cor_\Lambda)$ is $\shf$.  We have
already seen that there exists $F^\varepsilon \in \Derb_{(\Lambda)}(\cor_M)$
whose image in $\kss(\cor_\Lambda)$ is $\shf$ and such that
$\dot\SSi(F^\varepsilon) = \Lambda \cup T'_\varepsilon(\Lambda)$.  Since
$\Lambda$ arises from a compact exact Lagrangian submanifold of $T^*M$, we can
find a Hamiltonian isotopy $\phi\cl \dT^*(M\times\R) \times \R \to
\dT^*(M\times\R)$ such that, for all $s\geq 0$ and $\varepsilon>0$, we have
$\phi_s(\Lambda) = \Lambda$ and $\phi_s(T'_\varepsilon(\Lambda)) =
T'_{\varepsilon + s}(\Lambda)$.  By the main result of~\cite{GKS10} we can
quantize $\phi$ and compose the resulting kernel with $F^\varepsilon$.  We
obtain $F^{\varepsilon+s} \in \Derb(\cor_{M\times\R})$ such that
$\SSi(F^{\varepsilon+s}) = \Lambda \cup T'_{\varepsilon+s}(\Lambda)$.  For $s$
big enough there exists $a\in \R$ such that $\Lambda \subset T^*(M\times
\mo]-\infty,a[)$ and $T'_{\varepsilon+s}(\Lambda) \subset T^*(M\times
\mo]a,+\infty[)$.  Then $F \eqdot F^{\varepsilon+s}|_{M\times \mo]-\infty,a[}$
is a quantization of $\shf$, up to a choice of a suitable diffeomorphism
$\mo]-\infty,a\mc[ \simeq \R$.

\smallskip

Now, using standard properties of the microsupport, we can prove the following
isomorphism. Let $\shf,\shf' \in \kss(\cor_\Lambda)$ be given and let $F,F'\in
\Derb(\cor_{M\times\R})$ be quantizations of $\shf,\shf'$ obtained by the above
procedure. By construction $F|_{M\times \{t\}} \simeq 0$ for $t \ll 0$. We can
also see that, for $t \gg 0$, $F|_{M\times \{t\}}$ is independent of $t$ and is
locally constant on $M$.  We set $L = F|_{M\times \{t\}}$, $L' = F'|_{M\times
  \{t\}}$, for $t \gg 0$.  Then we have
\begin{equation}\label{eq:intro_homMhomLambda}
\RHom(L,L') \isofrom  \RHom(F,F')
 \isoto \rsect(\Lambda;\hom_{\kss(\cor_\Lambda)}(\shf,\shf') ).
\end{equation}
The $\hom$ sheaf in $\kss(\cor_\Lambda)$ is in fact a local system on
$\Lambda$. Hence we obtain an isomorphism between the cohomology of some local
systems on $M$ and corresponding local systems on $\Lambda$. If we take for $F$
a ``simple sheaf'' and $F'=F$ we have $\RHom(L,L) = \rsect(M;\cor_M)$ and
$\hom_{\kss(\cor_\Lambda)}(\shf,\shf) = \cor_\Lambda$, proving that $M$ and
$\Lambda$ have the same cohomology.

We prove that there exists $F_0\in \Derb_\Lambda(\cor_{M\times\R})$ such that
$F_0|_{M\times \{t\}} \simeq 0$ for $t \ll 0$ and $F_0|_{M\times \{t\}} \simeq
\cor_M$ for $t \gg 0$.  Let $\Derb_{loc}(\cor_M)$ be the full subcategory of
objects $L$ with locally constant cohomology sheaves, or equivalently, such that
$\dot\SSi(L) = \emptyset$.  Then, $L \mapsto F_0 \ltens \opb{p}L$ gives an
equivalence between $\Derb_{loc}(\cor_M)$ and $\Derb_\Lambda(\cor_{M\times\R})$.
On the other hand $F \mapsto \mu hom(F_0,F)$ induces an equivalence between
$\Derb_\Lambda(\cor_{M\times\R})$ and $\Derb_{loc}(\cor_\Lambda)$.  We deduce an
equivalence between local systems on $M$ and local systems on $\Lambda$, proving
that $M$ and $\Lambda$ have the same first homotopy groups.  This also proves
that the chosen $F_0$ is unique. Hence $\Lambda$ has a canonical quantization
$F_0$ such that $F_0|_{M\times \{t\}} \simeq 0$ for $t \ll 0$ and $F_0|_{M\times
  \{t\}} \simeq \cor_M$ for $t \gg 0$.  The final result is stated in
Theorem~\ref{thm:quant_canon} and Corollary~\ref{cor:equiv_homot}.

\subsection*{Triangulated orbit category}

We have explained how we build a quantization of a given global object of
$\kss(\cor_\Lambda)$.  The existence of such an object is equivalent to the
vanishing of the Maslov class and relative Stiefel-Whitney class of $\Lambda$.
If we take $\cor = \Z/2\Z$ the image of the Stiefel-Whitney class in
$H^2(\Lambda;\cor^\times)$ is zero, since the multiplicative group $\cor^\times$
is trivial.  The Maslov class appears in our framework as a shift functor in the
cohomological degrees of the objects of $\Derb(\cor_N)$.

It is possible to quotient $\Derb(\cor_N)$ by the shift functor, that is, define
a category $\Orb(\cor_N)$ with a functor $i_N \cl \Derb(\cor_N) \to
\Orb(\cor_N)$, such that $i_N(F) \simeq i_N(F[1])$ for all $F \in
\Derb(\cor_N)$.  This is a special case of a construction by Keller called the
triangulated hull of the orbit category.

We define $\Orb(\cor_N)$ for categories of sheaves and check that the usual
sheaf operations (direct, inverse images, tensor product and internal $\hom$)
make sense for $\Orb(\cor_N)$. We also check that the microsupport can be
defined for objects of $\Orb(\cor_N)$.  Almost all we have said up to now works
in the orbit categories (except the gluing property, which required a vanishing
$\Hom(F_i,F_i[d])$ for $d<0$, which has no meaning in $\Orb(\cor_N)$).  In
particular we can define a Kashiwara-Schapira stack $\kss^\orb(\cor_\Lambda)$
using $\Orb(\cor_N)$ instead of $\Derb(\cor_N)$.  Since the shift functor in
$\Derb(\cor_N)$ gives the identity functor in $\Orb(\cor_N)$, the Maslov class
gives no obstruction to the existence of a global object in
$\kss^\orb(\cor_\Lambda)$.  We take $\cor = \Z/2\Z$ and the Stiefel-Whitney
class gives no obstruction either. We see indeed that there exists a unique
simple sheaf in $\kss^\orb(\cor_\Lambda)$.

The category $\Orb(\cor_{M\times \R})$ carries of course less information than
$\Derb(\cor_{M\times \R})$ but, for example, the monodromy of locally constant
objects is not lost.  In particular the analog of~\eqref{eq:intro_homMhomLambda}
is enough to prove that the morphism $\pi_1(\Lambda) \to \pi_1(M)$ is injective.
This implies that taking a suitable cover $M' \to M$ we have vanishing of the
Maslov class and we can apply to $M'$ the results found for $M$ (in the usual
categories $\Derb(\cor_\bullet)$). We then obtain a quantization of the
pull-back of $\Lambda$ to $M'$. A quick study of this quantization shows that in
fact the Maslov class is zero.

To show the vanishing of the Stiefel-Whitney class, say $w$, we proceed in a
similar way. We work with $w$-twisted sheaves and obtain a quantization by
$w$-twisted sheaves.  Again a quick study of this quantization shows that the
twist $w$ has to be zero.

\bigskip\noindent
{\bf Acknowledgment.}  The starting point of this paper is a discussion with
Claude Viterbo.  He explained me how to construct a quantization in the sense of
this paper using Floer cohomology and asked whether it was possible to obtain it
with the methods of algebraic analysis.  Masaki Kashiwara gave me the idea to
glue locally defined simple sheaves under the assumption that the Maslov class
vanishes.  Claire Amiot explained me that it was possible to quotient by the
shift functor, hence avoiding this vanishing assumption, in the framework of the
triangulated orbit categories introduced by Keller.  I also thank Pierre
Schapira and Nicolas Vichery for many stimulating discussions.

\part{Sheaves and triangulated orbit categories}

\section{Microlocal theory of sheaves}
\label{section:mts}
In this section, we recall some definitions and results from
\cite{KS90}, following its notations with the exception of  slight
modifications. We consider a real manifold $M$ of class $C^\infty$.

\subsubsection*{Some geometrical notions  (\cite[\S 4.2,~\S 6.2]{KS90})}
For a locally closed subset $A$ of $M$, we denote by $\Int(A)$
its interior and by $\overline{A}$ its closure.
We denote by $\Delta_M$ or simply $\Delta$ the diagonal of $M\times M$. 

We denote by $\pi_M \cl T^*M\to M$ the cotangent bundle of $M$.  If $N\subset
M$ is a submanifold, we denote by $T^*_NM$ its conormal bundle, which
is naturally a fiber bundle over $N$ and a submanifold of $T^*M$. We identify
$M$ with $T^*_MM$, the zero-section of $T^*M$.
We set $\dT^*M = T^*M\setminus T^*_MM$ and we denote by
$\dot\pi_M\cl\dT^*M\to M$ the projection. For any subset $A$ of $T^*M$ we
define its antipodal $A^a = \{(x;\xi) \in T^*M$; $(x;-\xi) \in A\}$.

If $U \subset M$ is an open subset with smooth boundary we define its interior
and exterior conormal bundles $N_U^*, N_U^{*e} \subset T^*M$ by
\begin{equation}\label{eq:def_ext_con_bun}
N_U^* = \{(x;\lambda \, d\phi(x)); \; x \in \partial U, \, \lambda \geq 0\} ,
\qquad
N_U^{*e} = (N_U^*)^a,
\end{equation}
where $\phi\cl M\to \R$ is any $C^1$ function such that $U =
\opb{\phi}(]0,+\infty[)$ and $d\phi$ does not vanish on $\partial U$.

If $E\to M$ is a vector bundle and $A,B$ are subsets of $E$, we denote by $A+B$
the fiberwise sum, that is,
$$
A+B = \{ (x;e); \;  e = e_1+e_2
 \text{ for some $(x;e_1) \in A$, $(x;e_2) \in B$} \}.
$$

Let $f\cl M\to N$ be  a morphism of  real manifolds. It induces morphisms
on the cotangent bundles:
\begin{equation}\label{eq:def_derivee_morph}
T^*M \from[f_d] M\times_N T^*N \to[f_\pi] T^*N.
\end{equation}
Let $N\subset M$ be a submanifold and $A\subset M$ any subset. We denote by
$C_N(A) \subset T_NM$ the cone of $A$ along $N$.  If $M$ is a vector space,
$x_0\in N$ and $q\cl M \to T_{N,x_0}M$ denotes the natural quotient map, then
\begin{equation}\label{eq:form_cone1}
C_N(A) \cap T_{N,x_0}M 
= \bigcap_{U} \, \ol{\bigcup_{x\in A\cap (U\setminus \{x_0\})} q([x_0,x))},
\end{equation}
where $U$ runs over the neighborhoods of $x_0$ and $[x_0,x)$ denotes the half
line starting at $x_0$ and containing $x$.

If $A,B$ are two subsets of $M$, we set $C(A,B) = C_{\Delta_M}(A\times B)$.
Identifying $T_{\Delta_M}(M\times M)$ with $TM$ through the first projection,
we consider $C(A,B)$ as a subset of $TM$. If $M$ is a vector space and
$x_0\in M$, we have
\begin{equation}\label{eq:form_cone2}
C(A,B) \cap T_{x_0}M 
= \bigcap_{U} \, \ol{\bigcup_{x\in A \cap U,\; y\in B \cap U,\; x\not=y} q([y,x))},
\end{equation}
where $U$ runs over the neighborhoods of $x_0$.

The cotangent bundle $T^*M$ carries an exact symplectic structure.  We denote
the Liouville $1$-form by $\alpha_M$. It is given in local coordinates
$(x;\xi)$ by $\alpha_M = \sum_i \xi_i dx_i$.  We denote by
$H\cl T^*T^*M \isoto TT^*M$ the Hamiltonian isomorphism.  We have
$H(dx_i) = -\partial/\partial \xi_i$ and
$H(d\xi_i) = \partial/\partial x_i$. Following~\cite{KS90} we usually identify
$T^*T^*M$ and $TT^*M$ by $-H$.

\subsubsection*{Microsupport}
We consider a commutative unital ring $\cor$ of finite global dimension (we will
use $\cor=\Z$ or $\cor = \Z/2\Z$).  We denote by $\Mod(\cor)$ the category of
$\cor$-modules and by $\Mod(\cor_M)$ the category of sheaves of $\cor$-modules
on $M$.  We denote by $\Der(\cor_M)$ (resp.\ $\Derb(\cor_M)$) the derived
category (resp.\ bounded derived category) of $\Mod(\cor_M)$.

We recall the following notations, for the inclusion $j \cl Z \to M$ of a
locally closed subset of $M$ and for $F\in \Derb(\cor_M)$,
$$
F_Z  = \eim{j}\opb{j}F , \qquad \rsect_Z(F) = \roim{j} \epb{j}F.
$$
When $F= \cor_M$ is the constant sheaf we set for short $\cor_Z = (\cor_M)_Z =
\eim{j}(\cor_Z)$. We have
$$
F_Z  \simeq F \tens \cor_Z, \qquad \rsect_Z(F) \simeq \rhom(\cor_Z,F) .
$$

We recall the definition of the microsupport (or singular support) $\SSi(F)$ of
a sheaf $F$, introduced by M.~Kashiwara and P.~Schapira in~\cite{KS82}
and~\cite{KS85}.

\begin{definition}{\rm (see~\cite[Def.~5.1.2]{KS90})}
Let $F\in \Derb(\cor_M)$. We define $\SSi(F) \subset T^*M$ as the closure of the
set of points $(x_0;\xi_0) \in T^*M$ such that there exists a real
$C^1$-function $\phi$ on $M$ satisfying $d\phi(x_0) = \xi_0$ and
$(\rsect_{\{x;\, \phi(x)\geq \phi(x_0)\}} (F))_{x_0} \not\simeq 0$.

We set $\dot\SSi(F) = \SSi(F) \cap \dT^*M$.
\end{definition}
In other words, $p\notin\SSi(F)$ if the sheaf $F$ has no cohomology 
supported by ``half-spaces'' whose conormals are contained in a 
neighborhood of $p$. The following properties are easy consequences of
the definition: \\
(a) the microsupport is closed and $\R^+$-conic, that is,
invariant by the action of  $(\R^+,\times)$ on $T^*M$, \\
(b) $\SSi(F)\cap T^*_MM =\pi_M(\SSi(F))=\supp(F)$, \\
(c) the microsupport satisfies the triangular inequality:
if $F_1\to F_2\to F_3\to[+1]$ is a
distinguished triangle in  $\Derb(\cor_M)$, then 
$\SSi(F_i)\subset\SSi(F_j)\cup\SSi(F_k)$ for all $i,j,k\in\{1,2,3\}$
with $j\not=k$.

\begin{example}\label{ex:microsupport}
(i) If $F$ is a non-zero local system on a connected manifold $M$,
then $\SSi(F)=T^*_MM$, the zero-section. Conversely,
if $\SSi(F)\subset T^*_MM$, then the cohomology sheaves $H^i(F)$
are local systems, for all $i\in \Z$.

\smallskip\noindent
(ii) If $N$ is a smooth closed submanifold of $M$ and $F=\cor_N$, then 
$\SSi(F)=T^*_NM$.

\smallskip\noindent
(iii) Let $U \subset M$ be an open subset with smooth boundary. Then 
\begin{align*}
\SSi(\cor_U) & = (U\times_MT^*_MM) \cup N_U^{*e},\\
\SSi(\cor_{\ol U}) & = (U\times_MT^*_MM) \cup N_U^* .
\end{align*}
(iv) Let $\lambda$ be a closed convex cone with vertex at $0$ in $E=\R^n$.  Then
$\SSi(\cor_\lambda) \cap T^*_0\R^n = \lambda^\circ$, where $\lambda^\circ$ is
the polar cone of $\lambda$ and is defined by $\lambda^\circ = \{\xi\in E^*$; $\langle
v,\xi \rangle \geq 0$ for all $v\in E\}$.
\end{example}

\begin{notation}\label{not:micro_categories}
For a subset $S$ of $\dT^*M$ we denote by $\Derb_S(\cor_M)$ the full
triangulated subcategory of $\Derb(\cor_M)$ of the $F$ such that $\dot\SSi(F)
\subset S$.  We denote by $\Derb(\cor_M;S)$ the quotient of $\Derb(\cor_M)$ by
$\Derb_{\dT^*M\setminus S}(\cor_M)$.  If $p\in \dT^*M$, we write
$\Derb(\cor_M;p)$ for $\Derb(\cor_M;\{p\})$.  We denote by $\Derb_{(S)}(\cor_M)$
the full triangulated subcategory of $\Derb(\cor_M)$ of the $F$ for which there
exists a neighborhood $\Omega$ of $S$ in $T^*M$ such that $\SSi(F) \cap
\Omega\subset S$.
\end{notation}

Our notations differ slightly from those of~\cite{KS90} where $\Derb_Z(\cor_M)$,
for a given $Z \subset T^*M$, consists of the $F$ such that $\SSi(F) \subset Z$.
Hence our $\Derb_S(\cor_M)$ is the same as $\Derb_{S \cup T^*_MM}(\cor_M)$
in~\cite{KS90}.

\subsubsection*{Functorial operations}
Let $M$ and $N$ be two real manifolds. We denote by $q_i$ ($i=1,2$) the $i$-th
projection defined on $M\times N$ and by $p_i$ ($i=1,2$) the $i$-th projection
defined on $T^*(M\times N)\simeq T^*M\times T^*N$.

\begin{definition}
Let $f\cl M\to N$ be a morphism of manifolds and let $\Lambda\subset T^*N$
be a closed $\R^+$-conic subset. We say that $f$ is non-characteristic for
$\Lambda$ if $\opb{f_\pi}(\Lambda)\cap T^*_MN\subset M\times_NT^*_NN$.
\end{definition}
A morphism $f\cl M\to N$ is non-characteristic for a closed $\R^+$-conic subset
$\Lambda$ of $T^*N$ if and only if $\opb{f_d}(T^*_MM) \cap \opb{f_\pi}(\Lambda)
\subset M\times_NT^*_NN$.  It is equivalent to ask that $f_d\cl M\times_NT^*N\to
T^*M$ is proper on $\opb{f_\pi}(\Lambda)$ and in this case
$f_d\opb{f_\pi}(\Lambda)$ is closed and $\R^+$-conic in $T^*M$.

We denote by $\omega_M$ the dualizing complex on $M$.  Recall that $\omega_M$
is isomorphic to the orientation sheaf shifted by the dimension. We also use
the notation $\omega_{M/N}$ for the relative dualizing complex
$\omega_M\tens\opb{f}\omega_N^{\tens-1}$.  We have the duality functors
\begin{equation}\label{eq:dualfct}
\DD_M(\scbul)=\rhom(\scbul,\omega_M), \qquad
\DD'_M(\scbul)=\rhom(\scbul,\cor_M).
\end{equation}
For two manifolds $M,N$ and $F\in\Derb(\cor_M)$, $G\in\Derb(\cor_N)$ we define
$F \letens G \in \Derb(\cor_{M\times N})$ by
$$
F \letens G = \opb{q_1}F \ltens \opb{q_2}G .
$$

\begin{theorem}
\label{th:opboim}{\rm(See \cite[\S 5.4]{KS90}.)}
Let $f\cl M\to N$ be a morphism of manifolds, $F\in\Derb(\cor_M)$ and
$G\in\Derb(\cor_N)$.  Let $q_1\cl M\times N \to M$ and $q_2\cl M\times N \to N$
be the projections.
\begin{itemize}
\item [(i)] We have
\begin{gather*}
\SSi(F \letens G) \subset \SSi(F)\times\SSi(G),  \\
\SSi(\rhom(\opb{q_1}F,\opb{q_2}G))  \subset \SSi(F)^a\times\SSi(G).
\end{gather*}
\item [(ii)] We assume that $f$ is proper on $\supp(F)$. Then
  $\SSi(\reim{f}F)\subset f_\pi\opb{f_d}\SSi(F)$, with equality when $f$ is a
  closed embedding.
\item [(iii)] We assume that $f$ is non-characteristic with respect to
  $\SSi(G)$. Then the natural morphism $\opb{f}G \tens
  \omega_{M/N}\to\epb{f}(G)$ is an isomorphism. Moreover $\SSi(\opb{f}G) \cup
  \SSi(\epb{f}G) \subset f_d\opb{f_\pi}\SSi(G)$.  If $f$ is smooth, this
  inclusion is an equality.
\item [(iv)] We assume that $M=N\times I$, where $I$ is a contractible manifold,
  and that $f$ is the projection. Then $\SSi(F)\subset T^*_II\times T^*N$ if and
  only if $\opb{f} \roim{f} (F) \isoto F$.
\end{itemize}
\end{theorem}
\noindent
For the definition of \emph{cohomologically constructible} we refer
to~\cite[\S 3.4]{KS90}.

\begin{corollary}\label{cor:opboim}
Let $F,G\in\Derb(\cor_M)$. 
\begin{itemize}
\item [(i)] We assume that $\SSi(F)\cap\SSi(G)^a\subset T^*_MM$. Then
  $\SSi(F\ltens G)\subset \SSi(F)+\SSi(G)$.
\item [(ii)] We assume that $\SSi(F)\cap\SSi(G)\subset T^*_MM$. Then \\
  $\SSi(\rhom(F,G))\subset \SSi(F)^a+\SSi(G)$.  Moreover, assuming that
  $F$ is cohomologically constructible, the natural morphism $\DD'F \ltens
  G \to \rhom(F,G)$ is an isomorphism.
\end{itemize}
\end{corollary}
The next result follows immediately from Theorem~\ref{th:opboim}~(ii) and
Example~\ref{ex:microsupport}~(i). It is a particular case of the microlocal
Morse lemma (see~\cite[Cor.~5.4.19]{KS90}), the classical theory corresponding
to the constant sheaf $F=\cor_M$.
\begin{corollary}\label{cor:Morse}
Let $F\in\Derb(\cor_M)$, let $\phi\cl M\to\R$ be a function of class $C^1$ and
assume that $\phi$ is proper on $\supp(F)$.  Let $a<b$ in $\R$ and assume that
$d\phi(x)\notin\SSi(F)$ for $a\leq \phi(x)<b$. Then the natural morphisms
$\rsect(\opb{\phi}(\mo]-\infty,b[);F) \to \rsect(\opb{\phi}(\mo]-\infty,a[);F)$
and
$\rsect_{\opb{\phi}([b,+\infty[)}(M;F) \to \rsect_{\opb{\phi}([a,+\infty[)}(M;F)$
are isomorphisms.
\end{corollary}

Here we only explained the proper and non-cha\-rac\-teristic cases but we will
also need more general results.  We recall some notations
of~\cite[Def.~6.2.3]{KS90}. Let $i\cl M \to N$ be an embedding.  We let $\tau
\cl T^*_MN \to M$ be the projection and $\tau_d \cl T^*_MN \times_M T^*M \to
T^*(T^*_MN)$ the map defined in~\eqref{eq:def_derivee_morph}.  Since $T^*_MN$ is
a Lagrangian submanifold of $T^*N$, the Hamiltonian isomorphism induces $h \cl
T^*(T^*_MN) \isoto T_{T^*_MN}T^*N$.  We let $j \cl T^*M = M \times_M T^*M \to
T^*_MN \times_M T^*M$ be the inclusion of the zero section (of a vector bundle
over $T^*M$) and we define the inclusion $i' \cl T^*M \to T_{T^*_MN}T^*N$ as the
composition $i' = h \circ \tau_d \circ j$.  For closed conic subsets $A,B
\subset T^*N$ we set
\begin{align*}
i^\sharp(A) &= (i')^{-1}(C_{T^*_MN}(A)) ,  \\
A \hplus B &= \delta_N^\sharp(A\times B) .
\end{align*}
In local coordinates $A \hplus B$ is the set of $(x;\xi)$ such that there exist
two sequences $(x_n;\xi_n)$ in $A$ and $(y_n;\eta_n)$ in $B$ such that $x_n, y_n
\to x$, $\xi_n+\eta_n \to \xi$ and $|x_n-y_n| |\xi_n| \to 0$ when $n\to \infty$.

\begin{theorem}{\rm (See~\cite[Cor.~6.4.4, 6.4.5]{KS90}.)}
\label{thm:SSrhom}
Let $i\cl M \to N$ be an embedding and $F,G \in\Derb(\cor_N)$. Then
\begin{gather*}
\SSi(\opb{i}F) \subset i^\sharp(\SSi(F)), \\
\SSi(\rhom(F,G))\subset \SSi(F)^a \hplus \SSi(G) .
\end{gather*} 
\end{theorem}

\subsubsection*{Composition and convolution}
We will use two similar operations on sheaves, the composition $\circ$ and the
convolution $\star$, defined as follows. Let $M_i$, $i=1,2,3$, be manifolds and
let $q_{ij} \cl M_1\times M_2 \times M_3 \to M_i\times M_j$ be the projections,
for $1\leq i<j \leq 3$.  We define a functor
\begin{equation}\label{eq:def_comp_gene}
  \begin{split}
\Derb(\cor_{M_1\times M_2}) \times \Derb(\cor_{M_2\times M_3})
&\to \Derb(\cor_{M_1\times M_3})  \\
(K,L) &\mapsto K \circ L \eqdot
\reim{q_{13}}(\opb{q_{12}} K \ltens \opb{q_{23}} L) .
  \end{split}
\end{equation}
Let $M$ be a manifold and let $V$ be a vector space. We let $s\cl M \times V^2
\to M \times V$ be the sum, $(x,v_1,v_2) \mapsto (x,v_1+v_2)$.  In this
situation we define two functors
\begin{equation}\label{eq:def_conv_gene}
  \begin{split}
\Derb(\cor_{V}) \times \Derb(\cor_{M \times V})   &\to \Derb(\cor_{M\times V}), \\
(K,F) &\mapsto  K \star F \eqdot  \reim{s}(F \letens K) , \\
(K,F) &\mapsto  K \star' F \eqdot  \roim{s}(F \letens K) .
  \end{split}
\end{equation}

\subsubsection*{Cut-off}
We recall two ``cut-off'' results of~\cite{KS90}.  Let $M$ be a manifold,
$E=\R^d$ a vector space and $\gamma \subset E$ a closed convex cone.  Since
$0\in \gamma$ we have a morphism $\cor_\gamma \to \cor_{\{0\}}$, which induces
\begin{equation}\label{eq:morph_convgamma_id}
\cor_\gamma \star' F \to \cor_{\{0\}} \star' F \simeq F ,
\end{equation}
for any $F \in \Derb(\cor_{M\times E})$.

\begin{proposition}[``Microlocal cut-off lemma'' Prop.~5.2.3 of~\cite{KS90}]
\label{prop:cutoff}
Let $F \in \Derb(\cor_{M\times E})$.
Let $F'$ be the cone of the morphism~\eqref{eq:morph_convgamma_id}.
Then we have $\SSi(F') \cap (T^*M \times (E \times \Int(\gamma^{\circ}))) =
\emptyset$ and~\eqref{eq:morph_convgamma_id} is an isomorphism if and only if
$\SSi(F) \subset T^*M \times (E \times \gamma^{\circ})$.
\end{proposition}

In the situation of the proposition let us set $\Omega = T^*M \times (E \times
\Int(\gamma^{\circ}))$.  By Theorem~\ref{th:opboim} and
Example~\ref{ex:microsupport}~(iv), we have $\SSi(\cor_\gamma \star' F) \subset
\ol{\Omega}$.  Since the first part of the proposition implies
$\SSi(\cor_\gamma \star' F) \cap \Omega = \SSi(F) \cap \Omega$ we have
$\SSi(\cor_\gamma \star' F) = (\SSi(F) \cap \Omega ) \cup W$, where $W$ is a
subset contained in $\partial \Omega$.  The next proposition gives a local
version of this result with some control on $W$.

\begin{proposition}[Prop.~6.1.4 of \cite{KS90}]
\label{prop:cutoff-precis}
Let $N$ be a manifold and $x_0 \in N$. Let $C_0 \subset T^*_{x_0}N$ be an open
convex cone such that $\ol{C_0}$ is proper and let $\Omega \subset T^*N$ be a
conic neighborhood of $\ol{C_0} \setminus \{x_0\}$.  Let $S \subset \Omega$ be a
closed conic subset and let $W_0 \subset T^*_{x_0}N$ be a conic neighborhood of
$(\ol{S}\cap T^*_{x_0}N)\setminus \{x_0\}$.
Then there exist a neighborhood $U$ of $x_0$, $L,L' \in \Derb(\cor_{U\times
  N})$ and a distinguished triangle $L \to \cor_{\Delta_N} \to L' \to[+1]$ in
$\Derb(\cor_{U\times N})$ such that
\begin{itemize}
\item [(i)] for any $F\in \Derb(\cor_N)$ we have
  $\SSi(L' \circ F) \cap C_0 = \emptyset$,
\item [(ii)] for any $F\in \Derb(\cor_N)$ such that $\SSi(F) \cap \Omega \subset
  S$ we have
  \begin{align}
\label{eq:cutoff-precis0}
 \dot\SSi(L \circ F) \cap T^*_{x_0}N & \subset W_0 , \\
\label{eq:cutoff-precis1}
  \dot\SSi(L' \circ F) \cap T^*_{x_0}N &\subset 
((\SSi(F)\cap T^*_{x_0}N) \cup W_0 ) \setminus C_0 ,
  \end{align}
\item [(iii)] for any $F\in \Derb(\cor_N)$ such that $\SSi(F) \cap \Omega =
  \emptyset$ we have $L \circ F \simeq 0$.
\end{itemize}
\end{proposition}

This proposition is Proposition~6.1.4 of~\cite{KS90} although the statement
given in~\cite{KS90} only gives~(i) and~\eqref{eq:cutoff-precis0} for some $F'$
instead of $L\circ F$. However we can see in the proof that $F'$ is actually
given by composition with some $L$ as in the above proposition.
Then~\eqref{eq:cutoff-precis1} follows from~(i)
and~\eqref{eq:cutoff-precis0}. Finally~(iii) can be checked directly.

Let $E=\R^n$ be a vector space endowed with a norm $\|\cdot\|$.  We consider
coordinates $(x,t;\xi,\tau)$ on $T^*(E\times\R)$. For a given $a>0$ we let
$\gamma_a \subset E\times\R$ be the closed cone $\gamma_a = \{(x,t);\; t\geq a
\|x\| \}$.  Its polar cone is $\gamma_a^\circ =\{(\xi,\tau);\; \tau \geq
\|\xi\|/a\}$.  Let $t_0 \in \R$ be given. We set $U = E \times
\mo]t_0,+\infty[$. We choose $y \in U$ and we set $C = \Int(y-\gamma_a)$.

\begin{lemma}\label{lem:easycutoff}
Let $F \in \Derb(\cor_E)$ be such that $\SSi(F) \cap ((C\cap \ol{U}) \times
(\gamma_a^\circ\setminus \{0\})) = \emptyset$.  Then
$(\cor_{\gamma_a} \star (F_U)) |_C \simeq 0$.
\end{lemma}
\begin{proof}
(i) We have $\SSi(\cor_U) \subset E \times \gamma_a^\circ$. Hence we can apply
Corollary~\ref{cor:opboim} to bound $\SSi(F \tens \cor_U)$ and we obtain
$\SSi(F_U) \cap ((C\cap \ol{U}) \times (\gamma_a^\circ\setminus \{0\})) =
\emptyset$.  Since $F_U|_{E\setminus U} \simeq 0$ we have $\SSi(F_U) \cap (C
\times (\gamma_a^\circ\setminus \{0\})) = \emptyset$.

\medskip\noindent
(ii) Let us set $G = \cor_{\gamma_a} \star (F_U)$.  The map $s$
of~\eqref{eq:def_conv_gene} is proper on $\ol{U} \times \gamma$ and we can use
Theorem~\ref{th:opboim} to bound $\SSi(G)$.  Since $\opb{s}(C) \cap E \times
\gamma \subset C \times \gamma$ it is enough to know $\SSi(F_U|_C)$ to bound
$\SSi(G|_C)$. By~(i) and Example~\ref{ex:microsupport}~(iv) we obtain
$\SSi(G|_C) \subset T^*_CC$, that is, $G$ is locally constant on $C$. We see
easily that $\supp(G) \subset \ol{U}$ does not contain $C$. It follows that
$G|_C \simeq 0$.
\end{proof}

\subsubsection*{Microlocalization}
Let $N$ be a submanifold of $M$.  Sato's microlocalization
is a functor $\mu_N\cl \Derb(\cor_M) \to \Derb(\cor_{T^*_NM})$.
We refer to~\cite{KS90} for the definition and the main properties.
When $N$ is closed in $M$ we write $\mu_N$ for $\reim{j}\mu_N$, where $j$ is
the embedding of $T^*_NM$ in $T^*M$.
We let $i_M\cl M \to T^*M$ be the inclusion of the zero section.
Then, for any $F\in \Derb(\cor_M)$ we have
\begin{align}
\label{eq:proj_microltion1}
&\roim{\pi_M} \mu_N(F) \simeq \opb{i_M} \mu_N(F) \simeq  \rsect_N(F),  \\
\label{eq:proj_microltion2}
&\reim{\pi_M} \mu_N(F) \simeq \epb{i_M} \mu_N(F) \simeq  F \tens \omega_{N|M} ,
\end{align}
and we deduce Sato's distinguished triangle (triangle~(4.3.1) in~\cite{KS90}):
\begin{equation}\label{eq:SatoDT}
F \tens \omega_{N|M} \to \rsect_N(F)
\to \roim{\dot\pi_M{}} (\mu_N(F)|_{\dT^*M}) \to[+1].
\end{equation}

We recall the definition of the bifunctor $\mu hom$, which is a variant of the
microlocalization introduced in~\cite{KS90}.  Let 
$\Delta_M$ the diagonal of $M\times M$. Let $q_1,q_2\cl M\times M \to M$ be the
projections.  We identify $T^*_{\Delta_M}(M\times M)$ with $T^*M$ through the
first projection.  For $F,G\in \Derb(\cor_M)$ we have 
\begin{equation}\label{eq:def_muhom}
\mu hom (F,G) = \mu_{\Delta_M}(\rhom(\opb{q_2}F,\epb{q_1}G))
\quad \in \quad \Derb(\cor_{T^*M}). 
\end{equation}
For a submanifold $N$ of $M$ we have $\mu_N(G) \simeq \mu hom(\cor_N,G)$, for
any $G\in \Derb(\cor_M)$. The formulas~\eqref{eq:proj_microltion1}
and~\eqref{eq:proj_microltion2} give
\begin{align}
\label{eq:proj_muhom_oim}
\roim{\pi_M} \mu hom (F,G) &\simeq \rhom (F,G)  , \\
\label{eq:proj_muhom}
\reim{\pi_M} \mu hom (F,G) &\simeq 
\opb{\delta_M} \rhom(\opb{q_2}F,\opb{q_1}G) ,
\end{align}
where $\delta_M\cl M \to M\times M$ is the diagonal embedding.
If $F$ is cohomologically constructible, then
$\opb{\delta_M} \rhom(\opb{q_2}F,\opb{q_1}G) \simeq \DD'(F) \ltens G$
and Sato's distinguished triangle gives
\begin{equation}\label{eq:SatoDTmuhom}
\DD'(F) \ltens G  \to \rhom (F,G) 
\to \roim{\dot\pi_M{}} (\mu hom (F,G) |_{\dT^*M}) \to[+1].
\end{equation}

\begin{proposition}{\rm(Cor.~6.4.3 of~\cite{KS90}.)}
\label{prop:SSmuhom}
Let $F,G\in \Derb(\cor_M)$. Then
\begin{gather}
\label{eq:suppmuhom}
\supp \mu hom (F,G) \subset \SSi(F)\cap \SSi(G), \\
\label{eq:SSmuhom}
\SSi(\mu hom(F,G) ) \subset -H^{-1}(C(\SSi(G) , \SSi(F) )) .
\end{gather}
\end{proposition}

\subsubsection*{Homotopy colimit}

There is no projective or inductive limit in the derived category.  However it
is possible to define a notion of homotopy colimit, which is well defined up to
{\em non unique} isomorphism.  We will often consider variations on the
following case.
Let $M$ be a manifold and $\{U_n\}_{n\in \N} \subset M$ a sequence of open
subsets of $M$ such that $U_n \subset U_{n+1}$ for all $n\in \N$ and
$M=\bigcup_{n\in \N} U_n$. Then we can write $\cor_M$ as the homotopy colimit of
the $\cor_{U_n}$'s, which means that we have a distinguished triangle in
$\Derb(\cor_M)$
\begin{equation}\label{eq:homot-lim}
  \bigoplus_{n\in \N} \cor_{U_n} \to[i] \bigoplus_{n\in \N} \cor_{U_n}
 \to \cor_M \to[+1],
\end{equation}
where the $n^{th}$-component of $i$ is $(\id - i_n)$ and $i_n\cl \cor_{U_n} \to
\cor_{U_{n+1}}$ is the natural morphism.  This triangle can be used for example
as follows. Let $f\cl M \to N$ be a morphism of manifolds.  The functor
$\reim{f}$ commutes with direct sums. Taking the tensor product
of~\eqref{eq:homot-lim} with any $F\in \Derb(\cor_M)$ and applying $\reim{f}$
gives the distinguished triangle
\begin{equation}\label{eq:homot-lim-bis}
\bigoplus_{n\in \N} \reim{f}(F_{U_n}) \to \bigoplus_{n\in \N} \reim{f}(F_{U_n}) 
\to \reim{f}(F)  \to[+1],
\end{equation}
which shows how to recover $\reim{f}(F)$ from the $\reim{f}(F_{U_n})$'s up to
isomorphism.

\section{Triangulated orbit categories}
\label{sec:def_orb_cat}

We will use a very special case of the triangulated hull of an orbit category
as described by Keller in~\cite{Ke05}.  More precisely
Definition~\ref{def:categorie-orb} below is inspired by \S7 of~\cite{Ke05} that
we apply to the simple case where we quotient $\Derb(\cor_M)$ by the
autoequivalence $F \mapsto F[1]$ (in~\cite{Ke05} much more general equivalences
are considered). However we apply this construction for sheaves instead of
modules over an algebra.  We use the name triangulated orbit category for the
category $\Orb(\cor_M)$ introduced in Definition~\ref{def:categorie-orb} by
analogy with the categories introduced by Keller.  Actually we only show that
we have a functor $\iota_M \cl \Derb(\cor_M) \to \Orb(\cor_M)$ such that
$\iota_M(F) \simeq \iota_M(F)[1]$ and which satisfies
Corollary~\ref{cor:morph-QRF-QRG} below (the proof of this result is also
inspired by~\cite{Ke05}).

In this section we set $\cor = \Z/2\Z$ and $\Cor = \cor[X]/\langle X^2 \rangle$
(we have to be careful in applying results of~\cite{KS90} because $\Cor$ is not
of finite global dimension). We let $\varepsilon$ be the image of $X$ in
$\Cor$. Hence $\Cor = \cor[\varepsilon]$ with $\varepsilon^2=0$.  Let $M$ be a
manifold.  The natural ring morphisms $\cor \to \Cor$ and $\Cor\to \cor$ induce
two pairs of adjoint functors $(e_M,r_M)$ and $(E_M,R_M)$, where $e_M,E_M$ are
scalar extensions and $r_M,R_M$ restrictions of scalars:
\begin{alignat*}{3}
&\Derb(\cor_M) \overset{e_M}{\underset{r_M}{\rightleftarrows}} \Derb(\Cor_M),
&\qquad e_M(F) &= \Cor_M \tens_{\cor_M} F,
&\quad r_M(G) &= G, \\
&\Derb(\Cor_M) \xfrom[R_M]{} \Derb(\cor_M),
&& & R_M(F) &= F , \\
&\Der^-(\Cor_M) \overset{E_M}{\underset{R_M}{\rightleftarrows}} \Der^-(\cor_M),
& E_M(G) &= \cor_M \ltens_{\Cor_M} G,
& R_M(F) &= F . 
\end{alignat*}

Choosing an isomorphism $\rhom_{\cor_M}(\Cor_M,\cor_M) \simeq \Cor_M$ of
$\Cor_M$-modules we have a canonical isomorphism, for all $F\in \Derb(\cor_M)$,
\begin{equation}\label{eq:eMtorhom}
\Cor_M \tens_{\cor_M} F \simeq \rhom_{\cor_M}(\Cor_M,F).
\end{equation}
It follows that we also have an adjunction $(r_M,e_M)$, that is,
\begin{align}
\notag
\Hom_{\Derb(\Cor_M)}(G,e_M(F))
& \simeq \Hom_{\Derb(\Cor_M)}(G, \rhom_{\cor_M}(\Cor_M,F))  \\
\label{eq:adjrMeM}
& \simeq \Hom_{\Derb(\cor_M)}( G\tens_{\Cor_M} \Cor_M ,F)  \\
\notag
& \simeq \Hom_{\Derb(\cor_M)}(r_M(G),F)  .
\end{align}

\begin{definition}\label{def:categorie-orb}
We let $\perf(\Cor_M)$ be the full triangulated subcategory of $\Derb(\Cor_M)$
generated by the image of $e_M$, that is, by the objects of the form
$\Cor_M \tens_{\cor_M} F$ with $F\in \Derb(\cor_M)$.

We denote by $\Orb(\cor_M)$ the quotient $\Derb(\Cor_M)/\perf(\Cor_M)$.
We let $Q_M\cl \Derb(\Cor_M) \to \Orb(\cor_M)$ be the quotient functor
and we set $\iota_M = Q_M \circ R_M \cl \Derb(\cor_M) \to \Orb(\cor_M)$.
\end{definition}

Let $\perf'(\Cor_M)$ be the subcategory of $\Derb(\Cor_M)$ formed by the $P$
such that $Q_M(P) \simeq 0$. Then $\perf(\Cor_M) \subset \perf'(\Cor_M)$ and
$\Orb(\cor_M) \isoto \Derb(\Cor_M)/\perf'(\Cor_M)$. We do not know whether
$\perf'(\Cor_M) = \perf(\Cor_M)$. A general result (see~\cite{KS06} Ex.~10.11)
says that $P\in \perf'(\Cor_M)$ if and only if $P\oplus P[1] \in \perf(\Cor_M)$.

\begin{notation}
If the context is clear, we will not write the functor $Q_M$ or $R_M$, that is,
for $F\in \Derb(\cor_M)$ we often write $F$ instead of $R_M(F)$, and, for $G\in
\Derb(\Cor_M)$, we often write $G$ instead of $Q_M(G)$.  In particular for a
locally closed subset $Z\subset M$, we consider $\cor_Z = R_M(\cor_Z) \in
\Derb(\Cor_M)$ and $\cor_Z = Q_M(\cor_Z) \in \Orb(\cor_M)$.
\end{notation}

The exact sequence of $\Cor$-modules  $0\to \cor \to \Cor \to \cor \to 0$
induces a morphism
\begin{equation}\label{eq:def_shift_morph0}
s_M \cl \cor_M \to \cor_M[1] \qquad \qquad \text{in $\Derb(\Cor_M)$}
\end{equation}
and a distinguished triangle, for any $F\in \Derb(\cor_M)$,
\begin{equation}\label{eq:td_shift_morph0}
R_M(F) \to e_M(F)  \to  R_M(F) \to[s_M \tens \id_F] R_M(F)[1] .
\end{equation}
We thus obtain an isomorphism $s_M \tens \id_F \cl R_M(F) \isoto R_M(F)[1]$ in
$\Orb(\cor_M)$, for any $F\in \Derb(\cor_M)$.  This would work for any field
$\cor$. In characteristic $2$ we can generalize this isomorphism to any $F\in
\Derb(\Cor_M)$ (see Remark~\ref{rem:def_shift_morph2}).

\subsection*{Internal tensor product and homomorphism}

For two $\Cor$-modules $E_1,E_2$ we define $E_1 \epstens_\cor E_2 \in
\Mod(\Cor)$ as follows. The underlying $\cor$-vector space is $E_1 \otimes_\cor
E_2$ and $\varepsilon$ acts by
$$
\varepsilon \cdot(x\otimes y)
 = (\varepsilon x) \otimes y + x \otimes (\varepsilon y) .
$$
Since the characteristic is $2$, we can check that $\varepsilon^2$ acts by $0$
and this defines an object of $\Mod(\Cor)$ that we denote $E_1 \epstens_\cor
E_2$. We obtain in this way a bifunctor $\epstens_\cor \cl \Mod(\Cor) \times
\Mod(\Cor) \to \Mod(\Cor)$.  For $F_1,F_2 \in \Mod(\Cor_M)$, we define $F_1
\epstens_{\cor_M} F_2 \in \Mod(\Cor_M)$ as the sheaf associated with the
presheaf $U \mapsto F_1(U) \epstens_\cor F_2(U)$.  We remark that $r_M(F_1
\epstens_{\cor_M} F_2) \simeq r_M(F_1) \tens_{\cor_M} r_M(F_2)$, where $r_M$ is
seen here as functor $\Mod(\Cor_M) \to \Mod(\cor_M)$, and it follows easily that
$\epstens_{\cor_M}$ is an exact functor. We denote its derived functor in the
same way:
\begin{equation}\label{eq:def-epstens}
\epstens_{\cor_M} \cl \Derb(\Cor_M) \times \Derb(\Cor_M) \to \Derb(\Cor_M) .  
\end{equation}
For any $F,G \in \Derb(\Cor_M)$ we have canonical isomorphisms
\begin{gather}
\label{eq:unit-epstens}
\cor_M \epstens_{\cor_M} F \simeq F \epstens_{\cor_M} \cor_M \simeq F
\quad \text{ in $\Derb(\Cor_M)$,}  \\
\label{eq:tens-epstens}
  r_M(F  \epstens_{\cor_M} G) \simeq r_M(F) \tens_{\cor_M} r_M(G) 
\quad \text{ in $\Derb(\cor_M)$.}
\end{gather}
Using~\eqref{eq:unit-epstens} and the exact sequence $0\to \cor \to \Cor \to
\cor \to 0$, we obtain as
in~(\ref{eq:def_shift_morph0}-\ref{eq:td_shift_morph0}) a morphism $s_M(F) \cl F
\to F[1]$, for any $F\in \Derb(\Cor_M)$, and a distinguished triangle
\begin{equation}\label{eq:td_shift_morph}
F \to  \Cor_M \epstens_{\cor_M} F  \to  F \to[s_M(F)] F[1] .
\end{equation}
Using the adjunction $(e_M,r_M)$ and~\eqref{eq:tens-epstens} we have the
isomorphism, for any $F\in \Derb(\cor_M)$ and $G \in \Derb(\Cor_M)$,
\begin{equation}\label{eq:eM-epstens}
  \begin{split}
\Hom&_{\Derb(\Cor_M)} ( e_M(F \tens_{\cor_M} r_M(G)) , e_M(F)\epstens_{\cor_M} G)  \\
& \simeq
\Hom_{\Derb(\cor_M)}( F \tens_{\cor_M} r_M(G) , (r_Me_M(F))\epstens_{\cor_M} r_M(G) ) .
  \end{split}
\end{equation}
By adjunction we have a morphism $a_F \cl F \to r_Me_M(F)$.  The inverse image
of $a_F \otimes \id_{r_M(G)}$ by~\eqref{eq:eM-epstens} gives a canonical
morphism, for any $F\in \Derb(\cor_M)$ and $G \in \Derb(\Cor_M)$,
\begin{equation}\label{eq:eM-epstens2}
e_M(F \tens_{\cor_M} r_M(G)) \to  e_M(F)\epstens_{\cor_M} G .
\end{equation}

\begin{lemma}\label{lem:epstens-perf}
Let $F\in \Derb(\cor_M)$ and $G \in \Derb(\Cor_M)$.  Then the
morphism~\eqref{eq:eM-epstens2} is an isomorphism.  In the same way $G
\epstens_{\cor_M} e_M(F) \simeq e_M(r_M(G) \tens_{\cor_M} F)$. In particular,
for $F,G \in \Derb(\Cor_M)$ such that $F$ or $G$ belongs to $\perf(\Cor_M)$, we
have $F \epstens_{\cor_M} G \in \perf(\Cor_M)$ and $\epstens_{\cor_M}$ induces
a functor
$$
\epstens_{\cor_M} \cl \Orb(\cor_M) \times \Orb(\cor_M) \to \Orb(\cor_M) .
$$
\end{lemma}
\begin{proof}
(i) Let us denote by $u_{F,G}$ the morphism~\eqref{eq:eM-epstens2}. Using the
distinguished triangle $\tau_{\leq i}F \to F \to \tau_{>i}F \to[+1]$ and the
similar one for $G$ we can argue by induction on the length of $F$ and $G$ to
prove that $u_{F,G}$ is an isomorphism.  Then we are reduced to the case where
$F$ and $G$ are concentrated in degree $0$.
Writing $\Cor = \cor \oplus \varepsilon \cor$ we have $e_M(F \tens_{\cor_M}
r_M(G)) = (F\otimes G) \oplus \varepsilon (F\otimes G)$ and $e_M(F)
\epstens_{\cor_M} G = (F \oplus \varepsilon F) \otimes G$.  For sections $f$ and
$g$ of $F$ and $G$ we have
\begin{align*}
u_{F,G} ( f\otimes g) &= f\otimes g, \\
u_{F,G} ( \varepsilon(f\otimes g)) &=  \varepsilon \cdot (f\otimes g)
= (\varepsilon f)\otimes g + f \otimes (\varepsilon g )
\end{align*}
and we can check directly that $u_{F,G}$ is an isomorphism with inverse
$$
u_{F,G}^{-1} ((f_0+\varepsilon f_1)\otimes g) 
= (f_0\otimes g + f_1 \otimes (\varepsilon g)) + \varepsilon(f_1\otimes g) .
$$
(ii) If $F \in \perf(\Cor_M)$, there exist two sequences $F_0, F_1,\ldots,F_n=F
\in \Derb(\Cor_M)$ and $F'_0, F'_1,\ldots,F'_n \in \Derb(\cor_M)$ with $F_0 =
e_M(F'_0)$, and distinguished triangles involving $F_i$, $e_M(F'_{i+1})$ and
$F_{i+1}$, for $i=0,\ldots,n-1$.  Then~(i) and an induction on $n$ give $F
\epstens_{\cor_M} G \in \perf(\Cor_M)$.
\end{proof}
\begin{remark}\label{rem:def_shift_morph2}
An easy case of Lemma~\ref{lem:epstens-perf} is $F=\cor_M$
in~\eqref{eq:eM-epstens2}. We obtain $e_Mr_M(G) \isoto \Cor_M\epstens_{\cor_M}
G$, for any $G\in \Derb(\Cor_M)$.  Hence the distinguished
triangle~\eqref{eq:td_shift_morph} becomes
\begin{equation}\label{eq:td_shift_morph-bis}
F \to  e_Mr_M(F)  \to  F \to[s_M(F)] F[1]
\quad\text{ for any $F\in\Derb(\Cor_M)$}.
\end{equation}
Applying $Q_M$ to this triangle gives an isomorphism $s_M(F) \cl F \isoto F[1]$
in $\Orb(\cor_M)$.
\end{remark}

We can define an adjoint $\homeps$ to $\epstens$ by a similar construction.
For $F_1,F_2 \in \Mod(\Cor_M)$, we define $\homeps(F_1,F_2) \in \Mod(\Cor_M)$ as
the sheaf of $\cor$-vector spaces $\hom(F_1,F_2)$ with the action of
$\varepsilon$ given by
$$
(\varepsilon \cdot \varphi)(x)
 = \varepsilon \varphi(x) + \varphi(\varepsilon x), 
$$
where $\varphi$ is a section of $\hom(F_1,F_2)$ over an open set $U$ and $x$ a
section of $F_1$ over a subset $V$ of $U$.  Then we see that $\homeps$ is right
adjoint to $\epstens$, hence left exact.  We check also that its derived functor
$\rhomeps$ is right adjoint to $\epstens$ in $\Derb(\Cor_M)$ and that, for any
$F,G \in \Derb(\Cor_M)$,
\begin{equation}\label{eq:hom-homeps}
r_M(\rhomeps(F,G) ) \simeq \rhom(r_M(F), r_M(G)).
\end{equation}
We have the similar result as Lemma~\ref{lem:epstens-perf}.
\begin{lemma}\label{lem:homeps-perf}
Let $F,G \in \Derb(\Cor_M)$. We assume that $F$ or $G$ belongs to
$\perf(\Cor_M)$.  Then $\rhomeps(F,G) \in \perf(\Cor_M)$. The induced functor
$$
\rhomeps \cl \Orb(\cor_M)^{op} \times \Orb(\cor_M) \to \Orb(\cor_M) .
$$
is right adjoint to $\epstens_{\cor_M}$.
\end{lemma}

\subsection*{Morphisms in the triangulated orbit category}
We prove the formula~\eqref{eq:mor-orbA} which describes the morphisms in
$\Orb(\cor_M)$.

\begin{lemma}\label{lem:mor-perf-borne}
Let $F,P \in \Derb(\Cor_M)$. We assume that $P \in \perf(\Cor_M)$.
Then $\RHom(P,F)$ and $\RHom(F,P)$ belong to $\Derb(\Cor)$.
\end{lemma}
\begin{proof}
Since $\perf(\Cor_M)$ is generated by $e_M(\Derb(\cor_M))$, the same argument as
in~(ii) of the proof of Lemma~\ref{lem:epstens-perf} implies that we can assume
$P=e_M(Q)$, for some $Q\in \Derb(\cor_M)$. Then $\RHom_{\Cor}(P,F) \simeq
\RHom_\cor(Q,r_M(F))$ is bounded. Using~\eqref{eq:adjrMeM} the same proof gives
that $\RHom(F,P)$ is bounded.
\end{proof}

We define the following objects $L^{p,q}$ of $\Derb(\Cor)$, for any
two integers $p\leq q$, by
\begin{equation}\label{eq:def-Lpq}
  L^{p,q} \; = \; 0 \to \Cor \to[\varepsilon] \Cor \to[\varepsilon] \dots
\to[\varepsilon] \Cor \to 0,
\end{equation}
where the first $\Cor$ is in degree $p$ and the last one in degree $q$.  Then
$L^{p,q} \in \perf(\Cor)$ and we can see that there is a distinguished triangle
\begin{equation}\label{eq:dt-Lpq}
\cor[-p] \to L^{p,q} \to \cor[-q] \to[s^{p,q}] \cor[-p+1],
\end{equation}
with $s^{p,q} = s_{\pt}^{q-p+1}[-q]$, where $s_{\pt}$
is~\eqref{eq:def_shift_morph0}.  For $F \in \Derb(\Cor_M)$ we define
$s^{p,q}_M(F) \eqdot s^{p,q}\tens \id_F \cl F[-q] \to F[-p+1]$.  We deduce the
triangle, for any $F \in \Derb(\Cor_M)$ and any $n\geq 1$,
$$
L^{1,n}_M \epstens_{\cor_M} F \to F[-n] \to[{s_M^{1,n}(F)}] F \to[+1] .
$$
\begin{lemma}\label{lem:morph-orb}
We consider a distinguished triangle $P \to F' \to F \to[+1]$ in $\Derb(\Cor_M)$
and we assume that $P\in \perf(\Cor_M)$. Then there exist $n\in \N$ and a
morphism of triangles
$$
\xymatrix@C=1cm@R=5mm{
L^{1,n}_M \epstens_{\cor_M} F \ar[r]\ar[d] & F[-n] \ar[rr]^{s_M^{1,n}(F)} \ar[d]
&&  F \ar[r]^{+1} \ar@{=}[d] &  \\
P \ar[r] & F'  \ar[rr]  && F \ar[r]^{+1}  &  \pointdiag }
$$
\end{lemma}
\begin{proof}
We set for short $s_n = {s_M^{1,n}(F)}$. We consider the diagram
$$
\xymatrix@C=1.5cm@R=7mm{
& F[-n] \ar[r]^-{s_n}  \ar[dr]^{s_n}  \ar@{.>}[d]_a &  F \ar@{=}[d] \\
P \ar[r] & F'  \ar[r]  & F \ar[r]^w  &  P[1] \pointdiag }
$$
We have $w \circ s_n \in \Hom(F[-n],P[1])$.  Since $P\in \perf(\Cor_M)$ this
group vanishes for $n$ big enough, by Lemma~\ref{lem:mor-perf-borne}. In this
case we have $w \circ s_n =0$ and there exists a morphism $a$ as in the diagram
making the square commute.  Then we can extend the square to a commutative
diagram as in the lemma.
\end{proof}

For $F \in \Derb(\Cor_M)$ we have $s^{n,n}_M(F) \cl F[-n] \to F[-n+1]$.  Then
$\{F[-n], s^{n,n}_M(F)\}_{n\in \N}$ gives a projective system.  For $G\in
\Derb(\Cor_M)$ we define
\begin{equation}\label{eq:mor-Der-Orb}
\varinjlim_{n\in \N} \Hom_{\Derb(\Cor_M)}(F[-n],G) \to \Hom_{\Orb(\cor_M)}(F,G),
\end{equation}
by sending $\varphi_n \cl F[-n]\to G$ to $\varphi_n \circ (s_M^{1,n}(F))^{-1}$.
This is well defined since $s_M^{1,n}(F)$ becomes invertible in $\Orb(\cor_M)$
and $s_M^{1,n}(F) = s_M^{1,n-1}(F) \circ s_M^{n,n}(F)$.

\begin{proposition}\label{prop:morph-orb}
Let $F,G \in \Derb(\Cor_M)$. Then the inductive limit in the left hand side
of~\eqref{eq:mor-Der-Orb} stabilizes and the morphism~\eqref{eq:mor-Der-Orb} is
an isomorphism. More precisely, if $H^i(F) = H^i(G) =0$ for all $i$ outside an
interval $[a,b]$, then
\begin{equation}\label{eq:mor-orbA}
\Hom_{\Derb(\Cor_M)}(F[-n],G) \isoto \Hom_{\Orb(\cor_M)}(F,G),
\end{equation}
for all $n > 2(b-a) + \dim M +1$.
\end{proposition}
\begin{proof}
(i) We prove that the limit stabilizes. We chose $a\leq b$ such that $H^i(F) =
H^i(G) =0$ for all $i$ outside $[a,b]$.  By~\eqref{eq:td_shift_morph-bis} we
have the distinguished triangle
\begin{equation}\label{eq:morph-orb1}
\begin{split}
\Hom_{\Derb(\Cor_M)}( e_M r_M(F)[-n],G)
&\to \Hom_{\Derb(\Cor_M)}(F[-n],G) \\
\to[s'_n] {}& \Hom_{\Derb(\Cor_M)}(F[-n-1],G) \to[+1] ,
\end{split}
\end{equation}
for all $n\in \Z$. By adjunction we have
$$
\Hom_{\Derb(\Cor_M)}( e_M r_M(F)[-n],G) \simeq
\Hom_{\Derb(\cor_M)}(r_M(F)[-n],r_M(G) )
$$
and this is zero for $n> 2(b-a) + \dim M +1$ (recall that the flabby dimension
of a manifold $M$ is $\dim M +1$ and that injective over a field is the same as
flabby). It follows that the morphism $s'_n$ in~\eqref{eq:morph-orb1} is an
isomorphism for $n> 2(b-a) + \dim M +1$.

\medskip\noindent
(ii) We prove that~\eqref{eq:mor-Der-Orb} is an isomorphism.  We recall that
$$
\Hom_{\Orb(\cor_M)}(F,G) \simeq \varinjlim_{i\cl F'\to F}
\Hom_{\Derb(\Cor_M)}(F',G),
$$
where the limit runs over the morphisms $i\cl F'\to F$ whose cone belongs to
$\perf(\Cor_M)$ (and a morphism $u\cl F' \to G$ is send to $u \circ i^{-1}$ in
$\Orb(\cor_M)$).  By Lemma~\ref{lem:morph-orb} we can restrict to the family of
morphisms $s_M^{1,n}(F) \cl F[-n] \to F$ for $n\in \N$. This gives the result.
\end{proof}

\begin{corollary}\label{cor:morph-QRF-QRG}
Let $F,G \in \Derb(\cor_M)$. We recall that $\iota_M = Q_M \circ R_M$.
We have
$$
\Hom_{\Orb(\cor_M)}(\iota_M(F), \iota_M(G))
\simeq \bigoplus_{n\in\Z} \Hom_{\Derb(\cor_M)}(F[-n],G) .
$$
\end{corollary}
\begin{proof}
By Proposition~\ref{prop:morph-orb} the left hand side of the formula is
isomorphic to
$$
\Hom_{\Derb(\Cor_M)}(R_M(F)[-n_0], R_M(G))
\simeq \Hom_{\Der^-(\cor_M)}(E_M R_M(F)[-n_0], G),
$$
for any big enough $n_0\in\N$.  Using the resolution of $\cor$ as a
$\Cor$-module given by $\cdots\to \Cor \to[\varepsilon] \Cor \to[\varepsilon]
\Cor \to \cor \to 0$, we see that $E_MR_M(F) \simeq \bigoplus_{i\in \N} F[i]$.
The result follows easily.
\end{proof}

\subsection*{Direct sums}

We recall that we denote by $Q_M$ the quotient functor $\Derb(\Cor_M) \to
\Orb(\cor_M)$.
\begin{lemma}\label{lem:direct-sum}
Let $I$ be a small set and $\{F_i\}_{i\in I}$ a family in $\Derb(\Cor_M)$.  We
assume that there exists two integers $a\leq b$ such that $H^k(F_i)=0$ for all
$k$ outside $[a,b]$ and all $i\in I$.
Then $\bigoplus_{i\in I} Q_M(F_i)$ exists in $\Orb(\cor_M)$ and
$\bigoplus_{i\in I} Q_M(F_i) \simeq Q_M(\bigoplus_{i\in I} F_i)$.
\end{lemma}
\begin{proof}
By the hypothesis on the degrees the sum $\bigoplus_{i\in I} F_i$ exists in
$\Derb(\Cor_M)$ and we have $H^k(\bigoplus_{i\in I} F_i)=0$ for $k$ outside
$[a,b]$. Let $G \in \Derb(\Cor_M)$ and let $a'\leq a$, $b'\geq b$ be such that
$H^k(G)=0$ for $k$ outside $[a',b']$. We set $n=2(b'-a') + \dim M +2$.  If $F =
F_i$ for some $i\in I$, or $F = \bigoplus_{i\in I} F_i$, we have
$$
\Hom_{\Derb(\Cor_M)}(F[-n],G) \isoto \Hom_{\Orb(\cor_M)}(Q_M(F),Q_M(G))
$$
by Proposition~\ref{prop:morph-orb}.  Now the lemma follows from the universal
property of the sum.
\end{proof}

\subsection*{Direct and inverse images}

Let $f\cl M\to N$ be a morphism of manifolds.  We have functors $\roim{f}$,
$\reim{f}$, $\opb{f}$ and $\epb{f}$, between $\Derb(\Cor_M)$ and
$\Derb(\Cor_N)$. Since $\Cor$ is not of finite global dimension we give some
details.  

First, the functors $\roim{f}, \reim{f} \cl \Der(\Cor_M) \to \Der(\Cor_N)$
commute with the functors $r_N \cl \Derb(\Cor_N) \to \Derb(\cor_N)$ and $r_M$.
We remark that $r_N$ is conservative, which means that a morphism $u$ in
$\Der(\Cor_N)$ is an isomorphism if $r_N(u)$ is an isomorphism. Hence $\roim{f}$
and $\reim{f}$ induce functors $\Derb(\Cor_M) \to \Derb(\Cor_N)$ between the
bounded categories.

The case of $\opb{f}$ is clear since it is an exact functor.  To prove the
existence of $\epb{f}$, right adjoint to $\reim{f}$, it is enough, by
factorizing through the graph embedding, to consider the cases where $f$ is an
embedding or $f$ is a submersion. The usual formulas work in our case.  If $f$
is an embedding, then $\epb{f}(\cdot) = \opb{f} \homeps(\cor_M,\cdot)$ is
adjoint to $\reim{f}$. If $f$ is a submersion, then $\epb{f}(\cdot) =
\opb{f}(\cdot) \epstens \omega_{M|N}$.  We also remark that $\opb{f}$ and
$\epb{f}$ commute with $r_N$ and $r_M$.

\begin{lemma}\label{lem:im-dir-inv_Orb}
Let $f\cl M\to N$ be a morphism of manifolds.  Then the functors $\roim{f}$,
$\reim{f}$, $\opb{f}$ and $\epb{f}$, between $\Derb(\Cor_M)$ and
$\Derb(\Cor_N)$, preserves the categories $\perf(\Cor_M)$ and $\perf(\Cor_N)$.
They induce pairs of adjoint functors (that we denote in the same way) between
$\Orb(\cor_M)$ and $\Orb(\cor_N)$.
\end{lemma}
\begin{proof}
We only consider the case of $\roim{f}$, the other cases being similar.
For $F\in \Derb(\cor_M)$ we have a natural morphism $u\cl \Cor_N
\tens_{\cor_N}\roim{f} F \to \roim{f}(\Cor_M \tens_{\cor_M} F)$. Since
$r_N(\Cor_N) \simeq \cor_N^2$ we see easily that $r_N(u)$ is an
isomorphism. Since $r_N$ is conservative, $u$ is an isomorphism and we obtain
that $\roim{f}(\perf(\Cor_M)) \subset \perf(\Cor_N)$, as required.
\end{proof}

For $F \in \Orb(\cor_M)$ and $j\cl U \to M$ the inclusion of a locally closed
subset, we use the standard notations $F|_U = \opb{j}F$, $F_U =
\reim{j}\opb{j}F$, $\rsect_U(F) = \roim{j}\opb{j}F$, $\rsect(U;F) =
\roim{a_U}(F|_U) \in \Orb(\cor)$, where $a_U$ is $U \to \pt$, and $\rsect_c(U;F)
= \reim{a_U}(F|_U) \in \Orb(\cor)$. We have the same formulas as in
$\Derb(\cor_M)$:
$$
F_U  \simeq F \epstens \cor_U, \qquad \rsect_U(F) \simeq \homeps(\cor_U,F) .
$$

Let $N$ be a submanifold of $M$.  We recall that Sato's microlocalization is a
functor $\mu_N\cl \Derb(\cor_M) \to \Derb(\cor_{T^*_NM})$.  It is defined by
composing direct and inverse images functors and Lemma~\ref{lem:im-dir-inv_Orb}
implies that it induces functors, denoted in the same way:
\begin{align}
\mu_N&\cl \Derb(\Cor_M) \to \Derb(\Cor_{T^*_NM}) , \\
\mu_N&\cl \Orb(\cor_M) \to \Orb(\cor_{T^*_NM}) .  
\end{align}

\begin{definition}
Let $q_1,q_2\cl M\times M \to M$ be the projections.  We identify
$T^*_{\Delta_M}(M\times M)$ with $T^*M$ through the first projection.  For
$F,G\in \Orb(\cor_M)$ we define as in~\eqref{eq:def_muhom}
\begin{equation*}
\mu hom^\varepsilon (F,G) = \mu_{\Delta_M}(\rhomeps(\opb{q_2}F,\epb{q_1}G))
\quad \in \quad \Orb(\cor_{T^*M}). 
\end{equation*}
\end{definition}

The following result follows from the analog one in $\Derb(\Cor_M)$.
\begin{lemma}\label{lem:dir-sum_im-inv-dir}
Let $\{F_i\}_{i\in I}$ a small family in $\Derb(\Cor_M)$ satisfying the
hypothesis of Lemma~\ref{lem:direct-sum}.  Let $f\cl M' \to M$ and $g\cl M\to
M''$ be morphisms of manifolds and let $G\in \Orb(\cor_M)$.  Then we have
canonical isomorphisms
\newcommand{\somme}{{\textstyle\bigoplus_{i\in I}}\;}
\begin{align*}
\opb{f} (\somme Q_M(F_i)) &\simeq \somme \opb{f}Q_M(F_i), \\
\reim{g}(\somme Q_M(F_i)) &\simeq \somme \reim{g}\, Q_M(F_i), \\
(\somme Q_M(F_i)) \epstens_{\cor_M} G 
&\simeq \somme (Q_M(F_i)) \epstens_{\cor_M} G).
\end{align*}
\end{lemma}

\begin{lemma}\label{lem:modif-repr}
Let $U$ be an open subset of $M$. Let $F\in \Derb(\Cor_M)$ and $F' \in
\Derb(\Cor_U)$. We assume that there exists an isomorphism $F|_U\simeq F'$ in
$\Orb(\cor_U)$.  Then there exists $F_1 \in \Derb(\Cor_M)$ such that $F_1|_U
\simeq F'$ in $\Derb(\Cor_U)$ and $F_1\simeq F$ in $\Orb(\cor_M)$.
\end{lemma}
\begin{proof}
We let $j\cl U \to M$ be the inclusion and we set $Z=M\setminus U$.
Let $u \cl F|_U \to F'$ be an isomorphism in $\Orb(\cor_U)$.  By
Proposition~\ref{prop:morph-orb} there exist $n\in \Z$ and a morphism $F[-n] \to
F'$ in $\Derb(\Cor_M)$ which represents $u$.  Defining $P$ by the distinguished
triangle $F|_U[-n] \to F' \to P \to[+1]$ we have $Q_U(P) \simeq 0$.  We apply
$\eim{j}$ to this triangle and get~\eqref{eq:modif-repr2} below; we also
consider the excision triangle~\eqref{eq:modif-repr1} and the
triangle~\eqref{eq:modif-repr3} built on the composition $F_Z[-n-1] \to F_U[-n]
\to \eim{j} F'$:
\begin{gather}
\label{eq:modif-repr1}
F_Z[-n-1] \to[a] F_U[-n] \to F[-n] \to[+1] , \\
\label{eq:modif-repr2}
F_U[-n] \to[b]  \eim{j} F' \to \eim{j} P \to[+1] , \\
\label{eq:modif-repr3}
F_Z[-n-1] \to[b\circ a] \eim{j} F' \to F_1  \to[+1] .
\end{gather}
Then the octahedron axiom gives the triangle $F[-n] \to F_1 \to \eim{j} P
\to[+1]$. We have $Q_M(\eim{j} P) \simeq \eim{j} Q_M(P) \simeq 0$, hence
$F_1\simeq F[-n] \simeq F$ in $\Orb(\cor_M)$.  Applying $\opb{j}$ to the
triangle~\eqref{eq:modif-repr3} gives $F' \simeq F_1|_U$, as required.
\end{proof}

\begin{definition}\label{def:supporb}
For $F \in \Orb(\cor_M)$ we define $\supporb(F) \subset M$ as the complement of
the union of the open subsets $U\subset M$ such that $F|_U \simeq 0$.
\end{definition}
For an open subset $U\subset M$ we have $F|_U \simeq 0$ in $\Orb(\cor_U)$ if and
only if $F_U \simeq 0$ in $\Orb(\cor_M)$.  By the Mayer-Vietoris triangle we
deduce that, for a finite covering $U=\bigcup_{i=1}^n U_i$, we have $F|_U \simeq
0$ if and only if $F|_{U_i} \simeq 0$ for all $i$.  For an increasing countable
union $U=\bigcup_{i=1}^\infty U_i$ we have the same result using the
distinguished triangle
$$
\bigoplus_{i=1}^\infty F_{U_i}
 \to \bigoplus_{i=1}^\infty F_{U_i} \to F_U \to[+1] ,
$$
obtained by applying $F \epstens_{\cor_M}\cdot$ to~\eqref{eq:homot-lim} (with
$U$ instead of $M$).  We obtain finally, for any $F \in \Orb(\cor_M)$,
\begin{equation}\label{eq:supporb-OK}
F|_{M\setminus\supporb(F)} \simeq 0 \quad \text{and} \quad
F \isoto F_{\supporb(F)} .
\end{equation}

\section{Microsupport in the triangulated orbit categories}
\label{sec:micsup_orb_cat}

We define the microsupport of objects of $\Orb(\cor_M)$ and check that it
satisfies the same properties as the usual microsupport.  In~\cite{KS90} the
microsupport is defined when the coefficient ring is of finite global
dimension. The ring $\Cor$ is of infinite dimension and we use the forgetful
functor $r_M \cl \Derb(\Cor_M) \to \Derb(\cor_M)$ to define
\begin{equation}\label{eq:def-SSi-Cor}
  \SSi(F) \eqdot \SSi(r_M(F)) \quad \text{for $F\in \Derb(\Cor_M)$.}
\end{equation}
Since $r_M$ commutes with $\roim{f}, \reim{f}, \opb{f}, \epb{f}$ we see that
Theorem~\ref{th:opboim}~(ii) and~(iii) hold with this definition.
By~\eqref{eq:tens-epstens} and~\eqref{eq:hom-homeps} Theorem~\ref{th:opboim}~(i)
and Corollary~\ref{cor:opboim} hold if we replace $\tens$ and $\rhom$ by
$\epstens$ and $\rhomeps$. Since $r_M$ is conservative (that is, a morphism $u$
is an isomorphism if $r_M(u)$ is an isomorphism) we see that
Theorem~\ref{th:opboim}~(iv) also holds.

\subsection{Definition and first properties}
We define the microsupport $\SSo(F)$ of an object of $F\in\Orb(\cor_M)$ from the
microsupports of its representatives in $\Derb(\Cor_M)$.  We prove in
Proposition~\ref{prop:repr-bonmicsup} that, for a given $x_0\in M$ and $F\in
\Orb(\cor_M)$, we can find a representative $F'\in \Derb(\cor_M)$ with
$T^*_{x_0}M \cap \SSi(F')$ contained in a arbitrary neighborhood of $T^*_{x_0}M
\cap \SSo(F)$ .

\begin{definition}\label{def:SSorb}
Let $F \in \Orb(\cor_M)$. We define $\SSo(F) \subset T^*M$ by $\SSo(F) =
\bigcap_{F'} \SSi(F')$ where $F'$ runs over the objects of $\Derb(\Cor_M)$ such
that $F'\simeq F$ in $\Orb(\cor_M)$. We set $\dot\SSo(F) = \SSo(F) \cap \dT^*M$.
\end{definition}

We remark that $\SSo(F)$ is a closed conic subset of $T^*M$.  We deduce from
Lemma~\ref{lem:modif-repr} that $\SSo(F)$ is a local notion, that is, for
$U\subset M$ open, we have
\begin{equation}\label{eq:SSorb_local}
\SSo(F|_U) = \SSo(F) \cap T^*U.
\end{equation}
In other words $p=(x;\xi) \not\in \SSo(F)$ if and only if there exist a
neighborhood $U$ of $x$ and $F'\in \Derb(\Cor_U)$ such that $F|_U \simeq F'$ in
$\Orb(\cor_U)$ and $p\not\in \SSi(F')$.  For a given $F \in \Derb(\Cor_M)$ we
see in the next lemma that it is possible to require that this isomorphism $F|_U
\simeq F'$ in $\Orb(\cor_U)$ is the image of a morphism $F|_U \to F'$ in
$\Derb(\Cor_U)$.

\begin{lemma}\label{lem:troncationSSorb}
Let $x_0 \in M$ and let $C_0 \subset T^*_{x_0}M$ be a conic open subset such
that $\ol{C_0}$ is convex and proper.  Let $F, F' \in \Derb(\Cor_M)$.  We
assume that $\SSi(F') \cap \ol{C_0} = \emptyset$ and that there exists an
isomorphism $F\simeq F'$ in $\Orb(\cor_M)$.  Let $W_0 \subset T^*_{x_0}M$ be a
conic neighborhood of $(\SSi(F) \cap \ol{C_0}) \setminus \{x_0\}$.
Then there exist a neighborhood $U$ of $x_0$, $F'' \in \Derb(\Cor_U)$ and a
morphism $u \cl F|_U \to F''$ such that $u$ induces an isomorphism in
$\Orb(\cor_U)$ and
$$
\dot\SSi(F'') \cap T^*_{x_0}M \subset ((\SSi(F)\cap T^*_{x_0}M) \cup W_0 )
\setminus C_0 .
$$
\end{lemma}
\begin{proof}
Since $F\simeq F'$ in $\Orb(\cor_M)$ there exist two distinguished triangles
in $\Derb(\Cor_M)$
\begin{equation}\label{eq:lem-troncationSSorb1}
G \to F \to P \to[+1], \quad
G \to F' \to P_1 \to[+1], 
\end{equation}
with $P,P_1 \in \perf(\Cor_M)$.

We can find a conic open subset $\Omega_0 \subset T^*_{x_0}M$ such that
$\ol{C_0}\setminus \{x_0\} \subset \Omega_0$, $(\SSi(F) \cap \ol{\Omega_0})
\setminus \{x_0\} \subset W_0$ and $\SSi(F') \cap \Omega_0 = \emptyset$.

  We choose a conic open subset $\Omega \subset
T^*M$ such that $\Omega \cap T^*_{x_0}M = \Omega_0$ and $\SSi(F') \cap \Omega
= \emptyset$.  We set $S = \SSi(F) \cap \Omega$.  We let $U \subset M$ be a
neighborhood of $x_0$ and $L,L' \in \Derb(\cor_{U\times M})$ be given by
Proposition~\ref{prop:cutoff-precis}.  Composing with $L$ gives the triangles
\begin{equation}\label{eq:lem-troncationSSorb2}
L\circ G \to L\circ F \to L\circ P \to[+1], \quad
L\circ G \to L\circ F' \to L\circ P_1 \to[+1].
\end{equation}
By Proposition~\ref{prop:cutoff-precis}~(iii) we have $L\circ F' \simeq
0$. Hence $L\circ G \simeq L\circ P_1[-1]$. This proves that $L\circ G \in
\perf(\Cor_U)$ and the first triangle in~\eqref{eq:lem-troncationSSorb2} shows
that $L\circ F \in \perf(\Cor_U)$.

We set $F'' = L'\circ F$. Composing the triangle $L \to \cor_\Delta \to L'
\to[+1]$ with $F$ gives $L\circ F \to F|_U \to F'' \to[+1]$. We deduce that
$F|_U\isoto F''$ in $\Orb(\cor_U)$. The bound for $\SSi(F'')$
is~\eqref{eq:cutoff-precis1}.
\end{proof}

\begin{lemma}\label{lem:induction-troncation}
We set $E = \R^n$ and $\dot E = E \setminus \{0\}$.
Let $A \subset \dot E$ be a closed cone.  Let $A_1,\ldots,A_n$ be convex open
cones of $\dot E$ such that $\ol{A_i}$ is proper for each $i$ and $A \subset
\bigcup_{i=1}^n A_i$.  Then there exist open cones $C_i$, $W_i \subset A_i$,
$i=1,\ldots,n$, such that $\ol{C_i}$ is convex and proper for each $i$, and,
defining inductively $S_0 = \dot E$ and $S_i = (S_{i-1} \cup W_i) \setminus
C_i$, we have
\begin{itemize}
\item [(i)] $((S_{i-1} \cap \ol{C_i}) \setminus \{0\} )\subset W_i$, for
  all $i=1,\ldots,n$,
\item [(ii)] $A \cap S_n = \emptyset$.
\end{itemize}
\end{lemma}
\begin{proof}
(a) We choose conic subsets $D_i \subset C_i \subset A_i$ for $i=1,\ldots,n$
such that $A \subset \bigcup_{i=1}^n D_i$ and, for each $i$, $D_i$ is closed,
$C_i$ is open and convex and $\ol{C_i} \subset A_i\cup \{0\}$.
Then we define $S_0^+ = \dot E$, $S_i^+ = \dot E \setminus
\bigcup_{j=1}^i D_j$ and $W_i = S_{i-1}^+ \cap A_i$, for $i=1,\ldots,n$.

\medskip\noindent
(b) We define as in the lemma $S_0 = \dot E$ and $S_i = (S_{i-1} \cup
W_i) \setminus C_i$.  Let us prove by induction that $S_i \subset S_i^+$. This
is clear for $i=0$.  Assuming it holds for $i-1$ we obtain $S_i \subset
S_{i-1}^+ \setminus C_i$ since $W_i \subset S_{i-1}^+$.  Since $D_i \subset C_i$
we deduce $S_i \subset S_{i-1}^+ \setminus D_i = S_i^+$.

\medskip\noindent
(c) We deduce from~(b) that $((S_{i-1} \cap \ol{C_i}) \setminus \{0\}
)\subset S_{i-1}^+ \cap A_i = W_i$ which is the property~(i).  Since $A \subset
\bigcup_{i=1}^n D_i$ we also obtain $A \cap S_n \subset A \cap S^+_n =
\emptyset$, which is the property~(ii).
\end{proof}

\begin{proposition}\label{prop:repr-bonmicsup}
Let $M$ be a manifold, $x_0\in M$ and $A \subset \dT^*_{x_0}M$ a closed conic
subset.  Let $F\in \Orb(\cor_M)$ be such that $\SSo(F) \cap A = \emptyset$.
Then there exist a neighborhood $U$ of $x_0$, $F' \in \Derb(\Cor_U)$ and a
morphism $u \cl F|_U \to F'$ such that $u$ induces an isomorphism in
$\Orb(\cor_U)$ and $\SSi(F') \cap A = \emptyset$.
\end{proposition}
\begin{proof}
(i) Since $A \cap \SSo(F) = \emptyset$, for any $\xi \in A$ there exist a
neighborhood $V$ of $x_0$ in $M$ and $F_\xi \in \Derb(\Cor_V)$ such that $F_\xi
\simeq F|_V$ in $\Orb(\cor_V)$ and $(x_0;\xi) \not\in \SSi(F_\xi)$.  Let
$C_\xi \subset T^*_{x_0}M$ be a convex open cone such that $\ol{C_\xi}$ is
proper and $\SSi(F_\xi) \cap \ol{C_\xi} \subset \{x_0\}$.

We can cover $A$ by a finite number of such cones $C_{\xi_i}$, $i=1,\dots,n$.
We set $A_i = C_{\xi_i}$ and $F_i = F_{\xi_i}$.  We choose conic open subsets
$C_i$, $W_i$ of $A_i$ satisfying the properties of
Lemma~\ref{lem:induction-troncation} with $E = T^*_{x_0}M$.  We also define $S_0
= \dT^*_{x_0}M$ and $S_i = (S_{i-1} \cup W_i) \setminus C_i$, as in
Lemma~\ref{lem:induction-troncation}.

\medskip\noindent
(ii) Let us prove by induction on $i$ that there exist a neighborhood $U_i$ of
$x_0$, $F'_i \in \Derb(\Cor_{U_i})$ and a morphism $u_i \cl F|_{U_i} \to F'_i$
such that $u_i$ induces an isomorphism in $\Orb(\cor_{U_i})$ and $\dot\SSi(F'_i)
\cap T^*_{x_0}M \subset S_i$.

For $i=0$ it is enough to set $F'_0 = F$.  We assume that we have defined
$F'_{i-1}$ and $u_{i-1}$.  In particular we have isomorphisms $F'_{i-1} \simeq
F|_{U_{i-1}} \simeq F_i|_{U_{i-1}}$ in $\Orb(\cor_{U_{i-1}})$.  By the
property~(i) of Lemma~\ref{lem:induction-troncation} $W_i$ is a neighborhood of
$(\SSi(F'_{i-1}) \cap \ol{C_i}) \setminus \{x_0\}$.  Hence we can apply
Lemma~\ref{lem:troncationSSorb} with $M=U_{i-1}$, $F = F'_{i-1}$, $F' = F_i$,
$C_0 = C_i$ and $W_0 = W_i$.  We obtain a neighborhood $U$ of $x_0$, $F'' \in
\Derb(\Cor_U)$ and a morphism $u \cl F|_U \to F''$ such that $u$ induces an
isomorphism in $\Orb(\cor_U)$ and
\begin{equation}\label{eq:repr-bonmicsup1}
\dot\SSi(F'') \cap T^*_{x_0}M \subset ((\SSi(F)\cap T^*_{x_0}M) \cup W_0 )
\setminus C_0 .
\end{equation}
We set $U_i = U$, $F'_i = F''$, $u_i = u \circ u_{i-1}$.
Then~\eqref{eq:repr-bonmicsup1} translates into
$$
\dot\SSi(F'_i) \cap T^*_{x_0}M
\subset (S_{i-1} \cup W_i) \setminus C_i = S_i,
$$ 
as required. 

\medskip\noindent
(iii) We set $F' = F_n$, $u = u_n$.  The lemma follows from the
property~(ii) of Lemma~\ref{lem:induction-troncation}.
\end{proof}

\begin{proposition}\label{prop:trian-ineq-SSo}
Let $F' \to F \to F'' \to[+1]$ be a distinguished triangle
in $\Orb(\cor_M)$. Then $\SSo(F) \subset \SSo(F') \cup \SSo(F'')$.
\end{proposition}
\begin{proof}
(i) Let $p=(x_0;\xi_0) \in T^*M$ be given. We assume that $p\not\in \SSo(F')
\cup \SSo(F'')$ and we prove that $p\not\in \SSo(F)$.  Let us set for short $G
= F''$ and $H= F'[1]$.  Let $u \cl G \to H$ be the morphism of the triangle.
It is enough to find a neighborhood $U$ of $x_0$ and a morphism $u_0 \cl G_0
\to H_0$ in $\Derb(\Cor_U)$ whose image in $\Orb(\cor_U)$ is $u|_U$ and such
that $p\not\in \SSi(H_0) \cup \SSi(G_0)$. Indeed the cone of $u_0$ represents
$F$ and the result then follows from the triangular inequality for the usual
microsupport.

\medskip\noindent
(ii) By definition we can find $G_0 \in \Derb(\Cor_M)$ such that $p\not\in
\SSi(G_0)$ and $G_0 \simeq G$ in $\Orb(\cor_M)$. We can also find $H_1 \in
\Derb(\Cor_M)$ and a morphism $u_1 \cl G_0 \to H_1$ such $H_1 \simeq H$ in
$\Orb(\cor_M)$ and the image of $u_1$ is $u$.

By Proposition~\ref{prop:repr-bonmicsup} applied with $A=\rspos \xi_0$ there
exist a neighborhood $U$ of $x_0$ and a morphism $s \cl H_1|_U \to H_0$ such
that $s$ induces an isomorphism in $\Orb(\cor_U)$ and $p\not\in \SSi(H_0)$.  We
define $u_0 \cl G_0 \to H_0$ by $u_0 = s \circ u_1$. Then $u_0$ represents $u$
and we conclude by~(i).
\end{proof}

\begin{remark}\label{rem:supporb-SSo}
For $F \in \Orb(\cor_M)$ we have $\supporb(F) = T^*_MM \cap \SSo(F)$.  Indeed
we have $T^*_MM \cap \SSo(F) \subset \supporb(F)$
by~\eqref{eq:SSorb_local}. Conversely let us assume that $(x;0) \not\in
\SSo(F)$.  Then $F$ has a representative $F' \in \Derb(\Cor_M)$ such that
$(x;0) \not\in \SSi(F')$.  Then $F'$ vanishes in a neighborhood $U$ of $x$.
Hence $F$ also vanishes on $U$ and we have $x\not\in \supporb(F)$.
\end{remark}

\subsection{Functorial behavior}
We prove that $\SSo(\cdot)$ satisfies the same properties as $\SSi(\cdot)$
with respect to the usual sheaf operations.

\begin{proposition}\label{prop:im-inv-SSorb}
Let $f\cl M \to N$ be a morphism of manifolds.  Let $G\in \Orb(\cor_N)$. We
assume that $f$ is non-characteristic for $\SSo(G)$. Then $\SSo(\opb{f}G) \cup
\SSo(\epb{f}G) \subset f_d\opb{f_\pi}\SSo(G)$.
\end{proposition}
\begin{proof}
(i) The cases of $\opb{f}$ and $\epb{f}$ are similar and we only consider
$\opb{f}$.  We can write $f = p\circ i$ where $i\cl M \to M\times N$ is the
graph embedding and $p\cl M\times N \to N$ is the projection. Since the result
is compatible with the composition it is enough to consider the case of an
embedding and a submersion separately.

\medskip\noindent
(ii) We assume that $f$ is a submersion. Let $x \in M$ and set $y=f(x)$. Then
$$
\SSo(\opb{f}G)\cap T^*_xM
= \bigcap_{F'} \SSi(F')\cap T^*_xM 
\subset \bigcap_{G'} \SSi(\opb{f}G')\cap T^*_xM ,
$$
where $F'$ runs over the objects of $\Derb(\Cor_M)$ such that $F' \simeq
\opb{f}G$ in $\Orb(\cor_M)$ and $G'$ over the objects of $\Derb(\Cor_N)$ such
that $G' \simeq G$ in $\Orb(\cor_N)$.  Now the result follows from
Theorem~\ref{th:opboim}.

\medskip\noindent
(iii) We assume that $f$ is an embedding.  Let $(x_0;\xi_0) \in T^*M$ be such
that $(x_0;\xi_0) \not\in f_d(\SSo(G)\cap T^*_{x_0}N)$.  Let $l$ be the half
line $\rpos \cdot(x_0;\xi_0)$. Since $f$ is non-characteristic for $\SSo(G)$,
we have $\opb{f_d}(l) \cap \SSo(G) \subset \{x_0\}$.  By
Proposition~\ref{prop:repr-bonmicsup} there exist a neighborhood $V$ of $x_0$
and $G'\in \Derb(\Cor_V)$ such that $G' \simeq G$ in $\Orb(\cor_V)$ and
$\opb{f_d}(l) \cap \SSi(G') \subset \{x_0\}$. Then $\opb{f}G' \simeq \opb{f}G$
in $\Orb(\cor_{M\cap V})$ and $(x_0;\xi_0) \not\in f_d(\SSi(G')\cap
T^*_{x_0}N)$.  This proves $(x_0;\xi_0) \not\in \SSo(\opb{f}G)$, hence the
inclusion of the proposition.
\end{proof}

Let $E=\R^n$ be a vector space endowed with a norm $\|\cdot\|$.  We consider
coordinates $(x,t;\xi,\tau)$ on $T^*(E\times\R)$. For a given $a>0$ we let
$\gamma_a \subset E\times\R$ be the closed cone
\begin{equation}\label{eq:def-gammaa}
\gamma_a = \{(x,t);\; t\geq a \|x\| \}.
\end{equation}
Its polar cone is $\gamma_a^\circ =\{(\xi,\tau);\; \tau \geq \|\xi\|/a\}$.

\begin{lemma}\label{lem:im-dir-SSorb}
Let $F \in \Orb(\cor_{E\times\R})$. We assume that there exist a compact subset
$K$ of $E$ and $a>0$ such that $\SSo(F) \cap ((K\times \{0\}) \times
(\gamma_a^\circ \setminus \{0\})) = \emptyset$.  Then there exists $F' \in
\Derb(\Cor_{E\times\R})$ which represents $F$ such that $\SSi(F') \cap
((K\times \{0\}) \times (\gamma_a^\circ \setminus \{0\})) = \emptyset$.
\end{lemma}
\begin{proof}
(i) We set $M = E\times \R$.
Since microsupports are closed, the hypothesis holds with $a$ replaced by some
$b>a$.  By Proposition~\ref{prop:repr-bonmicsup}, for each $x\in K$ there exist
a neighborhood $U^x$ of $x$ and $F^x \in \Derb(\Cor_M)$ representing $F$ such
that $\SSi(F^x) \cap (U_x \times (\gamma_b^\circ \setminus \{0\})) =
\emptyset$.  We can find a finite number of points $x_i \in K$, $i \in I$, such
that $K \times \{0\} \subset \bigcup_{i\in I} U_{x_i}$.

Let $\varepsilon >0$ be given.  For $x\in E$ we let $C_x \subset E\times \R$ be
the truncated cone $C_x = ((x,\varepsilon) - \gamma_b) \cap (E \times
[-\varepsilon,\varepsilon])$.  We choose $\varepsilon > 0$ small enough such
that, for each $x\in K$, there exists $i\in I$ such that $C_x \subset
U_{x_i}$. We set $V = E \times \mo]-\varepsilon,+\infty[$, $W_x = \Int(C_x)$
and $W = \bigcup_{x\in K} W_x$.

\medskip\noindent
(ii) We choose $F_0 \in \Derb(\Cor_M)$ representing $F$.  We set $F' =
\cor_{\gamma_b\setminus \{0\}} \star (F_0)_V[1]$. Then we have the
distinguished triangle
\begin{equation}\label{eq:td_proj_cone_orb}
\cor_{\gamma_b} \star (F_0)_V \to (F_0)_V \to F' \to[+1].
\end{equation}
By Proposition~\ref{prop:cutoff} we have $\SSi(F') \cap (M \times
\Int(\gamma_b^{\circ})) = \emptyset$.  Hence $\SSi(F') \cap ((K\times \{0\})
\times (\gamma_a^\circ \setminus \{0\})) = \emptyset$.  In~(iii) we prove that
$F'|_W \simeq F|_W$ in $\Orb(\cor_W)$.  By Lemma~\ref{lem:modif-repr} we can
then extend $F'|_W$ to $F' \in \Derb(\Cor_{M})$ satisfying the conclusion of
the lemma.

\medskip\noindent
(iii) We define $\Phi \cl \Derb(\Cor_{M}) \to \Derb(\Cor_{M})$ by $\Phi(G) =
\cor_{\gamma_b} \star (G)_V$.  Then $\Phi$ induces a functor on
$\Orb(\cor_{M})$ such that $\Phi(Q_M(G)) \simeq Q_M(\Phi(G))$.

Since $W \subset V$, the assertion $F'|_W \simeq F|_W$ in $\Orb(\cor_W)$ that
we want to prove is equivalent to $Q_W(\Phi(F_0)|_W) \simeq 0$, by the
triangle~\eqref{eq:td_proj_cone_orb}.  By the definition of $W$ it is enough to
prove that $Q_{W_x}(\Phi(F_0)|_{W_x}) \simeq 0$ for any given $x\in K$.  Since
$F^{x_i}$ is a representative of $F$, for any $i\in I$, we have
$Q_{W_x}(\Phi(F_0)|_{W_x}) \simeq Q_{W_x}(\Phi(F^{x_i})|_{W_x})$. There exists
$i\in I$ such that $C_x \subset U_{x_i}$. Then $\SSi(F^{x_i}) \cap (C_x \times
(\gamma_b^\circ \setminus \{0\})) = \emptyset$ and Lemma~\ref{lem:easycutoff}
gives $\Phi(F^{x_i})|_{W_x} \simeq 0$. This implies $Q_{W_x}(\Phi(F_0)|_{W_x})
\simeq 0$, as required.
\end{proof}

\begin{proposition}\label{prop:im-dir-SSorb}
Let $f\cl M \to N$ be a morphism of manifolds.  Let $F\in \Orb(\cor_M)$.  We
assume that $f$ is proper on $\supporb(F)$. Then $\SSo(\reim{f}F)\subset
f_\pi\opb{f_d}\SSo(F)$.
\end{proposition}
\begin{proof}
(i) As in the proof of Proposition~\ref{prop:im-inv-SSorb} we can reduce the problem
to the cases where $f$ is an embedding or a projection. The case of an
embedding is similar to the part~(ii) of the proof of
Proposition~\ref{prop:im-inv-SSorb}.  We assume now that $M=M'\times N$ and $f$
is the projection.  Since the problem is local on $N$ we can assume that $N
\simeq \R^d$ for some $d$.  By first embedding $M'$ into some $\R^e$ and using
the embedding case we are reduced to the case of a projection $f \cl \R^n \to
\R^d$.

\medskip\noindent
(ii) We assume $f$ is the projection $M = M' \times N \to N$, with $M'=\R^e$ and
$N=\R^d$.

Let $q=(y;\eta) \in T^*N$ be such that $q\not\in f_\pi\opb{f_d}\SSo(F)$.  Let us
prove that $q\not\in \SSo(\reim{f}F)$. For this it is enough to give $F' \in
\Derb(\Cor_M)$ representing $F$ such that $f$ is proper on $\supp(F')$ and
$\SSi(F') \cap T^*_{M'}M' \times \{q\} = \emptyset$.

If $\eta = 0$, then $\supporb(F) \cap (M'\times\{y\}) = \emptyset$ by
Remark~\ref{rem:supporb-SSo}. We deduce by~\eqref{eq:supporb-OK} that $y
\not\in \supporb(\reim{f}F)$.  Hence $q\not\in \SSo(\reim{f}F)$.

We assume now $\eta\not= 0$.  We choose coordinates so that $y=(0,0)$ and $\eta
= (0,1)$.  We set $Z = \supporb(F) \cap \opb{f}(y)$.  Then $Z$ is compact and
we can find a cone $\gamma_a$ like~\eqref{eq:def-gammaa} such that $\dot\SSo(F)
\cap (Z \times \gamma_a^\circ) = \emptyset$.  We can also find $x \in \R^{e+d}$
such that $C = x-\gamma_a$ contains $Z$.  Then, in a neighborhood of $y$, we
have $\reim{f}F \simeq \reim{f}(F_C)$.

We set $K = C \cap (\R^{e+d-1}\times \{0\}) \cap \opb{f}(y)$.  By
Lemma~\ref{lem:im-dir-SSorb} there exists $F' \in \Derb(\Cor_{\R^{e+d}})$ which
represents $F$ such that $\dot\SSi(F') \cap ((K\times \{0\}) \times
\gamma^\circ) = \emptyset$.  Then $\SSi(F'_C)$ satisfies the same relation and
we deduce that $(y;\eta) \not\in \SSi(\reim{f}(F'_C))$.  Since $\reim{f}(F'_C)$
represents $\reim{f}F$ around $y$ this gives the result.
\end{proof}

\begin{proposition}\label{prop:SSo-tens-hom}
Let $F,G\in\Orb(\cor_M)$. 
\begin{itemize}
\item [(i)] We assume that $\SSo(F)\cap\SSo(G)^a\subset T^*_MM$. Then \\
  $\SSo(F\epstens_{\cor_M} G)\subset \SSo(F)+\SSo(G)$.
\item [(ii)] We assume that $\SSo(F)\cap\SSo(G)\subset T^*_MM$. Then \\
  $\SSo(\rhomeps(F,G))\subset \SSo(F)^a+\SSo(G)$. 
\end{itemize}
\end{proposition}
\begin{proof}
Let us prove~(i). Let $x_0 \in M$ and let $A , B \subset T^*_{x_0}M$ be conic
neighborhoods of $\SSo(F)\cap T^*_{x_0}M$, $\SSo(G)\cap T^*_{x_0}M$ such that
$A\cap B^a \subset \{x_0\}$.  By Proposition~\ref{prop:repr-bonmicsup} we can
find representatives $F', G' \in \Derb(\Cor_M)$ of $F,G$ such that $\SSi(F')\cap
T^*_{x_0}M \subset A$, $\SSi(G')\cap T^*_{x_0}M \subset B$.  Since microsupports
are closed we have $\SSi(F')\cap\SSi(G')^a\subset T^*_UU$ for some neighborhood
$U$ of $x_0$.  Then Corollary~\ref{cor:opboim} gives $\SSi(F'\epstens_{\cor_M}
G') \cap T^*_{x_0}M \subset A+B$.  Since $A$ and $B$ are arbitrarily close to
our microsupports we deduce~(i). The proof of~(ii) is the same.
\end{proof}

\subsection{Microsupport in the zero section}
In Proposition~\ref{prop:SSorb-sectnulle} we give a special case of
Theorem~\ref{th:opboim}~(iv) for $\SSo$.

\begin{lemma}\label{lem:recouvrt-cubes}
Let $C=[a,b]^d$ be a compact cube in $\R^d$ and let $\{U_i\}_{i\in I}$ be a
family of open subsets of $\R^d$ such that $C\subset \bigcup_{i\in I}
U_i$. Then there exists a finite family of open subsets $\{V_n\}$,
$n=1,\ldots,N$, such that
\begin{itemize}
\item [(i)] for each $n=1,\ldots,N$ there exists $i\in I$ such that $V_n
  \subset U_i$,
\item [(ii)] $C\subset \bigcup_{n=1}^N V_n$,
\item [(iii)] $V_n$ is contractible, for each $n=1,\ldots,N$,
\item [(iv)] $(\bigcup_{k=1}^{n} V_k) \cap V_{n+1}$ is contractible, for each
  $n=1,\ldots,N-1$.
\end{itemize}
\end{lemma}
\begin{proof}
For $x\in \R^d$ and $\varepsilon>0$ we set $C_x^\varepsilon =
x+\mo]-\varepsilon,\varepsilon[^d$.  We can choose $\varepsilon>0$ such that,
for any $x\in C$, there exists $i\in I$ satisfying $C_x^\varepsilon \subset
U_i$.  We let $x_n$, $n=1,\ldots,N$, be the points of the lattice $C\cap
(\varepsilon \Z)^d$ ordered by the lexicographic order of their coordinates.
Then the family $V_n = C_{x_n}^\varepsilon$, $n=1,\ldots,N$, satisfies the
required properties.
\end{proof}

\begin{proposition}\label{prop:SSorb-sectnulle}
Let $E=\R^d$ and $F\in \Orb(\cor_E)$.  We assume that $\SSo(F) \subset
T^*_EE$. Then there exists $A\in \Orb(\cor)$ such that $F\simeq A_E$.
\end{proposition}
\begin{proof}
(i) By Proposition~\ref{prop:repr-bonmicsup}, for any $x\in E$ there exists a
representative of $F$, say $F^x$, in $\Derb(\Cor_E)$ such that $\SSi(F^x) \cap
T^*_xE \subset \{x\}$.  Since microsupports are closed there exists an open
neighborhood of $x$, say $U_x$, such that $\SSi(F^x) \cap \dT^*U_x =
\emptyset$, that is, $F^x|_{U_x}$ is constant. In other words, there exists
$A^x\in \Derb(\Cor)$ such that $F^x|_{U_x} \simeq A^x_{U_x}$.  In particular we
have $F|_{U_x} \simeq A^x_{U_x}$ in $\Orb(\cor_{U_x})$.

\medskip\noindent
(ii-a) We set $x_0=0$ and $A = A^{x_0}$.  We also define $I_n = \mo]-n,n[^d$
for $n\in \N\setminus\{0\}$. Let us prove that $F|_{I_n} \simeq A_{I_n}$ in
$\Orb(\cor_{I_n})$ for any $n\geq 1$.

The family $\{U_x\}_{x\in \ol{I_n}}$ covers $\ol{I_n}$. Hence by
Lemma~\ref{lem:recouvrt-cubes} we can find a finite subcovering $\{V_k\}$,
$k=1,\ldots,K$, by contractible open subsets such that for each
$k=1,\ldots,K-1$ the set $(\bigcup_{j=1}^{k} V_j) \cap V_{k+1}$ is
contractible.  We set $W_k = \bigcup_{j=1}^{k} V_j$. In~(ii-b) we will prove by
induction on $k$ that $F|_{W_k} \simeq A' _{W_k}$ for some $A'\in
\Orb(\cor)$.  For $k=K$ we will have $F|_{I_n} \simeq A' _{I_n}$. Then
$F_{x_0} \simeq A' \simeq A$ and $F|_{I_n} \simeq A_{I_n}$, as claimed.

\medskip\noindent
(ii-b) Now we prove $F|_{W_k} \simeq A'_{W_k}$ for $k=1,\ldots,K$.  For $k=1$
we have $W_1=V_1 \subset U_{x_1}$ for some $x_1\in \ol{I_n}$.  Hence $F|_{W_1}
\simeq A^{x_1} _{W_1}$.  We set $A'=A^{x_1}$.  Let us assume $F|_{W_k} \simeq
A' _{W_k}$ for some $k$.  We set $Y_k = W_k \cap V_{k+1}$.  We choose $y$ such
that $V_{k+1} \subset U_y$ and we set $A''= A^y$.

Then $F|_{V_{k+1}} \simeq A''_{V_{k+1}}$.  Since $F_{Y_k} \simeq A'_{Y_k}\simeq
A''_{Y_k}$ we have $A' \simeq A''$.  Hence we have a Mayer-Vietoris triangle
$A'_{Y_k} \to[u] A'_{W_k}\oplus A'_{V_{k+1}} \to F|_{W_{k+1}} \to[+1]$ in
$\Orb(\cor_{W_{k+1}})$ where $u$ is of the form $(1,v)$ for some isomorphism $v
\in \Hom(A'_{Y_k},A'_{Y_k})$.
Since $Y_k$ and $V_{k+1}$ are contractible, Corollary~\ref{cor:morph-QRF-QRG}
gives
$$
\Hom(A'_{Y_k},A'_{Y_k}) \simeq \Hom(A',A')
 \simeq  \Hom(A'_{V_{k+1}},A'_{V_{k+1}}) .
$$
Hence $v$ can be extended to $V_{k+1}$ and we can define the commutative square
below:
$$
\xymatrix@C=1.5cm{
A'_{Y_k} \ar[r]^-{(1,v)} \ar@{=}[d]
& A'_{W_k}\oplus A'_{V_{k+1}} \ar[r] \ar[d]^{\left(\begin{smallmatrix}
    1 & 0 \\ 0 & v^{-1} \end{smallmatrix} \right) }_\wr
& F|_{W_{k+1}} \ar[r]^{+1}  \ar@{.>}[d] &  \\
A'_{Y_k} \ar[r]^-{(1,1)}
& A'_{W_k}\oplus A'_{V_{k+1}} \ar[r]
& A'_{W_{k+1}} \ar[r]^{+1} & \pointdiag
}
$$
We extend this square to an isomorphism of triangles and we obtain $F|_{W_{k+1}}
\simeq A'_{W_{k+1}}$, as required.

\medskip\noindent
(iii) As claimed in~(ii-a) we have $F_{I_n} \simeq A_{I_n}$ for all $n \in
\N\setminus\{0\}$.  We can assume that these isomorphisms are compatible with
the morphisms $i_n\cl (\cdot)_{I_n} \to (\cdot)_{I_{n+1}}$. Hence we obtain a
commutative square on the first two terms of the following triangles deduced
from~\eqref{eq:homot-lim}
$$
\xymatrix{
\oplus_n F_{I_n} \ar[r]^u\ar[d]^\wr & \oplus_n F_{I_n} \ar[r]\ar[d]^\wr
& F \ar[r]^{+1} \ar@{.>}[d] & \\
\oplus_n A_{I_n} \ar[r]^u & \oplus_n A_{I_n} \ar[r] & A_E \ar[r]^{+1} &,
}
$$
where the $n^{th}$-component of $u$ is $\id - i_n$. We extend this
square to an isomorphism of triangles and we see that $F \simeq A_E$.
\end{proof}

\part{The Kashiwara-Schapira stack}

In this part we define the Kashiwara-Schapira stack $\kss(\cor_\Lambda)$
quickly described in the introduction and we prove that it is equivalent to a
twisted stack of twisted local systems on $\Lambda$.  We prove it in the three
main steps described below.  For $p\in \Lambda$ we have the Lagrangian
subspaces of $T_pT^*M$ given by $\lambda_\Lambda(p) = T_p\Lambda$ and
$\lambda_0(p) = T_p\opb{\pi}\pi(p)$, where $\pi\cl T^*M \to M$ is the
projection.  Let $\sigma\cl \lag_M \to T^*M$ be the Lagrangian Grassmannian of
$T^*M$ and let $U_\Lambda$ be the open subset of $\lag_M|_\Lambda$ formed by
the $l$ which are transversal to $\lambda_\Lambda(p)$ and $\lambda_0(p)$, where
$p=\sigma(l)$.

Our first step is the definition of a functor
$$
m_\Lambda \cl \Derb_{(\Lambda)}(\cor_M) \to \Derb(\cor_{U_\Lambda}).
$$
We call $m_\Lambda(F)$  the sheaf of microlocal germs of $F$.
To explain the definition of $m_\Lambda(F)$ we first recall that
$\SSi(F)$ is the closure of the set of $p=(x;\xi)$ satisfying:
there exists $\varphi\cl M\to \R$ with $d\varphi_x =\xi$
such that $(\rsect_{\varphi\geq 0}F)_x \not= 0$.
We assume that $F\in \Derb_{(\Lambda)}(\cor_M)$ and $p\in \Lambda$.
 We set
$l= T_p\Lambda_\varphi$, where $\Lambda_\varphi = \{(x;d\varphi_x)\}$. It is
proved in~\cite{KS90} that, if $l \in U_\Lambda$, then
$(\rsect_{\varphi\geq 0}F)_x$ only depends on $l$.
We set $m_{\Lambda,l}(F) = (\rsect_{\varphi\geq 0}F)_x$. It is also proved
in~\cite{KS90} that all $m_{\Lambda,l}(F)$, for $l\in U_{\Lambda}$, are
isomorphic, up to a shift given by the Maslov index.
If $m_{\Lambda,l}(F)$ is concentrated in one degree, then $F$ is said pure
along $\Lambda$. If $m_{\Lambda,l}(F) \simeq \cor[d]$ for some $d\in\Z$, then
$F$ is said simple along $\Lambda$.
We prove that it is possible to define
$m_{\Lambda}(F) \in \Derb(\cor_{U_\Lambda})$ with stalks $m_{\Lambda,l}(F)$ and
with locally constant cohomology sheaves.
We let $\Dloc(\cor_{U_\Lambda})$ be the stack associated with the subprestack of
$\Derb(\cor_{U_\Lambda})$ formed by the complexes with 
locally constant cohomology sheaves. Then we obtain a functor
$\mks_\Lambda\cl \kss(\cor_\Lambda) \to \Dloc(\cor_{U_\Lambda})$.

The second step is to understand the link between the $\mks_\Lambda(F)|_U$ for
the different connected components $U$ of $U_\Lambda$.  We let $U_\Lambda^n$ be
the fiber product of $U_\Lambda$ over $\Lambda$, $n$ times. If $I\subset [1,n]$
is a set of indices, we let $q_I\cl U_\Lambda^n \to U_\Lambda^{|I|}$ be the
projection to the corresponding factors.
We introduce the Maslov sheaf of $\Lambda$,
$\shm_\Lambda \in \Dloc(\cor_{U_\Lambda^2})$, obtained as
$\shm_\Lambda = \mks_U(\shk_\Lambda)$, where $U$ is a neighborhood of the diagonal
of $\Lambda\times \Lambda$ and $\shk_\Lambda \in \kss(\cor_U)$ is a canonical
object defined on $U$.  For $(l,l') \in U_\Lambda^2$ we can see that
$(\shm_\Lambda)_{(l,l')} \simeq \cor[d]$, where $d \in \Z$ is given by some
Maslov index associated with $(l,l')$. Moreover we have isomorphisms
$$
\opb{\qq_{12}} \shm_\Lambda \dltens \opb{\qq_{23}} \shm_\Lambda \isoto
\opb{\qq_{13}} \shm_\Lambda
\text{ and }
\shm_\Lambda \dltens \opb{q_2} \mks_{\Lambda}(F) \isoto
\opb{q_1} \mks_{\Lambda}(F),
$$
for any $F \in \kss(\cor_{\Lambda})$, which satisfy natural commutative
diagrams.
In particular, for any $p\in \Lambda$ and any connected components
$U,V \subset U_\Lambda \cap\opb{\sigma}(p)$ the restriction
$\mks_\Lambda(F)|_U$ determines $\mks_\Lambda(F)|_V$.
We introduce the stack $\kss_{mg}(\cor_\Lambda)$ of pairs $(L,u)$,
where $L\in \Dloc(\cor_{U_\Lambda})$ and $u$ is an isomorphism
$\shm_\Lambda \dltens \opb{q_2} L \isoto \opb{q_1} L$ satisfying the
same diagram as $\mks_{\Lambda}(F)$.
Then we prove that $\mks_\Lambda$ induces an equivalence
$\mks'_\Lambda \cl \kss(\cor_\Lambda) \isoto \kss_{mg}(\cor_\Lambda)$.

The third step is a to give an equivalence between $\kss_{mg}(\cor_\Lambda)$
and the twisted stack $(\bigoplus_{i\in \Z}
\loceps(\cor_\Lambda)[i])_{m(\Lambda)}$ already defined in the introduction.
For this we embed $U_\Lambda$ in a fiber bundle $\shi_\Lambda$ and prove that,
up to shifts in the degrees, $\mks_\Lambda(F)$ extends as a local system on
$\shi_\Lambda$.

\section{Definition of the Kashiwara-Schapira stack}
\label{sec:KSstack}

We follow the notations of~\cite[\S 7.5]{KS90}.  Let $M$ be a manifold and
$\Lambda$ a locally closed conic Lagrangian submanifold of $\dT^*M$.  We
introduce the Kashiwara-Schapira stack of $\Lambda$.  The notion of stack used
here is that of ``sheaf of categories''.  We refer for example to~\cite[\S
19]{KS06}.  A prestack $\catc$ on a topological space $X$ consists of the data
of a category $\catc(U)$, for each open subset $U$ of $X$, restriction functors
$r_{V,U} \cl \catc(U) \to \catc(V)$, for $V\subset U$, and isomorphisms of
functors $r_{W,V} \circ r_{V,U} \simeq r_{W,U}$, for $W\subset V\subset U$,
satisfying compatibility conditions.

A stack is a prestack satisfying some gluing conditions.  In particular, if
$A,B \in \catc(U)$, then the presheaf $V \mapsto \Hom_{\catc(V)}(A|_V,B|_V)$ is
a sheaf on $U$.  Moreover, if $U = \bigcup_{i\in I} U_i$ and $A_i \in
\catc(U_i)$ are given objects with compatible isomorphisms between their
restrictions on the intersections $U_i \cap U_j$, then these objects glue into
an object of $\catc(U)$.

For any given prestack we can construct its associated stack, similar to the
associated sheaf of a presheaf.

\medskip 

We use the categories associated with a subset of $T^*M$ introduced in
Notation~\ref{not:micro_categories}.
\begin{definition}
\label{def:KSstack}
Let $\Lambda \subset T^*M$ be a locally closed conic subset.
We define a prestack $\kss^0_\Lambda$ on $\Lambda$ as follows.
Over an open subset $\Lambda_0$ of $\Lambda$ the objects of
$\kss^0_\Lambda(\Lambda_0)$ are those of $\Derb_{(\Lambda_0)}(\cor_M)$.
For $F,G \in \kss^0_\Lambda(\Lambda_0)$ we set
$$
\Hom_{\kss^0_\Lambda(\Lambda_0)}(F,G) \eqdot
\Hom_{\Derb(\cor_M;\Lambda_0)}(F,G) .
$$
We define the Kashiwara-Schapira stack of $\Lambda$ as the stack associated with
$\kss^0_\Lambda$. We denote it by $\kss(\cor_\Lambda)$ and, for
$\Lambda_0 \subset \Lambda$, we write $\kss(\cor_{\Lambda_0})$ instead of
$\kss(\cor_\Lambda)(\Lambda_0)$.

We denote by
$\kssfunc_\Lambda \cl \Derb_{(\Lambda)}(\cor_M) \to \kss(\cor_\Lambda)$ the
obvious functor. However, for $F \in \Derb_{(\Lambda)}(\cor_M)$, we often write
$F$ instead of $\kssfunc_\Lambda(F)$ if there is no risk of ambiguity.
\end{definition}

Several results in the next sections give links between $\kss(\cor_\Lambda)$
and stacks of the following type.

\begin{definition}
\label{def:Dloc}
Let $X$ be a topological space. We let $\Dloc^0(\cor_X)$ be the subprestack of
$U \mapsto \Derb(\cor_U)$, $U$ open in $X$, formed by the $F\in \Derb(\cor_U)$
with locally constant cohomologically sheaves.  We let $\Dloc(\cor_X)$ be the
stack associated with $\Dloc^0(\cor_X)$.
We denote by $\loc(\cor_X)$ the substack of $\Mod(\cor_X)$ formed by the
locally constant sheaves.
\end{definition}
We remark that $\Dloc(\cor_X)$ is only a stack of additive categories (the
triangulated structure is of course lost in the ``stackification'').  However
the cohomological functors $H^i \cl \Derb(\cor_U) \to \Mod(\cor_U)$ induce
functors of stacks $H^i \cl \Dloc(\cor_X) \to \loc(\cor_X)$ and the natural
embedding $\Mod(\cor_U) \hookrightarrow \Derb(\cor_U)$ induces $i \cl
\loc(\cor_X) \to \Dloc(\cor_X)$.  We have $H^0 \circ i \simeq
\id_{\loc(\cor_X)}$. Hence $i$ is faithful and $\loc(\cor_X)$ is a subcategory
of $\Dloc(\cor_X)$.

When $\Lambda$ is a locally closed conic Lagrangian submanifold of $\dT^*M$,
our main result on $\kss(\cor_\Lambda)$ is
Theorem~\ref{thm:KSstack=faisc_tordus} which says that it is equivalent to a
stack of twisted local systems. 

\medskip

For $\Omega \subset T^*M$, we recall the link between the sections of $\mu hom$
and the morphisms in $\Derb(\cor_M;\Omega)$. A morphism $u\cl F\to G$ in
$\Derb(\cor_M;\Omega)$ is represented by a triple $(F',s,u')$ with $F'\in
\Derb(\cor_M)$ and
\begin{equation*}
F\from[s] F' \to[u'] G
\end{equation*}
such that the $L$ defined (up to isomorphism) by the distinguished triangle
$F'\to[s] F \to L \to[+1]$ satisfies $\Omega \cap \SSi(L) = \emptyset$.
By~\eqref{eq:suppmuhom} we see that $s$ induces an isomorphism
$\mu hom(F,G)|_\Omega \isoto \mu hom(F',G) |_\Omega$. On the other
hand~\eqref{eq:proj_muhom_oim} gives a morphism
$\Hom(F',G) \to H^0(\Omega;\mu hom(F',G))$. Hence $u'$ induces an element in
$H^0(\Omega;\mu hom(F,G))$. Finally we obtain a well-defined morphism
\begin{equation}\label{eq:HomDkMOmega_vers_muhom}
\Hom_{\Derb(\cor_M;\Omega)}(F,G) \to H^0(\Omega;\mu hom(F,G)|_\Omega) .
\end{equation}
\begin{theorem}[Thm.~6.1.2 of~\cite{KS90}]
\label{thm:HomDkMp=muhomp}
If $\Omega = \{p\}$ for some $p\in T^*M$,
then~\eqref{eq:HomDkMOmega_vers_muhom} is an isomorphism.
\end{theorem}

For $F,G,H \in \Derb(\cor_M)$ we have a composition morphism
(see~\cite[Cor.~4.4.10]{KS90})
\begin{equation}\label{eq:comp_muhom}
  \mu hom(F,G) \ltens \mu hom(G,H) \to \mu hom(F,H) .
\end{equation}
It is compatible with the composition morphism for $\rhom$ through the
isomorphism $\rhom(F,G) \simeq \roim{\pi_M}\mu hom(F,G)$.
Hence for a given open subset $\Omega$ of $T^*M$ the morphism
induced by~\eqref{eq:comp_muhom} on the sections over $\Omega$ is compatible
with the composition in $\Derb(\cor_M;\Omega)$
through~\eqref{eq:HomDkMOmega_vers_muhom}.
\begin{notation}
\label{not:mucomposition}
For $F,G \in \Derb(\cor_M)$ and sections $a$ of $\mu hom(F,G)$ and $b$ of
$\mu hom(G,H)$, we denote by $b\mucirc a$ the image of $a\tens b$
by~\eqref{eq:comp_muhom}.
\end{notation}

It follows from Theorem~\eqref{thm:HomDkMp=muhomp} that the sheaf
associated with $\Omega \mapsto \Hom_{\Derb(\cor_M;\Omega)}(F,G)$ is
$H^0\mu hom(F,G)$, for given $F,G \in \Derb(\cor_M)$. We obtain
an alternative definition of $\kss(\cor_\Lambda)$:
\begin{corollary}\label{cor:defbiskss}
Let $\Lambda \subset T^*M$ be as in Definition~\ref{def:KSstack}.
We define a prestack $\kss^1_\Lambda$ on $\Lambda$ as follows.
Over an open subset $\Lambda_0$ of $\Lambda$ the objects of
$\kss^1_\Lambda(\Lambda_0)$ are those of $\Derb_{(\Lambda_0)}(\cor_M)$.
For $F,G \in \kss^1_\Lambda(\Lambda_0)$ we set
$\Hom_{\kss^1_\Lambda(\Lambda_0)}(F,G)
\eqdot H^0(\Lambda_0;\mu hom(F,G)|_{\Lambda_0})$.
The composition is induced by~\eqref{eq:comp_muhom}.
Then, the natural functor of prestacks $\kss^0_\Lambda \to \kss^1_\Lambda$
induces an isomorphism on the associated stacks.
\end{corollary}

\begin{remark}\label{rem:description_KSstack}
By Corollary~\ref{cor:defbiskss} an object of $\kss(\cor_\Lambda)$ is determined
by the data of an open covering $\{\Lambda_i\}_{i\in I}$ of $\Lambda$, objects
$F_i \in \Derb_{(\Lambda_i)}(\cor_M)$, for any $i\in I$, and sections $u_{ji}\in
H^0(\Lambda_{ij};\mu hom(F_i,F_j)|_{\Lambda_{ij}})$, for any $i,j\in I$, such that
\begin{itemize}
\item [(i)] $u_{ii}$ is induced by $\id_{F_i}$, for any $i\in I$,
\item [(ii)] $u_{kj} \mucirc u_{ji} = u_{ki}$, for any $i,j,k \in I$.
\end{itemize}
\end{remark}

\section{Simple sheaves}\label{sec:simpsheaves}

In this section we assume that $\Lambda$ is a locally closed conic Lagrangian
submanifold of $\dT^*M$.  We recall the definition of simple and pure sheaves
along $\Lambda$ and give some of their properties.

\subsection{Definition and first properties}
We first recall some notations from~\cite{KS90}.  For a function $\varphi\cl
M\to \R$ of class $C^\infty$ we define
\begin{equation}
\label{eq:def_Lambdaphi}
  \Lambda_\varphi = \{(x; d\varphi(x)); \; x\in M\}  .
\end{equation}
We notice that $\Lambda_\varphi$ is a closed Lagrangian submanifold of $T^*M$.
For a given point $p=(x;\xi) \in \Lambda \cap \Lambda_\varphi$ we have the
following Lagrangian subspaces of $T_p(T^*M)$
\begin{equation}\label{eq:def_lambdas_p}
\lambda_0(p) = T_p(T_x^*M), \qquad 
\lambda_\Lambda(p) = T_p\Lambda, \qquad
\lambda_\varphi(p) = T_p\Lambda_\varphi.
\end{equation}
We recall the definition of the inertia index (see for example \S A.3
in~\cite{KS90}).  Let $(E,\sigma)$ be a symplectic vector space and let
$\lambda_1, \lambda_2, \lambda_3$ be three Lagrangian subspaces of $E$. We
define a quadratic form $q$ on $\lambda_1\oplus \lambda_2\oplus \lambda_3$
by $q(x_1,x_2,x_3) = \sigma(x_1,x_2) + \sigma(x_2,x_3) + \sigma(x_3,x_1)$.
Then $\tau_E(\lambda_1, \lambda_2, \lambda_3)$ is defined as the signature
of $q$, that is, $p_+-p_-$, where $p_\pm$ is the number of $\pm 1$ in
a diagonal form of $q$.
We set
\begin{equation}\label{eq:def_tau}
\tau_{\varphi} = \tau_{p,\varphi} 
= \tau_{T_pT^*M}(\lambda_0(p), \lambda_\Lambda(p), \lambda_\varphi(p)).
\end{equation}

\begin{proposition}[Prop.~7.5.3 of~\cite{KS90}]
\label{prop:inv_microgerm}
Let $\varphi_0, \varphi_1\cl M\to \R$ be functions of class $C^\infty$, let
$p=(x;\xi) \in \Lambda$ and let $F\in \Derb_{(\Lambda)}(\cor_M)$.  We assume
that $\Lambda$ and $\Lambda_{\varphi_i}$ intersect transversally at $p$, for
$i=0,1$. 
Then $(\rsect_{\{ \varphi_1 \geq 0\}}(F))_x$ is isomorphic to
$(\rsect_{\{ \varphi_0 \geq 0\}}(F))_x
[\demi  (\tau_{\varphi_0} - \tau_{\varphi_1})]$.
\end{proposition}

\begin{definition}[Def.~7.5.4 of~\cite{KS90}]
\label{def:simple_pure}
In the situation of Proposition~\ref{prop:inv_microgerm} we say that $F$ is
pure at $p$ if $(\rsect_{\{ \varphi_0 \geq 0\}}(F))_x$ is concentrated in a
single degree, that is, $(\rsect_{\{ \varphi_0 \geq 0\}}(F))_x \simeq L[d]$,
for some $L\in \Mod(\cor)$ and $d\in \Z$. If moreover $L\simeq \cor$, we say
that $F$ is simple at $p$.

If $F$ is pure (resp. simple) at all points of $\Lambda$ we say that it is pure
(resp. simple) along $\Lambda$.
\end{definition}
We know from~\cite{KS90} that, if $\Lambda$ is connected and
$F\in \Derb_{(\Lambda)}(\cor_M)$ is pure at some $p\in \Lambda$, then $F$ is in
fact pure along $\Lambda$. Moreover the $L\in \Mod(\cor)$ in the above
definition is the same at every point.

\begin{example}
\label{ex:simple_sheaf}
The generic situation is easy. We consider the hypersurface
$X= \R^{n-1}\times \{0\}$ in $M=\R^n$. We let
$\Lambda = \{(\ul x, 0; 0, \xi_n); \xi_n>0\}$ be the ``positive'' half part of
$T^*_XM$. We set $Z= \R^{n-1}\times \rpos$.
Let $F\in \Derb_{\Lambda \cup T^*_MM}(\cor_M)$. Then, there exists
$L\in \Derb(\cor)$ such that the image of $F$ in the quotient category
$\Derb(\cor_M; \dT^*M)$ is isomorphic to $L_Z$.
\end{example}

For any $p\in \Lambda$ we can find an integral transform that sends a
neighborhood of $p$ in $\Lambda$ to the conormal bundle of a smooth
hypersurface.  Then, Theorem~7.2.1 of~\cite{KS90} reduces the general case to
Example~\ref{ex:simple_sheaf} and we can deduce:

\begin{lemma}\label{lem:simple_local}
Let $p=(x;\xi)$ be a given point of $\Lambda$. Then there exist a neighborhood
$\Lambda_0$ of $p$ in $\Lambda$ such that
\begin{itemize}
\item [(i)] there exists $F\in \Derb_{(\Lambda_0)}(\cor_M)$ which is
simple along $\Lambda_0$,
\item [(ii)] for any $G\in \Derb_{(\Lambda_0)}(\cor_M)$ there exist a
  neighborhood $\Omega$ of $\Lambda_0$ in $T^*M$ and an isomorphism
  $F\ltens L_M \isoto G$ in $\Derb(\cor_M;\Omega)$, where $L\in \Derb(\cor)$ is
  given by $L = \mu hom(F,G)_p$.
\end{itemize}
\end{lemma}

\begin{definition}\label{def:KSstacksimple}
Let $\Lambda \subset \dT^*M$ be a locally closed conic Lagrangian submanifold.
We let $\kss^p(\cor_\Lambda)$ (resp. $\kss^s(\cor_\Lambda)$) be the substack of
$\kss(\cor_\Lambda)$ formed by the pure (resp. simple) sheaves along $\Lambda$.
\end{definition}

Let $F,G\in \Derb_{(\Lambda)}(\cor_M)$. By~\eqref{eq:SSmuhom} we know that
$\mu hom(F,G)$ has locally constant cohomology sheaves on $\Lambda$
(see Lemma~\ref{lem:lagr-clean-inter} below).  Moreover,
for a given $p = (x;\xi) \in \Lambda$ we have
\begin{equation}\label{eq:stalk_muhom}
\mu hom(F,G)_p \simeq
\RHom( (\rsect_{\{ \varphi_0 \geq 0\}}(F))_x ,
(\rsect_{\{ \varphi_0 \geq 0\}}(G))_x ),
\end{equation}
where $\varphi_0$ is such that $\Lambda$ and $\Lambda_{\varphi_0}$ intersect
transversally at $p$ (see Proposition~\ref{prop:inv_microgerm}).  Hence, if $F$
and $G$ are simple along $\Lambda$ (or $F$ and $G$ are pure along $\Lambda$ and
$\cor$ is a field), then $\mu hom(F,G)$ is concentrated in one degree.

The functor $\kss^{0,opp}_\Lambda \times \kss^0_\Lambda \to
\Derb(\cor_\Lambda)$, $(F,G) \mapsto \mu hom(F,G)$ induces a functors of stacks
$$
\ol{\mu hom} \cl \kss^{opp}_\Lambda \times \kss_\Lambda \to \Dloc(\cor_\Lambda) .
$$
Lemma~\ref{lem:simple_local} implies the following result.

\begin{proposition}\label{prop:KSstack=Dloc}
Let $\Lambda \subset \dT^*M$ be a locally closed conic Lagrangian submanifold.
We assume  that there exists a simple sheaf $F\in \kss^s(\cor_\Lambda)$.
Then the functor $\ol{\mu hom}$ induces an equivalence of stacks
\begin{equation*}
\ol{\mu hom}(F,\cdot) \cl \kss(\cor_\Lambda) \isoto   \Dloc(\cor_\Lambda) ,
\qquad G \mapsto \ol{\mu hom}(F,G).
\end{equation*}
\end{proposition}

By Lemma~\ref{lem:simple_local} we know that simple sheaves exist locally around
a given point $p\in \Lambda$. When $\Lambda$ is in a good position we can see
that simple sheaves exist locally on the base, as follows.

\begin{lemma}\label{lem:simple_local_base}
Let $M$ be a manifold and let $\Lambda$ be a locally closed conic
Lagrangian submanifold of $\dT^*M$ such that the projection
$\Lambda/\R_{>0} \to M$ is finite. Let $p=(x;\xi) \in \Lambda$.
Then there exist a neighborhood $U$ of $x$ and $F\in \Derb(\cor_U)$ such that
$\dot\SSi(F) \subset \Lambda\cap T^*U$ and $F$ is simple along $\Lambda\cap
T^*U$.
\end{lemma}
\begin{proof}
(i) By hypothesis $\Lambda\cap T^*_xM$ consists of finitely many half-lines, say
$\rspos\cdot p_i$, with $p_i = (x;\xi_i)$, $i=1,\ldots,n$. Up to a restriction
to a neighborhood of $x$ we can assume that the $p_i$ belong to distinct
connected components of $\Lambda$, say $\Lambda_i$, $i=1,\ldots,n$. If $F_i$ is
simple along $\Lambda_i$, then the direct sum $\oplus_i F_i$ is simple along
$\Lambda$.  Hence we can assume that $\Lambda\cap T^*_xM = \rspos\cdot p$ for
some $p=(x;\xi)$.

\medskip\noindent
(ii) By Lemma~\ref{lem:simple_local} there exists a neighborhood $\Omega$ of $p$
in $T^*M$ and $F_0\in \Derb(\cor_M)$ such that $\SSi(F_0) \cap \Omega \subset
\Lambda$ and $F_0$ is simple along $\Lambda$ at $p$.  Up to shrinking $\Omega$ we
can assume that $T^*_xM \cap \Omega\cap \Lambda = \rspos\xi$.

We choose an open convex cone $V\subset T^*_xM$ such that $\xi\in V$, $\ol V$ is
proper and $\ol V \subset T^*_xM \cap \Omega$.  Hence $\ol V\cap \dot\SSi(F_0) =
\rspos\xi$.  In particular $W\eqdot V$ is a neighborhood of $\ol V\cap
\dot\SSi(F_0)$.  By Proposition~\ref{prop:cutoff-precis}, there exist $F\in
\Derb(\cor_M)$ and a distinguished triangle $F\to[u] F_0 \to G \to[+1]$ such
that $\SSi(G) \cap V =\emptyset$ and $T^*_xM \cap \dot\SSi(F) \subset W = V$.
Hence $T^*_xM \cap \dot\SSi(F) = \dot\SSi(F) \cap V = \dot\SSi(F_0) \cap V =
\rspos\xi$.

\medskip\noindent
(iii) Since microsupports are closed conic subsets there exists a conic
neighborhood $V_1$ of $V$ in $T^*M$ such that $\SSi(G) \cap V_1 =\emptyset$.
We can assume $V_1 \subset \Omega$.  Hence $\SSi(F) \cap V_1 = \SSi(F_0) \cap
V_1 \subset \Lambda \cap V_1$.

Since $\Lambda \setminus V_1$ is a closed subset of $\Lambda$ which does not
contain $p$, we can find a neighborhood $U_1$ of $x$ such that $\Lambda \cap
T^*U_1 \subset V_1$.  Then $(\SSi(F) \cap \dT^*U_1) \setminus \Lambda = (\SSi(F)
\cap \dT^*U_1) \setminus V_1$ is a closed subset of $\SSi(F) \cap \dT^*U_1$
which does not contain $p$.  Since $\dT^*_xM \cap \SSi(F) = \rspos p$, it
follows that $Z = \dot\pi_M((\SSi(F) \cap \dT^*U_1) \setminus \Lambda)$ is a
closed subset of $U_1$ which does not contain $x$.  We set $U = U_1 \setminus
Z$.  Then $F$ and $U$ satisfy the conclusion of the lemma.
\end{proof}

For a simple sheaf $F$ along $\Lambda$ we have seen that $\mu hom(F,F) \simeq
\cor_\Lambda$ in a neighborhood of $\Lambda$.  In the construction of a
quantization we will ``glue'' simple sheaves locally given on open subsets. For
this we need to understand $\mu hom(\rsect_UF,F')$ for simple sheaves $F,F'$
along $\Lambda$ and an open subset $U\subset M$.  This is done in
Proposition~\ref{prop:muhomFUF} below. We begin with some notations.

Let $U$ be an open subset of $M$ with smooth boundary.  We recall the notations
$N_U^*, N_U^{*e} \subset T^*M$ of~\eqref{eq:def_ext_con_bun} for the interior
and exterior conormal bundles of $\partial U$.  By Example~\ref{ex:microsupport}
we have $\SSi(\cor_U) = \ol{U} \times_M T^*_MM \cup N_U^{*e}$ and
$\SSi(\cor_{\ol{U}}) = \ol{U} \times_M T^*_MM \cup N_U^*$.

Let $\sigma\cl E\to M$ be a vector bundle.  For two given points $p=(x;e),
q=(x;f) \in E$ in the same fiber we define $V_p(q) \in T_pE$ by
\begin{equation}\label{eq:def-tgt-somme}
V_p(q) = (\partial / \partial t) (x; e+tf)|_{t=0} .
\end{equation}

\begin{lemma}\label{lem:tgt-somme}
Let $\sigma\cl E\to M$ be a vector bundle.
Let $\alpha \cl M \to E$ be a nowhere vanishing section of $\sigma$ and let
$\Lambda\subset E$ be a submanifold of $E$.  We assume that, for all $p=(x;e)
\in \Lambda$, we have $V_p(\alpha(x)) \not\in T_p\Lambda$.  Then the fiberwise
sum $\Lambda+ \R \cdot \alpha$ is a submanifold of $E$ in a neighborhood of
$\Lambda$ and we have, for all $p=(x;e) \in \Lambda$,
\begin{align*}
T_p(\Lambda+ \R \cdot \alpha) &= T_p\Lambda + \R\cdot V_p(\alpha(x)) ,  \\
C(\Lambda, \Lambda + \R_{\leq 0} \cdot\alpha)_p 
 &= T_p\Lambda + \R_{\geq 0} \cdot V_p(\alpha(x)) .
\end{align*}
\end{lemma}
\begin{proof}
Since the result is local we can assume $M = \R^n$, $E=\R^m$ and $\sigma$ is the
projection.  Then $\alpha$ determines a section $\theta$ of $TM$ over $\Lambda$
by $\theta(x;e) = \alpha(x)$ and $\Lambda+ \R \cdot \alpha = \{p+t\theta(p)$;
$p\in \Lambda$, $t\in\R\}$.  Then the result is a standard fact of differential
geometry.  (The change of sign from $\R_{\leq 0}$ from $\R_{\geq 0}$ is the last
formula is due to the convention taken in the definition of $C(A,B)$
in~\eqref{eq:form_cone2}.)
\end{proof}

The following result can be checked in local coordinates.  For a function $f$ on
$T^*M$ we let $X_f = H(df)$ be its Hamiltonian vector field.  We have $H(dx_i) =
-\partial/\partial \xi_i$ and $H(d\xi_i) = \partial/\partial x_i$.
\begin{lemma}\label{lem:con-bord-TU}
Let $\varphi \cl M \to \R$ be a function of class $C^1$ and let $X_{\varphi
  \circ \pi_M}$ be the Hamiltonian vector field of $\varphi \circ \pi_M \cl T^*M
\to \R$.

(i) Using the notation~\eqref{eq:def-tgt-somme} we have, for any $p=(x;\xi) \in
T^*M$, $V_p(d\varphi(x)) = - (X_{\varphi \circ \pi_M})_p$.

(ii) We define $U = \{x\in M;\, \varphi(x) >0\}$ and we assume that $\partial U$
is smooth.  Then, for $p \in \partial T^*U$, we have $H_p(N_{\dT^*U}^{*e}) =
\R_{\leq 0} \cdot (X_{\varphi \circ \pi_M})_p$, where $H_p \cl T^*_p T^*M \isoto
T_pT^*M$ is the Hamiltonian isomorphism.
\end{lemma}

\begin{proposition}\label{prop:muhomFUF}
Let $M$ be a manifold and let $\Lambda$ be a locally closed conic Lagrangian
submanifold of $\dT^*M$.  Let $F,F' \in \Derb(\cor_M)$ be such that $\SSi(F) =
\SSi(F') = \Lambda$.  Let $U \subset M$ be an open subset such that $\partial U$
is smooth. We assume
\begin{itemize}
\item [(i)]  $N_U^* \cap \Lambda = \emptyset$, respectively
$N_U^{*e} \cap \Lambda = \emptyset$,
\item [(ii)] $\dT^*_\Lambda \dT^*M \cap \dT^*_{\partial \dT^*U} \dT^*M =
  \emptyset$.
\end{itemize}
Then $\mu hom(F_U,F')|_{\dT^*M} \simeq (\mu hom(F,F'))_{\ol{T^*U}}|_{\dT^*M}$
and, respectively, \\
$\mu hom(\rsect_U F,F')|_{\dT^*M} \simeq (\mu hom(F,F'))_{T^*U}|_{\dT^*M}$.
\end{proposition}
\begin{proof}
(i) We first consider the case where $N_U^* \cap \Lambda = \emptyset$.
Since the problem is local we can assume that there exists a function
$\varphi \cl M \to \R$ such that $d\varphi$ does not vanish and $U = \{x\in
M;\, \varphi(x) > 0\}$.  The natural morphism $F_U \to F$ induces a morphism $u
\cl \mu hom(F,F') \to \mu hom(F_U,F')$.  We will apply
Lemma~\ref{lem:sousfaisceau} below, with $X = \dT^*M$, $Y =\Lambda$, $V =
\dT^*U$ and $\mu hom(F,F')$ instead of $F$, to see that $u|_X$ induces the
isomorphism of the proposition.  We have $\supp( \mu hom(F_U,F') ) \subset
\ol{V}$ and $u|_V$ is an isomorphism. Let us check that $\dot\SSi(
\muhom(F_U,F') ) \cap N_{\dT^*U}^{*e} = \emptyset$.  By
Corollary~\ref{cor:opboim} the hypothesis~(i) gives $\SSi(F_U) \subset
\Lambda + N_U^{*e}$.  Then the bound~\eqref{eq:SSmuhom} gives
\begin{align*}
-H(\SSi(\mu hom(F_U,F') )) \subset C(\SSi(F') , \SSi(F_U) )  
 \subset C(\Lambda, \Lambda + N_U^{*e})  .
\end{align*}
Hence by Lemma~\ref{lem:con-bord-TU}~(ii) it is enough to show that $(X_{\varphi
  \circ \pi_M})_p$ is not in $C(\Lambda, \Lambda + N_U^{*e})_p$, for all
$p=(x;\xi) \in \Lambda \cap \partial T^*U$.

We recall that $N_U^{*e} \subset \R_{\leq 0} \cdot d\varphi$. Through the
Hamiltonian isomorphism $H$ the hypothesis~(ii) is equivalent to $X_{\varphi
  \circ \pi_M} \not\in T\Lambda$. Hence we can use Lemma~\ref{lem:tgt-somme}
with $\alpha = d\varphi$ and we find (using also
Lemma~\ref{lem:con-bord-TU}~(i))
$$
C(\Lambda, \Lambda + N_U^{*e})_p 
\subset T_p\Lambda + \R_{\geq 0} \cdot V_p(d\varphi(x))
= T_p\Lambda + \R_{\leq 0} \cdot (X_{\varphi \circ \pi_M})_p  ,
$$
for any $p=(x;\xi) \in \Lambda$.  Hence $(X_{\varphi \circ \pi_M})_p \not\in
C(\Lambda, \Lambda + N_U^{*e})_p$, as required.

\medskip\noindent
(ii) Now we consider the case $N_U^{*e} \cap \Lambda = \emptyset$.  By
Corollary~\ref{cor:opboim} we have $\rsect_U F \simeq F_{\ol U}$.  We set $V = M
\setminus \ol{U}$.  We have distinguished triangles $F_V \to F \to F_{\ol U}
\to[+1]$ and $A_{T^*V} \to A \to A_{\ol{T^*U}} \to[+1]$, where $A = \mu
hom(F,F')$.  By part~(i) of the proof we have $\mu hom(F_V,F')|_{\dT^*M} \simeq
A_{T^*V}|_{\dT^*M}$ and we deduce $\mu hom(F_{\ol V},F')|_{\dT^*M} \simeq
A_{\ol{T^*V}}|_{\dT^*M}$, as required.
\end{proof}

\begin{lemma}\label{lem:sousfaisceau}
Let $X$ be a manifold and $Y$ a closed submanifold of $X$.  Let $F \in
\Derb(\cor_X)$ be such that $\SSi(F) \subset T^*_YX$.  Let $u\cl F \to F'$ be a
morphism in $\Derb(\cor_X)$ and let $V$ be an open subset of $X$ with smooth
boundary such that $\supp(F') \subset \ol{V}$, $u|_V$ is an isomorphism,
$\dT^*_YX \cap \dT^*_{\partial V}X = \emptyset$ and $\dot\SSi(F') \cap N_V^{*e}
= \emptyset$.  Then $F' \simeq F_{\ol V}$.
\end{lemma}
\begin{proof}
(i) Since $\supp(F') \subset \ol{V}$ the morphism $u$ induces $F_{\ol V} \to
F'_{\ol V} \simeq F'$.  We define $G$ by the distinguished triangle $F_{\ol V}
\to F' \to G \to[+1]$.  Since $u|_V$ is an isomorphism we have $\supp G
\subset \partial V$.  This implies $G \simeq i_!(G|_{\partial V})$, where $i
\cl \partial V \to X$ is the inclusion.  By Theorem~\ref{th:opboim}~(ii) we see
that $\SSi(G)$ contains $\supp(G) \times_M T^*_{\partial V}X$.

\medskip\noindent
(ii) Since $\dT^*_YX \cap N_V^{*e} = \emptyset$ we have $\dot\SSi(F_{\ol{V}})
\subset T^*_YX + N_V^*$ by Corollary~\ref{cor:opboim}. Hence
$\dot\SSi(F_{\ol{V}}) \cap N_V^{*e} = \emptyset$. By the hypothesis on
$\dot\SSi(F')$ we deduce $\dot\SSi(G) \cap N_V^{*e} = \emptyset$.  By the
result of~(i) we obtain $\supp(G) = \emptyset$, that is, $G\simeq 0$ and
$F_{\ol V} \isoto F'$.
\end{proof}

Now we check that the hypothesis~(i) and~(ii) of Proposition~\ref{prop:muhomFUF}
are generically satisfied. We prove an additional result which will be useful in
the proof of Lemma~\ref{lem:deform_bonne_pos}.

\begin{lemma}
\label{lem:hyp_muhomFUF_gen}
(i) Let $\Lambda$ be a closed conic Lagrangian submanifold of $\dT^*M$.  Let
$\varphi \cl M \to \R$ be a $\Cinf$ function. Then there exists a $\Cinf$
function $\phi \cl M\times \R^N \to \R$ such that, setting $\varphi_\varepsilon
= \phi(\cdot,\varepsilon)$ and $U_\varepsilon =
\opb{\varphi_\varepsilon}(\mo]0,+\infty[)$, we have $\varphi_0 = \varphi$ and
the set of $\varepsilon \in [-\delta,\delta]^N$ satisfying
$$
\text{$\partial U_\varepsilon$ is smooth, $T^*_{\partial U_\varepsilon}M \cap
  \Lambda = \emptyset$ and $\dT^*_\Lambda \dT^*M \cap \dT^*_{\partial
    \dT^*U_\varepsilon} \dT^*M = \emptyset$}
$$
has a complement of measure $0$ in $[-\delta,\delta]^N$, when $\delta>0$ is
small enough.

\smallskip\noindent
(ii) Let $\Lambda, \Lambda'$ be closed conic Lagrangian submanifolds of
$\dT^*M$.  We assume that $\Lambda/\rspos \to M$ is injective, that
$\Lambda'/\rspos \to M$ is finite and that $\pi_M(\Lambda) \cap \pi_M(\Lambda')$
is the closure of a submanifold of $M$ of codimension at least $2$.  Then we can
add the following properties in~(i): $\Lambda+ T^*_{\partial U_\varepsilon}M$ is
a Lagrangian manifold and $(\Lambda+ T^*_{\partial U_\varepsilon}M) \cap
\Lambda' = \emptyset$.

\smallskip\noindent
(iii) We assume that $M = M' \times \R$ and we take coordinates $(t;\tau)$ on
$T^*\R$.  If $\Lambda$ is transversal to the hypersurfaces $T^*M \times \R
\times \{\tau_0\}$, for all $\tau_0 \in \R$, then we can assume in~(i) and~(ii)
that $U_\varepsilon$ is a product $V \times \mo]a,b[$ for some $V \subset M'$
and $a,b \in \R$.
\end{lemma}
In particular we can find $U = U_\varepsilon$ as close as we want to $U_0$ which
satisfies~(i) and~(ii) of Proposition~\ref{prop:muhomFUF}.
\begin{proof}
(a) We choose $\Cinf$ functions $\varphi_1,\ldots,\varphi_N \cl M \to \R$ such
that the differentials $d\varphi_i$ generate $T^*M$ near $\opb{\varphi}(0)$.  We
define $\phi' \cl M\times \R^N \to \R$ by $\phi'(x,\varepsilon) = \varphi(x) +
\sum_{i=1}^N \varepsilon_i \, \varphi_i(x)$ and $\phi'_\varepsilon =
\phi'(\cdot,\varepsilon)$.  We define also $\Phi' \cl M \times \R^N \to T^*M$ by
$\Phi'(x,\varepsilon) = (x;d(\phi'_\varepsilon)(x)) = (x; d\varphi(x) +
\sum_{i=1}^N \varepsilon_i \, d\varphi_i(x))$.  Then $\Phi'$ is a submersion
near $\opb{\varphi}(0) \times \{0\}^N$, say over
$\opb{\varphi}([-\delta,\delta]) \times [-\delta,\delta]^N$ for some $\delta>0$.
In particular $\Phi'$ is transversal to $\Lambda$.  

\medskip\noindent
(b) We let $\Phi'_\varepsilon \cl M \to T^*M$ be the restriction
$\Phi'_\varepsilon = \Phi'(\cdot,\varepsilon)$.  By the transversality theorem
the subset $V \subset [-\delta,\delta]^N$ of $\varepsilon$ such that
$\Phi'_\varepsilon$ is transversal to $\Lambda$ has a complement of measure
$0$. For $\varepsilon \in V$ the intersection $\Lambda \cap
\Phi'_\varepsilon(M)$ is a discrete set of points, say $p_{\varepsilon,i}$,
$i\in I$. Hence, for $\eta \in\R$ outside a discret set, the boundary of
$U_{\varepsilon,\eta} \eqdot \opb{\phi_\varepsilon'}(\mo]\eta,+\infty[)$ does
not contain any $\pi_M(p_{\varepsilon,i})$.  This gives $T^*_{\partial
  U_{\varepsilon,\eta}} \cap \Lambda = \emptyset$.

\medskip\noindent
(c) For a given $p\in \Lambda$, the condition $(\dT^*_\Lambda \dT^*M)_p \cap
(\dT^*_{\partial \dT^*U_{\varepsilon,\eta}} \dT^*M)_p = \emptyset$ is equivalent
to $d(\phi'_\varepsilon \circ \pi_M)(p) \not\in (\dT^*_\Lambda \dT^*M)_p$, that
is, $(\phi'_\varepsilon \circ \pi_M)|_\Lambda$ is non singular at $p$.  Since
the set of singular values of any $\Cinf$ function is of measure $0$, the
condition $\dT^*_\Lambda \dT^*M \cap \dT^*_{\partial \dT^*U_{\varepsilon,\eta}}
\dT^*M = \emptyset$ is also satisfied for a generic $\eta$.

We define $\phi(x,\varepsilon,\eta) = \phi'(x,\varepsilon) - \eta$. Then $\phi$
satisfies the conclusions of~(i).

\medskip\noindent
(d) We set $\Lambda_1 =\eqdot \Lambda+ T^*_{\partial U_\varepsilon}M =
\{(x;\xi+t d(\phi'_\varepsilon)(x))$; $(x;\xi) \in \Lambda$, $t\in\R \}$, where
$t \in \R$ is well-defined because $\Lambda/\rspos \to M$ is injective.  Hence
multypling $t$ by $u>0$ gives an action of $\rspos$ on $\Lambda_1$, which
contracts $\Lambda_1$ on $\Lambda$. Since $\Lambda_1$ is a manifold near
$\Lambda$, by Lemma~\ref{lem:tgt-somme}, it is a manifold. We can check that it
is Lagrangian.

We set $Z = \pi_M(\Lambda) \cap \pi_M(\Lambda')$.  For $x\in Z$ we choose
generators $\xi$ of $\Lambda\cap T^*_xM$ and $\xi_i$, $i=1,\ldots,p$, of
$\Lambda'\cap T^*_xM$.  We define $P_x \subset T^*_xM$ as the finite union of
the planes $\R\xi \oplus \R\xi'_i$.  Then the condition $(\Lambda+ T^*_{\partial
  U_{\varepsilon,\eta}}M) \cap \Lambda' = \emptyset$ means that
$d(\phi'_\varepsilon)(x) \not\in P_x$, for each $x\in Z \cap \partial
U_{\varepsilon,\eta}$.  Since $Z$ is of codimension at least $2$ this is true
for a generic choice of $\varepsilon,\eta$ and we obtain~(ii).

\medskip\noindent
(e) By the hypothesis on $\Lambda$ it is enough to choose the functions
$\varphi_i$ independent of $t$ in part~(a) of the proof, such that the
differentials $d\varphi_i$ generate $T^*M'\times \{0\}$.  Then $\Phi'$ is no
longer a submersion but it is still transversal to $\Lambda$ and the same
arguments apply.
\end{proof}

\subsection{The canonical simple sheaf on the diagonal}
Let $\Lambda$ be a locally closed Lagrangian submanifold of $\dT^*M$.  One
question considered in the first part of this paper is to give conditions so
that $\kss(\cor_\Lambda)$ admits a global object. We prove in this paragraph
that $\kss(\cor_{\Lambda\times\Lambda^a})$ admits a canonical object defined on
some neighborhood of the diagonal
\begin{equation}\label{eq:de_Delta_lambda}
\Delta_\Lambda = \{(p,p^a) ; p\in  \Lambda\}.
\end{equation}

Let $X$ be a manifold and $Y,Z$ two submanifolds of $X$.  We recall that $Y$ and
$Z$ have a clean intersection if $W=Y\cap Z$ is a submanifold of $X$ and $TW =
TY \cap TZ$. This means that we can find local coordinates $(\ul x,\ul y, \ul
z,\ul w)$ such that $Y = \{ \ul x = \ul z =0\}$ and $Z = \{ \ul x = \ul y
=0\}$. Using these coordinates the following lemma is easy.

\begin{lemma}\label{lem:clean-cone}
Let $X$ be a manifold and $Y,Z$ two submanifolds of $X$ which have a clean
intersection. We set $W=Y\cap Z$. Then $C(Y,Z) = W\times_X TY + W\times_X TZ$.
\end{lemma}

\begin{lemma}\label{lem:lagr-clean-inter}
Let $X$ be a manifold and $\Lambda_1,\Lambda_2$ be two Lagrangian submanifolds
of $\dT^*X$. Let $F_1\in \Derb_{(\Lambda_1)}(\cor_X)$ and
$F_2\in \Derb_{(\Lambda_2)}(\cor_X)$.  We assume that $\Lambda_1$ and $\Lambda_2$
have a clean intersection and we set $\Xi= \Lambda_1 \cap \Lambda_2$.
Then there exists a neighborhood $U$ of $\Xi$ in $T^*X$ such that
$\SSi(\mu hom(F_1,F_2) |_U) \subset T^*_\Xi T^*X$, that is,
$\mu hom(F_1,F_2)|_U$ has locally constant cohomology sheaves on $\Xi$.
\end{lemma}
\begin{proof}
We have $\SSi(\mu hom(F_1,F_2) ) \subset -H^{-1}(C(\SSi(F_2) , \SSi(F_1) ))$
by the bound~\eqref{eq:SSmuhom}.
Let $U_i$ be a neighborhood of $\Lambda_i$ such that
$\SSi(F_i)\cap U_i  \subset \Lambda_i$, $i=1,2$.
Then $U = U_1 \cap U_2$ is a neighborhood of $\Xi$ and we have
$-H^{-1}(C(\SSi(F_2) , \SSi(F_1) )) \cap T^*U
 \subset -H^{-1}(C(\Lambda_2,\Lambda_1))$.

Since $\Lambda_i$ is Lagrangian we have
$-H^{-1}(T\Lambda_i) = T^*_{\Lambda_i}T^*X$, for $i=1,2$. In particular
$-H^{-1}( \Xi\times_{T^*X} T\Lambda_i) \subset T^*_\Xi T^*X$ and the result
follows from Lemma~\ref{lem:clean-cone}.
\end{proof}

Let $\Lambda_0$ be an open subset of $\Lambda$.  Let
$\Delta_M \subset M\times M$ be the diagonal.  Let
$F \in \Derb_{(\Lambda_0)}(\cor_M)$. Theorem~\ref{th:opboim}~(i) implies
$\DD'F \in \Derb_{(\Lambda_0^a)}(\cor_M)$.  We have $\Delta_{\Lambda_0} =
T^*_{\Delta_M}(M\times M) \cap (\Lambda_0\times \Lambda_0^a)$ and we deduce a
morphism
\begin{equation}
\label{eq:id_micro}
  \begin{split}
\Hom(F,F) & \simeq \Hom(\omega_{\Delta_M|M\times M},F\etens \DD'F) \\
&\to H^0(\Delta_{\Lambda_0};
\mu hom(\omega_{\Delta_M|M\times M},F\etens \DD'F) ).
  \end{split}
\end{equation}
We denote by $\delta_F \in H^0(\Delta_{\Lambda_0};
\mu hom(\omega_{\Delta_M|M\times M},F\etens \DD'F))$
the image of $\id_F$ by~\eqref{eq:id_micro}.

\begin{proposition}\label{prop:pre_Maslov_sheaf}
Let $\Lambda$ be a locally closed Lagrangian submanifold of $\dT^*M$. 
Let $\Lambda_0$ be an open subset of $\Lambda$.
Let $F,G,H \in \Derb_{(\Lambda_0)}(\cor_M)$. We assume that $F$ is simple
along $\Lambda_0$. Then there exists a unique
$$
\delta_{F,G} \in H^0(\Delta_{\Lambda_0}; \mu hom(F\etens \DD'F, G\etens \DD'G)),
$$
such that $\delta_F \mucirc \delta_{F,G} = \delta_G$ (where $\mucirc$ is defined
in Notation~\ref{not:mucomposition}).  If $G$ also is simple, then we have
$\delta_{F,G} \mucirc \delta_{G,H} = \delta_{F,H}$.
\end{proposition}
\begin{proof}
We set for short
$A_F \eqdot \mu hom(\omega_{\Delta_M|M\times M},F\etens \DD'F) )$
and $B \eqdot \mu hom(F\etens \DD'F, G\etens \DD'G)$.
Then~\eqref{eq:comp_muhom} gives a morphism $A_F \ltens B \to A_G$.

The intersection $\Delta_{\Lambda_0} = T^*_{\Delta_M}(M\times M) \cap
(\Lambda_0\times \Lambda_0^a)$ is clean. Hence by
Lemma~\ref{lem:lagr-clean-inter}, $A_F$ and $A_G$ are locally constant on
$\Delta_{\Lambda_0}$.  Since $F$ is simple we deduce by~\eqref{eq:stalk_muhom}
that, locally, $A_F \simeq \cor_{\Delta_{\Lambda_0}}$.  The same argument shows
that $B$ is locally constant on $\Lambda_0\times \Lambda_0^a$.
By~\eqref{eq:stalk_muhom} again the stalks of $A_G$ and $B$ are isomorphic.  It
follows that the composition $B_{\Delta_0} \to A_F \ltens B \to A_G$ given by $b
\mapsto \delta_F \tens b \mapsto \delta_F \mucirc b$ gives an isomorphism
$B_{\Delta_0} \isoto A_G$.  Then $\delta_{F,G}$ is the inverse image of
$\delta_G$ by this isomorphism.

The last formula follows from the unicity of $\delta_{F,G}$.
\end{proof}

\begin{corollary}
\label{cor:pre_Maslov_sheaf}
There exists a neighborhood $U$ of $\Delta_\Lambda$ in
$\Lambda\times \Lambda^a$ and $\shk_{\Delta_\Lambda}$ in
$\kss^s(\cor_U)$ such that,
for any open subset $\Lambda_0 \subset \Lambda$,
\begin{itemize}
\item [(i)] for any $F\in \Derb_{(\Lambda_0)}(\cor_M)$, there exist a
  neighborhood $V$ of $\Delta_{\Lambda_0}$ in $\Lambda\times \Lambda^a$ and a
  canonical morphism in $\kss^s(\cor_V)$:
$$
\gamma_F \cl \shk_{\Delta_\Lambda}|_V
\to (\kssfunc_{\Lambda_0\times \Lambda_0^a}(F\etens \DD'F))|_V ,
$$
which is an isomorphism as soon as $F$ is simple,
\item [(ii)] for $F, G\in \Derb_{(\Lambda_0)}(\cor_M)$ with $F$ simple along
  $\Lambda_0$, there exists a neighborhood $W$ of $\Delta_{\Lambda_0}$ in
  $\Lambda\times \Lambda^a$ such that
$\delta_{F,G}|_W \circ \gamma_F |_W = \gamma_G |_W$.
\end{itemize}
Moreover, for other $(U',\shk'_\Delta, \gamma'_F)$ satisfying (i) and (ii)
there exist a neighborhood $U_1$ of $\Delta_\Lambda$ in $\Lambda\times \Lambda^a$
and a unique isomorphism $\gamma\cl \shk_\Delta \to \shk'_\Delta$
in $\kss^s(\cor_{U_1})$ such that
$\gamma'_F|_{U_1} = \gamma \circ \gamma_F|_{U_1}$, for all $F$ as in~(i).
\end{corollary}
\begin{proof}
We can find a locally finite open covering
$\Lambda = \bigcup_{i\in I} \Lambda_i$ and $F_i \in \Derb_{(\Lambda_i)}(\cor_M)$
which is simple along $\Lambda_i$, for all $i\in I$. We set
$G_i =\kssfunc_{\Lambda_i\times \Lambda_i^a}(F_i\etens \DD'F_i)
\in \kss^s(\cor_{\Lambda_i\times \Lambda_i^a})$.
By Proposition~\ref{prop:pre_Maslov_sheaf}
and Corollary~\ref{cor:defbiskss}, for any 
$i,j\in I$, there exist a neighborhood $U^2_{ij}$ of $\Delta_{\Lambda_{ij}}$
in $\Lambda\times \Lambda^a$ and an isomorphism
$\delta_{ij}\cl G_i \isoto G_j$ in $\kss^s(\cor_{U^2_{ij}})$.
Moreover, for $i,j,k\in I$, there exists a neighborhood $U^3_{ijk}$ of
$\Delta_{\Lambda_{ijk}}$ in $\Lambda\times \Lambda^a$ such that
$\delta_{ik} = \delta_{jk} \circ \delta_{ij}$
in $\kss^s(\cor_{U^3_{ijk}})$.

Since the covering is locally finite we can find a neighborhood $U_i$ of
$\Delta_{\Lambda_i}$ in $\Lambda\times \Lambda^a$, for each $i\in I$, such that
$U_i\cap U_j \subset U^2_{ij}$ and $U_i\cap U_j \cap U_k \subset U^3_{ijk}$,
for all $i,j,k\in I$. Then, the $G_i$ glue into an object 
$\shk_{\Delta_\Lambda} \in \kss^s(\cor_{\bigcup_{i\in I}U_i})$.

Then (i) and (ii) follow from Proposition~\ref{prop:pre_Maslov_sheaf}.
The unicity follows easily from~(i) and~(ii).
\end{proof}

\subsection{Stalks of simple sheaves}

We prove that the stalks of a simple sheaf at a generic point are free.  Let $M$
be a manifold and let $\Lambda \subset \dT^*M$ be a smooth closed conic
Lagrangian submanifold. We set
\begin{align*}
Z_\Lambda = \{ &x\in \dot\pi_M(\Lambda);\;
\text{there exist a neighborhood $W$ of $x$ and a} \\
&\text{smooth hypersurface $S \subset W$ such that
$\Lambda \cap T^*W \subset T^*_SW$} \} .
\end{align*}

\begin{lemma}\label{lem:def_chemin}
Let $x,y\in M \setminus \dot\pi_M(\Lambda)$.
Let $I$ be an open interval containing $0$ and $1$.
Then there exists a $C^\infty$ embedding $c \cl I \to M$ such that $c(0) = x$,
$c(1) = y$ and $c([0,1])$ only meets $\dot\pi_M(\Lambda)$ at points of
$Z_\Lambda$, with a transverse intersection.
\end{lemma}
\begin{proof}
(i) Let $n$ be the dimension of $M$ and let $B \subset \R^{n-1}$ be the open
ball of radius $1$ and center $0$.  We choose a $C^\infty$ embedding $\gamma\cl
B\times I \to M$ such that $\gamma_0(0) = x$ and $\gamma_0(1) = y$, where
$\gamma_s \eqdot \gamma|_{\{s\} \times I}$ for $s\in B$.  For example we can
define $\gamma$ by integrating a vector field which admits an integral curve
from $x$ to $y$.  We also assume that $\gamma(B\times \{0\})$ and
$\gamma(B\times \{1\})$ do not meet $\dot\pi_M(\Lambda)$.

We choose a trivialization of the vector bundle $\gamma^*(T^*M)$. It gives a map
$\gamma'\cl B\times I \times \R^n \to T^*M$ and we set $\gamma'_s \eqdot
\gamma|_{\{s\} \times I\times \R^n}$. Then $\gamma'$ is an open embedding.  In
particular $\gamma'$ is transversal to $\Lambda$.

\medskip\noindent
(ii) By the transversality theorem there exists $s\in B$ such that $\gamma'_s$
is transversal to $\Lambda$.  In particular $\Lambda \cap \gamma'_s([0,1]\times
\R^n)$ consists of finitely many half lines, say $\rspos\cdot p_i$, $i=1,\ldots,
N$.  The transversality also implies that $V_i\eqdot T_{p_i}\Lambda \cap
T_{p_i}(\gamma'_s(I\times \R^n))$ is of dimension $1$.  We write $p_i =
\gamma'_s(t_i,v_i)$ and $x_i = \gamma_s(t_i)$.  Then $T^*_{x_i}M =
\gamma'_s(t_i,\R^n)$ and $T_{p_i}\Lambda \cap T_{p_i}T^*_{x_i}M$ is contained in
$V_i$, hence also of dimension $1$. This means that $\dot\pi_M|_{\Lambda} \cl
\Lambda \to M$ is of maximal rank $n-1$ at $p_i$.  Hence there exists a
neighborhood $\Omega_i$ of $p_i$ and a smooth hypersurface $S_i$ around $x_i$
such that $\Lambda \cap \Omega_i = T^*_{S_i}M \cap \Omega_i$.  If $x_i = x_j$
for $i\not= j$, then $S_i$ and $S_j$ meet transversally at $x_i$ and $S_i\cap
S_j$ is a submanifold of codimension $2$ in a neighborhood of $x_i$. Hence, by
deforming $\gamma_s$, we can assume moreover that $\gamma_s([0,1])$ avoids
$S_i\cap S_j$.  Then all $x_i$ are distinct and belong to $Z_\Lambda$.  We also
see that the intersection of $S_i$ and $\gamma_s(I)$ is transversal.  By joining
$x$ to $\gamma_s(0)$ and $y$ to $\gamma_s(1)$, we obtain the embedding $c$ of
the lemma.
\end{proof}

\begin{lemma}\label{lem:stalk_simple_sh}
We assume that $M$ is connected.  Let $F\in \Derb(\cor_M)$ be such that
$\dot\SSi(F) \subset \Lambda$ and $F$ is simple along $\Lambda$.  We set $U =
M\setminus \dot\pi_M(\Lambda)$.  We assume that there exists $x_0\in U$ such
that $H^iF_{x_0}$ is free of finite rank over $\cor$, for all $i\in \Z$.  Then
$H^iF_x$ is free of finite rank over $\cor$, for all $x\in U$ and all $i\in \Z$.
\end{lemma}
\begin{proof}
(i) Let $x\in U$ and let $I$ be an open interval containing $0$ and $1$.  
By Lemma~\ref{lem:def_chemin} we can choose a $C^\infty$ path
$\gamma\cl I \to M$ such that $\gamma(0) = x_0$, $\gamma(1) = x$ and
$\gamma([0,1])$ meets $\dot\pi_M(\Lambda)$ at finitely many points, all
contained in $Z_\Lambda$ and with a transversal intersection.
We denote these points by $\gamma(t_i)$, where $0< t_1< \cdots < t_k <1$.

\medskip\noindent
(ii) Since $F$ is locally constant on $U$, the stalk $F_{\gamma(t)}$ is constant
for $t\in \mo]t_i,t_{i+1}[$.  By Example~\ref{ex:simple_sheaf}, for $t_{i-1} < t
< t_i < u < t_{i+1}$, the stalks $F_{\gamma(t)}$ and $F_{\gamma(u)}$ differ by
$\cor[d_i]$, for some degree $d_i\in\Z$. Hence $H^iF_{\gamma(t)}$ is free of
finite rank over $\cor$, for all $i\in \Z$, if and ony if the same holds for
$F_{\gamma(u)}$.  The lemma follows.
\end{proof}

\section{Microlocal germs}
\label{sec:defmicrogerms}

We use the notations of~\cite[\S 7.5]{KS90}, in particular the
notations~\eqref{eq:def_lambdas_p} and~\eqref{eq:def_tau}.  Let $M$ be a
manifold of dimension $n$ and $\Lambda$ a locally closed conic Lagrangian
submanifold of $\dT^*M$. We let
\begin{equation}\label{eq:def_shl_M}
\sigma_{T^*M}\cl \lag_M \to T^*M
\end{equation}
be the fiber bundle of Lagrangian Grassmannian of $T^*M$. By definition the
fiber of $\lag_M$ over $p\in T^*M$ is the Grassmannian manifold of Lagrangian
subspaces of $T_pT^*M$. We let
\begin{equation}\label{eq:def_shl_M0}
\sigma_{T^*M}^0 \cl \lag_M^0 \to T^*M
\end{equation}
be the subbundle of $\lag_M$ whose fiber over $p\in T^*M$ is the set of
Lagrangian subspaces of $T_pT^*M$ which are transversal to $\lambda_0(p)$.
Then $\lag_M^0$ is an open subset of $\lag_M$. For a given $p\in T^*M$ we set
$V=T_{\pi_M(p)}M$ and we identify $T_pT^*M$ with $V\times V^*$.
We use coordinates $(\nu;\eta)$ on $T_pT^*M$.
Then we can see that any $l\in (\lag_M^0)_p$ is of the form
\begin{equation}\label{eq:LM0=matsym}
l = \{ (\nu;\eta)\in T_pT^*M;\; \eta = A\cdot \nu \},
\end{equation}
where $A \cl V \to V^*$ is a symmetric matrix.  This identifies the fiber
$(\lag_M^0)_p$ with the space of $n\times n$-symmetric matrices.

For a function $\varphi$ defined on a product $X\times Y$ and for a given
$x\in X$ we use the general notation $\varphi_x = \varphi|_{\{x\}\times Y}$.
\begin{lemma}\label{lem:exist-fcttest}
There exists a function $\varphi\cl \lag_M^0\times M \to \R$ of class
$C^\infty$ such that, for any $l\in \lag_M^0$ with $\sigma_{T^*M}(l) = (x;\xi)$,
$$
\varphi_l(x) = 0, \quad
d\varphi_l(x) = \xi, \quad
\lambda_{\varphi_l}( \sigma_{T^*M}(l)) = l.
$$
\end{lemma}
\begin{proof}
(i) We first assume that $M$ is the vector space $V=\R^n$. 
We identify $T^*M$ and $M\times V^*$. For $p=(x;\xi) \in M\times V^*$ the fiber
$(\lag_M^0)_p$ is identified with the space of quadratic forms on $V$
through~\eqref{eq:LM0=matsym}.  For $l\in (\lag_M^0)_p$ we let $q_l$ be the
corresponding quadratic form.  Now we define $\varphi_0$ by
$$
\varphi_0(l,y) = \langle y-x; \xi \rangle  + \pdemi\, q_l(y-x),
\quad \text{where $(x;\xi) = \sigma_{T^*M}^0(l)$.}
$$
We can check that $\varphi_0$ satisfies the conclusion of the lemma.

\medskip\noindent
(ii) In general we choose an embedding $i\cl M\hookrightarrow X\eqdot \R^N$.
For a given $p' =(x;\xi') \in M\times_X T^*X$ the subspace
$T_{p'}(M\times_X T^*X)$ of $T_{p'}T^*X$ is coisotropic.  The symplectic
reduction of $T_{p'}T^*X$ by $T_{p'}(M\times_X T^*X)$ is canonically identified
with $T_pT^*M$, where $p = i_d(p')$.  The symplectic reduction sends Lagrangian
subspaces to Lagrangian subspaces and we deduce a map, say
$r_{p'} \cl \lag_{X,p'} \to \lag_{M,p}$.
The restriction of $r_{p'}$ to the set of Lagrangian subspaces which are
transversal to $T_{p'}(M\times_X T^*X)$ is an actual morphism of manifolds.
In particular it induces a morphism $r^0_{p'} \cl \lag^0_{X,p'} \to \lag^0_{M,p}$.
We can see that $r^0_{p'}$ is onto and is a submersion.
When $p'$ runs over $M\times_X T^*X$ we obtain a surjective morphism of
bundles, say $r$:
$$
\xymatrix@R=.5cm{
\lag_X^0|_{M\times_X T^*X} \ar[r]^-r \ar[d] & \lag_M^0 \ar[d]\\
M\times_X T^*X \ar[r]^-{i_d} & T^*M. }
$$
We can see that $r$ is a fiber bundle, with fiber an affine space.
Hence we can find a section, say $j\cl \lag_M^0 \to \lag_X^0$.
For $(l,x) \in \lag_M^0\times M$ we set $\varphi(l,x) = \varphi_0(j(l), i(x))$,
where $\varphi_0$ is defined in~(i). Then $\varphi$ satisfies the conclusion
of the lemma.
\end{proof}
We come back to the Lagrangian submanifold $\Lambda$ of $\dT^*M$.  We let
\begin{equation}\label{eq:def_U_Lambda}
U_\Lambda \subset \lag_M^0|_\Lambda
\end{equation}
be the subset of $\lag_M^0|_\Lambda$ consisting of Lagrangian subspaces of
$T_pT^*M$ which are transversal to $\lambda_\Lambda(p)$. We define
$\sigma_\Lambda = \sigma_{T^*M}|_{U_\Lambda}$ and
$\tau_M = \pi_M|_\Lambda \circ \sigma_\Lambda$:
$$
\xymatrix{
U_\Lambda \ar[rr]^{\sigma_\Lambda} \ar[dr]_{\tau_M}
&&  \Lambda \ar[dl]^{\pi_M|_\Lambda}  \\
&M. }
$$
We note that $U_\Lambda$ is not a fiber bundle over $\Lambda$ but only an
open subset of $\lag_M^0|_\Lambda$. However, for a given $p\in \Lambda$, we
will use the notation
\begin{equation}\label{eq:def_U_Lambda_p}
U_{\Lambda,p} = \opb{\sigma_\Lambda}(p).
\end{equation}
We also introduce a notation for the graph of $\tau_M$ and the natural
``half-line bundle'' over it:
\begin{align}
\label{eq:ILambda}
I_\Lambda \subset U_\Lambda\times M , \;\;
I_\Lambda &=\{(l,\tau_M(l));\; l\in U_\Lambda \},  \\
\label{eq:JLambda}
J_\Lambda \subset \dT^*(U_\Lambda\times M) , \;\;
J_\Lambda &= \{ (l,x; 0, \lambda \xi);\;
(x;\xi) = \sigma_\Lambda(l), \lambda>0\}.
\end{align}

\begin{definition}\label{def:fcttest}
We let $\sht_\Lambda$ be the space of functions
$\varphi\cl U_\Lambda\times M \to \R$ of class $C^\infty$ such that, for any
$l\in U_\Lambda$ with $\sigma_\Lambda(l) = (x;\xi)$,
$$
\varphi_l(x) = 0, \quad
d\varphi_l(x) = \xi, \quad
\lambda_{\varphi_l}( \sigma_\Lambda(l)) = l.
$$
\end{definition}

\begin{lemma}\label{lem:fcttestLambda}
For any $\varphi \in \sht_\Lambda$ and any $l_0\in U_\Lambda$ we have
$$
\frac{\partial \varphi}{\partial l} (l_0,\tau_M(l_0)) = 0.
$$
\end{lemma}
\begin{proof}
For a given $l_0$ and $(x_0;\xi_0) = \sigma_\Lambda(l_0)$, we have the
transpose derivatives of $\tau_M, \sigma_\Lambda, \pi_M$:
\begin{gather*}
\tau_{M,d} \cl T^*_{x_0}M \to T^*_{l_0}(U_\Lambda), \qquad
\sigma_{\Lambda,d} \cl T^*_{(x_0;\xi_0)}\Lambda \to T^*_{l_0}(U_\Lambda),  \\
\pi_{M,d} \cl T^*_{x_0}M \to T^*_{(x_0;\xi_0)}T^*M , \qquad
(\pi_M|_\Lambda)_d \cl T^*_{x_0}M \to T^*_{(x_0;\xi_0)}\Lambda .
\end{gather*}
By definition we have $\varphi (l,\tau_M(l)) = 0$ for all $l \in U_\Lambda$.
By differentiation we obtain
$$
-\frac{\partial \varphi}{\partial l} (l_0,x_0)
=  \tau_{M,d}(\frac{\partial \varphi}{\partial x} (l_0,x_0) )
=  \tau_{M,d}(x_0; \xi_0 )
=  \sigma_{\Lambda,d} (\pi_M|_\Lambda)_d (x_0;\xi_0) .
$$
We remark that $\pi_{M,d} (x_0;\xi_0)$ is the Liouville $1$-form at
$(x_0;\xi_0)$. Since $\Lambda$ is conic Lagrangian it vanishes on $\Lambda$ and
$(\pi_M|_\Lambda)_d (x_0;\xi_0) =0$.
\end{proof}

\begin{lemma}\label{lem:clean-famille}
Let $X$ be a manifold and $p\cl E\to X$ a fiber bundle over $X$.  Let
$Y_1,Y_2\subset E$ be two submanifolds of $E$ and set $Y_3=Y_1\cap Y_2$.  We
assume that $Y_3$ is a submanifold, that $p|_{Y_3} \cl Y_3 \to X$ is a
submersion and that, for any $x\in X$, the submanifolds $Y_1\cap \opb{p}(x)$
and $Y_2\cap \opb{p}(x)$ have a clean intersection. Then $Y_1$ and $Y_2$ have a
clean intersection.
\end{lemma}
\begin{proof}
We have to check that $T_yY_3 = T_yY_1 \cap T_yY_2$, for all $y\in Y_3$.
Since this is a local problem we can write $E = X\times F$ and $y=(x,z)$.
We set $F_i = Y_i \cap \opb{p}(x) \subset F$. By hypothesis the projection
$T_yY_3 \to T_xX$ is onto, hence a fortiori $T_yY_i \to T_xX$, $i=1,2$.
We deduce $T_zF_i = T_yY_i \cap (\{0\} \times T_zF)$. 
Let $v = (v_x,v_z) \in T_xX \times T_zF$ be in $T_yY_1 \cap T_yY_2$.
Since $p|_{Y_3}$ is a submersion we can find $w_z\in  T_zF_3$ such that
$w=(v_x,w_z) \in T_yY_3$. Then $v-w=(0,v_z-w_z) \in T_zF_1 \cap T_zF_2$.
By hypothesis $v_z-w_z \in  T_zF_3$ and it follows that $v\in T_yY_3$.
This proves the lemma.
\end{proof}

For a function $\varphi\cl M\to \R$ of class $C^\infty$ we have introduced the
Lagrangian submanifold $\Lambda_\varphi$ in~\eqref{eq:def_Lambdaphi}. We also
define
\begin{equation}
\label{eq:def-Lambdaphip}
\Lambda'_\varphi = \{(x; \lambda \cdot d\varphi(x));\; x\in M, \;
 \lambda >0,\; \varphi(x) = 0\} .
\end{equation}
Since $\opb{\varphi}(0)$ is smooth at the points where $d\varphi$ is
not zero, $\Lambda'_\varphi \cap \dT^*M$ is a locally closed
conic Lagrangian submanifold of $\dT^*M$.

\begin{proposition}\label{prop:clean-inter}
For any $\varphi \in \sht_\Lambda$ there exists a neighborhood $V$ of $I_\Lambda$
(defined in~\eqref{eq:ILambda}) in $U_\Lambda\times M$ such that
\begin{itemize}
\item [(i)] $\SSi(\cor_{\opb{\varphi}([0,+\infty[)}) \cap \dT^*V = 
\Lambda'_\varphi \cap \dT^*V$ is a submanifold of $\dT^*V$,
\item [(ii)] $(T^*_{U_\Lambda}U_\Lambda \times \Lambda)
 \cap \Lambda'_\varphi \cap \dT^*V = J_\Lambda$, with $J_\Lambda$ given
in~\eqref{eq:JLambda},
\item [(iii)] $(T^*_{U_\Lambda}U_\Lambda \times \Lambda) \cap \dT^*V$
and $\Lambda'_\varphi \cap \dT^*V$ have a clean intersection.
\end{itemize}
\end{proposition}
\begin{proof}
(a) For a given $l\in U_\Lambda$, the manifolds $\Lambda_{\varphi_l}$ and $\Lambda$
have a transverse intersection at $\sigma_\Lambda(l) =(x_l;\xi_l)$. Hence we
can find a neighborhood $V_l$ of $x_l$ in $M$ such that:
\begin{itemize}
\item [(a-i)] $\SSi(\cor_{\opb{\varphi_l}([0,+\infty[)}) \cap \dT^*V_l = 
\Lambda'_{\varphi_l} \cap \dT^*V_l$  is a submanifold of $\dT^*V_l$,
\item [(a-ii)] $\Lambda \cap \Lambda'_{\varphi_l} \cap \dT^*V_l = \rspos \cdot
  \sigma_\Lambda(l)$,
\item [(a-iii)] $\Lambda \cap \dT^*V_l$ and $\Lambda'_{\varphi_l} \cap \dT^*V_l$
  have a clean intersection.
\end{itemize}
The assertion~(a-i) follows from Example~\ref{ex:microsupport}~(iii), the
assertions~(a-ii) and~(a-iii) from the transversality of $\Lambda_{\varphi_l}$ and
$\Lambda$.  We can also assume that
$V \eqdot \bigsqcup_{l\in U_\Lambda}\{l\} \times V_l$ is a neighborhood of
$I_\Lambda$ in $U_\Lambda \times M$.

\medskip\noindent
(b) Let $((l;0),(x;\xi)) \in
(T^*_{U_\Lambda}U_\Lambda \times \Lambda) \cap \dT^*V$.
If $((l;0),(x;\xi)) \in \Lambda'_{\varphi}$, then
$(x;\xi) \in \Lambda \cap \Lambda'_{\varphi_l} \cap \dT^*V_l$.
Hence, by~(a-ii), $((l;0),(x;\xi)) \in J_\Lambda$.
Conversely, Lemma~\ref{lem:fcttestLambda} implies
$J_\Lambda \subset T^*_{U_\Lambda}U_\Lambda \times \Lambda$
and we deduce~(ii).

Now~(iii) follows from (a-iii) and Lemma~\ref{lem:clean-famille} applied to
$E= T^*(U_\Lambda\times M)$, $X=U_\Lambda$, 
$Y_1 =(T^*_{U_\Lambda}U_\Lambda \times \Lambda)$ and
$Y_2 = \Lambda'_\varphi \cap \dT^*V$.
\end{proof}

\begin{theorem}\label{thm:SSmicrogerm}
Let $\varphi \in \sht_\Lambda$ and let $F\in \Derb_{(\Lambda)}(\cor_M)$.
Let $q_1\cl U_\Lambda \times M \to U_\Lambda$ and
$q_2\cl U_\Lambda \times M \to M$  be the projections. We set
\begin{alignat*}{2}
\shm_{\varphi,F} &=  \mu hom (\cor_{\opb{\varphi}([0,+\infty[)},\opb{q_2}F) 
&\quad  &\in \Derb(\cor_{T^*(U_\Lambda \times M)}), \\
\shn_{\varphi,F} &= (\rsect_{\opb{\varphi}([0,+\infty[)}(\opb{q_2}F))_{I_\Lambda}
&\quad  &\in \Derb(\cor_{U_\Lambda \times M}).
\end{alignat*}
Then there exists a neighborhood $V$ of $I_\Lambda$ in $U_\Lambda \times M$
such that
\begin{itemize}
\item [(i)] $\dT^*V \cap \supp ( \shm_{\varphi,F} ) \subset J_\Lambda$
and $\SSi(\shm_{\varphi,F}|_{\dT^*V})  \subset
T^*_{J_\Lambda}T^*(U_\Lambda\times M)$,
\item [(ii)] $\roim{\dot\pi_{V}{}}( \shm_{\varphi,F} |_ {\dT^*V})
\simeq \shn_{\varphi,F}$,
\item [(iii)] $\SSi(\roim{q_1} \shn_{\varphi,F}) \subset  T^*_{U_\Lambda}U_\Lambda$,
\item [(iv)] for any $l\in U_\Lambda$ we have
$(\roim{q_1} \shn_{\varphi,F})_l \simeq (\rsect_{\opb{\varphi_l}([0,+\infty[)}(F))_x$,
where $x= \tau_M(l)$.
\end{itemize}
\end{theorem}
\begin{proof}
(i) We take the neighborhood $V$ given by Proposition~\ref{prop:clean-inter}.
Then the result follows from Proposition~\ref{prop:clean-inter}
and Lemma~\ref{lem:lagr-clean-inter}.

\medskip\noindent
(ii) Sato's triangle~\eqref{eq:SatoDTmuhom} gives
$$
\DD'(\cor_{\opb{\varphi}([0,+\infty[)}) \tens \opb{q_2}F \tens \cor_{I_\Lambda}
\to \shn_{\varphi,F} \to
\roim{\dot\pi_{U_\Lambda \times M}{}}( \shm_{\varphi,F} )_{I_\Lambda}
\to[+1].
$$
By definition $d\varphi$ does not vanish in a neighborhood of $I_\Lambda$.
Hence $\opb{\varphi}(0)$ is a smooth hypersurface near $I_\Lambda$ and
$\DD'(\cor_{\opb{\varphi}([0,+\infty[)}) \simeq \cor_{\opb{\varphi}(]0,+\infty[)}$.
Since $I_\Lambda \subset \opb{\varphi}(0)$, the first term of the above triangle
is zero. By~(i) the support of
$\roim{\dot\pi_{V}{}}( \shm_{\varphi,F} |_ {\dT^*V})$ is already contained
in $I_\Lambda$. So we can forget the subscript $I_\Lambda$ in the third term and
we obtain~(ii).

\medskip\noindent
(iii) By (i) the cohomology sheaves of $\shm_{\varphi,F}$ are locally constant
sheaves on $J_\Lambda$. Since $J_\Lambda$ is a fiber bundle over $I_\Lambda$
with fiber $\rspos$ we deduce from~(ii) that $\shn_{\varphi,F}$ has
locally constant sheaves cohomology sheaves on $I_\Lambda$.
Since $q_1$ induces an isomorphism $I_\Lambda \isoto U_\Lambda$ we obtain~(iii).

\medskip\noindent
(iv) We first prove that
$i_l\cl \{l\}\times M \hookrightarrow U_\Lambda\times M$ is non-characteristic
for $\rsect_{\opb{\varphi}([0,+\infty[)}(\opb{q_2}F)$ in a neighborhood of
$x$. We use the bound in Theorem~\ref{thm:SSrhom}.
We set for short $A = T^*_{U_\Lambda}U_\Lambda \times \SSi(F)$ and
$B = T^*_{U_\Lambda}U_\Lambda \times \Lambda$.
Since microsupports are closed subsets it is enough to prove
\begin{itemize}
\item [(a)]\ \  $(A+\Lambda'_\varphi)  \cap 
(T^*_lU_\Lambda \times T^*_{M,x}M) \subset \{(l,x;0,0)\}$,
\item [(b)] \ \ $\dot\pi_{\pi} \opb{ \dot\pi_{d} }(-H^{-1}C(A,\Lambda'_\varphi))
 \cap  (T^*_lU_\Lambda \times T^*_{M,x}M) \subset \{(l,x;0,0)\}$.
\end{itemize}
Since $A\subset T^*_{U_\Lambda}U_\Lambda \times T^*M$ and
$\Lambda'_\varphi \cap T^*_{(l,x)}(U_\Lambda \times M) \subset
T^*_{U_\Lambda}U_\Lambda \times T^*M$, by Lemma~\ref{lem:fcttestLambda},
the statement (a) is clear.  Since $A\subset B$ in some neighborhood of
$A \cap\Lambda'_\varphi$, we may replace $A$ by $B$ in (b).
We have seen that $-H^{-1}C(B,\Lambda'_\varphi) \subset
T^*_{J_\Lambda}T^*(U_\Lambda \times \Lambda)$.
Hence $\dot\pi_{\pi} \opb{ \dot\pi_{d} } (-H^{-1}C(B,\Lambda'_\varphi))
\subset T^*_{I_\Lambda}(U_\Lambda \times \Lambda)$
and this gives (b).

Now the non-characteristicity implies
$$
\opb{i_l} \shn_{\varphi,F} \simeq 
(\epb{i_l} \rsect_{\opb{\varphi}([0,+\infty[)}(\opb{q_2}F))_x
 \tens \omega_{\{l\}|U_\Lambda}^{\tens-1}
\simeq (\rsect_{\opb{\varphi_l}([0,+\infty[)}(F))_x
$$
and we deduce (iv).
\end{proof}

For $\varphi\in \sht_\Lambda$ we define a functor, using the notations of
Theorem~\ref{thm:SSmicrogerm},
\begin{align*}
m_\Lambda^\varphi \cl \Derb_{(\Lambda)}(\cor_M) 
& \to \Derb(\cor_{U_\Lambda})   \\
F &\mapsto
\roim{q_1} (\rsect_{\opb{\varphi}([0,+\infty[)}(\opb{q_2}F)_{I_\Lambda}) .
\end{align*}
By Theorem~\ref{thm:SSmicrogerm} the cohomology sheaves of
$m_\Lambda^\varphi(F)$ are locally constant sheaves on $U_\Lambda$
and we have
$(m_\Lambda^\varphi(F))_l \simeq (\rsect_{\opb{\varphi_l}([0,+\infty[)}(F))_x$,
for any $l\in U_\Lambda$ and $x= \tau_M(l)$.

Let $\varphi_0, \varphi_1\in \sht_\Lambda$.
We define $\varphi\cl U_\Lambda\times M \times \R \to \R$ 
by $\varphi(l,x,t) = t\varphi_0(l,x) +(1-t) \varphi_1(l,x)$.
We let $q_{13} \cl U_\Lambda\times M \times \R \to U_\Lambda\times \R$
and $q_2 \cl U_\Lambda\times M \times \R \to M$ be the projections
and we define a functor
\begin{align*}
m_\Lambda^{\varphi_0,\varphi_1} \cl \Derb_{(\Lambda)}(\cor_M) 
& \to \Derb(\cor_{U_\Lambda\times \R})   \\
F &\mapsto
\roim{q_{13}} (\rsect_{\opb{\varphi}([0,+\infty[)}(\opb{q_2}F)_{I_\Lambda\times\R}) .
\end{align*}
Theorem~\ref{thm:SSmicrogerm} works as well with the parameter $t$
and we obtain:
\begin{lemma}\label{lem:mLambda_indpt_phi}
Let $\varphi_0, \varphi_1\in \sht_\Lambda$.
For $t\in \R$ we let
$i_t \cl U_\Lambda\times \{t\} \hookrightarrow U_\Lambda\times \R$
be the inclusion.
Then, for any $F\in \Derb_{(\Lambda)}(\cor_M)$, the cohomology sheaves of
$m_\Lambda^{\varphi_0,\varphi_1}(F)$ are locally constant sheaves on
$U_\Lambda\times\R$ and we have natural isomorphisms
$\opb{i_t} m_\Lambda^{\varphi_0,\varphi_1}(F) \simeq m_\Lambda^{\varphi_t}(F)$,
for all $t\in \R$.
In particular we have a canonical isomorphism
$m_\Lambda^{\varphi_0}(F) \simeq m_\Lambda^{\varphi_1}(F)$.
\end{lemma}
By this lemma the following definition is meaningful.
\begin{definition}\label{def:micro-germ}
Let $\Lambda$ be a locally closed conic Lagrangian submanifold of $T^*M$.
We let $m_\Lambda \cl \Derb_{(\Lambda)}(\cor_M) \to
\Derb(\cor_{U_\Lambda})$ be the functor $m_\Lambda^\varphi$ for an arbitrary
$\varphi\in \sht_\Lambda$. For a given $l\in U_\Lambda$ and
$F\in \Derb_{(\Lambda)}(\cor_M)$ we set
$m_{\Lambda,l}(F) = (m_\Lambda(F))_l$ and call it the microlocal germ of $F$ at
$l$.
\end{definition}

\begin{proposition}
\label{prop:microgerm_via_kss}
Let $\Lambda$ be a locally closed conic Lagrangian submanifold of $T^*M$.
Then the functors $m_{\Lambda_0} \cl \Derb_{(\Lambda_0)}(\cor_M) \to
\Derb(\cor_{U_{\Lambda_0}})$, where $\Lambda_0$ runs over the open subsets
of $\Lambda$, induce a functor of stacks
$$
\mks_\Lambda \cl \kss(\cor_\Lambda)
 \to \oim{\sigma_\Lambda}(\Dloc(\cor_{U_\Lambda})).
$$
In particular, for $F,G \in \Derb_{(\Lambda)}(\cor_M)$, we have a canonical
morphism 
\begin{equation}\label{eq:muhom_microgerm}
\opb{\sigma_\Lambda} H^0\mu hom(F,G)
\isoto H^0\rhom(m_\Lambda(F), m_\Lambda(G)) ,
\end{equation}
which is actually an isomorphism.
\end{proposition}
\begin{proof}
(i) Let $\Lambda_0$ be an open subset of $\Lambda$.
Let $F\in \Derb_{(\Lambda_0)}(\cor_M)$ be such that
$\SSi(F)\cap \Lambda_0 = \emptyset$. Then, by~(iv) of
Theorem~\ref{thm:SSmicrogerm}, we have $m_{\Lambda_0}(F) = 0$.
Hence the functor $m_{\Lambda_0}$
factorizes through $\Derb(\cor_M;\Lambda_0)$.  On the other hand, by~(iii) of
Theorem~\ref{thm:SSmicrogerm}, this functor takes value in the subcategory
of $\Derb(\cor_{U_{\Lambda_0}})$ of objects with locally constant cohomology
sheaves.  Hence we obtain a functor
$\mks^0_\Lambda \cl \kss^0_\Lambda \to \oim{\sigma_\Lambda}\Dloc^0(\cor_{U_\Lambda})$
between the prestacks of Definitions~\ref{def:KSstack} and~\ref{def:Dloc}.
We deduce $\mks_\Lambda$ as the composition of $(\mks^0_\Lambda)^a$
with the natural functor $(\oim{\sigma_\Lambda}\Dloc^0(\cor_{U_\Lambda}))^a
\to \oim{\sigma_\Lambda}\Dloc(\cor_{U_\Lambda})$, where $(\cdot)^a$ denotes
the associated stack.

\medskip\noindent
(ii) By Corollary~\ref{cor:defbiskss} the $\hom$ sheaf in the stack
$\kss(\cor_\Lambda)$ is $H^0\mu hom(\cdot,\cdot)$.
It follows from Definition~\ref{def:Dloc} that the 
$\hom$ sheaf in the stack $\Dloc(\cor_X)$ is $H^0\rhom(\cdot,\cdot)$.
This gives the morphism~\eqref{eq:muhom_microgerm}.
It is an isomorphism by~\eqref{eq:stalk_muhom}.
\end{proof}

Now let $M'$ be another manifold and $\Lambda'$ be a locally closed conic
Lagrangian submanifold of $T^*M'$. We have the obvious embedding
$i_{\Lambda,\Lambda'} \cl U_\Lambda \times U_{\Lambda'} \hookrightarrow
U_{\Lambda\times \Lambda'}$, $(l,l') \mapsto l\oplus l'$.
Proposition~7.5.10 of~\cite{KS90} gives
\begin{proposition}\label{prop:microgerm_prod}
There exists an isomorphism of functors \\
$\opb{i_{\Lambda,\Lambda'}} \circ \mks_{\Lambda\times \Lambda'}
\simeq \mks_\Lambda \etens \mks_{\Lambda'}$.
\end{proposition}

\section{The Maslov sheaf}

The functor $m_\Lambda \cl \Derb_{(\Lambda)}(\cor_M) \to
\Derb(\cor_{U_\Lambda})$ of Definition~\ref{def:micro-germ} send pure sheaves
along $\Lambda$ to local systems (with shifts in cohomological degrees) on
$U_\Lambda$.  When we apply this functor to the canonical object
$\shk_{\Delta_\Lambda}$ constructed in Corollary~\ref{cor:pre_Maslov_sheaf}, we
obtain a canonical local system on $\tU_\Lambda = \Delta_\Lambda
\times_{\Lambda\times \Lambda^a} U_{\Lambda\times \Lambda^a}$. We call it the
Maslov sheaf of $\Lambda$ because the degrees where it is concentrated is
related with the Maslov index.

In Proposition~\ref{prop:action_Maslov_sheaf} the Maslov sheaf is used to give
conditions on the locally constant objects of $\Derb(\cor_{U_\Lambda})$ which
are of the form $m_\Lambda(F)$.  This gives another description of
$\kss(\cor_\Lambda)$ (see Theorem~\ref{thm:mksprime_equiv}).

\subsection{Definition of the Maslov sheaf}
Let $\Lambda$ be a locally closed conic Lagrangian submanifold of $T^*M$.
By Corollary~\ref{cor:pre_Maslov_sheaf} we have a canonical object
$\shk_{\Delta_\Lambda}$ in $\kss^s(\cor_U)$, where $U$ is
a neighborhood of $\Delta_\Lambda$ in $\Lambda\times \Lambda^a$.

We set $\tU_\Lambda = 
\Delta_\Lambda \times_{\Lambda\times \Lambda^a} U_{\Lambda\times \Lambda^a}$.
The antipodal map $(\cdot)^a \cl T^*M \to T^*M$ induces a morphism on the 
Grassmannian of Lagrangian subspaces of $T_pT^*M$ and $T_{p^a}T^*M$.
We also denote this morphism by $(\cdot)^a \cl \lag_{M,p} \to \lag_{M,p^a}$,
$l\mapsto l^a$.
We define an embedding $\upsilon_\Lambda \cl U_\Lambda \times_\Lambda U_\Lambda
\to \tU_\Lambda$, $(l_1,l_2) \mapsto l_1\oplus l_2^a$.
\begin{definition}
\label{def:Maslovsheaf}
We define the Maslov sheaf $\tshm_\Lambda\in \Dloc(\cor_{\tU_\Lambda})$ by
$\tshm_\Lambda = 
\mks_{\Lambda\times\Lambda^a}(\shk_{\Delta_\Lambda})|_{\tU_\Lambda}$.  We also
set $\shm_\Lambda = \opb{\upsilon_\Lambda}(\tshm_\Lambda) \in
\Dloc(\cor_{U_\Lambda \times_\Lambda U_\Lambda})$.
\end{definition}

By Corollary~\ref{cor:pre_Maslov_sheaf}, for any open subset
$\Lambda_0 \subset \Lambda$ and $F\in \Derb_{(\Lambda_0)}(\cor_M)$, we have a
canonical morphism in $\Dloc(\cor_{\tU_\Lambda})$
\begin{equation}\label{eq:mor_can_Maslov}
\gamma'_F = m_{\Lambda\times \Lambda^a} (\gamma_F) 
\cl \tshm_\Lambda |_{\tU_{\Lambda_0}} \to 
m_{\Lambda\times \Lambda^a}(F\etens \DD'F) |_{\tU_{\Lambda_0}},
\end{equation}
which is  an isomorphism as soon as $F$ is simple along $\Lambda_0$.
The following results are well-known and can be deduced for example
from~\cite[Appendix]{KS90}.
\begin{proposition}\label{prop:Maslov_index}
Let $U \subset \tU_\Lambda$ be a connected component of $\tU_\Lambda$.
Let $l\in U$ and let $(p,p^a) \in \Delta_\Lambda$ be the projection
of $l$ to $\Delta_\Lambda$ (so that $l$ is a Lagrangian subspace
of $T_{(p,p^a)}T^*(M\times M)$).
Then the Maslov index
\begin{equation}\label{eq:def_tau_l}
\tau_\Lambda(l) \eqdot
\tau_{\Lambda\times \Lambda^a}( \lambda_0(p)\times \lambda_0(p^a),
\lambda_\Lambda(p)\times \lambda_{\Lambda^a}(p^a), l)
\end{equation}
is independent of $l\in U$. It is an even number.
Moreover, for any $p\in \Lambda$, $l_1 \in U_{\Lambda,p}$ and  
$l_2 \in U_{\Lambda^a,p^a}$, we have
\begin{equation}\label{eq:def_tau_l1l2}
\tau_\Lambda(l_1,l_2) = \tau_\Lambda( \lambda_0(p), \lambda_\Lambda(p), l_1) - 
\tau_{\Lambda^a}( \lambda_0(p^a), \lambda_{\Lambda^a}(p^a), l_2).
\end{equation}
\end{proposition}

\begin{notation}\label{not:tau_U}
For a connected component $U$ of $\tU_\Lambda$ we set $\tau(U)= \tau_\Lambda(l)
\in 2\Z$, for any $l\in U$.  Let $U_1$ and $U_2$ be connected components of
$U_\Lambda$ and $U_{\Lambda^a}$.  Let $\Lambda_0$ be a connected component of
$\sigma_\Lambda(U_1) \cap \sigma_{\Lambda^a}(U_2)$.  Then $U_1\times_{\Lambda_0}
U_2$ is contained in a connected component, say $U$, of $\tU_\Lambda$. In this
case we set $\tau_{\Lambda_0}(U_1,U_2)= \tau(U)$.  For $k\in 2\Z$ we define
\begin{equation}\label{eq:def_tUk}
\tU_\Lambda^k = \{ l\in \tU_\Lambda; \; \tau_\Lambda(l) = k \}.
\end{equation}
\end{notation}

\begin{proposition}
\label{prop:comp_conn_tULambda}
Let $\tilde\sigma_\Lambda \cl \tU_\Lambda \to \Lambda$ be the projection to the
base. Then we have, for any $k\in 2\Z$,
\begin{itemize}
\item [(i)] for any $p\in \Lambda$, the fiber
$\tU_\Lambda^k \cap \opb{\tilde\sigma_\Lambda}(p)$ is connected,
\item [(ii)] the image of the restriction
$\tilde\sigma_\Lambda|_{\tU_\Lambda^k} \cl \tU_\Lambda^k \to \Lambda$ is 
$$
\{ p \in \Lambda;\;
\codim_{\lambda_0(p)}(\lambda_0(p) \cap \lambda_\Lambda(p)) \geq |k|/2 \} .
$$
\end{itemize}
In particular, the restriction $\tilde\sigma_\Lambda|_{\tU_\Lambda^0} \cl
\tU_\Lambda^0 \to \Lambda$ is onto.
\end{proposition}

\begin{proposition}\label{prop:shift_Maslovsheaf}
Let $U \subset \tU_\Lambda$ be a connected component of $\tU_\Lambda$.
Then $\tshm_\Lambda |_U \in \Dloc(\cor_U)$ is a local system of rank $1$
concentrated in degree $\demi \tau(U)$.
\end{proposition}
\begin{proof}
This follows from the formula $(m_\Lambda^\varphi(F))_l \simeq
(\rsect_{\opb{\varphi_l}([0,+\infty[)}(F))_x$, for $l\in U_\Lambda$ and $x=
\tau_M(l)$, and from~\cite[Prop.~7.5.3]{KS90}.
\end{proof}

\subsection{Composition with the Maslov sheaf}

All Grothendieck operations do not induce functors on the categories
$\Dloc(\cor_\bullet)$. But this works for the tensor product and the inverse
image.
Let $X$ be a topological space. The functor $\ltens$ on $\Derb(\cor_X)$ clearly
induces a bifunctor on the subprestack $\Dloc^0(\cor_X)$.  Taking the
associated stack we obtain a bifunctor on $\Dloc(\cor_X)$. We denote it also by
$\ltens$ since it commutes with the natural functor
$\Dloc^0(\cor_X) \to \Dloc(\cor_X)$.

Let $f\cl X \to Y$ be a continuous map between topological spaces.  Then
$\opb{f}$ induces a functor of prestacks
$\opb{f} \cl \Dloc^0(\cor_Y) \to \oim{f}\Dloc^0(\cor_X)$.
We note that we have a functor of stacks
$(\oim{f}\Dloc^0(\cor_X) )^a \to \oim{f}\Dloc(\cor_X)$, where $(\cdot)^a$
denotes the associated stack. Hence $\opb{f}$ induces a functor of stacks
$\opb{f} \cl \Dloc(\cor_Y) \to \oim{f}\Dloc(\cor_X)$, which commutes
with the functor $\Dloc^0(\cor_\bullet) \to \Dloc(\cor_\bullet)$.

We denote by $U_\Lambda^n$ the fiber product of $n$ factors $U_\Lambda$ over
$\Lambda$. In Propositions~\ref{prop:prod_Maslov_sheaf}
and~\ref{prop:action_Maslov_sheaf} below we denote by
$\qq_{ij} \cl U_\Lambda^3 \to U_\Lambda^2$ the projection to the factors $i$ and
$j$. We use similar notations $q_i\cl U_\Lambda^2 \to U_\Lambda$,
$\qqq_{ij}\cl U_\Lambda^4 \to U_\Lambda^2$ and
$\qqq_{ijk}\cl U_\Lambda^4 \to U_\Lambda^3$.

\begin{proposition}
\label{prop:prod_Maslov_sheaf}
There exists a canonical isomorphism in $\Dloc(\cor_{U_\Lambda^3})$
$$
u\cl \opb{\qq_{12}} \shm_\Lambda \dltens \opb{\qq_{23}} \shm_\Lambda 
 \isoto \opb{\qq_{13}} \shm_\Lambda
$$
such that the following diagram commutes in $\Dloc(\cor_{U_\Lambda^4})$
\begin{equation}\label{eq:diag_struc_Maslov}
\vcenter{\xymatrix@C=2.3cm{
\opb{\qqq_{12}} \shm_\Lambda \dltens \opb{\qqq_{23}} \shm_\Lambda
\dltens \opb{\qqq_{34}} \shm_\Lambda
\ar[d]_{\opb{\qqq_{123}}u \tens \id} \ar[r]^-{\id \tens \opb{\qqq_{234}}u}
&
\opb{\qqq_{12}} \shm_\Lambda \dltens \opb{\qqq_{24}} \shm_\Lambda
\ar[d]^{\opb{\qqq_{124}}u}
\\
\opb{\qqq_{13}} \shm_\Lambda \dltens \opb{\qqq_{34}} \shm_\Lambda 
\ar[r]^-{\opb{\qqq_{134}}u}
& \opb{\qqq_{14}} \shm_\Lambda .
}}
\end{equation}
\end{proposition}

If we had a good notion of direct image, we could define a composition of
kernels in $\Dloc(\cor_\bullet)$ as in~\eqref{eq:def_comp_gene}.  We could
restate the result as $\shm_\Lambda \circ \shm_\Lambda \simeq \shm_\Lambda$
which justifies the title of this paragraph.

\begin{proof}
By Proposition~\ref{prop:shift_Maslovsheaf} $\shm_\Lambda$ is concentrated in
a single degree over each connected component of $U_\Lambda^2$. Hence it is
enough to define $u$ locally.

\medskip\noindent
(i) We consider an open subset $\Lambda_0$ of $\Lambda$ such
that there exists $F\in \Derb_{(\Lambda_0)}(\cor_M)$ which is simple along
$\Lambda_0$.  Then we have the isomorphism
$\gamma'_F \cl \tshm_\Lambda|_{\tU_{\Lambda_0}}
 \isoto m_{\Lambda_0 \times \Lambda_0^a}(F\etens \DD'F)|_{\tU_{\Lambda_0}}$
given in~\eqref{eq:mor_can_Maslov}.
Let us set for short $V= U_{\Lambda_0}$ and $L= m_{\Lambda_0}(F)$.
Then Proposition~\ref{prop:microgerm_prod} gives
$\shm_\Lambda \simeq L\etens_{\Lambda_0} \DD'L$.
On each connected component $V_i$ of $V$ we have $L|_{V_i} \simeq L_i[d_i]$
where $L_i$ is a local system of rank $1$ and $d_i\in \Z$. Hence we have
a canonical isomorphism $\DD'L \tens L \simeq \cor_V$ and we deduce
the sequence of isomorphisms:
\begin{equation}\label{eq:prod_Maslov_sheaf}
  \begin{split}
\opb{\qq_{12}} \shm_\Lambda \dltens & \opb{\qq_{23}} \shm_\Lambda  \\
&\simeq (L\etens_{\Lambda_0} \DD'L \etens_{\Lambda_0} \cor_V) \tens
(\cor_V \etens_{\Lambda_0} L\etens_{\Lambda_0} \DD'L)  \\
&\simeq L\etens_{\Lambda_0} (\DD'L \tens L) \etens_{\Lambda_0} \DD'L  \\
&\simeq L\etens_{\Lambda_0} \cor_V \etens_{\Lambda_0} \DD'L \\
&\simeq \opb{\qq_{13}} \shm_\Lambda .
  \end{split}
\end{equation}
We denote by $u(F)$ the composition in~\eqref{eq:prod_Maslov_sheaf}.

\medskip\noindent
(ii) Let us check that $u(F)$ is independent of $F$.
Let $F' \in \Derb_{(\Lambda_0)}(\cor_M)$ be another simple sheaf along
$\Lambda_0$. Up to shrinking $\Lambda_0$ we have $F\simeq F'[i]$
in $\Derb(\cor_M;\Omega)$, for some neighborhood $\Omega$ of
$\Lambda_0$ and some $i\in\Z$. Since morphisms in 
$\Derb(\cor_M;\Omega)$ are compositions of morphisms in
$\Derb(\cor_M)$ and their inverses, we may even assume that we have
$v\cl F \to F'[i]$ in $\Derb(\cor_M)$ which induces an isomorphism
in $\Derb(\cor_M;\Omega)$, hence in $\kss(\cor_{\Lambda_0})$.
Then 
$w = m_{\Lambda_0}(v) \cl L= m_{\Lambda_0}(F) \isoto L'= m_{\Lambda_0}(F'[i])$
is an isomorphism and we have
$\gamma'_{F'} \circ (\gamma'_F)^{-1} = w\etens \DD'(w^{-1})$.
Using $w$ it is easy to draw a commutative square between
any two consecutive lines of the definitions of  $u(F)$ and $u(F')$
in~\eqref{eq:prod_Maslov_sheaf}.
Then we obtain $u(F) = u(F')$ as required.
In particular taking a covering of $\Lambda$ by open subsets
like $\Lambda_0$ we can glue the local definitions $u(F)$ and
obtain our morphism $u$.

\medskip\noindent
(iii) The commutativity of the diagram~\eqref{eq:diag_struc_Maslov} is also a
local question. It is a consequence of the sequence~\eqref{eq:prod_Maslov_sheaf}
which defines $u$.
\end{proof}

\begin{definition}\label{def:kssmg}
We define a stack $\kss_{mg}(\cor_\Lambda)$ as follows (the subscript ``$mg$''
stands for microlocal germs). For an open subset $\Lambda_0$ of $\Lambda$, the
objects of $\kss_{mg}(\cor_{\Lambda_0})$ are the pairs $(\mgL,u_\mgL)$, where
$\mgL\in \Dloc(\cor_{U_{\Lambda_0}})$ and $u_\mgL$ is an isomorphism in
$\Dloc(\cor_{U_{\Lambda_0}^2})$
\begin{equation}
\label{eq:def_morph_struct_kssmg}
u_\mgL \cl \shm_\Lambda  \dltens \opb{q_2} \mgL  \isoto \opb{q_1} \mgL 
\end{equation}
such that the following diagram commutes in $\Dloc(\cor_{U_{\Lambda_0}^3})$
\begin{equation}
\label{eq:def_diag_kssmg}
\vcenter{\xymatrix@C=2.5cm{
\opb{\qq_{12}} \shm_\Lambda \dltens \opb{\qq_{23}} \shm_\Lambda
\dltens \opb{\qq_{3}} \mgL
\ar[d]_{u \tens \opb{\qq_{3}}\id} \ar[r]^-{\id \tens \opb{\qq_{23}}u_\mgL}
&
\opb{\qq_{12}} \shm_\Lambda \dltens \opb{\qq_{2}} \mgL
\ar[d]^{\opb{\qq_{12}}u_\mgL}
\\
\opb{\qq_{13}} \shm_\Lambda \dltens \opb{\qq_{3}} \mgL
\ar[r]^-{\opb{\qq_{13}}u_\mgL}
& \opb{\qq_{1}} \mgL . }}
\end{equation}
We define $\Hom_{\kss_{mg}(\cor_{\Lambda_0})} ((\mgL,u_\mgL), (\mgL',u_{\mgL'}))$, for
two objects $(\mgL,u_\mgL)$ and $(\mgL',u_{\mgL'})$ of $\kss_{mg}(\cor_{\Lambda_0})$,
as the set of $v \in \Hom_{\Dloc(\cor_{U_{\Lambda_0}})} (\mgL,\mgL')$ such that
\begin{equation}
\label{eq:def_mor_kssmg}
\vcenter{\xymatrix{
\shm_\Lambda  \dltens \opb{q_2} \mgL
 \ar[r]^-{u_\mgL}   \ar[d]_{\id \tens \opb{q_2}(v)}
&  \opb{q_1} \mgL   \ar[d]^{\opb{q_1}(v)}  \\
\shm_\Lambda  \dltens \opb{q_2} \mgL' \ar[r]^-{u_{\mgL'}}
&  \opb{q_1} \mgL' }}
\end{equation}
is commutative.
\end{definition}

\begin{proposition}\label{prop:action_Maslov_sheaf}
Let $\Lambda_0$ be an open subset of $\Lambda$ and
$F \in \kss(\cor_{\Lambda_0})$. Then we have a canonical isomorphism
in $\Dloc(\cor_{U_\Lambda^2})$
$$
u_F \cl \shm_\Lambda  \dltens \opb{q_2} \mks_{\Lambda_0}(F)
 \isoto \opb{q_1} \mks_{\Lambda_0}(F) 
$$
such that $(\mks_{\Lambda_0}(F),u_F) \in \kss_{mg}(\cor_{\Lambda_0})$.
In other words $\mks_\Lambda$ induces a functor
\begin{equation}\label{eq:functor_mksprime}
\mks'_\Lambda \cl \kss(\cor_\Lambda) \to \kss_{mg}(\cor_\Lambda) .
\end{equation}
\end{proposition}
\begin{proof}
By Lemma~\ref{lem:simple_local} we can write locally $F= G\ltens L_M$, where $G$
is simple along $\Lambda$ and $L\in \Derb(\cor)$.  Then we have locally the
canonical isomorphisms $m_\Lambda(F) \simeq m_\Lambda(G) \tens L_{U_\Lambda}$
and $\tshm_\Lambda \simeq m_{\Lambda\times \Lambda^a}(G\etens \DD'G)$.  Now the
proof goes on like the proof of Proposition~\ref{prop:prod_Maslov_sheaf}.
\end{proof}

Since $\sigma_\Lambda \circ q_1 = \sigma_\Lambda \circ q_2$ we have a canonical
isomorphism $\opb{q_1} \opb{\sigma_{\Lambda}} F \simeq \opb{q_2}
\opb{\sigma_\Lambda} F \simeq$ for all $F \in \Dloc(\cor_\Lambda)$.  Now let
$\Lambda_0$ be an open subset of $\Lambda$ and let $(\mgL_0,u_{\mgL_0})\in
\kss_{mg}(\cor_{\Lambda_0})$ be given.  Then we obtain a functor of stacks
\begin{equation}\label{eq:dloc-to-kssmg}
t_{\mgL_0} \cl \Dloc(\cor_{\Lambda_0}) \to \kss_{mg}(\cor_{\Lambda_0}),
\quad
L \mapsto (\mgL_0 \tens \opb{\sigma_{\Lambda^0}}(L), u_{\mgL_0} \tens \id_L).
\end{equation}

\begin{lemma}\label{lem:kssmg=dloc}
Let $\Lambda_0$ be an open subset of $\Lambda$. We assume that there exists a
connected component $U$ of $U_{\Lambda_0}$ such that
$\sigma_\Lambda(U) = \Lambda_0$.  We also assume that there exists
$(\mgL_0,u_{\mgL_0})\in \kss_{mg}(\cor_{\Lambda_0})$ such that
$(\mgL_0)_l \simeq \cor$ for $l\in U$.  Then the functor $t_{\mgL_0}$
of~\eqref{eq:dloc-to-kssmg} is an equivalence of stacks.
\end{lemma}
\begin{proof}
Since the result is local on $\Lambda_0$ we may as well assume that $\Lambda_0$
is contractible and that there exists a section $i\cl \Lambda_0 \to U$
of $\sigma_\Lambda|_U$.
For $(\mgL,u_{\mgL})\in \kss_{mg}(\cor_{\Lambda_0})$
we set $\mgLp = \mgL \tens \DD'\mgL_0 \in \Dloc(\cor_{U_{\Lambda_0}})$
and $j_{\mgL_0}(\mgL) = \opb{i}(\mgLp) \in \Dloc(\cor_{\Lambda_0})$.
This defines a functor $j_{\mgL_0} \cl \kss_{mg}(\cor_{\Lambda_0})
 \to \Dloc(\cor_{\Lambda_0})$.
Let us prove that $t_{\mgL_0}$ and $j_{\mgL_0}$ are mutually inverse.

It is clear that $j_{\mgL_0}\circ t_{\mgL_0}$ is the identity functor.  So we have
to prove that, for a given $(\mgL,u_{\mgL})\in \kss_{mg}(\cor_{\Lambda_0})$, there
exists an isomorphism 
$(\mgL,u_{\mgL}) \simeq t_{\mgL_0} \circ j_{\mgL_0} (\mgL,u_{\mgL})$.
We set $\mgLp = \mgL \tens \DD'\mgL_0$, $L = \opb{i}(\mgLp)$
and $\mgL' = j_{\mgL_0}(L)$.

The isomorphism $u_{\mgL_0}$ in~\eqref{eq:def_morph_struct_kssmg} gives an
isomorphism $\shm_\Lambda \simeq \mgL_0 \etens_{\Lambda_0} \DD'\mgL_0$.
Hence $u_\mgL$ induces
\begin{equation}\label{eq:isomgLp}
  u_\mgL \tens \id_{\DD'\mgL_0} \cl \opb{q_2} \mgLp \isoto \opb{q_1} \mgLp .
\end{equation}
We define $i_2 \cl U_{\Lambda_0} \to U_{\Lambda_0}^2$ by
$i_2(l) = (l, i(\sigma_\Lambda(l)))$. Then
$q_1 \circ i_2 = \id_{U_{\Lambda_0}}$ and  $q_2 \circ i_2 = i\circ \sigma_\Lambda$.
Applying $\opb{i_2}$ to~\eqref{eq:isomgLp} we deduce
$\mgLp \isoto \opb{\sigma_\Lambda} L$ or, as well,
$v\cl \mgL \simeq \mgL_0 \tens \opb{\sigma_{\Lambda^0}}(L)$.

It only remains to see that $u_\mgL = u_{\mgL_0} \tens \id_L$.
We define $i_3 \cl U_{\Lambda_0}^2 \to U_{\Lambda_0}^3$ by
$i_3(l,l') = (l,l', i(\sigma_\Lambda(l)))$. Then applying $\opb{i_2}$
to~\eqref{eq:def_diag_kssmg} we obtain
\begin{equation*}
\vcenter{\xymatrix@C=2.5cm{
\shm_\Lambda \dltens \opb{q_{2}} \mgL_0 \tens \opb{\sigma_{\Lambda^0}}(L)
\ar[d]_{u_{\mgL_0} \tens \opb{\qq_{3}}\id} \ar[r]^-{\id \tens \opb{q_{2}}v}
& \shm_\Lambda \dltens \opb{q_{2}} \mgL \ar[d]^{u_\mgL}  \\
\opb{q_{1}} \mgL_0 \tens \opb{\sigma_{\Lambda^0}}(L) \ar[r]^-{\opb{q_{1}}v}
& \opb{q_{1}} \mgL , }}
\end{equation*}
which implies $u_\mgL = u_{\mgL_0} \tens \id_L$ and concludes the proof.
\end{proof}

\begin{theorem}\label{thm:mksprime_equiv}
Let $\Lambda$ be a locally closed conic Lagrangian submanifold of $T^*M$.
Then the functor $\mks'_\Lambda$ of Proposition~\ref{prop:action_Maslov_sheaf}
is an equivalence of stacks.
\end{theorem}
\begin{proof}
This is a local problem and we may restrict to an open subset $\Lambda_0$ of
$\Lambda$ such that there exists a simple sheaf $F$ along $\Lambda_0$.
We set $\mgL_0 = \mks_{\Lambda_0}(F)$.
Then $\mks'_\Lambda = t_{\mgL_0} \circ \ol{\mu hom}(F,\cdot)$, where
$t_{\mgL_0}$ is defined in~\eqref{eq:dloc-to-kssmg} and $\ol{\mu hom}(F,\cdot)$
in Proposition~\ref{prop:KSstack=Dloc}.  Now the result follows from
Proposition~\ref{prop:KSstack=Dloc} and Lemma~\ref{lem:kssmg=dloc}.
\end{proof}

\section{Monodromy morphism}

For $F\in \Derb_{(\Lambda)}(\cor_M)$ we know that $m_\Lambda(F)$ is a locally
constant object on $U_\Lambda$.  Here we describe the monodromy of its
restriction to a fiber $U_{\Lambda,p}$ of $\sigma_\Lambda \cl U_\Lambda \to
\Lambda$.

\medskip

We first recall well-known results on locally constant sheaves and introduce
some notations.  Let $X$ be a manifold and $L\in \Derb(\cor_X)$ such that
$\SSi(L) \subset T^*_XX$, that is, $L$ has locally constant cohomology sheaves.
Then any path $\gamma \cl [0,1] \to X$ induces an isomorphism
\begin{equation}\label{eq:def_action_chemin}
M_\gamma(L) \cl L_{\gamma(0)} \isofrom \rsect([0,1]; \opb{\gamma}L)
\isoto L_{\gamma(1)}.
\end{equation}
Moreover, $M_\gamma(L)$ only depends on the homotopy class of $\gamma$ with
fixed ends. We will use the notation $M_{[\gamma]}(L) \eqdot M_\gamma(L)$,
where $[\gamma]$ is the class of $\gamma$.
For another path $\gamma' \cl [0,1] \to X$ such that $\gamma'(0) = \gamma(1)$,
we have $M_{\gamma'\cdot \gamma}(L) = M_{\gamma'}(L) \circ M_\gamma(L)$.
In particular, if we fix a base point $x_0\in X$, we obtain the monodromy
morphism
\begin{equation}\label{eq:def_monodr}
  \begin{split}
M(L) \cl \pi_1(X;x_0) &\to \Iso(L_{x_0})  \\
\gamma & \mapsto M_\gamma(L),
  \end{split}
\end{equation}
where $\pi_1(X;x_0)$ is the fundamental group of $(X,x_0)$ and $\Iso(L_{x_0})$
is the group of isomorphisms of $L_{x_0}$ in $\Derb(\cor)$.
For any $\cor$-module $M$ we have the sign morphism
$s_M \cl \Z/2\Z \to \Iso(M)$, which sends $1\in \Z/2\Z$ to the multiplication
by $-1 \in \cor$.
When we have a morphism $\varepsilon \cl \pi_1(X;x_0) \to \Z/2\Z$, we will say
``$L$ has monodromy $\varepsilon$'' if $M(L) = s_{L_{x_0}} \circ \varepsilon$.
We remark that a morphism $\pi_1(X;x_0) \to \Z/2\Z$ is necessarily 
invariant by conjugation in $\pi_1(X;x_0)$. Hence we do not have to
choose a base point and we will write abusively
$\varepsilon \cl \pi_1(X) \to \Z/2\Z$.

Now we go back to the situation of section~\ref{sec:defmicrogerms}.
In particular $M$ is a manifold of dimension $n$ and $\Lambda$ is a locally
closed conic Lagrangian submanifold of $\dT^*M$.
For $F\in \Derb_{(\Lambda)}(\cor_M)$ we have defined
$m_\Lambda(F) \in \Derb(\cor_{U_\Lambda})$, which has locally constant
cohomology sheaves.  For a given $p\in \Lambda$ we will describe
$M(m_\Lambda(F)|_{U_{\Lambda,p}})$.  We first define an embedding of
$U_{\Lambda,p}$ into a connected manifold.

By definition, an element $l\in U_{\Lambda,p}$ is a Lagrangian subspace of
$T_p\Lambda$ which is transversal to $\lambda_0(p)$ and $\lambda_\Lambda(p)$.
The decomposition $T_p\Lambda \simeq l \oplus \lambda_\Lambda(p)$ gives a
projection $T_p\Lambda \to \lambda_\Lambda(p)$ and its restriction to
$\lambda_0(p)$ is an isomorphism that we denote by
\begin{equation}\label{eq:def_upl}
  u_p(l) \cl \lambda_0(p) \isoto \lambda_\Lambda(p).
\end{equation}

For two vector spaces $V,W$ of dimension $n$ we denote by $\Iso(V,W)$ the space
of isomorphisms from $V$ to $W$. If $V$ and $W$ are oriented, we let
$\Iso^+(V,W)$ be the connected component of orientation preserving
isomorphisms. We remark that $V\otimes \Lambda^n V$ has a canonical orientation
and we set
\begin{equation}\label{eq:iso_plus}
 \tIso(V,W) = \Iso^+(V\otimes \Lambda^n V,W\otimes \Lambda^n W).
\end{equation}
Now we have a natural embedding
\begin{equation}\label{eq:def_u_p}
  \begin{split}
i_{U_{\Lambda,p}} \cl  U_{\Lambda,p} & \to  \tIso(\lambda_0(p),\lambda_\Lambda(p)) \\
l & \mapsto u_p(l) \otimes  \Lambda^n u_p(l). 
  \end{split}
\end{equation}
The topological space $\tIso(\lambda_0(p),\lambda_\Lambda(p))$ is isomorphic to
$GL_n^+(\R)$. The fundamental group $\pi_1(GL_n^+(\R))$ is trivial for $n=1$,
isomorphic to $\Z$ for $n=2$ and to $\Z/2\Z$ for $n\geq 3$. In any case we have
a canonical morphism $\pi_1(GL_n^+(\R)) \to \Z/2\Z$ which does not depend on the
choice of a base point. So we obtain canonical morphisms, for any connected
component $U_{\Lambda,p}^0$ of $ U_{\Lambda,p}$
\begin{equation}\label{eq:def_sign_loop}
\vcenter{
\xymatrix@C=.5cm@R=.7cm{
\pi_1(U_{\Lambda,p}^0) \ar[dr]_{\varepsilon'_p}
\ar[rr]^-{\pi_1(i_{U_{\Lambda,p}} )}
&& \pi_1( \tIso(\lambda_0(p),\lambda_\Lambda(p)) )
\ar[dl]^{\varepsilon_p}  \\
&\Z/2\Z.
}}
\end{equation}

\begin{proposition}\label{prop:monodr=sign}
Let $F\in \Derb_{(\Lambda)}(\cor_M)$ and $p\in \Lambda$.
Then, for any connected component $U_{\Lambda,p}^0$, the monodromy of
$m_\Lambda(F)|_{U_{\Lambda,p}^0}$ is $\varepsilon'_p$.
\end{proposition}
\begin{proof}
(i) We let $U_\Lambda^0$ be the connected component of $U_\Lambda$ which contains
$U_{\Lambda,p}^0$. Since $U_\Lambda^0$ is open in $\lag_M^0$ 
we can deform any loop $\gamma$ in $U_{\Lambda,p}^0$ into a loop $\gamma'$
in a nearby fiber $U_{\Lambda,q}^0$. This does not change the monodromy.
We also have $\varepsilon'_p(\gamma) = \varepsilon'_q(\gamma')$.
Hence we may as well assume that $p$ is a generic point of $\Lambda$,
that is, in a neighborhood of $p$ we have $\Lambda = T^*_NM$, for a
submanifold $N \subset M$.
Then $F$ is isomorphic to $L_N$ in $\kss(\cor_\Lambda)$, for some
$L\in \Derb(\cor)$, and we have $m_\Lambda(F) \simeq m_\Lambda(\Z_N)\tens_\Z L$.
Hence we can also assume that $\cor=\Z$ and $F=\Z_N$.

\medskip\noindent
(ii) We take coordinates $(x_1,\ldots,x_n)$ so that
$N = \{x_1 =\cdots= x_k = 0\}$ and $p= (0;1,0)$. We identify $(\lag_M^0)_p$
with a space of matrices as in~\eqref{eq:LM0=matsym}.
Then $U_{\Lambda,p}$ is the space of symmetric matrices $A$
such that $\det(A_k) \not= 0$, where $A_k$ is the matrix obtained from
$A$ by deleting the $k$ first lines and columns.
Choosing a connected component $U_{\Lambda,p}^0$ means prescribing the
signature of $A_k$. We can choose a base point $B \in U_{\Lambda,p}^0$
represented by a diagonal matrix
$B = \operatorname{diag}(0,\ldots,0,1,\ldots,1,-1,\ldots,-1)$ with $k$ zeroes
and $l$ $1$'s.
We choose indices $i<j$ such that $B_{ii} = 1$ and $B_{jj}=-1$.
For $\theta \in [0,2\pi]$, we define the matrix
$B(\theta)$ which is equal to $B$ except
$$
\begin{pmatrix}
B_{ii}(\theta) & B_{ij}(\theta) \\
B_{ji}(\theta) & B_{jj}(\theta)
\end{pmatrix}
=
\begin{pmatrix}
\cos(\theta) & \sin(\theta) \\ 
\sin(\theta) & -\cos(\theta)
\end{pmatrix} .
$$
Then $\gamma\cl \theta \mapsto B(\theta)$ defines a loop in $U_{\Lambda,p}^0$ and
$\pi_1(U_{\Lambda,p}^0)$ is generated by loops of this form, where $i,j$ run
over the possible indices.
We have $\varepsilon'_p(\gamma) =1 \in \Z/2\Z$. Hence it remains to
prove that the monodromy of $m_\Lambda(\Z_N)$ around $\gamma$ is $-1$.
Since $m_\Lambda(\Z_N)$ has stalk $\Z$ up to some shift, the monodromy
can only be $1$ or $-1$. So we only have to check that the monodromy
of $m_\Lambda(\Z_N)$ around $\gamma$ is not trivial.

\medskip\noindent
(iii) We define $\varphi\cl [0,2\pi] \times M \to \R$ by
$\varphi(\theta, \ul x) = x_1 + \ul x \, B(\theta) \, {}^t\ul x$.
Then $\{\varphi_\theta\geq 0\} \cap N$ is a quadratic cone which is
homotopically equivalent to the subspace
$V_\theta = \langle e_\theta, e_p; p=k+1,\ldots, k+l, \, p\not= i \rangle$,
where $e_p = (0,\underset{\scriptscriptstyle p}{1},0)$
and $e_\theta = (0, \underset{\scriptscriptstyle i}\cos(\frac{\theta}{2}),0,
 \underset{\scriptscriptstyle j}\sin(\frac{\theta}{2}) , 0)$.
The stalk of $m_\Lambda(\Z_N)$ at $B(\theta) \in U_{\Lambda,p}^0$ is
\begin{equation}
\label{eq:calcul_monodr}
m_\Lambda(\Z_N)_{B(\theta)} \simeq  \rsect_{\{\varphi_\theta\geq 0\}}(\Z_N)
\simeq \rsect_{V_\theta}(\Z_N) \simeq \Z[k+l-n]
\end{equation}
and the choice of the isomorphism~\eqref{eq:calcul_monodr} is equivalent
to the choice of an orientation of $V_\theta$. Since we can not choose
compatible orientations of all $V_\theta$, $\theta \in [0,2\pi]$, the monodromy
of $m_\Lambda(\Z_N)$ is not $1$, as required.
\end{proof}

Let $L_p$ be the locally constant sheaf on
$\tIso(\lambda_0(p),\lambda_\Lambda(p))$ with stalk $\cor$ and monodromy
$\varepsilon_p$.
Let $F\in \Derb_{(\Lambda)}(\cor_M)$ be simple along $\Lambda$.
Then, for any connected component $U_{\Lambda,p}^0$ of $U_{\Lambda,p}$,
$m_\Lambda(F)|_{U_{\Lambda,p}^0}$ is concentrated in a single degree.
Proposition~\ref{prop:monodr=sign} says that
\begin{equation}\label{eq:coller_mg_ponctlmt}
m_\Lambda(F)|_{U_{\Lambda,p}^0} \simeq L_p|_{U_{\Lambda,p}^0}[d_0],
\end{equation}
for some integer $d_0$.

\section{Extension of microlocal germs}\label{sec:extension_mg}

We have seen in~\eqref{eq:coller_mg_ponctlmt} that the sheaf of microlocal
germs $m_\Lambda(F)$ of a simple sheaf $F$ extends 
from $U_{\Lambda,p}$ to $\tIso(\lambda_0(p),\lambda_\Lambda(p))$, via
the natural embedding $i_{U_{\Lambda,p}}$, as a local system on
$\tIso(\lambda_0(p),\lambda_\Lambda(p))$.
In this section we prove that such an extension exists not only over $p$
but globally over $\Lambda$ if the Maslov class of $\Lambda$ is zero.
We also prove that in the case of the Maslov sheaf we can also extend the
structural morphism of Proposition~\ref{prop:prod_Maslov_sheaf}.
We deduce a description of a twisted version of the Kashiwara-Schapira
stack in Corollary~\ref{cor:loceps=mgprime}.

\subsection{Stabilisation}\label{sec:stabilisation}

We let $\shi_\Lambda$ be the fiber bundle over $\Lambda$ whose fiber over a
point $p$ is $\tIso(\lambda_0(p),\lambda_\Lambda(p))$.  The inclusions
$i_{U_{\Lambda,p}}$ in~\eqref{eq:def_u_p} give 
\begin{equation}\label{eq:def_uLambda}
i_{U_\Lambda} \cl  U_\Lambda  \hookrightarrow  \shi_\Lambda .
\end{equation}

For an integer $N>0$ we define $\Xi_N = \{(\ul x, 0; 0,\xi_N)$; $\xi_N>0\}$
$\subset \dT^*_{\R^{N-1}}\R^N$ and $p_N = (0;0,1)\in \Xi_N$.
We have the inclusions
$$
\xymatrix{
 U_\Lambda \times   U_{\Xi_N}
\ar[r] \ar[d]_{i_{U_\Lambda} \times i_{\Xi_N}}
& U_{\Lambda\times \Xi_N} \ar[d]^{i_{U_{\Lambda\times \Xi_N}}}  \\
\shi_\Lambda \times \shi_{\Xi_N} \ar[r]
& \shi_{\Lambda\times \Xi_N}  , }
\quad
\xymatrix{
 U_\Lambda \times   U_{\Xi_N}|_{\{p_N\}} 
\ar[r] \ar[d]
& U_{\Lambda\times \Xi_N}|_{\Lambda\times \{p_N\}}  \ar[d]  \\
\shi_\Lambda \times \shi_{\Xi_N}|_{\{p_N\}}  \ar[r]
& \shi_{\Lambda\times \Xi_N} |_{\Lambda\times \{p_N\}} . }
$$
We remark that $\Xi_N \simeq \R^N$ and that the diagram in the left hand side is
the product of the one in the right hand side by $\Xi_N$.  The Maslov index
$\tau(\lambda_0, \lambda_{\Xi_N},l)$ is constant for $l$ in a given connected
component of $U_{\Xi_N}$ and takes distinct values for distinct components.  It
can take the values $-N+1, -N+3,\ldots, N-1$. Hence $U_{\Xi_N}$ has $N$
connected components that we label by the Maslov index $U_{\Xi_N}^{-N+1},\ldots,
U_{\Xi_N}^{N-1}$.

\begin{proposition}\label{prop:reunir_comp_conn}
Let $\Lambda$ be a locally closed conic Lagrangian submanifold of $\dT^*M$.  We
assume that the Maslov class of $\Lambda$ is zero and that $U_\Lambda$ has a
finite number of connected components, say $U_i$, $i\in I$.  Then there exist an
integer $N$ and a family of connected components $V_i \subset U_{\Xi_N}$, $i\in
I$, such that all the products $U_i\times V_i$ are in the same connected
component $W \subset U_{\Lambda\times \Xi_N}$.  Moreover the projection $W\to
\Lambda\times \Xi_N$ is onto.
\end{proposition}
\begin{proof}
(i) For $l\in U_\Lambda$ and $p = \sigma_\Lambda(l) \in \Lambda$ we set for
short $\tau_\Lambda(l) = \tau(\lambda_0(p), \lambda_\Lambda(p),l)$.
For $l \in U_{\Xi_N}$ we define $\tau_{U_{\Xi_N}}(l)$ in the same way.

Let $p\in \Lambda$ and $q\in \Xi_N$ be two given points and let 
$U_1,U_2 \subset U_\Lambda$ and $V_1,V_2 \subset U_{\Xi_N}$ be connected
components such that $U_i\cap U_{\Lambda,p}$ and $V_i\cap U_{\Xi_N,q}$,
$i=1,2$, are non empty. We choose $l_i \in U_i\cap U_{\Lambda,p}$ and
$l'_i \in V_i\cap U_{\Xi_N,q}$, $i=1,2$.
Since the connected components of $U_{\Lambda\times \Xi_N, (p,q)}$ are
distinguished by the Maslov index, we see that $U_1\times V_1$ and
$U_2\times V_2$ are in the same connected component of
$U_{\Lambda\times \Xi_N}$ if $\tau_\Lambda(l_1) - \tau_{U_{\Xi_N}}(l'_1)
= \tau_\Lambda(l_2) - \tau_{U_{\Xi_N}}(l'_2)$.

\medskip\noindent
(ii) We set $\Lambda'_i = \sigma_\Lambda(U_i)$.  We recall the
notations~\ref{not:tau_U}: for $l_i \in U_i$, $l_j \in U_j$ and $\Lambda''_{ij}$
a connected component of $\Lambda'_{ij}$ we have defined
$\tau_{\Lambda''_{ij}}(U_i,U_j) = \tau_\Lambda(l_i) - \tau_\Lambda(l_j)$, which
is an even integer independent of the choice of $l_i, l_j$.
This defines a \v Cech cocycle on the covering $\Lambda = \bigcup_{i\in I}
\Lambda'_i$.  Its class in $H^1(\Lambda;\Z_\Lambda)$ is twice the Maslov class of
$\Lambda$.  By hypothesis it is zero, so we can find a family of integers $n_i$,
$i\in I$, such that $n_i -n_j = \demi (\tau_\Lambda(l_i) - \tau_\Lambda(l_j))$.

We choose $N$ odd such that $N > \sup\{2 |n_i|$; $i\in I\}$. By~(i) the products
$U_i\times U_{\Xi_N}^{2n_i}$ are all in the same connected component of
$U_{\Lambda\times \Xi_N}$.  It is clear that $U_i\times U_{\Xi_N}^{2n_i}$ maps
surjectively onto $\Lambda\times \Xi_N$ for any $i\in I$ and this implies the
last assertion.
\end{proof}

\subsection{Twisted local systems}

We consider a fiber bundle $p\cl E\to X$ over a connected manifold $X$. We let
$E_x$ be the fiber over $x\in X$. We assume that $E_x$ is path connected.  We
let $\ul H_1(E) \in \Mod(\Z_X)$ be the local system with stalk $H_1(E_x;\Z)$ (we
have $H_1(E_x;\Z) \simeq (\pi_1(E_x;b_x))^{ab}$, for any base point $b_x\in
E_x$).

\begin{definition}
\label{def:loc_syst_monodr_eps}
We assume to be given a morphism of local systems
$\varepsilon \cl \ul H_1(E) \to (\Z/2\Z)_X$.
We consider an open subset $U\subset E$ such that, for all $x\in X$,
the ``fiber'' $U_x \eqdot U \cap E_x$ is non-empty and connected.
We let $\Dloceps(\cor_{U|X})$ be the substack of $\oim{p}(\Dloc(\cor_U))$
formed by the $F$ such that, for all $x\in X$, $F|_{U_x}$ has monodromy
$\varepsilon$.
Similarly we let $\loceps(\cor_{U|X})$ be the substack of
$\oim{p}(\loc(\cor_U))$ formed by the local systems $F$ such that, for all
$x\in X$, $F|_{U_x}$ has monodromy $\varepsilon$.
\end{definition}

\begin{lemma}\label{lem:equiv_monodr_eps}
In the situation of Definition~\ref{def:loc_syst_monodr_eps}
the restriction map from $E$ to $U$ gives an equivalence of stacks
$\loceps(\cor_{E|X}) \isoto \loceps(\cor_{U|X})$.
\end{lemma}
\begin{proof}
This is a local problem on $X$. Hence we may assume that $X$ is contractible.
Then $\loceps(\cor_{E|X})$ contains a unique object with stalk $\cor$ (up to
isomorphism), say $L$.  Taking the tensor product with $L$ induces an
equivalence $\Mod(\cor) \isoto \loc^{\varepsilon}(\cor_{E|W})$,
$M \mapsto M_E \tens L$.  The same holds with $E$ replaced by $U$ and we obtain
the commutative diagram
$$
\xymatrix{
& \Mod(\cor) \ar[dl]_{ \cdot \tens L}^\sim \ar[dr]^{ \cdot \tens L|_U}_\sim \\
\loceps(\cor_{E|W}) \ar[rr]^{r}  && \loceps(\cor_{U|X}) ,}
$$
which shows that the restriction map $r$ is an equivalence.
\end{proof}

\subsection{Microlocal germs of the trivial simple sheaf}

We use the notations of the paragraph~\ref{sec:stabilisation},
in particular $\Xi_N \subset \dT^*_{\R^{N-1}}\R^N$ for an integer $N$ and
$U_{\Xi_N}^{-N+1}, U_{\Xi_N}^{-N+3},\ldots,
U_{\Xi_N}^{N-1}$, the $N$ components of $U_{\Xi_N}$.

We set $F_N = \cor_{\R^{N-1}} \in \Derb_{(\Xi_N)}(\cor_M)$.  For a given $p_N
\in \Xi_N$ the inclusion of the fiber $\shi_{\Xi_N,p_N} \subset \shi_{\Xi_N}$ is
a homotopy equivalence. Hence $\loceps(\cor_{\shi_{\Xi_N}|\Xi_N})$ contains a
unique object, say $L_N$, up to isomorphism, with stalks isomorphic to $\cor$.
In the same way $\loceps(\cor_{\tshi_{\Xi_N}|\Xi_N})$ contains a unique object
with stalks isomorphic to $\cor$, that we denote by $K_N$

In this case we can compute $m_{\Xi_N}(F_N)$ as in the proof of
Proposition~\ref{prop:monodr=sign}. We find the following description.

\begin{lemma}\label{lem:extension_m_XiN}
There exist isomorphisms
\begin{itemize}
\item [(i)] $\alpha_k\cl m_{\Xi_N}(F_N)|_{U_{\Xi_N}^k} \isoto L_N|_{U_{\Xi_N}^k} [i_k]$,
where $i_k = \lfloor \demi k \rfloor$ and $\lfloor\cdot \rfloor$
is the integer part, for $k= -N+1, -N+3, \ldots,N-1$,
\item [(ii)] $\beta_l\cl m_{\Xi_N\times \Xi_N^a}(F_N\etens \DD'F_N)|_{\tU_{\Xi_N}^{2l}}
\isoto K_N|_{\tU_{\Xi_N}^{2l}} [-l]$,
for $l= -N+1, -N+2, \ldots,N-1$,
\item [(iii)] $\gamma\cl L_N\etens_{\Xi_N} \DD'L_N \isoto K_N|_{\shi_{\Xi_N}^2}$,
\end{itemize}
such that, for all $k,k'= -N+1, -N+3, \ldots,N-1$ we have
\begin{equation}\label{eq:iso_ext_m_XiN}
\gamma|_{{U_{\Xi_N}^k}\times_{\Xi_N} {U_{\Xi_N}^{k'}}} 
\circ (\alpha_{k} \etens \DD'(\alpha_{k'}^{-1}))
= \beta_{(k+k')/2}|_{{U_{\Xi_N}^k}\times_{\Xi_N} {U_{\Xi_N}^{k'}}}.
\end{equation}
Moreover, for other $(L_N',K_N',\alpha',\beta',\gamma')$ as in~(i)-(iii)
satisfying~\eqref{eq:iso_ext_m_XiN} there exists a unique isomorphism $u\cl L_N
\isoto L_N'$ in $\loceps(\cor_{\shi_{\Xi_N}|\Xi_N})$ such that $\alpha'_k =
u|_{U_{\Xi_N}^k} \circ \alpha_k$, for all connected components $U_{\Xi_N}^k$ of
$U_{\Xi_N}$.
\end{lemma}

\subsection{Extension of the Maslov sheaf}

Recall the fiber bundle $\shi_\Lambda \to \Lambda$ and the inclusion
$i_{U_\Lambda} \cl U_\Lambda \hookrightarrow \shi_\Lambda$ defined
in~\eqref{eq:def_uLambda}.
We also set $\tshi_\Lambda = \shi_{\Lambda\times \Lambda^a}|_{\Delta_\Lambda}$
and we denote by $i_{\tU_\Lambda} \cl  \tU_\Lambda  \hookrightarrow  \tshi_\Lambda$
the natural inclusion.
Similarly we set $\shi_\Lambda^2 = \shi_\Lambda \times_\Lambda \shi_\Lambda$,
$\shi_\Lambda^3 = \shi_\Lambda \times_\Lambda \shi_\Lambda^2$
and we have the inclusions
$i_{U_\Lambda^2} \cl  U_\Lambda^2  \hookrightarrow  \shi_\Lambda^2$,
$i_{U_\Lambda^3} \cl  U_\Lambda^3  \hookrightarrow  \shi_\Lambda^3$.

We recall the notation $\tU_\Lambda^k = \{ l\in \tU_\Lambda; \; \tau_\Lambda(l)
= k \}$ for $k\in 2\Z$, introduced in~\eqref{eq:def_tUk}.

\begin{lemma}\label{lem:extension_m_lambda+}
Let $\Lambda$ be a locally closed conic Lagrangian submanifold of $\dT^*M$.  We
assume that the Maslov class of $\Lambda$ is zero and that $U_\Lambda$ has a
finite number of connected components, say $U_i$, $i\in I$.
Let $F\in \Derb_{(\Lambda)}(\cor_M)$ be simple along $\Lambda$.
Then there exist $L \in \loceps(\cor_{\shi_\Lambda|\Lambda})$,
$K \in \loceps(\cor_{\tshi_\Lambda|\Lambda})$ and isomorphisms
\begin{itemize}
\item [(i)] $\alpha_U\cl m_\Lambda(F)|_U \isoto L|_U [d_U]$, for any connected
  component $U$ of $U_\Lambda$, where $d_U$ is some integer,
\item [(ii)] $\beta_l\cl m_{\Lambda\times \Lambda^a}(F\etens \DD'F)|_{\tU_\Lambda^{2l}}
\isoto K|_{\tU_\Lambda^{2l}} [-l]$, for $l\in\Z$,
\item [(iii)] $\gamma\cl L\etens_\Lambda \DD'L \isoto K|_{\shi_\Lambda^2}$,
\end{itemize}
such that, for all connected components, $U,V$, of $U_\Lambda$, 
\begin{equation}\label{eq:iso_ext_m_lambda+}
\gamma|_{U\times_\Lambda V} \circ (\alpha_U \etens \DD'(\alpha_V^{-1}))
= \beta_l|_{U\times_\Lambda V},
\end{equation}
where $l = \demi \tau(U,V)$.

Moreover, for other $(L',K',\alpha_U',\beta',\gamma')$ as in~(i)-(iii)
satisfying~\eqref{eq:iso_ext_m_lambda+} there exists a unique isomorphism
$u\cl L \isoto L'$ in $\loceps(\cor_{\shi_\Lambda|\Lambda})$ such that
$\alpha'_U = u|_U \circ \alpha_U$, for all connected components $U$ of
$U_\Lambda$.
\end{lemma}
\begin{proof}
(i) We choose an integer $N$ and a connected component $W$ of $U_{\Lambda\times
  \Xi_N}$ given by Proposition~\ref{prop:reunir_comp_conn}.  Recall that for any
component $U_i$ of $U_\Lambda$ there exists a component $V_i$ of $U_{\Xi_N}$
such that $U_i\times V_i \subset W$.

We use the notation $F_N = \cor_{\R^{N-1}}$ of Lemma~\ref{lem:extension_m_XiN}.
We set $L^0 = m_{\Lambda\times \Xi_N}(F \etens F_N )|_W [d] \in
\loceps(\cor_{W|\Lambda\times \Xi_N})$, where the shift $d$ is chosen so that
$L^0$ is in degree $0$.
By Lemma~\ref{lem:equiv_monodr_eps} we have an equivalence of stacks
$\loceps(\cor_{\shi_{\Lambda\times \Xi_N}|\Lambda\times \Xi_N}) \isoto
\loceps(\cor_{W|\Lambda\times \Xi_N})$. 
Hence there exists
$L^1 \in \loceps(\cor_{\shi_{\Lambda\times \Xi_N}|\Lambda\times \Xi_N})$ such
that $L^1|_W \simeq L^0$.

We recall that $L_N$ is the unique object of
$\loceps(\cor_{\shi_{\Xi_N}|\Xi_N})$ with stalks isomorphic to $\cor$.  Hence
$L^1|_{\shi_\Lambda \times \shi_{\Xi_N}}$ decomposes in a unique way as
$L^1|_{\shi_\Lambda \times \shi_{\Xi_N}} \simeq L \etens L_N$, with $L\in
\loceps(\cor_{\shi_\Lambda|\Lambda})$. We obtain finally
\begin{equation}\label{eq:extension_m_lambda1}
\begin{split}
m_\Lambda(F)|_{U_i} \etens m_{\Xi_N}(F_N)|_{V_i} 
& \simeq m_{\Lambda\times \Xi_N}(F \etens F_N)|_{U_i \times V_i}  \\
& \simeq (L \etens L_N)|_{U_i \times V_i}.
\end{split}
\end{equation}
We have defined $\alpha_k$ in Lemma~\ref{lem:extension_m_XiN}~(i).  There exists
a unique isomorphism $\alpha_{U_i}$ as in~(i) of the current lemma whose tensor
product with $\alpha_k$ gives~\eqref{eq:extension_m_lambda1}.  We have $d_{U_i}
= -i_k$ for the $k$ such that $V_i = U_{\Xi_N}^k$.

\medskip\noindent
(ii) The same argument as in~(i) with $F \etens F_N$ replaced by
$F \etens F_N \etens \DD'F \etens \DD'F_N$ gives
$K \in \loceps(\cor_{\tshi_\Lambda|\Lambda})$ and isomorphisms
$$
m_{\Lambda\times \Lambda^a}(F\etens \DD'F)|_{\tU_\Lambda^{2l}}
\etens
m_{\Xi_N\times \Xi_N^a}(F_N\etens \DD'F_N)|_{\tU_{\Xi_N}^{-2l}}
\simeq
 K|_{\tU_\Lambda^{2l}} \etens K_N|_{\tU_{\Xi_N}^{-2l}} ,
$$
for $l\in \Z$.  Using $\beta_l$ in Lemma~\ref{lem:extension_m_XiN}~(ii) we
deduce the isomorphism $\beta_l$ of~(ii) in the current lemma.  Then~(iii)
and~\eqref{eq:iso_ext_m_lambda+} follow from
Lemma~\ref{lem:extension_m_XiN}~(iii) and~\eqref{eq:iso_ext_m_XiN}.
\end{proof}

By Proposition~\ref{prop:shift_Maslovsheaf} we can shift $\shm_\Lambda$ to
obtain an object of $\loc(\cor_{U_\Lambda^2})$ as follows.  We define
$\shm'_\Lambda \in \loc(\cor_{U_\Lambda^2})$ by $\shm'_\Lambda|_U \eqdot
\shm_\Lambda|_U \, [\demi \tau(U)]$, for any connected component $U$ of
$U_\Lambda^2$.  Then Proposition~\ref{prop:prod_Maslov_sheaf} defines an
isomorphism:
\begin{equation}\label{eq:prod_Maslov_sheaf_prime}
u' \cl \opb{\qq_{12}} (\shm'_\Lambda) \tens \opb{\qq_{23}} (\shm'_\Lambda) 
 \isoto \opb{\qq_{13}} (\shm'_\Lambda)  .
\end{equation}

\begin{theorem}
\label{thm:extension_maslov_sheaf}
There exist $\nshm_\Lambda \in \loceps(\cor_{\shi_\Lambda^2|\Lambda})$
together with two isomorphisms
$\alpha\cl \shm'_\Lambda \isoto \opb{i_{U_\Lambda^2}} \nshm_\Lambda$ and
$v\cl \opb{q_{12}}  (\nshm_\Lambda) \tens \opb{q_{23}}  (\nshm_\Lambda)
\isoto \opb{q_{13}}  (\nshm_\Lambda)$,
such that the following diagram is commutative
\begin{equation}\label{eq:diag_ext_mas_sheaf}
\vcenter{\xymatrix@C=2cm{
\opb{q_{12}} (\shm'_\Lambda) \tens \opb{q_{23}} (\shm'_\Lambda) 
 \ar[r]^-{u'}_-\sim \ar[d]_\wr 
& \opb{q_{13}} (\shm'_\Lambda)  \ar[d]^\wr \\
\opb{q_{12}} \opb{i_{U_\Lambda^2}} (\nshm_\Lambda)
\tens \opb{q_{23}} \opb{i_{U_\Lambda^2}} (\nshm_\Lambda)
 \ar[r]^-{v'}_-\sim
& \opb{q_{13}} \opb{i_{U_\Lambda^2}} (\nshm_\Lambda) ,   }}
\end{equation}
where $u'$ is~\eqref{eq:prod_Maslov_sheaf_prime},
$v' = \opb{i_{U_\Lambda^3}}(v)$ and the vertical arrows are induced by $\alpha$.
Moreover $(\nshm_\Lambda,v)$ satisfies a commutative diagram similar
to~\eqref{eq:diag_struc_Maslov}.
\end{theorem}
\begin{proof}
(i) We first define $\nshm_\Lambda$ locally. Let $\Lambda_0$ be an open subset
of $\Lambda$ such that
\begin{equation}\label{eq:hyp_rec_Lambda}
\begin{minipage}{11cm}
the Maslov class of $\Lambda_0$ is zero, $U_{\Lambda_0}$ has a finite number of
connected components and there exists $F_0\in \Derb_{(\Lambda_0)}(\cor_M)$ which
is simple along $\Lambda_0$.
\end{minipage}
\end{equation}
We choose $L_0 \in \loceps(\cor_{\shi_{\Lambda_0}|\Lambda_0})$ and
$\alpha_U\cl m_\Lambda(F_0)|_U \isoto L_0|_U [d_U]$, for $U$ any connected
component of $U_{\Lambda_0}$, given by Lemma~\ref{lem:extension_m_lambda+}.
We set $\nshm_0 = L_0 \etens_{\Lambda_0} \DD'L_0$. 
The contraction $L_0 \tens \DD'L_0 \to \cor_{U_{\Lambda_0}}$ induces
$v_0\cl \opb{q_{12}}  (\nshm_0) \tens \opb{q_{23}}  (\nshm_0)
\isoto \opb{q_{13}}  (\nshm_0)$
and the $\alpha_U$'s induce
$\alpha_0\cl m_{\Lambda_0}(F_0) \etens  m_{\Lambda_0}(\DD'F_0)
\isoto \opb{i_{U_\Lambda^2}} \nshm_0$.

\medskip\noindent
(ii) Let $F'_0\in \Derb_{(\Lambda_0)}(\cor_M)$ be another simple sheaf and
$a\cl F_0 \to F'_0$ a morphism such that $m_\Lambda(a)$ is an isomorphism.
We choose $L'_0 \in \loceps(\cor_{\shi_{\Lambda_0}|\Lambda_0})$ and
$\alpha'_U\cl m_\Lambda(F_0)|_U \isoto L'_0|_U [d_U]$ as in part~(i) of the
proof.
Then there exists an isomorphism $b\cl L_0 \isoto L'_0$ such that
$\alpha'_U \circ m_\Lambda(a)|_U = b|_U \circ \alpha_U$, for all connected
components $U$ of $U_{\Lambda_0}$.  Defining $\nshm'_0$, $v'_0$ and $\alpha'_0$
the same way as $\nshm_0$, $v_0$ and $\alpha_0$, we see that $b$ induces
$\beta\cl \nshm_0 \isoto \nshm'_0$ such that the following diagrams commutes
$$
\xymatrix{
m_{\Lambda_0}(F_0) {\etens}  m_{\Lambda_0}(\DD'F_0)
\ar[r]^-{\alpha_0} \ar[d] 
&  \opb{i_{U_\Lambda^2}} \nshm_0  \ar[d]  \\
m_{\Lambda_0}(F'_0) {\etens}  m_{\Lambda_0}(\DD'F'_0)
\ar[r]^-{\alpha'_0}
&  \opb{i_{U_\Lambda^2}} \nshm'_0   }
\quad
\xymatrix{
\opb{q_{12}}  (\nshm_0) {\tens} \opb{q_{23}}  (\nshm_0)
\ar[r]^-{v_0} \ar[d] 
& \opb{q_{13}}  (\nshm_0) \ar[d]  \\
\opb{q_{12}}  (\nshm'_0) {\tens} \opb{q_{23}}  (\nshm'_0)
\ar[r]^-{v'_0}
& \opb{q_{13}}  (\nshm'_0) ,
 }
$$
where the vertical arrows are induced by $a$ or $\beta$.

\medskip\noindent
(iii) We cover $\Lambda$ by open subsets $\Lambda_i$, $i\in I$,
satisfying~\eqref{eq:hyp_rec_Lambda}.  We choose simple sheaves $F_i\in
\Derb_{(\Lambda_i)}(\cor_M)$ and we construct $\nshm_i$, $v_i$ and $\alpha_i$
the same way as $\nshm_0$, $v_0$ and $\alpha_0$.  By part~(ii) we can glue them
in $\nshm$, $v$ and $\alpha$.  The commutativity of the
diagram~\eqref{eq:diag_ext_mas_sheaf}, as well as the last assertion, are local
statements and follow the construction of $v_0$ and $\alpha_0$.
\end{proof}

\begin{definition}\label{def:kssmgprime}
Replacing $\shm_\Lambda$ by $\shm'_\Lambda$ in Definition~\ref{def:kssmg},
we define a stack $\kss'_{mg}(\cor_\Lambda)$ similar to $\kss_{mg}(\cor_\Lambda)$:
for an open subset $\Lambda_0$ of $\Lambda$, the objects of
$\kss'_{mg}(\cor_{\Lambda_0})$ are the pairs $(L,u_L)$, where
$L\in \Dloc(\cor_{U_{\Lambda_0}})$ and $u_L$ is an isomorphism in
$\Dloc(\cor_{U_{\Lambda_0}^2})$
\begin{equation}
\label{eq:def_morph_struct_kssmgprime}
u_L \cl \shm'_\Lambda  \dltens \opb{q_2} L  \isoto \opb{q_1} L 
\end{equation}
such that the diagram obtained from~\eqref{eq:def_diag_kssmg} by replacing
$\shm_\Lambda$ by $\shm'_\Lambda$ is commutative.  The morphisms in
$\kss'_{mg}(\cor_\Lambda)$ are defined as in the case of
$\kss_{mg}(\cor_\Lambda)$, replacing $\shm_\Lambda$ by $\shm'_\Lambda$ in the
diagram~\eqref{eq:def_mor_kssmg}.
\end{definition}

In the next lemma we use the same notations $q_i$,\dots for the projections
from $\shi_\Lambda^p$ to $\shi_\Lambda^q$ as in the case of $U_\Lambda$.
\begin{lemma}\label{lem:mor_str_dloceps}
Let $L\in \Dloceps(\cor_{\shi_\Lambda|\Lambda})$. Then there exists a unique
isomorphism $v_L\cl \nshm_\Lambda \dltens \opb{p_2} L  \isoto \opb{p_1} L$
such that
$$
\xymatrix@C=2.5cm{
\opb{\qq_{12}} \nshm_\Lambda \dltens \opb{\qq_{23}} \nshm_\Lambda
\dltens \opb{\qq_{3}} L
\ar[d]_{v \tens \opb{\qq_{3}}\id} \ar[r]^-{\id \tens \opb{\qq_{23}}v_L}
&
\opb{\qq_{12}} \nshm_\Lambda \dltens \opb{\qq_{2}} L
\ar[d]^{\opb{\qq_{12}}v_L}
\\
\opb{\qq_{13}} \nshm_\Lambda \dltens \opb{\qq_{3}} L
\ar[r]^-{\opb{\qq_{13}}v_L}
& \opb{\qq_{1}} L  }
$$
commutes.
\end{lemma}
\begin{proof}
Since the statement contains the unicity of $v_L$ it is enough to prove the
result locally on $\Lambda$. So we can assume that $\Lambda$ is contractible.
Hence $\loceps(\cor_{\shi_\Lambda|\Lambda})$ contains a unique object with
stalk $\Z$, say $L_0$. Then $L \simeq L_0 \tens \opb{\sigma_\Lambda}(L')$
for a unique $L'\in \Dloc(\cor_\Lambda)$. We also have
$\nshm_\Lambda \simeq L_0 \etens_{\Lambda_0} \DD'L_0$ and the result is obvious.
\end{proof}

Theorem~\ref{thm:extension_maslov_sheaf}
and Lemma~\ref{lem:mor_str_dloceps} imply the following result.

\begin{corollary}
\label{cor:loceps=mgprime}
The restriction functor $\Dloceps(\cor_{\shi_\Lambda|\Lambda}) \to
\Dloceps(\cor_{U_\Lambda|\Lambda})$ induces an equivalence of stacks
$\Dloceps(\cor_{\shi_\Lambda|\Lambda}) \isoto \kss'_{mg}(\cor_\Lambda)$.
\end{corollary}

\section{Twisted stacks}

\begin{definition}\label{def:twisted_stack}
Let $X$ be a topological space and $\catc$ a stack over $X$ equipped with an
autoequivalence $T\cl \catc \isoto \catc$. Let $\{U_i\}_{i\in I}$ be a locally
finite covering of $X$ and let $\check a = \{a_{ij}\}_{i,j\in I}$ be a \v Cech
cocycle of degree $1$ with values in $\Z$.  We define a stack $\catc_{T,\check
  a}$ (denoted $\catc_{\check a}$ if there is no ambiguity about $T$) on $X$ by
the data of stacks $\catc_i$ on $U_i$, $\catc_i \eqdot \catc|_{U_i}$ and the
equivalences $\catc_i|_{U_{ij}} \isoto \catc_j|_{U_{ij}}$, $F \mapsto
T^{a_{ij}}(F)$.
\end{definition}

We quote the following easy lemma.

\begin{lemma}\label{lem:eq_tw_stacks}
In the situation of Definition~\ref{def:twisted_stack} let $\check a, \check b$
be two cohomologous cocycles and let $\check c = \{c_i\}_{i\in I}$ be a cochain
such that $\check b =  \check a + \partial \check c$.
Then the equivalences $T^{c_i} \cl \catc_i \isoto \catc_i$ induce an equivalence
$\phi_{\check c} \cl \catc_{\check a} \isoto \catc_{\check b}$.
Moreover, for any $d\in \Z$, we have $\phi_{\check c + d} \simeq T^d \circ
\phi_{\check c}$, where $\check c +d = \{c_i+d\}_{i\in I}$.
\end{lemma}

In what follows we consider the notion of twisted stack for the case of the
Kashiwara-Schapira stack and the autoequivalence $T$ induced by the shift
functor $F\mapsto F[1]$. We write $(\kss(\cor_\Lambda) )_{\check a}$ for
$(\kss(\cor_\Lambda) )_{T, \check a}$.

\begin{remark}\label{rem:description_twKSstack}
For an open covering $\{\Lambda_i\}_{i\in I}$ of $\Lambda$ and a \v Cech cocycle
$\check a = \{a_{ij}\}_{i,j\in I}$, we have a description of the objects of
$(\kss(\cor_\Lambda))_{\check a}$ as in Remark~\ref{rem:description_KSstack} as
follows.  An object of $(\kss(\cor_\Lambda) )_{\check a}$ is determined by the
data of objects $F_i \in \Derb_{(\Lambda_i)}(\cor_M)$, for any $i\in I$, and
sections $u_{ji}\in H^{a_{ij}}(\Lambda_{ij};\mu hom(F_i,F_j)|_{\Lambda_{ij}})$,
for any $i,j\in I$, such that
\begin{itemize}
\item [(i)] $u_{ii}$ is induced by $\id_{F_i}$, for any $i\in I$,
\item [(ii)] $u_{kj} \mucirc u_{ji} = u_{ki}$, for any $i,j,k \in I$.
\end{itemize}
For a cochain $\check c = \{c_i\}_{i\in I}$ and $ \check b = \check a + \partial
\check c$ as in Lemma~\ref{lem:eq_tw_stacks} the equivalence $\phi_{\check c}\cl
(\kss(\cor_\Lambda) )_{\check a} \isoto (\kss(\cor_\Lambda) )_{\check b}$ is
given by $F_i \mapsto F_i[c_i]$.
\end{remark}

\begin{lemma}\label{lem:kssmg'=kssmgtwisted}
Let $m_\Lambda \in H^1(\Lambda;\Z_\Lambda)$ be the Maslov class of $\Lambda$.
Then there exist a covering $\{\Lambda_i\}_{i\in I}$ of $\Lambda$ and
a \v Cech cocycle $\check m_\Lambda$ representing  $m_\Lambda$ such that
we have $\kss_{mg}(\cor_\Lambda) \simeq (\kss'_{mg}(\cor_\Lambda))_{\check
  m_\Lambda}$.
\end{lemma}
\begin{proof}
(i) We choose a covering $\{\Lambda_i\}_{i\in I}$ of $\Lambda$ such that there
exist simple sheaves $F_i\in \Derb_{(\Lambda_i)}(\cor_M)$ for all $i\in I$.
We let $U_i^\alpha$, $\alpha\in A_i$, be the connected components of
$U_{\Lambda_i}$.  The object
$m_{\Lambda_i}(F_i)|_{U_i^\alpha} \in \Derb(\cor_{U_i^\alpha})$ is concentrated
in a single degree that we denote by $d_i^\alpha$.
Then, for $i,j \in I$ and $\alpha\in A_i, \beta\in A_j$ such that
$U_i^\alpha \cap U_j^\beta \not= \emptyset$, the difference
$c_{ij} = d_i^\alpha - d_j^\beta$ only depends on $i$ and $j$. The \v Cech
cochain $\{c_{ij}\}_{i,j\in I}$ is a cocycle on $\Lambda$ with represents
$m_\Lambda$.

\medskip\noindent
(ii) For any given $i\in I$ we define
$\phi_i \cl \kss_{mg}(\cor_\Lambda)|_{\Lambda_i}
\isoto \kss'_{mg}(\cor_\Lambda)|_{\Lambda_i}$,
$(\mgL,u_\mgL) \mapsto (\mgL',u_{\mgL'})$, where we set
$\mgL'|_{U_i^\alpha} = \mgL|_{U_i^\alpha}[-d_i^\alpha]$ for all $\alpha\in A_i$.
By~(i) we see that the $\phi_i$'s, $i\in I$, glue together into an equivalence
of stacks $\kss_{mg}(\cor_\Lambda) \simeq (\kss'_{mg}(\cor_\Lambda))_{m_\Lambda}$.
\end{proof}

Putting together Theorem~\ref{thm:mksprime_equiv},
Lemma~\ref{lem:kssmg'=kssmgtwisted} and Corollary~\ref{cor:loceps=mgprime} we
obtain a description of the Kashiwara-Schapira stack.

\begin{theorem}\label{thm:KSstack=faisc_tordus}
Let $\Lambda$ be a locally closed conic Lagrangian submanifold of $\dT^*M$.
Then there exists a covering $\{\Lambda_i\}_{i\in I}$ of $\Lambda$ and
a \v Cech cocycle $\check m_\Lambda$ representing the Maslov class of $\Lambda$
such that we have an equivalence of stacks
$\kss(\cor_\Lambda) \simeq (\Dloceps(\cor_{\shi_\Lambda|\Lambda}))_{\check m_\Lambda}$.

If $m_\Lambda = 0$, we obtain $\kss(\cor_\Lambda) \simeq
\Dloceps(\cor_{\shi_\Lambda|\Lambda}) \simeq \bigoplus_{i\in \Z}
\loceps(\cor_{\shi_\Lambda|\Lambda})[i]$ and $\kss^p(\cor_\Lambda)$ is the
substack formed by the objects concentrated in one degree.
\end{theorem}

\begin{remark}\label{rem:thmKSstack=faisc_tordus}
(i) By Lemma~\ref{lem:eq_tw_stacks} the last equivalence of
Theorem~\ref{thm:KSstack=faisc_tordus} is only defined up to a shift $d\in\Z$.

\smallskip\noindent

(ii) If $\cor = \Z/2\Z$, then $-1=1$ and the monodromy $\varepsilon$ is trivial.
Hence the inverse image by the projection $\shi_\Lambda \to \Lambda$ induces a
functor of stacks $\Dloc(\cor_\Lambda) \to
\Dloceps(\cor_{\shi_\Lambda|\Lambda})$, which is easily seen to be an
equivalence.  We deduce $(\kss(\cor_\Lambda))_{-\check m_\Lambda} \simeq
\Dloc(\cor_{\Lambda})$.  Since $GL_1(\cor) = \cor^\times$ is trivial there is in
fact no non trivial local system of rank $1$. It follows that
$(\kss(\cor_\Lambda))_{-\check m_\Lambda}$ has global simple objects, which are
all isomorphic up to a shift in degree.

\smallskip\noindent
(iii) If $\cor = \Z/2\Z$ and the Maslov class $m_\Lambda$ vanishes, we obtain
$\kss(\cor_\Lambda) \simeq \Dloc(\cor_{\Lambda}) \simeq \bigoplus_{i\in \Z}
\loc(\cor_{\Lambda})[i]$ and $\kss(\cor_\Lambda)$ has global simple objects,
unique up to shift as in~(ii).
\end{remark}

\subsection*{Topological obstructions}
We quickly recall the definition of the second Stiefel-Whitney class.
We first recall some facts on topological obstructions.  We consider 
a connected manifold $X$ and a fiber
bundle $p\cl E\to X$ with path connected fibers.
We assume to be given a morphism of local systems
$\varepsilon \cl \ul H_1(E) \to (\Z/2\Z)_X$.
The second obstruction class of $E$ and $\varepsilon$ is the obstruction for
the stack $\loceps(\Z_X)$ to have a global object, locally free of rank $1$.
It is defined as follows.
We first remark that, if $U$ is contractible, $\loceps(\Z_U)$ has only one rank
$1$ locally free object, up to isomorphism. Clearly the automorphism group of
this object is $\Z/2\Z = \{\pm 1\}$.
We consider a covering $X = \bigcup_{i\in I} U_i$ such that, for all
$i,j,k \in I$, the open sets $U_i$, $U_{ij}$ and $U_{ijk}$ are contractible.
We choose $L_i \in \loceps(\Z_{U_i})$ and isomorphisms
$u_{ij} \cl L_i|_{U_{ij}} \isoto L_j|_{U_{ij}}$, for all $i,j\in I$.
Then, for $i,j,k \in I$, the composition
$c_{ijk} = u_{ki} \circ u_{jk} \circ u_{ij}$ belongs to $\Z/2\Z$.
We can check that $\{c_{ijk}\}_{i,j,k\in I}$ defines a \v Cech cocyle
over $X$ and that its class in $H^2(X;\Z/2\Z)$ is independent of the
choices of the covering, the objects $L_i$ and the morphisms $u_{ij}$.
\begin{definition}\label{def:sec_obstr}
For a bundle $E\to X$ and a morphism
$\varepsilon \cl \ul H_1(E) \to (\Z/2\Z)_X$,
we let $o_2(E,\varepsilon) \in H^2(X;\Z/2\Z)$ be the class defined
by the cocyle $\{c_{ijk}\}_{i,j,k\in I}$.
\end{definition}
By construction $\loceps(\Z_X)$ has a global object locally free of rank $1$ if
and only if $o_2(E,\varepsilon) =0$.

We will use this obstruction class in the following case. Let $F_1,F_2 \to X$ be
two real vector bundles over $X$ of the same rank, say $r$.  Let
$\shi_{F_1,F_2}$ be the fiber bundle with fiber $\tIso(F_1(x),F_2(x))$ at $x\in
X$ (see~\eqref{eq:iso_plus}).  The fiber is isomorphic to $GL_r^+(\R)$ and, as
explain before~\eqref{eq:def_sign_loop}, we have a morphism $\varepsilon \cl \ul
H_1(\shi_{F_1,F_2}) \to (\Z/2\Z)_X$.
\begin{definition}\label{def:relStieWhit}
We define the relative second Stiefel-Whitney class of $F_1$ and $F_2$ by
$rw_2(F_1,F_2) \eqdot o_2(\shi_{F_1,F_2},\varepsilon) \in H^2(X;\Z/2\Z)$.
\end{definition}
If $F_1$ is the trivial vector bundle we have
$rw_2(F_1,F_2) = w_2(F_2 \otimes \Lambda^r F_2) = w_1^2(F_2) + w_2(F_2)$.

Now we deduce from Theorem~\ref{thm:KSstack=faisc_tordus}:
\begin{corollary}\label{cor:objet_global_KSstack}
In the situation of Theorem~\ref{thm:KSstack=faisc_tordus} we assume that the
Maslov class of $\Lambda$ is zero.
Then the stack of simple sheaves $\kss^s(\cor_\Lambda)$ has a global object if
and only if the image of $rw_2(\lambda_0,\lambda_\Lambda)$ in
$H^2(\Lambda;\cor^\times)$ is zero.
\end{corollary}

\section{The Kashiwara-Schapira stack for orbit categories}

In this section we set $\cor = \Z/2\Z$.
We have defined the usual sheaf operations for the triangulated orbit categories
$\Orb(\cor_M)$ and we can define a Kashiwara-Schapira stack in this situation.
We give quickly the analogs of the results obtained in the previous sections.

For a subset $S$ of $T^*M$ we recalled in Notations~\ref{not:micro_categories}
the categories $\Derb_{S}(\cor_M)$, $\Derb_{(S)}(\cor_M)$ and
$\Derb(\cor_M;S)$. We define $\OrbL{S}(\cor_M)$, $\OrbL{(S)}(\cor_M)$ and
$\Orb(\cor_M;S)$ in the same way, replacing $\Derb$ by $\Orb$ and $\SSi$ by
$\SSo$.  Let $\Lambda \subset \dT^*M$ be a locally closed conic Lagrangian
submanifold. We define a stack $\kss^\orb(\cor_\Lambda)$ on $\Lambda$ as in
Definition~\ref{def:KSstack}, again replacing $\Derb$ by $\Orb$.  It comes with
a functor $\kssfunc^\orb_\Lambda \cl \OrbL{(\Lambda)}(\cor_M) \to
\kss^\orb(\cor_\Lambda)$.

We say that $F\in \Orb(\cor_M)$ is simple along $\Lambda$ if $\SSo(F)\cap \dT^*M
\subset \Lambda$ and, for any $p\in \Lambda$, there exists $F'\in \Derb(\Cor_M)$
such that $Q(F') \simeq F$ and $F'$ is simple along $\Lambda$ in a neighborhood
of $p$.  As in section~\ref{sec:KSstack} we can define the substack
$\kss^{\orb,s}(\cor_\Lambda)$ of $\kss^\orb(\cor_\Lambda)$ associated with the
simple sheaves.

\begin{lemma}\label{lem:HomDkMp=muhomporb}
In the situation of~\eqref{eq:HomDkMOmega_vers_muhom} we have a morphism \\
$\Hom_{\Orb(\cor_M;\Omega)}(F,G) \to \Hom_{\Orb(\cor_\Omega)}(\cor_\Omega,\mu
hom^\varepsilon(F,G)|_\Omega) $.  If $\Omega = \{p\}$ for some $p\in T^*M$, then
it is an isomorphism.
\end{lemma}

\begin{lemma}\label{lem:unique-fais-simple-orb}
Let $\Lambda \subset \dT^*M$ be a locally closed conic Lagrangian
submanifold. We assume that $\Lambda$ is contractible.  Let $F,F' \in
\Orb(\cor_M)$ be two simple sheaves along $\Lambda$ and let $\Omega$ be a
neighborhood of $\Lambda$ such that $\SSo(F) \cap \Omega =\SSo(F') \cap \Omega =
\Lambda$.  Then we have a unique isomorphism $\mu hom^\varepsilon(F,F')|_\Omega
\simeq \cor_\Lambda$ in $\Orb(\cor_\Omega)$.
\end{lemma}

By Lemma~\ref{lem:unique-fais-simple-orb} there exists a unique simple sheaf in
$\kss^{\orb}(\cor_{\Lambda_0})$ for any contractible open subset $\Lambda_0
\subset \Lambda$, up to a unique isomorphism.  In other words
$\kss^{\orb,s}(\cor_{\Lambda})$ has locally a unique object with the identity as
unique isomorphism. Hence gluing is trivial.  Since
$\kss^{\orb,s}(\cor_{\Lambda})$ is a stack it follows that it has a unique
global object.

We recall that $\loc(\cor_X)$ is the substack of $\Mod(\cor_X)$ formed by the
locally constant sheaves.

\begin{definition}\label{def:Orbloc}
Let $X$ be a manifold. We let $\Oloc^0(\cor_X)$ be the subprestack of
$U \mapsto \Orb(\cor_U)$, $U$ open in $X$, formed by the $F\in \Orb(\cor_U)$
such that $\SSo(F) \subset T^*_UU$.  We let $\Oloc(\cor_X)$ be the
stack associated with $\Oloc^0(\cor_X)$.
\end{definition}

The composition of the functors $\loc(\cor_U) \to \Mod(\cor_U) \to
\Derb(\cor_U)$ and $\iota_U \cl \Derb(\cor_U) \to \Orb(\cor_U)$ (see
Definition~\ref{def:categorie-orb}), for any $U\subset X$, induce a functor of
stacks $i_X\cl \loc(\cor_X) \to \Oloc(\cor_X)$.

Let us denote by $\psh(\cor_X)$ the prestack of presheaves of $\cor$-vector
spaces on $X$. For a given $F\in \Orb(\cor_X)$ we define a presheaf $h_X(F)$ by
$h_X(F)(U) = \Hom_{\Orb(\cor_U)}(\cor_U,F|_U)$, for any open subset $U$ of $X$.
Then $F\mapsto h_X(F)$ induces a functor of prestack $h_X \cl \Orb(\cor_X) \to
\psh(\cor_X)$.

\begin{lemma}\label{lem:equiv_loc-Oloc}
The functor of prestacks $h_X \cl \Orb(\cor_X) \to \psh(\cor_X)$ induces a
functor of stacks, denoted in the same way $h_X \cl \Oloc(\cor_X) \to
\loc(\cor_X)$.  The functors $i_X$ and $h_X$ are mutually inverse equivalences
of stacks.
\end{lemma}
\begin{proof}
The first assertion follows from Proposition~\ref{prop:SSorb-sectnulle}.
To prove the last claim, it is enough to see that $i_X$ is locally an
equivalence, that is, essentially surjective and fully faithful, and that $h_X
\circ i_X \simeq \id_{\loc(\cor_X)}$.

The functor $\Mod(\cor) \to \Orb(\cor)$ is an equivalence.  Hence
Proposition~\ref{prop:SSorb-sectnulle} also implies that $i_X|_U$ is essentially
surjective as soon as $U$ is contractible.  Let us prove that $h_X \circ i_X
\simeq \id_{\loc(\cor_X)}$. We recall the formula of
Corollary~\ref{cor:morph-QRF-QRG}
\begin{align*}
\Hom_{\Orb(\cor_U)}(\cor_U,F|_U)
& \simeq \bigoplus_{n\in\Z}  \Hom_{\Derb(\cor_U)}(\cor_U[-n],F|_U) \\
& \simeq \bigoplus_{n\in\Z}  H^n(U;F) .
\end{align*}
If $F$ is a local system we deduce $\Hom_{\Orb(\cor_U)}(\cor_U,F|_U) \simeq
F(U)$ for all contractible open subsets $U$. Hence
$h_X \circ i_X \simeq \id_{\loc(\cor_X)}$ as claimed.

For two local systems $F,G \in \loc(\cor_X)$ the formula of
Corollary~\ref{cor:morph-QRF-QRG} gives in the same way, for $U$ contractible,
\begin{align*}
\Hom_{\Orb(\cor_U)}(i_x(F)|_U,i_x(G)|_U) 
& \simeq \bigoplus_{n\in\Z} \Hom_{\Derb(\cor_U)}(F|_U[-n],G|_U) \\
& \simeq \Hom_{\loc(\cor_U)}(F|_U,G|_U).
\end{align*}
Hence $i_x$ is fully faithful and the lemma is proved.
\end{proof}

As we remarked after Lemma~\ref{lem:unique-fais-simple-orb}
$\kss^{\orb,s}(\cor_\Lambda)$ has a unique global object.
Theorem~\ref{thm:KSstack=faisc_tordus} becomes trivial in this case.

\begin{proposition}\label{prop:KSstackorb}
The stack $\kss^{\orb,s}(\cor_\Lambda)$ has a unique object, say $\shf_0$,
defined over $\Lambda$.  Moreover the functor $\OrbL{(\Lambda)}(\cor_M) \to
\Orb(\cor_\Lambda)$, $F \mapsto \mu hom^\varepsilon(F_0,F)|_\Lambda$, where
$F_0$ is local representative of $\shf_0$, and the functor $h_\Lambda$ of
Lemma~\ref{lem:equiv_loc-Oloc} induce an equivalence of stacks
$\kss^\orb(\cor_\Lambda) \isoto \loc(\cor_\Lambda)$.
\end{proposition}

\part{Convolution and microlocalization}
\label{part:conv_mic}

Let $N$ be a manifold and $\Lambda \subset \dT^*N$ a locally closed conic
Lagrangian submanifold.  As explained in Remarks~\ref{rem:description_KSstack}
and~\ref{rem:description_twKSstack} the objects of $\kss^s(\cor_\Lambda)$ are
described by simple sheaves along a covering, $F_i \in
\Derb_{(\Lambda_i)}(\cor_N)$, and gluing ``isomorphisms'' $u_{ji}\in
H^0(\Lambda_{ij};\mu hom(F_i,F_j)|_{\Lambda_{ij}})$.  For a given object of
$\kss^s(\cor_\Lambda)$ we want to find a representative in
$\Derb_{(\Lambda)}(\cor_N)$, or better, $\Derb_{\Lambda}(\cor_N)$.  For this we
would like to glue the $F_i$'s in the category $\Derb(\cor_N)$ instead of
$\kss^s(\cor_\Lambda)$.  A first step for this is to find other representatives
of the $\kssfunc_{\Lambda_i}(F_i)$'s for which the $u_{ji}$ arise from morphisms
in $\Derb(\cor_N)$.  In this part we introduce a functor, $\Psi$, which gives an
answer to this question (see Theorem~\ref{thm:muhom=hompsi} and
Corollary~\ref{cor:muhom=hompsi} below). To define $\Psi$ we need to choose a
direction on $N$ and we assume that $N$ is decomposed $N=M \times\R$.

\medskip

We set for short $\rspos = \mo]0,+\infty[$ and $\rpos = [0,+\infty[$. We usually
endow $\R$ with the coordinate $t$ and $\rpos$, $\rspos$ with the coordinate
$u$. The associated coordinates in the cotangent bundles are $(t;\tau)$ for
$T^*\R$ and $(u;\upsilon)$ for $T^*\rspos$.  We set $T^*_{\tau\geq 0}\R = \{
(t,\tau) \in T^*\R;$ $\tau\geq 0\}$ and we define $T^*_{\tau>0}\R$ similarly.
For a manifold $M$ and an open subset $U\subset M\times\R$ we define
\begin{equation}
  \label{eq:def_tau_positif}
  \begin{alignedat}{2}
T^*_{\tau\geq 0}U &= (T^*M \times T^*_{\tau\geq 0}\R) \cap T^*U, 
& \qquad T^*_{\tau\leq 0}U & = (T^*_{\tau\geq 0}U)^a ,  \\
T^*_{\tau >0}U &= (T^*M \times T^*_{\tau >0}\R) \cap T^*U ,
& \qquad T^*_{\tau < 0}U & = (T^*_{\tau > 0}U)^a.
\end{alignedat}
\end{equation}
\begin{definition}\label{def:derbtp}
Let $U$ be an open subset of $M\times\R$. We let $\Derbtp(\cor_U)$
(resp. $\Derbtpn(\cor_U)$) be the full subcategory of $\Derb(\cor_U)$ of
sheaves $F$ satisfying $\dot\SSi(F) \subset T^*_{\tau >0}U$
(resp. $\SSi(F) \subset T^*_{\tau \geq 0}U$).
\end{definition}

\section{The functor $\Psi$}\label{sec:func_psi}

The convolution product is a variant of the ``composition of kernels''
considered in~\cite{KS90} (denoted by $\circ$ -- see the
notations~\eqref{eq:def_comp_gene} and~\eqref{eq:def_conv_gene}).  It is used
in~\cite{T08} to study the localization of $\Derb(\cor_{M\times\R})$ by the
objects with microsupport in $T^*_{\tau\leq 0}(M\times\R)$, in a framework
similar to the present one. Namely, Tamarkin proves that the functor $F\mapsto
\cor_{M\times[0,+\infty[} \star F$ is a projector from $\Derb(\cor_{M\times\R})$
to the left orthogonal of the subcategory $\Derb_{T^*_{\tau\leq
    0}(M\times\R)}(\cor_{M\times\R})$ of objects with microsupport in
$T^*_{\tau\leq 0}(M\times\R)$ (see~\cite{GS11} for a survey).  We will use a
variant of Tamarkin's definition.

We will use the product~\eqref{eq:def_conv_gene} in the following special
situation.  We define the subsets of $\R\times \rspos$:
\begin{equation}\label{eq:def_cones}
  \begin{split}
\gammaof & = \{(t,u); \; 0\leq t <u\}, \\
  \lambda_0 & = \{0\} \times \rspos, \quad
\lambda_1=  \{(t,u) \in \R\times\rspos; t=u\} .
  \end{split}
\end{equation}

\begin{definition}\label{def:conv_gamma}
Let $M$ be a manifold and let $U subset M\times\R$ be the an open subset.  We
define $U_\gammaof \subset M\times\R \times \rspos$ by
\begin{equation*}
U_\gammaof 
= \{ (x,t,u) \in M\times \R \times \rspos;\; \{x\}\times [t-u,t] \subset U\}.
\end{equation*}
For $F \in \Derb(\cor_U)$ and $G \in \Derb(\cor_{\R\times \rspos})$ with
$\supp(G) \subset \ol{\gamma}$, we define $G\star F \in
\Derb(\cor_{U_\gammaof})$ by
\begin{equation}\label{eq:defGstarF}
G\star F = (\reim{s}(F \letens G))|_{U_\gammaof} ,
\end{equation}
where $s\cl U \times \R \times \rspos \to M\times \R \times \rspos$ is the sum
$s(x,t_1,t_2,u) = (x,t_1+t_2,u)$.
We define the functor $\Psi_U \cl \Derb(\cor_U) \to \Derb(\cor_{U_\gammaof})$ by 
\begin{equation}\label{eq:defPsiU}
\Psi_U(F) = \cor_{\gammaof} \star F
= (\reim{s}(F\etens \cor_{\{(t,u);\; 0\leq t < u \}}))|_{U_\gammaof}  .
\end{equation}
\end{definition}

\begin{remark}\label{rem:def-Psi}
(i) We see easily on the definition of $U_\gammaof$ that, for any submanifold
$M'$ of $M$, we have $U_\gammaof \cap (M' \times\R \times \rspos) = (U \cap (M'
\times\R))_\gammaof$.  For a disjoint union $U = \bigsqcup_{i\in I} U_i$ we
also have $U_\gammaof = \bigsqcup_{i\in I} U_{i,\gammaof}$.  Hence we can
reduce the description of $U_\gammaof$ to the case where $M$ is a point and $U=
\mo]a,b[$ is an interval of $\R$. Then we have
\begin{equation}\label{eq:Ugamma-interval}
  (]a,b[)_\gammaof = \{ (t,u) \in \R\times\rspos;\; a+u<t<b \}.
\end{equation}

\medskip\noindent
(ii) We have $\opb{s}(U_\gammaof) \cap (M\times \R \times \ol{\gammaof}) \subset U
\times \ol{\gammaof}$.  Since $\supp(G) \subset \ol{\gamma}$, it follows that we
also have $G\star F = (\reim{s'}(F' \letens G))|_{U_\gammaof}$ where $F' \in
\Derb(\cor_{M\times\R})$ is any object such that $F'|_U = F$ and $s'\cl M \times
\R^2 \times \rspos \to M\times \R \times \rspos$ is the sum.

\medskip\noindent
(iii) For the same reason the restriction of $s$ to $\opb{s}(U_\gammaof) \cap (M\times
\R \times \ol{\gammaof})$ is a proper map.  Hence we can replace $\reim{s}$ by
$\roim{s}$ in~\eqref{eq:defGstarF}.
\end{remark}

We define the projections
\begin{equation}\label{eq:def_proj}
  \begin{alignedat}{2}
q \cl M\times \R \times \rspos &\to M\times \R ,
& \qquad  (x,t,u) &\mapsto (x,t),  \\
r \cl M\times \R \times \rspos &\to M\times \R ,
 & (x,t,u) &\mapsto (x,t-u) 
  \end{alignedat}
\end{equation}
and we denote by $q_U,r_U \cl U_\gammaof \to U$ the restrictions of $q,r$ to
$U_\gammaof$. Using the notations~\eqref{eq:def_cones} we have $\cor_{\lambda_0} \star F
\simeq \opb{q_U}(F)$ and $\cor_{\lambda_1} \star F \simeq \opb{r_U}(F)$, for
any $F\in \Derb(\cor_U)$.  Since $\lambda_0 \subset \gammaof$ and $\lambda_1
\subset \ol{\gammaof} \setminus \gammaof$, we have natural morphisms
$\cor_\gammaof \to \cor_{\lambda_0}$ and $\cor_{\lambda_1} [-1] \to
\cor_\gammaof$.  They induce morphisms, for all $F\in \Derb(\cor_U)$,
\begin{equation}\label{eq:def_alphabeta}
\alpha(F) \cl \Psi_U(F) \to  \opb{q_U}(F),
\qquad
\beta(F) \cl \opb{r_U}(F)[-1] \to \Psi_U(F) .
\end{equation}
The morphism $\cor_{\lambda_1} [-1] \to \cor_\gammaof$ factorizes through
$\cor_{\lambda_1} [-1] \to \cor_{\Int(\gammaof)}$ and $\cor_{\Int(\gammaof)}
\to \cor_\gammaof$.  These morphisms induce $\beta'(F) \cl \opb{r_U}(F)[-1] \to
\cor_{\Int(\gammaof)} \star F$ and $\beta''(F) \cl \cor_{\Int(\gammaof)} \star
F \to \Psi_U(F)$.  The excision triangle for the inclusion $\lambda_0 \subset
\gammaof$ induce the distinguished triangle
\begin{equation}\label{eq:dtrPsiq0}
\cor_{\Int(\gammaof)} \star F  \to[\;\;\beta''(F)\;\;] \Psi_U(F) \to[\;\;\alpha(F)\;\;] 
 \opb{q_U}(F)  \to[+1] .
\end{equation}

\begin{lemma}\label{lem:dtgrPsiq}
For $F\in \Derbtpn(\cor_U)$ the morphism $\beta'(F) \cl \opb{r_U}(F)[-1] \to
\cor_{\Int(\gammaof)} \star F$ is an isomorphism and~\eqref{eq:dtrPsiq0} gives
the distinguished triangle
\begin{equation}\label{eq:dtgamqr2}
\opb{r_U}(F)[-1] \to[\;\;\beta(F)\;\;] \Psi_U(F) \to[\;\;\alpha(F)\;\;] 
 \opb{q_U}(F)  \to[+1] .
\end{equation}
\end{lemma}
\begin{proof}
(i) We recall that $\cor_{\lambda_1} \star F \simeq \opb{r_U}(F)$.  We set
$\gamma' = \ol{\gammaof} \setminus \lambda_0$.  Applying $\cdot \star F$ to the
excision triangle given by $\lambda_1 \subset \gamma'$ gives the distinguished
triangle
$$
 \opb{r_U}(F)[-1] \to[\;\;\beta'(F)\;\;] \cor_{\Int(\gammaof)} \star F
\to \cor_{\gamma'} \star F  \to[+1]  .
$$
Hence the first assertion follows from the vanishing of $\cor_{\gamma'} \star
F$, which we prove in~(ii).  Since the second assertion follows from the first,
this will conclude the proof.

\medskip\noindent
(ii) For $x\in M$ and $u>0$ we define $i_{(x,u)} \cl \R \to M\times\R\times
\rspos$, $t\mapsto (x,t,u)$.  To prove that $\cor_{\gamma'} \star F \simeq 0$,
it is enough to see that $\opb{i_{(x,u)}}(\cor_{\gamma'} \star F ) \simeq 0$,
for all $(x,u)$.  By the base change formula we have
$\opb{i_{(x,u)}}(\cor_{\gamma'} \star F ) \simeq \cor_{]0,u]} \star F$ and we
conclude with Lemma~\ref{lem:noyaustar} below.
\end{proof}

\begin{lemma}\label{lem:noyaustar}
Let $a<b\in \R$ and let $F\in \Derbtpn(\cor_\R)$.
Then $\cor_{]a,b]} \star F \simeq 0$.
\end{lemma}
\begin{proof}
By the definition of $\cor_{]a,b]} \star F$ its germs at some $x\in \R$ are
\begin{align*}
(\cor_{]a,b]} \star F)_x 
& \simeq  \rsect_c(\opb{s}(x); (F\etens \cor_{]a,b]}) |_{\opb{s}(x)} ) \\
& \simeq \rsect_c(\R;  F\tens \cor_{[x-b,x-a[}) .
\end{align*}
The excision triangle applied to the inclusion $\{x-a\} \subset [x-b,x-a]$
shows that $(\cor_{]a,b]} \star F)_x$ is the cone of the restriction morphism
$\rsect([x-b,x-a];F) \to F_{x-a}$, which is an isomorphism by the hypothesis on
$\SSi(F)$ and by Corollary~\ref{cor:Morse}.  Hence $(\cor_{]a,b]} \star F)_x$
vanishes for all $x\in\R$ and this proves the lemma.
\end{proof}

Let $V$ be an open subset of $U$.  Let $N$ be a submanifold of $M$ and $U' = U
\cap (N\times \R)$. We have
\begin{align}
\label{eq:restr-PsiU-ouvert}
\Psi_V(F|_V) &\simeq (\Psi_U(F))|_{V_\gammaof} ,  \\
\label{eq:restr-PsiU-sousvar}
\Psi_{U'}(F|_{U'}) &\simeq (\Psi_U(F))|_{U'_\gammaof} ,
\end{align}
where the first isomorphism follows from supports estimates as in
Remark~\ref{rem:def-Psi}~(ii) and the second one follows from the base change
formula.

\medskip

In the next lemma we use an analog of the convolution for sets.  For $A \subset
M\times \R$ and $B \subset \R\times \rspos$ we define $B\star A \subset M\times
\R\times \rspos$ by
\begin{equation}\label{eq:1}
B\star A = s(A\times B) ,
\end{equation}
where $s\cl U \times \R \times \rspos \to M\times \R \times \rspos$ is the sum
as Definition~\ref{def:conv_gamma}.

\begin{lemma}\label{lem:suppPsiuL}
Let $F\in \Derb(\cor_U)$ and let $V\subset U$ be an open subset.  We assume
that
\begin{equation}\label{eq:hyp_F_vert_cst}
\text{$F|_{V\cap (\{x\}\times \R)}$ is locally constant, for any $x\in M$.}
\end{equation}
Then $\Psi_U(F)|_{V_\gammaof} \simeq 0$.  As a special case, if $\SSi(F|_V)
\subset T^*_VV$, then $\Psi_U(F)|_{V_\gammaof} \simeq 0$. In particular
$\supp(\Psi_U(F)) \subset (\gammaf \star \dot\pi_U(\dot\SSi(F))) \cap
U_\gammaof$.
\end{lemma}
\begin{proof}
We set $V_x = V \cap (\{x\}\times \R)$. Then $V_\gammaof = \bigsqcup_{x\in M}
(V_x)_\gammaof$ and we have to prove $\Psi_U(F)|_{(V_x)_\gammaof} \simeq 0$,
for all $x\in M$. By~\eqref{eq:restr-PsiU-sousvar} we have
$\Psi_U(F)|_{(V_x)_\gammaof} \simeq \Psi_{V_x}(F|_{V_x})$.  The set $V_x$ is a
disjoint union of open intervals of $\R$ and $F|_{V_x}$ is constant on each of
these intervals. A direct computation gives $\Psi_{V_x}(F|_{V_x}) \simeq 0$ and
we obtain the result.
\end{proof}

\begin{lemma}\label{lem:annulation_qr}
Let $F \in \Derb(\cor_U)$.
\begin{itemize}
\item [(i)] We have $\reim{q_U} \epb{q_U} (F) \isoto F$ and
  $\reim{r_U}(\Psi_U(F)) \simeq 0$.
\item [(ii)] If $F\in \Derbtpn(\cor_U)$, then $\reim{q_U} \opb{r_U} (F)$
  satisfies~\eqref{eq:hyp_F_vert_cst} (with $V=U$). In particular $\Psi_U(
  \reim{q_U} \opb{r_U} (F)) \simeq 0$.
\item [(iii)] We assume that $U = M \times \R$, that $F\in \Derbtpn(\cor_U)$
  and that $\supp(F) \subset M \times [a,+\infty[$ for some $a\in \R$.  Then
  $\reim{q_U} \opb{r_U} (F) \simeq 0$.
\end{itemize}
\end{lemma}
\begin{proof}
(i) The first morphism is the adjunction for $(\reim{q_U}, \epb{q_U})$.
It is an isomorphism because the fibers of $q_U$ are intervals.
Let us prove the second isomorphism.  We define $r'\cl M\times\R^2 \times\rspos
\to M\times\R^2$, $(x,t_1,t_2,u) \mapsto (x,t_1,t_2-u)$ and $s' \cl
M\times\R^2\to M\times\R$, $(x,t_1,t_2) \mapsto (x,t_1+t_2)$.  We have the
commutative diagram
$$
\xymatrix{
M\times\R^2 \times\rspos \ar[r]^-{r'} \ar[d]_s
& M\times\R^2 \ar[d]^{s'} \\
M\times\R \times\rspos \ar[r]_-r & M\times\R . }
$$
We let $j\cl U \to M\times\R$ be the inclusion and we set $F' = \eim{j}F$.
Then $\Psi_U(F) \simeq (\cor_{\gammaof}\star F')|_{U_\gammaof}$ and we have
\begin{align*}
\reim{r_U}(\Psi_U(F))
&\simeq \reim{r} ( (\reim{s}(F' \etens \cor_{\gammaof}))_{U_\gammaof})  \\
&\simeq \reim{(r\circ s)} ( (F'\etens \cor_{\R\times\rspos}) \tens
 \cor_{\opb{q_2}\gammaof \cap \opb{s}U_\gammaof})  \\
&\simeq \reim{(s'\circ r')}( \opb{r'}(F'\etens \cor_\R) \tens 
 \cor_{\opb{q_2}\gammaof \cap \opb{s}U_\gammaof})  \\
&\simeq \reim{s'} ( (F'\etens \cor_\R) \tens
\reim{r'}( \cor_{\opb{q_2}\gammaof \cap \opb{s}U_\gammaof})).
\end{align*}
Hence it is enough to prove that
$\reim{r'}( \cor_{\opb{q_2}\gammaof \cap \opb{s}U_\gammaof}) \simeq 0$.
Our hypothesis on $U$ means that for any $x\in U' = p_M(U)$, there exists
$a_x,b_x\in\R$ such that $(\{x\}\times\R) \cap U = \{x\}\times ]a_x,b_x[$.
Then we have
$U_\gammaof = \{(x,t,u) \in M\times\R\times\rspos$; $x\in U', a_x+u<t<b_x\}$.
  For
any $(x,t_1,t_2) \in M\times\R^2$ the fiber $\opb{r'}(x,t_1,t_2) \cap
(\opb{q_2}\gammaof \cap \opb{s}U_\gammaof)$ is identified with
\begin{align*}
\{u>0; \; (x,t_1,&t_2+u,u) \in \opb{q_2}\gammaof \cap \opb{s}U_\gammaof\} \\
&= \{u>0; \; 0\leq t_2+u<u \text{ and } (x,t_1+t_2+u,u) \in U_\gammaof\} \\
&= \{u>0; \;  -t_2\leq u \text{ and }  u < b_x-t_1-t_2 \},
\end{align*}
where we assume $t_2<0$ and $a_x<t_1+t_2$ (otherwise the fiber is empty).  Since
$-t_2>0$ we see that the fiber is either empty or a half closed interval. This
implies $\reim{r'}( \cor_{\opb{q_2}\gammaof \cap \opb{s}U_\gammaof}) \simeq 0$,
as required.

\medskip\noindent
(ii) We choose $x\in M$ and we set $U_x = U\cap (\{x\}\times \R)$.  By the base
change formula we have $(\reim{q_U} \opb{r_U} (F))|_{U_x} \simeq \reim{q_{U_x}}
\opb{r_{U_x}} (F|_{U_x})$. Hence we can assume that $M$ is a point and that $U$
is an interval, say $U = \mo]a,b[$.  By Example~\ref{ex:microsupport}~(i), to
prove that $\reim{q_U} \opb{r_U} (F)$ is constant it is enough to see that its
microsupport is contained in the zero section.

We let $j\cl U \to \R$ be the inclusion and we set for short $q=q_\R, r=r_\R
\cl \R^2 \to \R$.  We set $F' = \roim{j}F$. Then we have $\reim{q_U} \opb{r_U}
(F) \simeq (\reim{q} (\opb{r} (F') \tens \cor_{\R\times \rspos}))|_U$.

We have $(\rsect_{]-\infty,a]}(F'))_a \simeq 0$ and we can deduce that
$\SSi(F') \cap T^*_a\R \subset T^*_{\tau\geq 0}\R$.  Hence $\SSi(F') \cap
T^*(\mo]-\infty,b[) \subset T^*_{\tau\geq 0}\R$.  By Theorem~\ref{th:opboim}
and Corollary~\ref{cor:opboim} we obtain successively, with coordinates
$(t,u;\tau,\upsilon)$ on $T^*\R^2$,
\begin{align*}
 \SSi(\opb{r}F'|_{U\times \R}) &\subset \{(\tau,-\tau); \;\tau\geq 0\} , \\
 \SSi((\opb{r} (F') \tens \cor_{\R\times \rspos})|_{U\times \R})
&\subset \{(\tau, -\tau + \upsilon); \;\tau\geq 0,\, \upsilon \leq 0\} , \\
 \SSi(\reim{q_U} \opb{r_U} (F)) &\subset \{\tau =  0\} ,
\end{align*}
which proves that $\reim{q_U} \opb{r_U} (F)$ is constant.

\medskip\noindent
(iii) By~(ii) $\reim{q_U} \opb{r_U} (F)$ is constant on the fibers $\{x\}
\times \R$ for all $x\in M$.  By the hypothesis its restriction to $M\times
\{a-1\}$ vanishes.  Hence it is zero.
\end{proof}

\begin{lemma}\label{lem:SSPsiF}
Let $F\in \Derbtpn(\cor_U)$. Then 
\begin{equation}\label{eq:SSPsiF}
\dot\SSi(\Psi_U(F)) =
(q_d\opb{q_\pi}(\dot\SSi(F)) \cup r_d\opb{r_\pi}(\dot\SSi(F))) \cap  \dT^*U_\gammaof .
\end{equation}
\end{lemma}
\begin{proof}
This follows from the triangle~\eqref{eq:dtgamqr2}, the triangular inequality
for the microsupport and the fact that $\dot\SSi(\opb{q}F)$ and
$\dot\SSi(\opb{r}F)$ are disjoint.
\end{proof}

\section{Adjunction properties}\label{sec:adj_prop}

Let $U\subset M\times \R$ be an open subset.  We let
$\Derbra(\cor_{U_\gammaof})$ be the full subcategory of
$\Derbtpn(\cor_{U_\gammaof})$ consisting of $\reim{r}$-acyclic objects, that is,
the objects $G$ such that $\reim{r}G \simeq 0$.  This is a triangulated
category.  By Lemma~\ref{lem:annulation_qr}~(i) the functor $\Psi_U$ takes
values in $\Derbra(\cor_{U_\gammaof})$.  By Theorem~\ref{th:opboim} the functor
$\reim{q_U}$ sends $\Derbtpn(\cor_{U_\gammaof})$ into $\Derbtpn(\cor_U)$ ($q_U$
is not proper, but we can use a homotopy colimit triangle
like~\eqref{eq:homot-lim-bis}, with the $U_n$'s relatively compact, before we
apply Theorem~\ref{th:opboim}).  Moreover the morphism of functors $\alpha \cl
\Psi_U \to \opb{q_U}$ in~\eqref{eq:def_alphabeta} and the adjunction morphism
$\reim{q_U} \epb{q_U} \simeq \reim{q_U} \opb{q_U} [1] \to \id$ induce:
\begin{equation}\label{eq:psiq-id}
  b_U(F) \cl \reim{q_U} \Psi_U (F) [1] \to F,
\qquad \text{for all $F\in \Derbtpn(\cor_U)$.}
\end{equation}

\begin{lemma}\label{lem:adj-Psi-q}
The functor $\reim{q_U}[1] \cl \Derbra(\cor_{U_\gammaof}) \to \Derbtpn(\cor_U)$ is
left adjoint to $\Psi_U \cl \Derbtpn(\cor_U) \to \Derbra(\cor_{U_\gammaof})$.
In particular we have an adjunction morphism
\begin{equation}\label{eq:id-psiq}
  b'_U(G) \cl G \to \Psi_U \reim{q_U} (G) [1] ,
\qquad \text{for all $G\in \Derbra(\cor_{U_\gammaof})$.}
\end{equation}
\end{lemma}
\begin{proof}
Since $\Derbtpn(\cor_U)$ and $\Derbra(\cor_{U_\gammaof})$ are full
subcategories of $\Derb(\cor_U)$ and $\Derb(\cor_{U_\gammaof})$, it is enough
to prove
\begin{equation}\label{eq:psiq-id1}
\Hom_{\Derb(\cor_{U_\gammaof})}(G, \Psi_U(F))
\simeq \Hom_{\Derb(\cor_U)}(\reim{q_U}G[1],F)
\end{equation}
for any $F\in \Derbtpn(\cor_U)$ and $G\in \Derbra(\cor_{U_\gammaof})$.  Since
$r_U$ is a smooth map with fibers homeomorphic to $\R$ we have a canonical
isomorphism of functors $\epb{r_U}F\simeq \opb{r_U}[1]$; hence an adjunction
$(\reim{r_U},\epb{r_U}[1])$. The same holds for $q_U$.  Applying
$\RHom(G,\cdot)$ to~\eqref{eq:dtgamqr2} we obtain the distinguished triangle
\begin{align*}
\RHom(G,\opb{r_U}F[-1]) \to \RHom(G, \Psi_U(F)) \to
\RHom(G, \opb{q_U}F)  \to[+1] .
\end{align*}
The adjunction $(\reim{r_U},\opb{r_U}[1])$ and the hypothesis $G\in
\Derbra(\cor_{U_\gammaof})$ give $\RHom(G,\opb{r_U}F[-1]) \simeq 0$. We deduce
$$
\RHom(G, \Psi_U(F)) \isoto \RHom(G, \opb{q_U}F) \simeq \RHom(\reim{q_U}G[1],F),
$$
which implies~\eqref{eq:psiq-id1}.
\end{proof}

\begin{lemma}\label{lem:qpsi-projA}
Let $F\in \Derbtpn(\cor_U)$.  Then the morphisms $\Psi_U(b_U(F))$ and
$b'_U(\Psi_U (F))$ are mutually inverse isomorphisms:
\begin{align*}
\Psi_U(b_U(F)) &\cl \Psi_U \reim{q_U} \Psi_U (F) [1] \isoto \Psi_U(F) , \\
b'_U(\Psi_U (F)) &\cl \Psi_U (F) \isoto \Psi_U \reim{q_U} \Psi_U (F) [1] .
\end{align*}
\end{lemma}
\begin{proof}
(i) We prove the first isomorphism.
We apply $\reim{q_U}[1]$ to the distinguished triangle~\eqref{eq:dtgamqr2}.
Since $q_U$ has fibers isomorphic to $\R$ the adjunction morphism $\reim{q_U}
\epb{q_U}(F) \to F$ is an isomorphism and we obtain the distinguished triangle:
\begin{equation}\label{eq:qpsi-projA1}
L \to \reim{q_U} \Psi_U (F) [1] \to[b_U(F)]  F \to[+1],
\end{equation}
where $L = \reim{q_U} \opb{r_U}(F)$. By Lemma~\ref{lem:annulation_qr}~(ii) we
have $\Psi_U(L) \simeq 0$. Hence applying $\Psi_U$ to~\eqref{eq:qpsi-projA1}
gives the lemma.

\medskip\noindent
(ii) The composition $\Psi_U(b_U(F)) \circ b'_U(\Psi_U (F))$ is the identity
morphism of $\Psi_U(F)$, by general properties of adjunctions.  Hence
the lemma follows from~(i).
\end{proof}

\begin{proposition}\label{prop:Psipresqueff}
We assume that $U = M \times \R$, that $F\in \Derbtpn(\cor_U)$ and that
$\supp(F) \subset M \times [a,+\infty[$ for some $a\in \R$.  Then the
adjunction morphism $b_U(F) \cl \reim{q_U} \Psi_U (F) [1] \to F$
of~\eqref{eq:psiq-id} is an isomorphism and for any $G\in \Derbtpn(\cor_U)$ we
have
$$
\Hom(F,G) \isoto \Hom(\Psi_U(F), \Psi_U(G)) .
$$
\end{proposition}
\begin{proof}
By Lemma~\ref{lem:annulation_qr}~(i) and~(iii) we have $\reim{q_U} \epb{q_U} (F)
\isoto F$ and $\reim{q_U} \opb{r_U} (F) \simeq 0$.  Hence the first part
follows from by applying $\reim{q_U}$ to the distinguished
triangle~\eqref{eq:dtgamqr2}.  Then the second part is given by the adjunction
$(\reim{q_U}, \Psi_U)$ of Lemma~\ref{lem:adj-Psi-q}.
\end{proof}

\section{Link with microlocalization}

In this section we prove Theorem~\ref{thm:muhom=hompsi} which gives the sections
of $\mu hom$ outside the zero section in terms of homomorphism between the image
by $\Psi_U$. 

We first look at the difference between the sections of $\mu hom(F,G)$ outside
the zero section and $\rhom(F,G)$, that is, the third term of Sato's
distinguished triangle, which is $\DD'(F) \ltens G$ in the constructible case
by~\eqref{eq:SatoDTmuhom}.  In general it is given by~\eqref{eq:not_homprime}
below and we check that it has a similar behaviour as $\DD'(F) \ltens G$.

\begin{definition}\label{def:homprime}
Let $X$ be a manifold. Let $q_{X,1},q_{X,2}\cl X\times X \to X$ be the
projections and $\delta_X \cl X \to X\times X$ the diagonal embedding.  Let
$F,F'\in \Derb(\cor_X)$. We set
\begin{equation}\label{eq:not_homprime}
\hom'(F,F') \eqdot \opb{\delta_X}\rhom(\opb{q_{X,2}}F,\opb{q_{X,1}}F').
\end{equation}
\end{definition}
Then~\eqref{eq:proj_microltion1}-\eqref{eq:proj_muhom} give a distinguished
triangle, for any $F,F'\in \Derb(\cor_X)$,
\begin{equation}\label{eq:SatoDTmuhom2}
  \begin{split}
 \hom'(F,F') \to \rhom (F,F') &  \\
\to \roim{\dot\pi_M{}} & (\mu hom  (F,F') |_{\dT^*M}) \to[+1].
  \end{split}
\end{equation}

\begin{lemma}\label{lem:im_inv_homprime}
(i) Let $f\cl X\to Y$ be a morphism of manifolds. Let $F,F'\in \Derb(\cor_Y)$
such that $f$ is non-characteristic for $\SSi(F)$ and $\SSi(F')$. Then
$$
\opb{f}\hom'(F,F') \isoto \hom'(\opb{f}F,\opb{f}F') .
$$
(ii) For $F,F'\in \Derb(\cor_X)$ we have
 $\SSi(\hom'(F,F'))\subset \SSi(F)^a \hplus \SSi(F')$.
\end{lemma}
\begin{proof}
(i) We use the notations of Definition~\ref{def:homprime}.
We set \\ $G = \rhom(\opb{q_{Y,2}}F,\opb{q_{Y,1}}F')$. Then
$\SSi(\opb{q_{Y,1}}F')$ and $\SSi(G) \subset \SSi(F)^a \times \SSi(F')$ are
non-characteristic for $f\times f$. By Theorem~\ref{th:opboim} we deduce
$\epb{(f\times f)} \opb{q_{Y,1}}F' \simeq \opb{(f\times f)} \opb{q_{Y,1}}F'
\tens \omega_{X\times X| Y\times Y}$ and $\opb{(f\times f)} G \simeq
\epb{(f\times f)}G \tens \omega_{X\times X| Y\times Y}^{\otimes -1}$.  This
gives the third isomorphism in the following sequence:
\begin{align*}
\opb{f}\hom'(F,F')
& \simeq  \opb{f}\opb{\delta_Y}\rhom(\opb{q_{Y,2}}F,\opb{q_{Y,1}}F')  \\
& \simeq \opb{\delta_X}\opb{(f\times f)} \rhom(\opb{q_{Y,2}}F,\opb{q_{Y,1}}F') \\
& \isoto \opb{\delta_X}
\rhom(\opb{(f\times f)}\opb{q_{Y,2}}F,\opb{(f\times f)}\opb{q_{Y,1}}F') \\
& \simeq \opb{\delta_X}\rhom(\opb{q_{X,2}}\opb{f}F,\opb{q_{X,1}}\opb{f}F').
\end{align*}

\noindent
(ii) follows from Theorem~\ref{th:opboim}~(i) and Theorem~\ref{thm:SSrhom}
applied with $i = \delta_X$.
\end{proof}

\begin{lemma}\label{lem:muhom_Psi_q}
Let $M$ be a manifold and $U$ an open subset of $M\times\R$.  Let $F\in
\Derbtpn(\cor_U)$ and $G\in \Derbtp(\cor_U)$. Then 
$\alpha(F)$ induces an isomorphism
\begin{equation}\label{eq:lem-muhom_Psi_q}
\begin{split}
\mu hom(\Psi_U(F),\opb{q_U}(G)) |_{\dT^*U_\gammaof} 
&\isoto \mu hom(\opb{q_U}(F),\opb{q_U}(G)) |_{\dT^*U_\gammaof}  \\
& \simeq (\eim{{q_{U,d}}} \, \opb{q_{U,\pi}}(\mu hom(F,G)))|_{\dT^*U_\gammaof} .
\end{split}
\end{equation}
\end{lemma}
\begin{proof}
The first morphism in~\eqref{eq:lem-muhom_Psi_q} appears in the distinguished
triangle deduced from~\eqref{eq:dtgamqr2} by applying the functor $\mu
hom(\cdot, \opb{q_U}(G))$. To see that it is an isomorphism we check that $\mu
hom(\opb{r_U}(F), \opb{q_U}(G))$ vanishes on $\dT^*U_\gammaof$.  By
Proposition~\ref{prop:SSmuhom} this object is supported by $S =
\SSi(\opb{r_U}(F)) \cap \SSi(\opb{q_U}(G))$.  By Theorem~\ref{th:opboim} and by
the hypothesis on $G$ we have
$$
S \cap \dT^*U_\gammaof \subset
\{ (x,t,u;\xi,\tau,-\tau) \} \cap \{ (x,t,u;\xi,\tau,0); \; \tau>0 \}
= \emptyset .
$$
Hence we have proved the first isomorphism. The second one is a general fact
stated in~\cite[Prop.~4.4.7]{KS90} which explains the behaviour of $\mu hom$
under an inverse image by a submersion.
\end{proof}

We will consider ``boundary values'' of sheaves on $U_\gammaof$: for $G \in
\Derb(\cor_{U_\gammaof})$ its boundary value is $\opb{i} \roim{j}(G) \in
\Derb(\cor_U)$, where $i,j$ are the maps defined by
\begin{equation}\label{eq:def_imm}
  \begin{alignedat}{2}
i = i_U\cl U  &\to U \times \rpos  ,
& \qquad  (x,t) &\mapsto (x,t,0) ,   \\
j = j_U \cl U \times \rspos &\to U \times \rpos ,
 & (x,t,u) &\mapsto (x,t,u) .
  \end{alignedat}
\end{equation}

\begin{lemma}\label{lem:formules-muhom_Psi_q}
Let $F \in \Derb(\cor_U)$  and $\shf\in \Derb(\cor_{T^*U})$. We have the
canonical isomorphisms:
\begin{gather*}
\roim{\dot\pi_{U_\gammaof}{}} 
((\eim{{q_{U,d}}} \, \opb{q_{U,\pi}}(\shf))|_{\dT^*U_\gammaof})
\simeq \opb{q_U}\roim{\dot\pi_U{}}(\shf|_{\dT^*U}) , \\
  \opb{i} \roim{j} \opb{q_U} F \simeq F .
\end{gather*}
\end{lemma}
\begin{proof}
The first isomorphism follows from the base change formula.
Let us prove the second one.  Let $q_1 \cl U \times \R \to U$ be the projection.
Then $\opb{i} \roim{j} \opb{q_U} F \simeq \opb{i} \rsect_{U \times
  \rspos}(\opb{q_1} F)$.  We have $\SSi(\cor_{U \times \rspos}) \subset T^*_UU
\times T^*\R$.  Hence Theorem~\ref{th:opboim} and Corollary~\ref{cor:opboim}
give
$$
\rsect_{U \times \rspos}(\opb{q_1} F) \simeq \rhom(\cor_{U \times \rspos},
\opb{q_1} F) \simeq \cor_{U \times \rpos} \tens \opb{q_1} F
$$
and we obtain $\opb{i} \roim{j} \opb{q_U} F \simeq \opb{i} (\cor_{U \times
  \rpos} \tens \opb{q_1} F) \simeq F$.
\end{proof}

\begin{lemma}\label{lem:hompPsiq}
Let $F,G \in \Derb(\cor_U)$ and $\shf = \mu hom(\Psi_U(F),\opb{q_U}(G))
\in \Derb(\cor_{U_\gammaof})$.
Then the natural morphism
\begin{align*}
\opb{i}\roim{j} \rhom(\Psi_U(F),\opb{q_U}(G))
& \simeq \opb{i}\roim{j} \roim{\pi_{U_\gammaof}}(\shf) \\
& \to  \opb{i}\roim{j} \roim{\dot\pi_{U_\gammaof}{}}(\shf|_{\dT^*U_\gammaof} )
\end{align*}
is an isomorphism.
\end{lemma}
\begin{proof}
(i) By the triangle~\eqref{eq:SatoDTmuhom2} the cone of the morphism of the
lemma is $\opb{i} \roim{j} \hom'(\Psi_U(F),\opb{q_U}(G))$.  Let us prove that it
vanishes.

Here we consider $\gammaof$ as a subset of $\R^2$ rather than $\R \times
\rspos$.  We also consider the sum $s$, $(x,t_1,t_2,u) \mapsto (x,t_1+t_2,u)$,
as a map from $U \times\R^3$ to $M\times \R^2$ and we define $\Psi'_U(F) \in
\Derb(\cor_{M\times\R^2})$ by $\Psi'_U(F) = \reim{s}(F\etens
\cor_{\gammaof})$.  Then $\Psi_U(F) \simeq \Psi'_U(F)|_{U_\gammaof}$ and,
setting $A = \hom'(\Psi'_U(F),\opb{q_U}(G))$, we have
$$
\roim{j} \hom'(\Psi_U(F),\opb{q_U}(G)) \simeq \rsect_{M\times\R\times\rspos} A .
$$
Hence we only have to prove $\opb{i} \rsect_{M\times\R\times\rspos} A \simeq 0$.

\medskip\noindent
(ii) Setting $V = (U\times \mo]-\infty,0]) \cup U_\gammaof$ we see that $s$ is
proper as a map from $\opb{s}(V) \cap (U \times \ol{\gammaof})$ to $V$.  Hence
we can use Theorem~\ref{th:opboim} to bound $\SSi(\Psi'_U(F)|_V)$.  We have
$\SSi(\cor_\gammaof) \subset \R^2\times C$, where $C$ is the subset of
$(\R^2)^*$ given by $C = \{(\tau,\sigma)$; $-\tau \leq \sigma \leq 0\}$.  We
deduce that $\SSi(\Psi'_U(F)|_V) \subset T^*M \times (\R^2\times C)$.  We also
have $\SSi(\opb{q_U}(G)) \subset \{\sigma = 0\}$ and
Lemma~\ref{lem:im_inv_homprime}~(ii) then implies $\SSi(A) \subset \{\sigma \geq
0\}$.

\medskip\noindent
(iii) By Corollary~\ref{cor:opboim}~(ii), used with $F =
\cor_{M\times\R\times\rspos}$ and $G=A$ we deduce
$\rsect_{M\times\R\times\rspos} A \simeq A_{M\times\R\times\rpos}$. It follows
that $\opb{i} \rsect_{M\times\R\times\rspos} A \simeq \opb{i} A$.  By
Lemma~\ref{lem:im_inv_homprime}~(i) we have $\opb{i} A \simeq \hom'(\opb{i}
\Psi'_U(F),\opb{i}\opb{q_U}(G))$. 
Since $\cor_\gammaof |_{\R \times \{0\}} =0$, the base change formula implies
$\opb{i} \Psi'_U(F) \simeq 0$ and this concludes the proof.
\end{proof}

\begin{proposition}\label{prop:muhom_Psi_q}
Let $M$ be a manifold and $U$ an open subset of $M\times\R$.  Let $F\in
\Derbtpn(\cor_U)$ and $G\in \Derbtp(\cor_U)$. Then we have an isomorphism
\begin{equation}\label{eq:muhom_Psi_q-vers-hom2}
\opb{i} \roim{j}  \rhom(\Psi_U(F),\opb{q_U}(G))  \isoto 
 \roim{\dot\pi_U{}}( \mu hom(F,G)|_{\dT^*U} ) 
\end{equation}
which is functorial in $F$ and $G$.
\end{proposition}
\begin{proof}
By Lemma~\ref{lem:hompPsiq}  the left hand side
of~\eqref{eq:muhom_Psi_q-vers-hom2} is isomorphic to
$$
\opb{i}\roim{j} \roim{\dot\pi_{U_\gammaof}{}}(
\mu hom(\Psi_U(F),\opb{q_U}(G)) |_{\dT^*U_\gammaof} ) .
$$
By Lemma~\ref{lem:muhom_Psi_q} this is again isomorphic to
$$
\opb{i}\roim{j} \roim{\dot\pi_{U_\gammaof}{}}
(\eim{{q_{U,d}}} \, \opb{q_{U,\pi}}(\mu hom(F,G))|_{\dT^*U_\gammaof} )
$$
and we conclude with  Lemma~\ref{lem:formules-muhom_Psi_q}.
\end{proof}

\begin{theorem}\label{thm:muhom=hompsi}
Let $M$ be a manifold and $U$ an open subset of $M\times\R$.  Let $F\in
\Derbtpn(\cor_U)$ and $G\in \Derbtp(\cor_U)$. Then we have an isomorphism
$$
\opb{i}\roim{j} \rhom(\Psi_U(F),\Psi_U(G))
\isoto  \roim{\dot\pi_U{}}(\mu hom(F,G)|_{\dT^*U} )
$$
which is functorial in $F$ and $G$.
\end{theorem}
\begin{proof}
(i) The morphism of the theorem is induced by~\eqref{eq:muhom_Psi_q-vers-hom2} and
the morphism $\alpha(G) \cl \Psi_U(G) \to \opb{q_U}(G)$.
Since~\eqref{eq:muhom_Psi_q-vers-hom2} is an isomorphism, we only have to show
that
$$
\opb{i}\roim{j} \rhom(\Psi_U(F),\Psi_U(G))
\to \opb{i} \roim{j}  \rhom(\Psi_U(F),\opb{q_U}(G)) 
$$
also is an isomorphism, that is, that its cone vanishes.  By the distinguished
triangle~\eqref{eq:dtgamqr2} its cone is $A = \opb{i}\roim{j} \rhom(\Psi_U(F),
\opb{r_U}(G))$.

\smallskip\noindent
(ii) For a given $(x,t) \in U$ and $k\in \Z$ we have
\begin{equation}\label{eq:germopbiroimj}
  H^k(A)_{(x,t)} \simeq \varinjlim_W \Hom(\Psi_U(F)|_W, \opb{r_U}(G)|_W[k]),
\end{equation}
where $W$ runs over the open subsets of $M\times\R\times\rspos$ such that
$\ol{W}$ is a neighborhood of $(x,t,0)$ in $U\times [0,+\infty[$.  We may take
$W = V_\gammaof$, where $V$ runs over the open neighborhoods of $(x,t)$ in $U$.
By~\eqref{eq:restr-PsiU-ouvert} we have $\Psi_U(F)|_{V_\gammaof} \simeq
\Psi_V(F|_V)$.  We also have $\opb{r_U}(G)|_{V_\gammaof} \simeq
\opb{r_V}(G|_V)$.  Since $\epb{r_V} \simeq \opb{r_V}[1]$, the adjunction
$(\reim{r_V}, \epb{r_V})$ gives
$$
\Hom(\Psi_V(F|_V), \opb{r_V}(G|_V)[k])
 \simeq \Hom(\reim{r_V}\Psi_V(F|_V), G|_V[k-1]).
$$
By Lemma~\ref{lem:annulation_qr} we have $\reim{r_V}\Psi_V(F|_V) \simeq 0$ and
we deduce the vanishing of~\eqref{eq:germopbiroimj} for all $(x,t)$ in $U$.
Hence $A \simeq 0$, as required.
\end{proof}

If $Z$ is a compact subset of $U$ and $\shf \in \Derb(\cor_{U_\gammaof})$, we
have the isomorphism $H^k(Z; (\opb{i}\roim{j} \shf) |_Z) \simeq \varinjlim_W
H^k(W;\shf)$ where $W$ is open and $Z \subset \ol{W}$. We deduce:

\begin{corollary}\label{cor:muhom=hompsi}
In the situation of Theorem~\ref{thm:muhom=hompsi} let $Z$ be a compact subset
of $U$ and let $k\in \Z$.  We have
$$
H^k(\opb{\dot\pi_U}(Z); \mu hom(F,G))
\simeq \varinjlim_W \Hom(\Psi_U(F)|_W,\Psi_U(G)[k]|_W) ,
$$
where $W$ runs over the open subsets of $M\times\R\times\rspos$ such that
$\ol{W}$ is a neighborhood of $Z$ in $U \times [0,+\infty[$.
\end{corollary}

\part{Quantization of Lagrangian submanifolds}
\label{part:quant}

In this part and the following we consider a a manifold $M$ and a closed conic
Lagrangian submanifold $\Lambda \subset T^*_{\tau >0}(M\times \R)$ which is the
cone over a compact exact Lagrangian submanifold of $T^*M$ (see
Lemma~\ref{lem:cond_Lambda_exact} below).  We will prove in
Theorem~\ref{thm:quant_canon} that there exists a unique simple sheaf $F$ along
$\Lambda$ such that $\SSi(F) =\Lambda$ and $F|_{M \times \{t_0\}} \simeq 0$ for
$t_0\ll 0$ and $F|_{M \times \{t_0\}} \simeq \cor_M$ for $t_0\gg 0$.  We recover
known results on the topology of $\Lambda$, namely that its Maslov class and
relative Stiefel-Whitney class vanish and that the projection to $M$ is a
homotopy equivalence.

We proceed in several steps.  Let $m_\Lambda$ be the Maslov class of $\Lambda$.
We first assume $\cor = \Z/2\Z$ and construct $F \in
\Derb_{(\Lambda_0)}(\cor_{M\times\R})$ for open subsets $\Lambda_0 \subset
\Lambda$ such that $m_\Lambda|_{\Lambda_0} = 0$.  We do this by a gluing
procedure similar to the gluing of perverse sheaves on a complex manifolds. We
do not work directly on $M\times\R$. In fact we glue the sheaves $\Psi_U(F_i)$
on $M\times\R\times \mo]0,\varepsilon[$ (and take the restriction to $M \times
\R \times \{u\}$ at the end of the construction).  Then $\dot\SSi(F)$ has a
bound given by $\Lambda^+ \eqdot q_d\opb{q_\pi}(\Lambda) \cup
r_d\opb{r_\pi}(\Lambda)$ as in Lemma~\ref{lem:SSPsiF}.  Doing this for two
subsets $\Lambda_0, \Lambda_1 \subset \Lambda$ such that $m_\Lambda|_{\Lambda_0}
= m_\Lambda|_{\Lambda_1} = 0$ and $\Lambda_0 \cup \Lambda_1 = \Lambda$, we can
construct $F \in \OrbL{\Lambda^+}(\cor_{M\times\R \times \mo]0,\varepsilon[})$.
Then $F_u \eqdot F|_{M \times \R \times \{u\}}$ has microsupport $\Lambda_u =
\Lambda \cup T'_u(\Lambda)$ where $T'_u$ is the translation by $u$ is the factor
$\R$ of $M\times \R$.

We can find a Hamiltonian isotopy $\Phi$ of $\dT^*(M\times \R)$ such that
$\Phi_s(\Lambda) = \Lambda$ and $\Phi_s(T'_u(\Lambda)) = T'_{u+s}(\Lambda)$ for
$s,u >0$.  Using~\cite{GKS10} we can associate with $\Phi$ a sheaf on $(M\times
\R)^2$ whose composition with $F_u$ gives $F_{u+s}$ with microsupport
$\Lambda_{u+s}$. For $s\gg 0$ there exists $A\in \R$ such that $\Lambda \subset
T^*(M \times \mo]-\infty,A[)$ and $T'_{u+s}(\Lambda) \subset T^*(M \times
\mo]A,+\infty[)$.  Then $F_{u+s}|_{M \times \mo]-\infty,A[}$ is our quantization
(using a suitable diffeomorphism $\R \simeq \mo]-\infty,A[$).  The result at
this step is stated in Theorem~\ref{thm:quantOrb}: for any $\shf\in
\kss^\orb(\cor_\Lambda)$ there exists $F \in \OrbL{\Lambda}(\cor_{M\times \R})$
such that $\kssfunc^\orb_{\Lambda}(F) \simeq \shf$.

The category $\Orb(\cor_{M\times \R})$ carries of course less information than
$\Derb(\cor_{M\times \R})$ but, as we saw in Lemma~\ref{lem:equiv_loc-Oloc}, the
monodromy of locally constant objects is not lost.  The object $F_+ = F|_{M
  \times \{t_0\}}$ for $t_0\gg 0$ is locally constant (since $\Lambda \subset
T^*(M \times \mo]-\infty,A[)$).  Using some results of microlocal sheaf theory
we can see that $\Hom(F,F') \simeq \Hom(F_+, F'_+)$ for $F, F' \in
\OrbL{\Lambda}(\cor_{M\times \R})$.  Then we deduce a relation between the
monodromy of $\shf\in \kss^\orb(\cor_\Lambda)$ along $\Lambda$ and the monodromy
of $F_+$ along $M$: we prove that the morphism $\pi_1(\Lambda) \to \pi_1(M)$ is
injective.

The last result implies that, for a suitable cyclic cover $r\cl M' \to M$, the
components of $r^*(\Lambda)$ have Maslov class zero.  Then we can quantize them
by $F' \in \Der(\cor_{M'\times \R})$.  We see that a non zero Maslov class for
$\Lambda$ would imply that $F'$ is unbounded, though $F'_+$ is locally constant
and locally bounded, hence bounded. This gives a contradiction and shows that
$m_\Lambda =0$.

Since $m_\Lambda =0$, what we have done in $\Orb(\cor_{M\times \R})$ can be
performed in $\Derb(\cor_{M\times \R})$ (still with $\cor = \Z/2\Z$).  The
relation $\Hom(F,F') \simeq \Hom(F_+, F'_+)$ and
Corollary~\ref{cor:muhom=hompsi} then imply $\rsect(M;\cor_M) \simeq
\rsect(\Lambda;\cor_\Lambda)$.  In particular $H^2(M;\Z/2\Z) \simeq
H^2(\Lambda;\Z/2\Z)$.  Hence the relative Stiefel-Whitney class of $\Lambda$ is
the inverse image of a class $c \in H^2(M;\Z/2\Z)$.  Working with $c$-twisted
sheaves we can define a quantization $F$ with coefficients $\Z$.  Then $F_+$ is
a non zero $c$-twisted locally constant sheaf and this implies $c=0$.  Hence the
relative Stiefel-Whitney class also vanishes and we can finally construct a
quantization of $\Lambda$. We deduce $\rsect(M;\cor_M) \simeq
\rsect(\Lambda;\cor_\Lambda)$ for any ring $\cor$.  We can also prove an
equivalence between local systems on $M$ and $\Lambda$. Hence $\pi_1(\Lambda)
\simeq \pi_1(M)$ and the projection to $\Lambda \to M$ is a homotopy
equivalence.

\section{Deformation of the microsupport}\label{sec:def_micsup}

We recall here a result of~\cite{GKS10} which says that a deformation of a
microsupport of a sheaf by some Hamiltonian isotopy induces a ``deformation'' of
the sheaf. In particular, for a given conic Lagrangian submanifold $\Lambda
\subset \dT^*N$ and a homogeneous Hamiltonian isotopy $\phi$ of $\dT^*N$, the
categories $\Derb_\Lambda(\cor_N)$ and $\Derb_{\phi(\Lambda)}(\cor_N)$ are
equivalent.  The same result is used in Corollary~\ref{cor:restr_it_equiv} to
extend a sheaf with microsupport $\Lambda_\varepsilon \eqdot
(q_d\opb{q_\pi}(\Lambda) \sqcup r_d\opb{r_\pi}(\Lambda)) \cap T^*(M\times \R
\times \mo]0,\varepsilon[)$, for small $\varepsilon$, to a sheaf with
microsupport $\Lambda_\varepsilon$, for all $\varepsilon$.

\medskip

Let $N$ be a manifold and let $I$ be an open interval of $\R$ and let $u_0 \in
I$ be given.  We consider a homogeneous Hamiltonian isotopy $\phi\cl
\dT^*N\times I\to \dT^*N$ of class $C^\infty$, that is, $\phi$ is a
$C^\infty$-map and, denoting by $\phi_u \cl \dT^*N\times \{u\}\to \dT^*N$ the
restriction at time $u$, we have
\begin{itemize}
\item [(i)]  $\phi_{u_0} = \id_{\dT^*N}$,
\item [(ii)] $\phi_u$ is a homogeneous symplectic isomorphism for each $u\in I$.
\end{itemize}
We let $\Gamma'_\phi = \{(\phi(x;\xi),(x;-\xi),u)$; $(x,\;\xi) \in \dT^*N$,
$u\in I\}$ be the union of the graphs of $\phi_u$, $u\in I$. This is a subset of
$(\dT^*N^2) \times I$.  We can check that there exists a unique closed conic
Lagrangian submanifold $\Gamma_\phi \subset \dT^*(N^2\times I)$ which is
identified with $\Gamma'_\phi$ through the projection induced by $T^*I\to I$:
\begin{equation}\label{eq:def_Gamma_phi}
\vcenter{\xymatrix@R=5mm@C=1.5cm{
\Gamma_\phi \ar@{^{(}->}[r] \ar_=[d] & \dT^*(N^2\times I) \ar[d] \\
\Gamma'_\phi \ar@{^{(}->}[r]  &  (\dT^*N^2) \times I .}}
\end{equation}
Now we let $L_{u_0} \subset \dT^*N$ be a closed conic subset and we set $L_u =
\phi_u(L_0)$, for all $u\in I$.  We consider the disjoint union of these
subsets, say $L' = \bigsqcup_{t\in I} (L_u \times \{u\})$, as a subset of
$(\dT^*N) \times I$.  We define
\begin{equation}\label{eq:def_L}
L = \Gamma_\phi \circ L_{u_0} \quad \subset \quad \dT^*(N\times I),
\end{equation}
which is a closed conic subset whose projection to $(\dT^*N) \times I$ is $L'$.
For a given $u \in I$ we let $i_u \cl N \to N\times I$ be the inclusion $x
\mapsto (x,u)$.  Then $L$ is non-characteristic for $i_u$ and we have $L_u =
(i_u)_d(\opb{(i_u)_\pi}(L))$, for any $u\in I$.  By Theorem~\ref{th:opboim} the
inverse image of sheaves by $i_u$ gives a functor
\begin{equation}\label{eq:restr_iu}
  \opb{i_u} \cl \Derlb_L(\cor_{N\times I}) \to  \Derlb_{L_u}(\cor_N).
\end{equation}

\begin{proposition}{\rm(Prop.~3.12 of~\cite{GKS10})}
\label{prop:restr_it_equiv}
For any $u\in I$ the functor~\eqref{eq:restr_iu} is an equivalence of
categories.
\end{proposition}

\begin{remark}
\label{rem:hamisot-kernel}
Proposition~\ref{prop:restr_it_equiv} has the following consequence (which is
also stated in~\cite{GKS10} as corollary of the main theorem): the categories
$\Derlb_{L_0}(\cor_N)$ and $\Derlb_{L_u}(\cor_N)$ are equivalent.  In the same
way, the categories $\Derlb_{(L_0)}(\cor_N)$ and $\Derlb_{(L_u)}(\cor_N)$ are
equivalent. Hence, to prove to prove the existence of an object in
$\Derlb_{\Lambda}(\cor_N)$ or $\Derlb_{(\Lambda)}(\cor_N)$, we may assume that
$\Lambda$ is in a generic position.
\end{remark}

We will use Proposition~\ref{prop:restr_it_equiv} in a case where the set $L$
of~\eqref{eq:restr_iu} is obtained from a Lagrangian submanifold of $T^*_{\tau
  >0}(M\times \R)$ which satisfies the equivalent conditions of the following
easy lemma.

\begin{lemma}\label{lem:cond_Lambda_exact}
Let $\Lambda$ be a closed conic Lagrangian submanifold of $T^*_{\tau >0}(M\times
\R)$.  The following conditions are equivalent:
\begin{itemize}
\item [(i)] the map $T^*_{\tau >0}(M\times\R) \to T^*M$, $(x,t;\xi,\tau) \mapsto
  (x;\xi/\tau)$, induces an injection $\Lambda/\rspos \hookrightarrow T^*M$ and
  $\Lambda/\rspos$ is compact,
\item [(ii)] there exists a compact exact Lagrangian submanifold
  $\widetilde\Lambda \subset T^*M$ and $f\cl \widetilde\Lambda \to \R$ such that
  $df = \alpha_M|_{\widetilde\Lambda}$ and
$$
\Lambda = \{(x,t;\xi,\tau);\; \tau>0, \; (x;\xi/\tau) \in \widetilde\Lambda,
\; t = -f(x;\xi/\tau) \} .
$$
\end{itemize}
\end{lemma}

For $u\in\R$ we define the translation $T_u \cl M\times\R \to M\times\R$, $(x,t)
\mapsto (x,t+u)$. We denote by $T'_u \cl T^*(M\times\R) \to T^*(M\times\R)$,
$(x,t;\xi,\tau) \mapsto (x,t+u;\xi,\tau)$, the induced map on the cotangent
bundle.  We also introduce some notations, for $\Lambda \subset
T^*_{\tau >0}(M\times\R)$:
\begin{equation}\label{eq:def-Lambdaplus}
\begin{split}
\Lambda_u &= \Lambda \cup T'_u(\Lambda) \quad \subset
 \quad T^*_{\tau >0}(M\times\R), \quad \text{for $u>0$,}  \\
\Lambda^+ &= q_d\opb{q_\pi}(\Lambda) \sqcup r_d\opb{r_\pi}(\Lambda)
\quad \subset \quad T^*_{\tau >0}(M\times\R\times\rspos) .
\end{split}
\end{equation}
By Lemma~\ref{lem:SSPsiF}, if $U\subset M\times\R$ and $F \in \Derbtpn(\cor_U)$
is such that $\dot\SSi(F) \subset \Lambda$, we have $\dot\SSi(\Psi_U(F)) \subset
\Lambda^+$.  In particular $\Lambda^+$ is the natural bound of the microsupport
of the object given by Proposition~\ref{prop:quant_ouvert1}.  We remark that
$\Lambda^+$ is non-characteristic for the inclusions $i_u$, $u>0$, and that
$\Lambda_u = (i_u)_d(\opb{(i_u)_\pi}(\Lambda^+))$.

\begin{lemma}\label{lem:isotopy_transl}
Let $\Lambda$ be a closed conic Lagrangian submanifold of $T^*_{\tau >0}(M\times
\R)$ which satisfies the conditions of Lemma~\ref{lem:cond_Lambda_exact}.
Then there exists a homogeneous Hamiltonian isotopy $\phi\cl \dT^*(M\times\R)
\times \rspos \to \dT^*(M\times\R)$ such that $\phi_1 = \id$ and, using the
notations~\eqref{eq:def_Gamma_phi} and~\eqref{eq:def-Lambdaplus}, we have
$\Gamma_\phi \circ \Lambda_1 = \Lambda^+$.  In particular $\phi_u(\Lambda_1) =
\Lambda_u$, for all $u>0$.
\end{lemma}
\begin{proof}
(i) We set $I=\rspos$ (to distinguish between the set of parameters and the
group $\rspos$ acting in the fibers).  Since the map $(x,t;\xi,\tau) \mapsto
(x;\xi/\tau)$ induces an injection $\Lambda/\rspos \hookrightarrow T^*M$, the
sets $\Lambda$ and $T'_u(\Lambda)$ are disjoint for all $u>0$.  Considering all
$u>0$ at once we define the following closed subsets of $\dT^*(M\times \R)
\times I$:
$$
\Lambda^0 = \Lambda \times \rspos ,  \qquad
\Lambda^1 = \bigsqcup_{u>0} (T'_u(\Lambda) \times \{u\}) .
$$
Then $\Lambda^0$ and $\Lambda^1$ are disjoint and the projections
$\Lambda^i/\rspos \to I$ are proper for $i=0,1$.  Hence we can find a conic
neighborhood $\Omega$ of $\Lambda^1$ in $\dT^*(M\times \R) \times I$ such that
$\ol\Omega \cap \Lambda^0 = \emptyset$ and the projection $\ol\Omega/\rspos \to
I$ is proper, that is, $\ol\Omega \cap (\dT^*(M\times \R) \times \{u\})$ is
compact for all $u>0$.

\medskip\noindent
(ii) We choose a $C^\infty$-function $h\cl \dT^*(M\times\R) \times I \to \R$
such that,
\begin{itemize}
\item [(a)] $h_u \eqdot h|_{\dT^*(M\times\R) \times \{u\}}$ is homogeneous of
  degree $1$, for all $u\in I$,
\item [(b)] $h$ vanishes outside $\Omega$,
\item [(c)] there exists a neighborhood $\Omega'$ of $\Lambda^1$ such that
  $h(x,t;\xi,\tau) = -\tau$, for all $(x,t;\xi,\tau) \in \Omega'$.
\end{itemize}
By~(a), (b) and the compactness of $(\ol\Omega \cap (\dT^*(M\times \R) \times
\{u\})) / \rspos$, the Hamiltonian flow of $h$, say $\phi$, is defined on $I$.
We choose for initial time $t_0=1$.  Then $\phi_u$ is the identity map outside
$\Omega$ for all $u\in I$.  Since the Hamiltonian vector field of the function
$-\tau$ is $H_{-\tau} = \partial/\partial t$ we have $\phi_u(x,t;\xi,\tau) =
(x,t+u-1;\xi,\tau)$, for all $((x,t;\xi,\tau),u) \in \Omega'$.  We deduce the
formula for the unique conic Lagrangian $\Gamma_\phi$ above the graph of $\phi$
(introduced in the diagram~\eqref{eq:def_Gamma_phi}):
\begin{align*}
&\Gamma_\phi \cap (\dT^*(M\times\R) \setminus \Omega)^2 \times T^*I
 =  \{((x,t;\xi,\tau),(x,t;-\xi,-\tau),(u;0))\} ,\\
&\Gamma_\phi \cap (\Omega')^2 \times T^*I
 = \{((x,t+u;\xi,\tau),(x,t;-\xi,-\tau),(u;-\tau))\} .
\end{align*}
Since $\Lambda^0 \subset (\dT^*(M\times\R) \setminus \Omega)$ and
$\Lambda^1 \subset \Omega'$ the lemma follows.
\end{proof}

\begin{corollary}\label{cor:restr_it_equiv}
Let $\Lambda$ be a closed conic Lagrangian submanifold of $T^*_{\tau >0}(M\times
\R)$ which satisfies the conditions of Lemma~\ref{lem:cond_Lambda_exact} (or the
conclusions of Lemma~\ref{lem:isotopy_transl}).  Let $\Lambda_u$, $u>0$, and
$\Lambda^+$ be the sets defined in~\eqref{eq:def-Lambdaplus}.
Then we have:

\noindent
{\rm (i)} The inverse image functor induces an equivalence of categories
\begin{equation}\label{eq:restr_iu2}
  \opb{i_u} \cl \Derlb_{\Lambda^+}(\cor_{M\times \R\times \rspos})
 \to  \Derlb_{\Lambda_u}(\cor_{M\times \R}),
\end{equation}
for any $u>0$. In particular, for all $F,G \in
\Derlb_{\Lambda^+}(\cor_{M\times\R\times \rspos})$ we have
\begin{equation}\label{eq:hLambda-hLambdaI2}
  \RHom(F,G) \isoto \RHom(F|_{M\times \R \times \{u\}},
G|_{M\times \R \times \{u\}}) .
\end{equation}

\smallskip\noindent
{\rm (ii)} The restriction functor induces an equivalence of categories
\begin{equation}\label{eq:restr_iu3}
\Derlb_{\Lambda^+}(\cor_{M\times\R\times \rspos}) \to
\Derlb_{\Lambda^+}(\cor_{M\times\R\times\mo]0,u[})
\end{equation}
for any $u>0$. In particular, for all $F,G \in
\Derlb_{\Lambda^+}(\cor_{M\times\R\times \rspos})$ we have
\begin{equation}\label{eq:hLambda-hLambdaI}
  \RHom(F,G) \isoto \RHom(F|_{M\times \R \times \mo]0,u[},
G|_{M\times \R \times \mo]0,u[}) .
\end{equation}
\end{corollary}
\begin{proof}
We choose $0<u<u'$. By Proposition~\ref{prop:restr_it_equiv} and
Lemma~\ref{lem:isotopy_transl}, applied with $I=\rspos$ or $I=\mo]0,u'[$, we
obtain that the functors $A$ and $B$ in the following diagram
$$
\xymatrix@C=0mm{
\Derlb_{\Lambda^+}(\cor_{M\times\R\times \rspos}) 
\ar[rr]^R \ar[dr]_A 
&& \Derlb_{\Lambda^+}(\cor_{M\times\R\times\mo]0,u'[}) \ar[dl]^B \\
& \Derlb_{\Lambda_u}(\cor_{M\times \R})
}
$$
are equivalences of categories. It follows that the restriction functor $R$ is
an equivalence of categories.  The functor $A$ is~\eqref{eq:restr_iu2} and the
functor $R$ is~\eqref{eq:restr_iu3}.  Hence this proves the proposition.
\end{proof}

\begin{corollary}\label{cor:Hom-isomorphes}
We keep the hypothesis of Corollary~\ref{cor:restr_it_equiv}.
Then for any $F_1,F_2 \in \Derb_{\Lambda}(\cor_{M\times\R})$ we have
isomorphisms
\begin{equation*}
\RHom(F_1,F_2)  \isoto \RHom(\opb{q}F_1, \opb{r}F_2) 
 \isoto \RHom(F_1, \oim{T_u}F_2)  ,
\end{equation*}
for any $u\geq 0$.
\end{corollary}
\begin{proof}
Since $r$ is a submersion with fibers diffeomorphic to $\rspos$ we have
$\RHom(F_1,F_2) \isoto \RHom(\opb{r}F_1, \opb{r}F_2)$.  By hypothesis we have
$\Lambda \subset \{\tau>0\}$ and we can consider the distinguished
triangle~\eqref{eq:dtgamqr2} (for $U=M\times\R$).  Applying $\RHom(\cdot,
\opb{r}F_2)$ to~\eqref{eq:dtgamqr2} we obtain the triangle
\begin{equation}\label{eq:Hom-isomorphes1}
\begin{split}
\RHom(\opb{r}F_1, \opb{r}F_2) \to {}&\RHom(\opb{q}F_1, \opb{r}F_2) \\
& \to \RHom(\Psi_{M\times\R} (F_1), \opb{r}F_2) \to[+1] .
\end{split}
\end{equation}
By Lemma~\ref{lem:annulation_qr}~(ii) we have $\reim{r}(\Psi_{M\times\R}(F_1))
\simeq 0$. Hence the third term in~\eqref{eq:Hom-isomorphes1} vanishes and we
obtain the first isomorphism of the corollary.  The second one follows
from~\eqref{eq:hLambda-hLambdaI2} applied with $F = \opb{q}F_1$ and
$G=\opb{r}F_2$.
\end{proof}

In the situation of Corollary~\ref{cor:restr_it_equiv} we will need a stronger
condition than ``locally bounded'' on the objects of
$\Derlb_{\Lambda^+}(\cor_{M\times \R\times \rspos})$.  This is precised in the
following lemma.

\begin{lemma}\label{lem:restr_eq_bounded}
We keep the hypothesis and notations of Corollary~\ref{cor:restr_it_equiv}.
Let $U \subset M$ be a connected relatively compact open subset and let
$\varepsilon>0$ be given.  Let $F \in \Derlb_{\Lambda^+}(\cor_{M\times \R\times
  \rspos})$.  We assume that $F|_{U \times ]-A,A[\times ]0,\varepsilon[}$ is
bounded for any $A>0$.  Then $F|_{U\times \R\times \rspos} \in
\Derb(\cor_{U\times \R\times \rspos})$.
\end{lemma}
\begin{proof}
We choose $A>0$ such that $\dot\pi_{M\times\R}(\Lambda) \subset M \times
\mo]-A,A[$.  Then we also have $\dot\pi_{M\times\R}(\Lambda) \subset M \times
\mo]-A,A[$. We choose $B,C$ such that $A<B<C$.  We introduce the following open
subsets of $U\times \R \times \rspos$:
\begin{align*}
U_1 &= \opb{r}(U \times \mo]A,+\infty[),
\quad\qquad U_2 = \opb{q}(U \times \mo]-\infty,-A[), \\
U_3 &= \opb{r}(U \times \mo]-\infty,-A[) 
\cap \opb{q}(U \times \mo]A,+\infty[), \\
V_1 &= \opb{r}(U \times \mo]-B,B[) \cap (U \times \R \times \mo]2B,+\infty[), \\
V_2 &= U \times \mo]-B,B\mc[ \times \mo]2B,+\infty[ , \\
W_1 &= U \times \mo]-C,3C\mc[ \times \mo]0,\varepsilon[  ,
\qquad W_2 = U \times \mo]-C,3C\mc[ \times \mo]\varepsilon/2,2C[ .
\end{align*}
By hypothesis $F|_{W_1}$ is bounded. Since $\ol{W_2}$ is compact we also know
that $F|_{W_2}$ is bounded. We have $\dot\pi_{M\times\R}(\Lambda) \cap U_i =
\emptyset$ for $i=1,2,3$ and we deduce that $F$ is locally constant on $U_i$,
$i=1,2,3$. Moreover the sets $U_1,U_2,U_3$ are connected and meet $W_2$. We
deduce that $F|_{U_i}$ is bounded, $i=1,2,3$.

By Theorem~\ref{th:opboim} the restriction $F|_{V_1}$ is of the type
$\opb{r}(F_1)$ for some $F_1 \in \Derlb(\cor_{U \times \mo]-B,B[})$.  We set $N
= V_1 \cap (M \times \R \times \{C+B\})$.  Then $N \subset W_2$ and we
deduce that $F|_N$ is bounded. It follows that $F_1$ and then $F|_{V_1}$ are
bounded.  The same argument shows that $F|_{V_2}$ is bounded.
Since $U\times \R\times \rspos$ is covered by the sets $U_i$, $V_i$, $W_i$ this
concludes the proof.
\end{proof}

\section{Genericity hypothesis}

We consider a manifold $M$ and a closed conic Lagrangian submanifold $\Lambda
\subset T^*_{\tau >0}(M\times \R)$. Our aim is to find $F\in
\Derb(\cor_{M\times\R})$ such that $\dot\SSi(F) = \Lambda$. Since the Maslov
class of $\Lambda$ could a priori be non zero and give a topological obstruction
for the existence of such an $F$, we first restrict to an open subset
$\Lambda_0$ of $\Lambda$ and look for $F\in
\Derb_{(\Lambda_0)}(\cor_{M\times\R})$.  The local geometry is easy if we ask
that any $p \in \partial\Lambda_0$ has a neighborhood where we have $\Lambda_0 =
\Lambda \cap T^*U_0$ for some $U_0 \subset M\times\R$ depending on $p$.  Then
$F$ is locally of the form $\rsect_{U_0}(F')$ for some $F'$ with $\dot\SSi(F) =
\Lambda$ (around $p$).  By Corollary~\ref{cor:opboim} we have
$\SSi(\rsect_{U_0}(F')) \subset \Lambda + N^*_{\partial U_0}$ under some
non-characteristicity hypothesis, where the interior conormal bundle
$N^*_{\partial U_0}$ is defined in~\eqref{eq:def_ext_con_bun}. We want that
$\Lambda$ be open in $\Lambda + N^*_{\partial U_0}$. This fails if $\Lambda +
N^*_{\partial U_0}$ has self-intersections. We also want that the hypothesis of
Proposition~\ref{prop:muhomFUF} are satisfied. This motivates the following
assumption.

\begin{assumption}\label{assu:bonne_pos}
Let $\Lambda \subset T^*_{\tau >0}(M\times \R)$ be a closed conic Lagrangian
submanifold and let $\Lambda_0 \subset \Lambda$ be an open subset.  We consider
the following hypothesis on $(\Lambda_0,\Lambda)$: there exists a family of open
subsets $U_i$, $i\in I$, of $M\times\R$ such that $\Lambda \subset \bigcup_{i
  \in I} T^*U_i$ and, for each $i \in I$, we have
\begin{itemize}
\item [(i)] denoting by $\Lambda_{i,j}$, $j \in J_i$, the connected components
  of $\Lambda \cap T^*U_i$ and setting $W_{i,j} = p_M(\dot\pi_{U_i}(\Lambda_0
  \cap \Lambda_{i,j}))$, $V_{i,j} = U_i \cap (W_{i,j} \times \R)$, we have
  $\Lambda_0 \cap \Lambda_{i,j} = T^*V_{i,j} \cap \Lambda_{i,j}$,
\item [(ii)] for each $j \in J_i$, the boundary $\partial V_{i,j} \cap U_i$
  is smooth and
$$
\dT^*U_i \cap \Lambda_{i,j} \cap N_{V_{i,j}}^{*e} =  \emptyset,
 \qquad
\dT^*_{\Lambda_{i,j}} \dT^*U_i \cap \dT^*_{\partial \dT^*V_{i,j}} \dT^*U_i = \emptyset,
$$
\item [(iii)] $\dT^*U_i \cap (\Lambda_{i,j} + N_{V_{i,j}}^*) \cap (\Lambda_{i,j'} +
  N_{V_{i,j'}}^*) = \emptyset$, for all $j\not= j' \in J_i$.
\end{itemize}
\end{assumption}
Our assumption on $(\Lambda,\Lambda_0)$ could as well be written: for any $x \in
M\times \R$ there exists a neighborhood $U$ of $x$ such that~(i)-(iii) hold with
$U_i$ replaced by $U$.

By~(i) we can write $\Lambda \cap T^*U_i = \bigsqcup _{j \in J_i} \Lambda_{i,j}$
and $\Lambda_0 \cap T^*U_i = \bigsqcup _{j \in J_i} \Lambda_{i,j} \cap
T^*V_{i,j}$.  As explained before the statement of the
assumption~\ref{assu:bonne_pos} we have the following bound for the
microsupports of some sheaves associated with $\Lambda_0$:
\begin{equation}\label{eq:def_cloture-ms-Lambda0}
  \Lambda_0^\clms 
= \bigcup_{i\in I} \bigcup_{j\in J_i} (\dT^*U_i \cap (\Lambda_{i,j} + N_{V_{i,j}}^*)) ,
\end{equation}
Near a point of $\partial V_{i,j} \cap U_i$, the set $V_{i,j}$ only depends on
$\Lambda_0$. Hence $\Lambda_0^\clms$ only depends on $(\Lambda_0, \Lambda)$ and
not on the choice of a family $U_i$ satisfying~\mbox{(i)-(iii)}.  By~(iii) we
have $\Lambda_0^\clms \cap \Lambda = \ol{\Lambda_0}$.

\begin{lemma}\label{lem:deform_bonne_pos}
Let $\Lambda \subset T^*_{\tau >0}(M\times \R)$ be a conic Lagrangian submanifold
such that $\Lambda/\rspos$ is compact.  Let $\Lambda_1 ,\Lambda_2$ be open
subsets of $\Lambda$ such that $\ol{\Lambda_1} \subset \Lambda_2$.
Then there exists an open subset $\Lambda_0$ of $\Lambda$ such that $\Lambda_1
\subset \Lambda_0 \subset \Lambda_2$ and $(\Lambda_0,\Lambda)$ satisfies the
assumption~\ref{assu:bonne_pos}.
\end{lemma}
\begin{proof}
We cover $\pi_{M\times \R}(\Lambda_2)$ by a finite number of open subsets
$U_i = \opb{\varphi_i}(]0,+\infty[) \times \mo]a_i,b_i[$, $i=1,\ldots,N$, where
the $\varphi_i \cl M \to \R$ are $\Cinf$ functions.  We take the $U_i$ small
enough so that, setting $I_1 = \{i$; $\Lambda_1\cap T^*U_i \not= \emptyset\}$,
we have $(\Lambda \cap \bigcup_{i\in I_1} \ol{T^*U_i}) \subset \Lambda_2$.

We let $\Lambda_{i,j}$, $j\in J_i$, be the components of $\Lambda \cap
T^*U_i$. We set $\varphi_{i,j} = \varphi_i$ and $V_{i,j} = U_i$.  By
Lemma~\ref{lem:hyp_muhomFUF_gen} we can modify $\varphi_{1,1}$ such that
\begin{equation}\label{eq:deform_bonne_pos1}
T^*_{\partial V_{1,1}}(M\times \R) \cap \Lambda = \emptyset,
\quad
\dT^*_\Lambda \dT^*(M\times\R) \cap \dT^*_{\partial \dT^*V_{1,1}} \dT^*(M\times\R)
= \emptyset .
\end{equation}
We set $\Xi_{1,1} = \Lambda + T^*_{\partial V_{1,1}}(M\times \R)$.  By
Lemma~\ref{lem:hyp_muhomFUF_gen} again, applied with the function
$\varphi_{1,2}$ and $\Lambda ' = \Lambda \cup \Xi_{1,1}$, we can modify
$\varphi_{1,2}$ such that~\eqref{eq:deform_bonne_pos1} holds with $V_{1,1}$
replaced by $V_{1,2}$ and $\Lambda$ replaced by $\Lambda \cup \Xi_{1,1}$.  We
set $\Xi_{1,2} = \Lambda + T^*_{\partial V_{1,2}}(M\times \R)$.

We go on with Lemma~\ref{lem:hyp_muhomFUF_gen} applied with the function
$\varphi_{1,3}$ and $\Lambda' = \Lambda \cup \Xi_{1,1} \cup \Xi_{1,2}$.  We
obtain $V_{1,3}$ and $\Xi_{1,3} = \Lambda + T^*_{\partial V_{1,3}}(M\times
\R)$, etc.  We modify all $V_{i,j}$ in this way by induction.  At the end we
define $\Lambda_0 = \Lambda \cap \bigcup_{i\in I_1,\, j\in J_i} T^*V_{i,j}$.
Then $\Lambda_0$ has the required property.
\end{proof}

\section{Gluing}\label{sec:gluing}

In general the objects of the derived category of sheaves on a space $X$ are not
determined by their restrictions to the open subsets of a covering of $X$.  In
particular we can not glue a family $F_i$, $i\in I$, of locally defined objects
into a global object.  Perverse sheaves are examples of objects of the derived
category of sheaves where gluing is possible (see for
example~\cite[Prop.~10.2.9]{KS90}).  More generally this is possible when the
$\rhom(F_i,F_i)$ are concentrated in non negative degrees.

Let $X$ be a manifold.  Let $\{U_i^N \}_{N\in \N}$, $i\in I$, be a finite number
of decreasing sequences of open subsets of $X$: $U_i^{N+1} \subset U_i^N$, for
all $i \in I$ and all $N\in \N$.  We set $U^N = \bigcup_{i\in I} U_i^N$.

\begin{lemma}\label{lem:gluingNsheaves}
We consider $F_i\in \Derb(\cor_{U^{N_0}_i})$, $i\in I$, defined for some given
$N_0 \in \N$, and morphisms $\varphi_{ji}\cl F_i|_{U^{N_0}_{ij}} \to
F_j|_{U^{N_0}_{ij}}$, for all $i,j \in I$, such that, $\varphi_{ii} = \id_{F_i}$
and $\varphi_{kj}|_{U^{N_0}_{ijk}}\circ \varphi_{ji}|_{U^{N_0}_{ijk}} =
\varphi_{ki}|_{U^{N_0}_{ijk}}$, for all $i,j,k \in I$.  We assume that
for any open subset $V\subset X$, any $k<0$ and any $i\in I$, we have
\begin{equation}\label{eq:gluingNsheaves1}
\varinjlim_N H^k(U^N_i \cap V; \rhom(F_i,F_i) ) \simeq 0 .
\end{equation}
Then there exist $N' \in \N$ and $F\in \Derb(\cor_{U^{N'}})$ together with
isomorphisms $\varphi_i\cl F|_{U^{N'}_i} \isoto F_i|_{U^{N'}_i}$, $i \in I$,
such that 
\begin{itemize}
\item [(i)] $ \varphi_{ji}|_{U^{N'}_{ij}} = \varphi_j|_{U^{N'}_{ij}} \circ
  \varphi_i^{-1}|_{U^{N'}_{ij}}$ for all $i\not= j \in I$,
\item [(ii)] for any $G\in \Derb(\cor_{U^{N'}})$ such that for any $V\subset X$,
  $k<0$ and $i\in I$
\begin{equation*}
\varinjlim_N H^k(U^N_i \cap V; \rhom(G,F_i) ) \simeq 0 ,
\end{equation*}
we have the exact sequence
\begin{equation}\label{eq:gluingNsheaves2}
\begin{split}
0 \to  \varinjlim_N \Hom(G|_{U^N},F|_{U^N})
\to & \bigoplus_{i \in I}  \varinjlim_N \Hom(G|_{U^N_i}, F_i|_{U^N_i}) \\
& \to \bigoplus_{i\not= j \in I} \varinjlim_N \Hom(G|_{U^N_{ij}}, F_i|_{U^N_{ij}} ).
\end{split}
\end{equation}
\end{itemize}
Moreover, for other $F'$ and $\varphi'_i$ satisfying~(i)-(ii) there exist $N''$
and an isomorphism $\varphi \cl F|_{U^{N''}} \isoto F'|_{U^{N''}}$ such that
$\varphi_i = \varphi'_i \circ \varphi$ for all $i\in I$.
\end{lemma}
\begin{proof}
(a) We proceed by induction on $|I|$.  For $i,j \in I$ we let $u_i \cl U_i^{N_0}
\to X$ and $u_{ij} \cl U_{ij}^{N_0} \to X$ be the inclusions.  Then
the morphism $\varphi_{ji}$ induces $\varphi'_{ji} \cl \roim{u_i}(F_i) \to
\roim{u_{ij}}(F_j|_{U_{ij}^{N_0}})$.  When $I = \{1,2\}$ we define $F$ by the
distinguished triangle
\begin{equation}\label{eq:gluingNsheaves3}
F \to[f]  \roim{u_1}(F_1) \oplus  \roim{u_2}(F_2)
  \to[g] \roim{u_{12}}(F_2|_{U_{12}^{N_0}}) \xto[+1]{h} ,
\end{equation}
where $g = (\varphi'_{21} ,- n_2)$ and $n_2 \cl \roim{u_2}(F_2) \to
\roim{u_{12}}(F_2)$ is the natural morphism.  The morphism $f$ induces
$\varphi_i \cl F|_{U_i^{N_0}} \to F_i$, $i=1,2$.  Since $n_2|_{U_1^{N_0}}$ and
$\varphi'_{21}|_{U_2^{N_0}}$ are isomorphisms, the restrictions of the triangle to
$U_1^{N_0}$ and $U_2^{N_0}$ split. We deduce that $\varphi_1$ and $\varphi_2$ are
isomorphisms and also that $\varphi_{21} = \varphi_1|_{U_{12}^{N_0}} \circ
\varphi_2^{-1}|_{U_{12}^{N_0}}$.

\medskip\noindent
(b) Let $G$ be given as in~(ii) and let $u_1 \cl G|_{U_1^{N}} \to
F_1|_{U_1^{N}}$ , $u_2 \cl G|_{U_2^{N}} \to F_2|_{U_2^{N}}$ be given such that
$(g \circ (u_1,u_2))|_{U_{12}^{N'}} = 0$ for some $N' \geq N$.  We have to prove
that $(u_1,u_2)$ factorizes through  $u \cl G|_{U^{N''}} \to
F|_{U^{N''}}$, for some $N'' \geq N'$, and $u$ is unique in the limit over $N$.

The existence of $u$ results from properties of distinguished triangle (namely
that $\Hom(G,\cdot)$ turn triangles into long exact sequences).  For the same
reason, if $u'$ is another factorization of $(u_1,u_2)$ defined on $U^{N''}$,
then there exists
$$
v \cl G|_{U^{N''}} \to \roim{u_{12}}(F_2|_{U_{12}^{N''}}) [-1]
$$
such that $u|_{U^{N''}} - u' = h[-1] \circ v$.  We remark that $v$ belongs to
$H^{-1}(U_2^{N''} \cap U_{12}^{N_0} ; \rhom(G,F_2))$. The condition on $G$
gives $v|_{U^{N'''}} =0$ for some $N''' \geq N''$.  Hence $u|_{U^{N'''}} =
u'|_{U^{N'''}}$, as required.

\medskip\noindent
(c) Now we assume $I = \{1,\ldots,k\}$.  We let $F_0$ be the $F$ given in~(a)
for $F_1, F_2$.  By~(b) the morphisms $\varphi_{1i}$ and $\varphi_{2i}$, for
$i>2$, factorize through a unique morphism $\varphi_{0,i} \cl F_i|_{U_{0i}^{N'}}
\to F_0|_{U_{0i}^{N'}}$, for some $N'$.  By the unicity, the compatibility
conditions on the $\varphi_{ij}$'s, for $i,j \in I$, give compatibility
conditions on the $\varphi_{ij}$'s, for $i,j \in \{0,3,\ldots,k\}$.  By the
triangle~\eqref{eq:gluingNsheaves3}, the vanishing
condition~\eqref{eq:gluingNsheaves1} on $F_1, F_2$ yields the same condition on
$F_0$.  Then the induction proceeds.
\end{proof}

\section{Quantization}

In Theorem~\ref{thm:quant} we prove the existence of a quantization of an object
$\shf \in \kss(\cor_{\Lambda})$, that is, $F \in \Derb_{\Lambda}(\cor_{M\times
  \R})$ such that $\kssfunc_{\Lambda}(F) \isoto \shf$.  We first built a
quantization in $\Derb_{(\Lambda_0)}(\cor_{M\times \R})$, or rather,
$\Derb_{(q_d\opb{q_\pi}(\Lambda_0))}(\cor_{M\times\R\times\mo]0,\varepsilon[})$
in Proposition~\ref{prop:quant_ouvert1}, for an open subset $\Lambda_0 \subset
\Lambda$.  This will be used in the case of orbit categories in
Theorem~\ref{thm:quantOrb}.  We have already seen that the existence of a non
trivial $\shf \in \kss(\cor_{\Lambda})$ implies the vanishing of the Maslov
class of $\Lambda$, whereas $\kss^\orb(\cor_{\Lambda})$ always has a simple
global object.

We will represent objects of a twisted stack $(\kss(\cor_\Lambda))_{\check a}$
as in Remark~\ref{rem:description_twKSstack}. However we will use the stronger
assumption that the microsupport of the representatives coincides with
$\Lambda$ outside the zero-section.

\begin{definition}\label{def:goodrepr_twKSstack}
Let $\Lambda \subset T^*_{\tau >0}(M\times \R)$ be a closed conic Lagrangian
submanifold and let $\Lambda_0 \subset \Lambda$ be an open subset.  A good
representative of an object $\shf \in (\kss(\cor_{\Lambda_0}) )_{\check c}$ is
the data of open subsets $U_i$ for $i \in I$, and  $\Lambda_{i,j}$ for
$i\in I$, $j \in J_i$ satisfying the assumption~\ref{assu:bonne_pos} together
with $F_{i,j} \in \Derb(\cor_{U_i})$ and $u_{jj'}\in H^{c_{jj'}}(\Lambda_0 \cap
\Lambda_{jj'};\mu hom(F_{i,j},F_{i',j'}))$ such that $\dot\SSi(F_j) \subset
\Lambda_j$, for each $i\in I$ and $j\in J_i$, and the data
$\{(F_i,u_{ij})\}_{i,j\in I}$ represent $\shf$ as in
Remark~\ref{rem:description_twKSstack}.
\end{definition}

By Lemmas~\ref{lem:simple_local_base} and~\ref{lem:deform_bonne_pos} the objects
of $(\kss(\cor_{\Lambda_0}) )_{\check c}$ have good representatives, up to
putting $\Lambda$ in a generic position.

For $\varepsilon>0$ we denote by $q\cl M\times\R \times \mo]0,\varepsilon\mc[
\to M\times\R$ the projection.  For $\Lambda \subset T^*_{\tau >0}(M\times \R)$
we have $q_d\opb{q_\pi}(\Lambda) \subset T^*_{\tau >0}(M\times \R\times
\mo]0,\varepsilon\mc[)$ and the inverse image $\opb{q}$ induces a functor,
denoted in the same way
\begin{equation}\label{eq:opbkss}
  \opb{q} \cl \kss(\cor_\Lambda) \isoto \kss(\cor_{q_d\opb{q_\pi}(\Lambda)}) ,
\end{equation}
which is an equivalence by the following argument.  We recall that a choice of
simple sheaf $F$ gives a local equivalence between $\kss(\cor_\Lambda)$ and
$\Dloc(\cor_\Lambda)$. Then $\opb{q}F$ gives an equivalence between
$\kss(\cor_{q_d\opb{q_\pi}(\Lambda)})$ and
$\Dloc(\cor_{q_d\opb{q_\pi}(\Lambda)})$ and our functor corresponds to the usual
inverse image from $\Dloc(\cor_\Lambda)$ to
$\Dloc(\cor_{q_d\opb{q_\pi}(\Lambda)})$.  Since $q_d\opb{q_\pi}(\Lambda) \simeq
\Lambda \times \mo]0,\varepsilon[$, this last functor is an equivalence.

\begin{proposition}\label{prop:quant_ouvert1}
Let $\Lambda \subset T^*_{\tau >0}(M\times \R)$ be a closed conic Lagrangian
submanifold such that $\Lambda/\rspos$ is compact and let $\Lambda_0 \subset
\Lambda$ be an open subset which satisifes the assumption~\ref{assu:bonne_pos}.
Let $U_i$, $i \in I$, and $V_{j}$, $\Lambda_{j}$, $j \in J_i$ be as in the
assumption~\ref{assu:bonne_pos}.  We set $J = \bigsqcup_{i\in I} J_i$ and $J_0 =
\{j\in J$; $\Lambda_{j} \cap \Lambda_0 \not= \emptyset\}$.  Let $\check c$ be a
\v Cech cocycle over $\Lambda_0$ associated with the covering $\{\Lambda_0 \cap
\Lambda_{j}\}_{j\in J_0}$ of $\Lambda_0$.  Let $\shf \in
(\kss(\cor_{\Lambda_0}))_{\check c}$ be an object with a good representative
$\{F_{j}$, $u_{jj'}\}_{j,j'\in J_0}$ as in
Definition~\ref{def:goodrepr_twKSstack}.  We assume that $\check c$ is a
coboundary and we choose a cochain $\check b = \{b_j\}_{j\in J_0}$ such that
$\check c|_{\{\Lambda_j\}_{j\in J_0}} = \partial \check b$. We assume that
$\shf$ is pure (see Definition~\ref{def:simple_pure}).  Then there exist
$\varepsilon>0$, $F \in
\Derb_{(q_d\opb{q_\pi}(\Lambda))}(\cor_{M\times\R\times\mo]0,\varepsilon[})$ and
isomorphisms, for all $i\in I$,
$$
\varphi_i \cl F|_{U_{i,\varepsilon}} 
\isoto \bigoplus_{j\in J_i\cap J_0} \Psi_{U_i}(\rsect_{V_j}F_j[b_j])|_{U_{i,\varepsilon}},
$$
where
$U_{i,\varepsilon} = U_{i,\gammaof} \cap (M\times\R\times\mo]0,\varepsilon[)$,
such that
\begin{itemize}
\item [(a)] $\supp(F) \subset \ol{\gammaof} \star
  \dot\pi_{M\times\R}(\ol{\Lambda_0})$ and
$$
\dot\SSi(F) \subset   q_d\opb{q_\pi}(\Lambda_0^\clms)
 \cup r_d\opb{r_\pi}( \Lambda_0^\clms) 
\cup T^*M \times T^*_{\R\times\rspos}(\R\times\rspos) ,
$$
where $\Lambda_0^\clms$ is defined in~\eqref{eq:def_cloture-ms-Lambda0},
\item [(b)] for $i,i' \in I$, the morphism $\varphi_{i'} \circ \varphi_i^{-1}
  |_{U_{ii'}}$ represents $(u_{j'j})_{(j',j) \in K_{i'i}}$ through the
  isomorphism of Theorem~\ref{thm:muhom=hompsi}, where \\ $K_{i'i} = \{ (j',j)
  \in J_{i'} \times J_i$; $\Lambda_j \cap \Lambda_{j'} \cap \Lambda_0 \not=
  \emptyset\}$,
  \item [(c)] the $\varphi_i$'s induce an isomorphism
    $\kssfunc_{q_d\opb{q_\pi}(\Lambda_0)}(F) \isoto \opb{q}(\shf)$ in
    $\kss(\cor_{q_d\opb{q_\pi}(\Lambda_0)})$, where $\opb{q}$ is defined
    in~\eqref{eq:opbkss}.
\end{itemize}
Moreover the $F$ and $\varphi_i$'s satisfying~{\rm(a)-(c)} are unique up to
isomorphism.
\end{proposition}
\begin{proof}
(i) We can assume that each $F_i$ is defined on a neighborhood, say $U'_i$, of
$\ol{U_i}$.  We can also assume that $\Lambda$ is non-characteristic for
$\partial U_i$ for each $i$. Then we have $(\cor_\Lambda)_{\ol{\dT^*U_i}} \simeq
\rsect_{\dT^*U_i}(\cor_\Lambda)$ on $\dT^*(M\times\R)$. By the
assumption~\ref{assu:bonne_pos} we also have $(\cor_\Lambda)_{\ol{\dT^*V_j}}
\simeq \rsect_{\dT^*V_j}(\cor_\Lambda)$ on $\dT^*(M\times\R)$.

\medskip\noindent
(ii) For each $i\in I$ we choose a decreasing sequence of open subsets, $U_i^N$,
$N\in \N$, of $U'_{i,\gammaof}$ such that $U_i^N$ contains $\ol{U_{i}} \times
\mo]0,1/N]$, for $N$ bigger than some $N_0$, and that $(M\times\R \times \{0\})
\cap \bigcap_{N\in \N} \ol{U_i^N} = \ol{U_i}$, where $\ol{U_i^N}$ is the closure
of $U_i^N$ in $M\times\R \times\rpos$.  We define $G_i \in \Derb(\cor_{U_i^0})$
by $G_i = \bigoplus_{j\in J_i} \Psi_{U'_i}(\rsect_{V_j}F_j[b_j])|_{U_i^0}$.

We recall from the assumption~\ref{assu:bonne_pos} that $V_j = U_i \cap (W_{i,j}
\times \R)$. We deduce $\Psi_{U_i}(\rsect_{V_j}F) \simeq \rsect_{V_{j,\gammaof}}
\Psi_{U_i}(F)$ and $\Psi_{U_i}(F_{V_j}) \simeq (\Psi_{U_i}(F))_{V_{j,\gammaof}}$
for any $F \in \Derb(\cor_{U_i})$.  For $i,i' \in I$ and $j\in J_i$, $j'\in
J_{i'}$ we obtain
\begin{align*}
\rhom(\Psi_{U_i}(\rsect_{V_j}F_j) &, \Psi_{U_{i'}}(\rsect_{V_{j'}}F_{j'}) ) \\
&\simeq
\rhom(\Psi_{U_i}((\rsect_{V_j}F_j)_{V_{j'}}) , \Psi_{U_{i'}} (F_{j'}) ) .
\end{align*}

\medskip\noindent
(iii) For $i,i' \in I$ we denote by $J_{ii'}$ the set of pairs $(j,j') \in J_i
\times J_{i'}$ such that $\Lambda_j \cap \Lambda_{j'} \cap \Lambda_0 \not=
\emptyset$.  If $(j,j') \not\in J_{ii'}$, we have
\begin{align*}
\supp (\mu hom(\rsect_{V_j}F_j, F_{j'}) )
&\subset \dot\SSi((\rsect_{V_j}F_j)_{V_{j'}}) \cap \dot\SSi(F_{j'})  \\
&\subset (\Lambda_j + \dT^*_{\partial V_{j}} U_i + \dT^*_{\partial V_{j'}} U_i )
 \cap \Lambda_{j'} = \emptyset ,
\end{align*}
by the assumption~\ref{assu:bonne_pos}~(iii).  If $(j,j') \in J_{ii'}$, then
$\Lambda_j \cap T^*U_{ii'} = \Lambda_{j'} \cap T^*U_{ii'}$ and $V_j \cap U_{ii'}
= V_{j' }\cap U_{ii'}$.  Hence $\rsect_{V_j}(F_j)_{V_{j'}} |_{U_{ii'}} \simeq
(F_j)_{V_{j}} |_{U_{ii'}}$. By Proposition~\ref{prop:muhomFUF} we have
$$
\mu hom((F_j)_{V_j}, F_{j'})|_{T^*U_{ii'}}
 \simeq \mu hom(F_j, F_{j'})_{\ol{T^*V_j}}|_{T^*U_{ii'}} .
$$
We remark that $T^*U_{ii'} \cap \Lambda_j \cap T^*V_j = T^*U_{ii'} \cap
\Lambda_{j'} \cap T^*V_{j'} = T^*U_{ii'} \cap \Lambda_j \cap \Lambda_0$. Hence
Corollary~\ref{cor:muhom=hompsi} gives
\begin{align*}
\varinjlim_N H^l&(U_{ii'}^N ; \rhom(G_i,G_{i'})) \\
& \simeq \bigoplus_{j,j'} H^l(\ol{T^*U_{ii'}} \cap \dT^*(M\times\R) ;
 \mu hom((\rsect_{V_j}(F_j))_{V_{j'}}, F_{j'}) )  \\
& \simeq \bigoplus_{(j,j') \in J_{ii'}} H^l(\ol{T^*U_{ii'}} \cap \dT^*(M\times\R) ;
 \mu hom(F_j, F_{j'})_{\ol{T^*V_j}} ) \\
& \simeq
\bigoplus_{(j,j') \in J_{ii'}} H^l(\dT^*U_{ii'} \cap \Lambda_j \cap \Lambda_0;
 \mu hom(F_j, F_{j'})) ,
\end{align*}
where the last isomorphism follows from the hypothesis in part~(i) of the proof.
Since $\shf$ is pure, $\mu hom(F_j, F_{j'})$ is concentrated in degree $0$ (and
is a locally constant sheaf on $\Lambda$).  Hence its cohomology over any subset
vanishes in negative degrees and the hypothesis of
Lemma~\ref{lem:gluingNsheaves} is satisfied.  Moreover the sections $u_{jj'}$ of
$\mu hom(F_j, F_{j'})$ induce morphisms $\varphi_{i'i} \cl G_i|_{U_{ii'}^N} \to
G_{i'}|_{U_{ii'}^N}$ for $N\gg 0$ satisfying the cocycle condition
$\varphi_{i''i'} \varphi_{i'i} = \varphi_{i''i}$.  We apply
Lemma~\ref{lem:gluingNsheaves} to glue the $G_i$ into an object $F$ defined over
the open subset $U^{N'} = \bigcup_{i\in I} U^{N'}_i$, for some $N' \gg 0$.
There exists $\varepsilon>0$ such that $U_{i,\varepsilon} \subset U^{N'}$ for
all $i\in I$.  Then the extension by $0$ of $F|_{U^{N'} \cap
  (M\times\R\times\mo]0,\varepsilon[)}$ satisfies the conclusions of the
proposition.
\end{proof}

When we have two open subsets $\Lambda'_0 \subset \Lambda_0 \subset \Lambda$
there exists a morphism between the objects given by
Proposition~\ref{prop:quant_ouvert1} for $\Lambda_0$ and $\Lambda'_0$.  The next
result follows from Proposition~\ref{prop:quant_ouvert1} and part~(ii) of
Lemma~\ref{lem:gluingNsheaves}.

\begin{proposition}\label{prop:quant_ouvert2}
We consider the situation of Proposition~\ref{prop:quant_ouvert1} and we let
$\Lambda'_0 \subset \Lambda_0$ be an open subset such that
$(\Lambda,\Lambda'_0)$ also satisfies the assumption~\ref{assu:bonne_pos}, with
respect to the same family of open subsets $U_i$, $i \in I$.  We let $V'_{j}$,
$j \in J_i$ be as in the assumption~\ref{assu:bonne_pos} and we let $\shf' \in
(\kss(\cor_{\Lambda'_0}))_{\check c}$ be the restriction of $\shf$. We let
$\varepsilon>0$, $F' \in
\Derb_{(q_d\opb{q_\pi}(\Lambda))}(\cor_{M\times\R\times\mo]0,\varepsilon[})$
and isomorphisms
$$
\varphi'_i \cl F'|_{U_{i,\varepsilon}} 
\isoto \bigoplus_{j\in J_i\cap J_0} \Psi_{U_i}(\rsect_{V'_j}F_j[b_j])|_{U_{i,\varepsilon}},
\quad i\in I,
$$
be given by Proposition~\ref{prop:quant_ouvert1}
such that 
$\kssfunc_{q_d\opb{q_\pi}(\Lambda'_0)}(F') \isoto \opb{q}(\shf')$.
Then there exists a unique morphism $u\cl F \to F'$ such that, for each $i\in
I$, $u|_{U_{i,\varepsilon}}$ is the sum of the morphisms induced by the natural
morphisms $\rsect_{V_j}F_j \to \rsect_{V'_j}F_j$, $j\in J_i\cap J_0$.
\end{proposition}

\begin{theorem}\label{thm:quant}
Let $\Lambda$ be a closed conic Lagrangian submanifold of $T^*_{\tau >0}(M\times
\R)$ which satisfies the conditions of Lemma~\ref{lem:cond_Lambda_exact}.  Then,
for any $\shf \in \kss(\cor_{\Lambda})$ there exists $F \in
\Derb(\cor_{M\times\R})$ such that $\dot\SSi(F) = \Lambda$, $F|_{M\times \{t\}}
\simeq 0$ for $t\ll 0$ and $\kssfunc_{\Lambda}(F) \simeq \shf$.
\end{theorem}
\begin{proof}
(i) By Remark~\ref{rem:hamisot-kernel} and Lemmas~\ref{lem:simple_local_base}
and~\ref{lem:deform_bonne_pos} we can assume that $\Lambda$ satisfies the
hypothesis of Proposition~\ref{prop:quant_ouvert1}.  Let $\Lambda^+ \subset
T^*_{\tau >0}(M\times\R\times\rspos)$ be the set defined
in~\eqref{eq:def-Lambdaplus}.  We consider $\varepsilon>0$ and $F_1 \in
\Derb(\cor_{M\times\R\times\mo]0,\varepsilon[})$ with $\dot\SSi(F_1) \subset
\Lambda^+ \cap \dT^*(M\times\R\times \mo]0,\varepsilon[)$ given by
Proposition~\ref{prop:quant_ouvert1}.  By the equivalence of
categories~\eqref{eq:restr_iu3} there exists $F_2 \in
\Derlb(\cor_{M\times\R\times\rspos})$ extending $F_1$ and such that
$\dot\SSi(F_2) \subset \Lambda^+$.  Since $F_1$ is bounded,
Lemma~\ref{lem:restr_eq_bounded} says that $F_2$ is also bounded.

We choose $A,B$ such that $\Lambda \subset T^*(M \times \mo]A,B\mc[)$.  By~(a)
of Proposition~\ref{prop:quant_ouvert1} we have $\supp(F_1) \subset M\times
[A,+\infty\mc[ \times \mo]0,\varepsilon[$.  We also have $\dot\SSi(F_2) \subset
\Lambda^+ \subset T^*(M \times \mo]A,+\infty\mc[ \times \rspos)$ and it follows
that $F_2$ is locally constant on $M \times \mo]-\infty,A\mc[ \times
\rspos$. Since $F_2$ coincides with $F_1$ on $M\times \R
\times\mo]0,\varepsilon[$ we obtain that $\supp(F_2) \subset [A,+\infty\mc[
\times \rspos$.

We also have $\kssfunc_{q_d\opb{q_\pi}(\Lambda)}(F_2) \isoto \opb{q}(\shf)$ in
$\kss(\cor_{q_d\opb{q_\pi}(\Lambda)})$, where $\opb{q}$ is defined
in~\eqref{eq:opbkss}.

\medskip\noindent
(ii) We choose $u> 2+B-A$ and we let $i_u \cl M \times \R \times \{u\} \to M
\times \R \times \rspos$ be the inclusion. In the notations
of~\eqref{eq:def-Lambdaplus} we have $(i_u)_d(\opb{(i_u)_\pi}(\Lambda^+)) =
\Lambda_u = \Lambda \cup T'_u(\Lambda)$.  We have $\Lambda \subset T^*(M \times
\mo]-\infty,B[)$ and $T'_u(\Lambda) \subset T^*(M \times \mo]B+2,+\infty[)$.

We set $F_3 = (\opb{i_u} F_2)|_{M \times \mo]-\infty,B+2[}$.  By~(i) we have
$\dot\SSi(F_3) \subset \Lambda$, $\supp(F_3) \subset [A,B+2\mc[$ and
$\kssfunc_{\Lambda}(F_3) \isoto \shf$.

We choose a diffeomorphism $f \cl \R \to \mo]-\infty,B+2[$ such that $f$ is the
identity on $\mo]-\infty,B[$ and we set $F = \opb{(\id_M \times f)}(F_3)$. Then
$F$ satisfies the conclusions of the theorem.
\end{proof}

\section{Restriction at infinity}

Let $\Lambda$ be a closed conic Lagrangian submanifold of $T^*_{\tau >0}(M\times
\R)$ which satisfies the conditions of Lemma~\ref{lem:cond_Lambda_exact} or,
more generally, the conclusions of Lemma~\ref{lem:isotopy_transl}.

Since $\Lambda/\rspos$ is compact we can choose $A>0$ such that $\Lambda \subset
T^*_{\tau >0}(M\times \mo]-A,A[)$.  Then, for any $F \in
\Derlb_\Lambda(\cor_{M\times\R})$, the restrictions $F|_{M\times
  \mo]-\infty,-A[}$ and $F|_{M\times \mo]A,+\infty[}$ have locally constant
cohomology sheaves.

\begin{definition}\label{def:DerlbLambdaplus}
For $F \in \Derlb_\Lambda(\cor_{M\times\R})$ we define $F_-, F_+ \in
\Derlb(\cor_{M})$ by $F_- = F|_{M\times \{-t\}}$, $F_+ = F|_{M\times \{t\}}$,
for any $t\in [A,+\infty[$.  Then $F_-, F_+$ are indeed independent of $t \in
[A,+\infty[$ and have locally constant cohomology sheaves.  We let
$\Derlb_{\Lambda,+}(\cor_{M\times\R})$ be the full subcategory of
$\Derlb_\Lambda(\cor_{M\times\R})$ consisting of the $F$ such that $F_- \simeq
0$.
\end{definition}

For $F \in \Derlb_{\Lambda,+}(\cor_{M\times\R})$ we have by definition
\begin{equation}\label{eq:F_restr-infini}
F|_{M\times ]A,+\infty[} \simeq F_+\etens \cor_{]A,+\infty[} ,
\qquad  F|_{M\times ]-\infty,-A[} \simeq 0 .
\end{equation}

\begin{lemma}\label{lem:annulation_reim-pM}
Let $F\in \Derlb(\cor_{M\times\R})$. We assume that there exists $A>0$ such that
$\supp(F) \subset M\times [-A,A]$. We also assume either
$\SSi(F) \subset T^*_{\tau\geq 0}(M\times\R)$ or
$\SSi(F) \subset T^*_{\tau\leq 0}(M\times\R)$.
Let $p_M \cl M\times\R \to M$ be the projection.
Then $\reim{p_M}(F) \simeq \roim{p_M}(F) \simeq 0$.
\end{lemma}
\begin{proof}
By base change we may assume that $M$ is a point.
Then the result follows from the ``Morse lemma'' Corollary~\ref{cor:Morse}.
\end{proof}

\begin{theorem}\label{thm:restr-infini-ff}
Let $F,F' \in \Derlb_{\Lambda,+}(\cor_{M\times\R})$. We let $F_+, F'_+ \in
\Derlb(\cor_{M})$ be their restrictions to $M\times\{t\}$, $t\gg0$, as in
Definition~\ref{def:DerlbLambdaplus}.  Then
\begin{equation}\label{eq:restr-infini-ff1}
\RHom(F,F')  \isoto \RHom(F_+,F'_+) .
\end{equation}
In particular the functor $\Derlb_{\Lambda,+}(\cor_{M\times\R}) \to
\Derlb(\cor_{M})$ given by $F \mapsto F_+$ is fully faithful and we have:
$F\simeq F'$ if and only if $F_+ \simeq F'_+$.
\end{theorem}
\begin{proof}
Let $p_M\cl M\times\R \to M$ be the projection.  Let us choose $A>0$ so
that~\eqref{eq:F_restr-infini} holds for $F$ and $F'$ and let $u>2A$.  Hence
$\supp(\oim{T_u} F') \subset M\times ]A,+\infty[$ and we obtain by
Corollary~\ref{cor:Hom-isomorphes} (applied with $M=M$, $\rec=\id_M$)
\begin{equation}\label{eq:restr-infini-ffA}
  \begin{split}
\RHom(F,F') &\simeq  \RHom(F,\oim{T_u} F') \\
& \simeq \RHom(\opb{p_M}(F_+),\oim{T_u} F')  \\
&\simeq \RHom(F_+, \roim{p_M} \oim{T_u} F') .
  \end{split}
\end{equation}
Let us set $G = (\oim{T_u} F')\tens \cor_{M\times ]-\infty, A+u[}$.
By~\eqref{eq:F_restr-infini} we know that $\oim{T_u} F'$ has locally constant
cohomology sheaves in a neighborhood of $M\times \{A+u\}$. We deduce that
$\SSi(G) \subset T^*_{\tau\geq 0}(M\times\R)$.  By
Lemma~\ref{lem:annulation_reim-pM} we obtain $\roim{p_M}(G) \simeq 0$.
By~\eqref{eq:F_restr-infini} again we have the distinguished triangle $G \to
\oim{T_u} F' \to F'_+ \etens \cor_{[A+u,+\infty[} \to[+1]$.  Hence we obtain
$\roim{p_M} \oim{T_u} F' \simeq \roim{p_M} ( F'_+ \etens \cor_{[A+u,+\infty[} )
\simeq F'_+$ and~\eqref{eq:restr-infini-ffA} translates
into~\eqref{eq:restr-infini-ff1}.
\end{proof}

\begin{theorem}\label{thm:muhom=hom}
We assume moreover that $\Lambda/\rspos$ is compact.
Let $F,F' \in \Derb_{\Lambda,+}(\cor_{M\times\R})$. Then we have an isomorphism
\begin{equation}\label{eq:muhom=hom}
\RHom(F,F')  \isoto \rsect(\Lambda; \mu hom(F,F'))  .
\end{equation}
Its composition with~\eqref{eq:restr-infini-ff1} gives a canonical isomorphism
\begin{equation}\label{eq:muhom=hom-infini1}
\RHom(F_+,F'_+) \simeq \rsect(\Lambda; \mu hom(F,F')). 
\end{equation}
\end{theorem}
\begin{proof}
The second isomorphism follows from the first and
from~\eqref{eq:restr-infini-ff1}. Let us prove that the natural
morphism~\eqref{eq:muhom=hom} is an isomorphism.  Let $\Lambda^+ \subset
T^*_{\tau >0}(M\times\R\times\rspos)$ be the set defined
in~\eqref{eq:def-Lambdaplus}.  By Lemma~\ref{lem:SSPsiF} $\Psi_{M\times \R}(F)$
and $\Psi_{M\times \R}(F')$ belong to $\Derb_{\Lambda^+}(\cor_{M\times\R\times
  \rspos})$. The isomorphism~\eqref{eq:hLambda-hLambdaI} and
Theorem~\ref{thm:muhom=hompsi} give, setting $N_\varepsilon = M\times \R \times
\mo]0,\varepsilon[$,
\begin{align*}
H^i\rsect(\Lambda; \mu hom(F,F')) &\simeq \varinjlim_{\varepsilon>0}
H^i\RHom(\Psi_{M\times \R}(F) |_{N_\varepsilon},
 \Psi_{M\times \R}(F')|_{N_\varepsilon} )) \\
&\simeq H^i\RHom(F,F'),
\end{align*}
for any $i\in \Z$. This implies~\eqref{eq:muhom=hom}.
\end{proof}

\begin{remark}\label{rem:comp_muhom=hom-infini}
We have recalled in~\eqref{eq:comp_muhom} that $\mu hom$ admits a composition
morphism (denoted by $\mucirc$ in Notation~\ref{not:mucomposition}) compatible
with the composition morphism for $\rhom$.  In particular the
isomorphism~\eqref{eq:muhom=hom} is compatible with the composition
morphisms $\circ$ and $\mucirc$.  Since~\eqref{eq:restr-infini-ff1} is
clearly compatible with $\circ$, we deduce that~\eqref{eq:muhom=hom-infini1}
also is compatible with $\circ$ and $\mucirc$.
\end{remark}

\section{The triangulated orbit category case}

We checked in Sections~\ref{sec:def_orb_cat} and~\ref{sec:micsup_orb_cat} that
the results we used in the category $\Derb(\cor_M)$ have analogs in the category
$\Orb(\cor_M)$.  In particular we have already defined a Kashiwara-Schapira
stack $\kss^\orb(\cor_{\Lambda})$ in this situation.  In the same way the
results of Part~\ref{part:quant} also hold in the triangulated orbit category,
except the gluing procedure of Section~\ref{sec:gluing}, since the hypothesis of
Lemma~\ref{lem:gluingNsheaves} makes no sense in this case.  However we can
prove the existence part of Proposition~\ref{prop:quant_ouvert1} for the
triangulated orbit category in Proposition~\ref{prop:quant_orb} below. We first
remark in Lemma~\ref{lem:decomp_deuxouverts} below that we can decompose
$\Lambda$ in two open subsets with vanishing Maslov classes.

\begin{lemma}\label{lem:decomp_deuxouverts}
Let $X$ be a compact manifold and let $c\in H^1(X;\Z_X)$.
Then there exist two open subsets $U_1, U_2$ of $X$ such that $U_1$, $U_2$ and
$U_1\cap U_2$ have a finite number of connected components and the restrictions
$c|_{U_i} \in H^1(U_i;\Z_{U_i})$ vanish, for $i=1,2$.
\end{lemma}
\begin{proof}
We represent the image of $c$ in $H^1(X;\R_X)$ by a $1$-form $\alpha$.  Let
$r\cl X' \to X$ be the universal covering of $X$ and let $f\cl X' \to \R$ be a
primitive of $r^*(\alpha)$.  Then, for any $x_1,x_2 \in X'$ such that
$r(x_1)=r(x_2)$ we have $f(x_1)-f(x_2) = \langle c,\gamma \rangle$, where
$\gamma$ is the loop at $f(x_1)$ determined by $x_1,x_2$. Hence $f(x_1)-f(x_2)$
is an integer and $f$ descends to a map $g\cl X \to S^1$ and we have
$c=g^*(\delta)$, where $\delta \in H^1(S^1;\Z_{S^1})$ is the canonical class.

We choose a covering of $S^1$ by two open intervals $I_1,I_2$ such that the
boundaries points of $I_1$ and $I_2$ are regular values of $g$.  We set $U_i =
\opb{g}(I_i)$, $i=1,2$.  Then the boundaries of $U_1$, $U_2$ and $U_1\cap U_2$
are smooth compact hypersurfaces, hence with a finite number of connected
components. We see also that the boundaries of two distinct components of $U_1$
(or $U_2$ or $U_1\cap U_2$) are disjoint. Hence $U_1$, $U_2$ and $U_1\cap U_2$
have a finite number of components.
By construction we have $c|_{U_i} = g^*(\delta|_{I_i}) = 0$.
\end{proof}

\begin{proposition}\label{prop:quant_orb}
Let $\Lambda \subset T^*_{\tau >0}(M\times \R)$ be a closed conic Lagrangian
submanifold which satisfies the conditions of Lemma~\ref{lem:cond_Lambda_exact}.
Let $\shf \in \kss^\orb(\cor_{\Lambda})$.  Then there exist $\varepsilon>0$ and
$F \in \Orb(\cor_{M\times\R\times\mo]0,\varepsilon[})$ such that
\begin{itemize}
\item [(a)] $\supp(F) \subset \ol{\gammaof} \star \dot\pi_{M\times\R}(\Lambda)$
  and $\dot\SSi(F) \subset q_d\opb{q_\pi}(\Lambda) \cup
  r_d\opb{r_\pi}(\Lambda)$,
\item [(b)] we have $\kssfunc_{q_d\opb{q_\pi}(\Lambda)}(F) \simeq \opb{q}(\shf)$
  in $\kss^\orb(\cor_{q_d\opb{q_\pi}(\Lambda)})$.
\end{itemize}
\end{proposition}
\begin{proof}
We let $m \in H^1(\Lambda;\Z_\Lambda)$ be the  Maslov class of $\Lambda$.
By Lemma~\ref{lem:decomp_deuxouverts} we can find two open subsets $\Lambda^0$,
$\Lambda^1$ of $\Lambda$ such that $m|_{\Lambda^0} = m|_{\Lambda^1} = 0$.
By Lemma~\ref{lem:deform_bonne_pos}, up to moving $\Lambda$ by a
Hamiltonian isotopy, we can choose a family of open subsets $U_i$, $i\in I$, of
$M\times\R$ such that $(\Lambda^0,\Lambda)$ and $(\Lambda^1,\Lambda)$ satisfy
the assumption~\ref{assu:bonne_pos} with respect to this family.  With the
notations of the assumption~\ref{assu:bonne_pos} we set $J= \bigsqcup J_i$ and
we set for short $\Lambda_j = \Lambda_{i,j}$.  We can assume that $\Lambda^0$
and $\Lambda^1$ are unions of some $\Lambda_j$, that is, $\Lambda^0 =
\bigcup_{j\in J^0} \Lambda_j$ and $\Lambda^1 = \bigcup_{j\in J^1} \Lambda_j$ for
$J^0,J^1 \subset J$.

Then we can decompose $\Lambda^0 \cap \Lambda^1 = \Lambda^+ \sqcup \Lambda^-$
and represent $m$ by a \v Cech cocycle $\{c_{j,j'}\}_{j,j' \in J}$ such that
$c_{j,j'} = 1$ if $\Lambda_j \cap \Lambda_{j'} \subset \Lambda^-$ and $c_{j,j'}
= 0$ else.

By Propositions~\ref{prop:quant_ouvert1} and~\ref{prop:quant_ouvert2} there
exist $\varepsilon>0$ and objects, $F^a$, of
$\Derb_{(\Lambda^a)}(\cor_{M\times\R \times \mo]0,\varepsilon[})$, for $a
=0,1,+,-$, such that $F^a$ represents $\shf|_{\Lambda^a}$ and such that we have
morphisms
\begin{equation*}
\varphi_0^+ \cl F^0 \to F^+,  \; \varphi_0^- \cl F^0 \to F^-,  
\; \varphi_1^+ \cl F^1 \to F^+, \; \varphi_1^- \cl F^1 \to F^-[1], 
\end{equation*}
which induce isomorphisms in the Kashiwara-Schapira stack.

In $\Orb(\cor_{M\times\R\times\mo]0,\varepsilon[})$ we have $F^- \simeq F^-[1]$
and we can define $F \in \Orb(\cor_{M\times\R\times\mo]0,\varepsilon[})$ by the
distinguished triangle
$$
F \to F^0 \oplus F^1 \to[\left( \begin{smallmatrix}
  \varphi_0^+ & \varphi_0^- \\   \varphi_1^+ & \varphi_1^-
\end{smallmatrix} \right)] F^+ \oplus F^- \to[+1] .
$$
Then $F$ satisfies the required properties.
\end{proof}

The results of Section~\ref{sec:def_micsup} also hold for the category $\Orb$
and we deduce the analogs of Theorems~\ref{thm:quant}
and~\ref{thm:restr-infini-ff}.

\begin{theorem}\label{thm:quantOrb}
Let $\Lambda \subset T^*_{\tau>0}(M\times\R)$ be a closed conic connected
Lagrangian submanifold which satisfies the conditions of
Lemma~\ref{lem:cond_Lambda_exact}.  Let $\shf\in \kss^\orb(\cor_\Lambda)$.
Then there exists $F \in \Orb(\cor_{M\times \R})$ such that $\dot\SSo(F) =
\Lambda$, $F|_{M\times \{t\}} \simeq 0$ for $t\ll 0$ and
$\kssfunc^\orb_{\Lambda}(F) \simeq \shf$.
\end{theorem}

\begin{theorem}\label{thm:restr-infini-ff-Orb}
Let $F,F' \in \OrbL{\Lambda}(\cor_{M\times\R})$.
We assume that $\dot\SSo(F)$, $\dot\SSo(F') \subset \Lambda$ and that
$F|_{M\times \{t\}} \simeq F'|_{M\times \{t\}} \simeq 0$ for $t\ll0$.  We
define $F_+, F'_+ \in \Orb(\cor_M)$ by $F_+ = F|_{M\times \{t\}}$, $F'_+ =
F'|_{M\times \{t\}}$, for any $t\gg0$. Then we have the isomorphism
\begin{equation}
\label{eq:restr-infini-ff-Orb1}
\Hom_{\Orb(\cor_{M\times\R})}(F,F')  \isoto \Hom_{\Orb(\cor_M)}(F_+,F'_+) .
\end{equation}
In particular $F\simeq F'$ if and only if $F_+ \simeq F'_+$.
\end{theorem}

\part{Topological consequences}

In this part we let $M$ be a connected manifold and $\Lambda \subset
T^*_{\tau>0}(M\times\R)$ a closed conic connected Lagrangian submanifold which
satisfies the conditions of Lemma~\ref{lem:cond_Lambda_exact}.  We recall that
this means that $\Lambda$ is obtained from a compact exact Lagrangian
submanifold $\widetilde\Lambda \subset T^*M$ by adding a variable.  We recover
results of~\cite{FSS08} and~\cite{A12} which say that the projection $\Lambda
\to M$ is a homotopy equivalence, assuming the vanishing of the Maslov class of
$\Lambda$, and also~\cite{Kr13} which says that, indeed, the Maslov class of
$\Lambda$ vanishes.  We also see that the relative Stiefel-Whitney class of
$\Lambda$ vanishes.

\section{Poincar\'e groups}

We let $\pi_\Lambda \cl \Lambda \to M$ be the projection to the base and we
denote by $\pi_1(\pi_\Lambda) \cl \pi_1(\Lambda) \to \pi_1(M)$ the induced
morphism of Poincar\'e groups.

\begin{proposition}\label{prop:morph-Poincare-inj}
The morphism $\pi_1(\pi_\Lambda) \cl \pi_1(\Lambda) \to \pi_1(M)$ is injectif.
\end{proposition}
\begin{proof}
(i) We set $\cor=\Z/2\Z$ and $G= \pi_1(\Lambda)$.  We let $\rho \cl G \to
GL(\cor[G])$ be the regular representation of $G$. This means that $\cor[G]$ is
the vector space with basis $\{e_g\}_{g \in G}$ and the action of $G$ is given
by $g\cdot e_h = e_{gh}$, for all $g,h \in G$.  We let $\shl_\rho$ be the local
system on $\Lambda$ with stalks $\cor[G]$ corresponding to this representation
$\rho$.

\medskip\noindent
(ii) By Proposition~\ref{prop:KSstackorb} we have a unique simple object
$\shf_0 \in \kss^\orb(\cor_\Lambda)$.  Then the functor $\mu
hom^\varepsilon(F_0,\cdot)$ induces an equivalence $\kss^\orb(\cor_\Lambda)
\isoto \loc(\cor_\Lambda)$.  We let $\shf_\rho \in \kss^\orb(\cor_\Lambda)$ be
the object associated with $\shl_\rho$ by this equivalence.  By
Theorem~\ref{thm:quantOrb}, there exist $F_0, F_\rho \in \Orb(\cor_{M\times
  \R})$ such that $\kssfunc^\orb_\Lambda(F_0) \simeq \shf_0$ and
$\kssfunc^\orb_\Lambda(F_\rho) \simeq \shf_\rho$.  We then have $\mu
hom^\varepsilon(F_0,F_\rho)|_\Lambda \simeq \shl_\rho$.  We define $L_0, L_1 \in
\Orb(\cor_M)$ by $L_0 = F_0|_{M\times \{t\}}$ and $L_1 = F_\rho|_{M\times
  \{t\}}$ for $t\gg 0$.  We let $p\cl M\times\R \to M$ be the projection and we
set $F=F_0 \epstens \opb{p}L_1$ and $F' = F_\rho \epstens \opb{p}L_0$.  Then
$F|_{M\times \{t\}} \simeq L_0 \epstens L_1 \simeq F'|_{M\times \{t\}}$ for
$t\gg 0$ and Theorem~\ref{thm:restr-infini-ff-Orb} implies
\begin{equation}\label{eq:morph-Poincare-inj1}
  F_0 \epstens \opb{p}L_1 \simeq F_\rho \epstens \opb{p}L_0 .
\end{equation}

\noindent
(iii) As in Lemma~\ref{lem:equiv_loc-Oloc}, we let $L'_i$ be the sheaf on $M$
associated with the presheaf $U \mapsto \Hom_{\Orb(\cor_U)}(\cor_U,L_i)$, for
$i=0,1$.  Then $L'_0$ and $L'_1$ are local systems on $M$. Applying
$\kssfunc^\orb_\Lambda$ to~\eqref{eq:morph-Poincare-inj1} and the equivalence
$\kss^\orb(\cor_\Lambda) \isoto \loc(\cor_\Lambda)$, we find $\opb{\pi_\Lambda}
L'_1 \simeq \shl_\rho \tens \opb{\pi_\Lambda} L'_0$.

\medskip\noindent
(iv) We let $\rho'_0$ and $\rho'_1$ be the representations of $\pi_1(M)$
corresponding to the local systems $L'_0$ and $L'_1$.  They induce
representations of $G= \pi_1(\Lambda)$, say $\rho''_0$ and $\rho''_1$, through
the morphism $\pi_1(\pi_\Lambda)$. Then the result of~(iii) gives the
isomorphism of representations of $G$, $\rho''_1 \simeq \rho \tens \rho''_0$.
We restrict these representations to the subgroup $K =
\ker(\pi_1(\pi_\Lambda))$ of $G$. Then $\rho''_0|_K$ and $\rho''_1|_K$ are
trivial representations and we deduce that $\rho|_K$ also is trivial.  Since
$\rho$ is a faithful representation of $G$, this gives $K= \{1\}$, as required.
\end{proof}

Let $r\cl M' \to M$ be a covering. The derivative of $r$ induces a covering $r'
\cl T^*M' \to T^*M$.  We let $\Lambda'_0$ be a connected component of
$\opb{r'}(\Lambda')$.  Then $\Lambda'_0 \to \Lambda$ is a covering and
$\pi_1(\Lambda'_0)$ is a subgroup of $\pi_1(\Lambda)$.  We have the commutative
diagram
\begin{equation}\label{eq:diag_Poinc-groups}
\vcenter{\xymatrix{
\pi_1(\Lambda'_0) \ar@{^{(}->}[r]  \ar[d]
  & \pi_1(\Lambda)  \ar@{^{(}->}[d]^{\pi_1(\pi_\Lambda)}  \\
\pi_1(M') \ar[r] & \pi_1(M) ,  }}
\end{equation}
where $\pi_1(\pi_\Lambda)$ is injective by
Proposition~\ref{prop:morph-Poincare-inj}.  This implies that the morphism
$\pi_1(\Lambda'_0) \to \pi_1(M')$ is injective.  In particular, if $M'$ is the
universal cover of $M$, then $\pi_1(\Lambda'_0)$ vanishes, that is, $\Lambda'_0$
is the universal cover of $\Lambda$.

We denote by $m_\Lambda$ the Maslov class of $\Lambda$, which is a group
morphism $m_\Lambda \cl \pi_1(\Lambda) \to \Z$.
\begin{corollary}\label{cor:revet-Maslov}
We assume that $m_\Lambda \not= 0$. Then there exist covering maps $f\cl M_0
\to M_1$, $g\cl M_1 \to M$ where $f$ is a cyclic cover of group $\Z$ and closed
conic connected Lagrangian submanifolds $\Lambda_i \subset \dT^*(M_i \times \R)$
for $i=0,1$, such that the derivatives of $f$ and $g$ induce a cyclic cover of
group $\Z$, $\Lambda_0 \to \Lambda_1$, and an isomorphism $\Lambda_1 \isoto
\Lambda$:
\begin{equation}\label{eq:diag_revet-Maslov0}
\vcenter{\xymatrix@C=1.5cm{
\Lambda_0  \ar[r]^{\Z} \ar[d] & \Lambda_1  \ar[r]^\sim \ar[d] 
& \Lambda  \ar[d]  \\
M_0 \ar[r]^{\Z}_f  & M_1 \ar[r]_g  & M .   }}
\end{equation}
Moreover the isomorphism $\Lambda_1 \isoto \Lambda$ identifies the Maslov
classes of $\Lambda_1$ and $\Lambda$ and the Maslov class of $\Lambda_0$ is
zero.
\end{corollary}
\begin{proof}
We set $K = \ker(m_\Lambda)$. Since $m_\Lambda \not= 0$ we have $\pi_1(\Lambda)
/ K \simeq \Z$.  We let $M'$ be the universal cover of $M$ and we define
$\Lambda'_0$ as in the diagram~\eqref{eq:diag_Poinc-groups}.  Hence $\Lambda'_0$
is the universal cover of $\Lambda$.  Since $K$ and $\pi_1(\Lambda)$ are
subgroups of $\pi_1(M)$ they act freely on $M'$.  These actions commute with
their actions on $\Lambda'_0$ through the map $\Lambda'_0 \to M$.  We obtain the
diagram of quotient manifolds
\begin{equation*}
\vcenter{\xymatrix@C=1.5cm{
\Lambda'_0  \ar[r] \ar[d] & \Lambda'_0/K  \ar[r]^{f'} \ar[d] 
& \Lambda'_0 / \pi_1(\Lambda)  \ar[d]  \\
M' \ar[r]  & M'/K \ar[r]^f  & M' / \pi_1(\Lambda) ,   }}
\end{equation*}
where $f'$ and $f$ are covering maps with group $\pi_1(\Lambda) / K \simeq \Z$.
We set $\Lambda_0 = \Lambda'_0/K$, $M_0 = M'/K$, $\Lambda_1 =
\Lambda'_0/\pi_1(\Lambda)$ and $M_1 = M'/\pi_1(\Lambda)$.  Then $\Lambda_1$ is
identified with $\Lambda$ since it is the quotient of the universal cover of
$\Lambda$ by its Poincar\'e group.  This gives the
diagram~\eqref{eq:diag_revet-Maslov0}. The claim on the Maslov classes follows
easily.
\end{proof}

\section{Vanishing of the Maslov class}

We recall that $m_\Lambda \cl \pi_1(\Lambda) \to \Z$ is the Maslov class of
$\Lambda$.  Since $\Lambda$ satisfies the conditions of
Lemma~\ref{lem:cond_Lambda_exact}, it is actually obtained from a compact exact
Lagrangian submanifold $\widetilde\Lambda \subset T^*M$ by adding a variable.
We have $\widetilde\Lambda \simeq \Lambda/\rspos$.  Hence $\Lambda$ and
$\widetilde\Lambda$ are homotopic and we can see that they have the same Maslov
class.  In~\cite{Kr13} Kragh and Abouzaid (using a result of~\cite{A11}) prove
that the Maslov class of any compact exact Lagrangian submanifold of a cotangent
bundle vanishes.  Now we can give a new proof of this result.

\begin{theorem}\label{thm:Maslovzero}
We have $m_\Lambda = 0$.
\end{theorem}
\begin{proof}
(i) We set $\cor=\Z/2\Z$.  We apply Corollary~\ref{cor:revet-Maslov} and we
replace $M$ by $M_1$. Hence we have cyclic covers of group $\Z$
\begin{equation}\label{eq:diag-Maslovzero}
\vcenter{\xymatrix@C=1.5cm{
\Lambda_0  \ar[r]^{\Z} \ar[d] & \Lambda  \ar[d]_{\pi_\Lambda}  \\
M_0 \ar[r]^{\Z}_f  & M   }}
\end{equation}
such that $m_{\Lambda_0} =0$.  Hence have an exact sequence $\pi_1(\Lambda_0)
\to \pi_1(\Lambda) \to[m_\Lambda] \Z$. The diagram~\eqref{eq:diag-Maslovzero}
induces the isomorphisms $\pi_1(\Lambda) / \pi_1(\Lambda_0) \simeq \Z \simeq
\pi_1(M) / \pi_1(M_0)$ and we deduce that $m_\Lambda$ factorizes through a
morphism $m\cl \pi_1(M) \to \Z$ such that $\ker(m) \simeq \pi_1(M_0)$.

We let $c \in H^1(M;\Z)$ be the cohomology class corresponding to $m$. Then
$\pi_\Lambda^*(c)$ is the Maslov class of $\Lambda$ (viewed as a cohomology
class) and $f^*(c) =0$.

We will obtain a contradiction by constructing a quantization $G$ of
$\Lambda_0$ on $M_0 \times \R$ with opposite properties: (i) $G$ should be
unbounded because of the non-vanishing of $c$ and (ii) $G$ sould be bounded
because $G|_{M_0 \times \{t\}}$, $t\gg0$, is a locally bounded locally constant
object.

The construction of $G$ consists in checking that Theorem~\ref{thm:quant} holds
when we replace $M$ by a cyclic cover.

\medskip\noindent
(ii) By Lemma~\ref{lem:decomp_deuxouverts} there exist two open subsets $U_1,
U_2$ of $M$ such that $U_1$, $U_2$ and $U_1\cap U_2$ have a finite number of
connected components and the restrictions $c|_{U_i} \in H^1(U_i;\Z_{U_i})$
vanish, for $i=1,2$.  We can decompose $U_1 \cap U_2 = V_+ \sqcup V_-$ into two
disjoint open subsets such that $c$ is represented by the \v Cech cocyle $c_{12}
\cl U_{12} \to \Z$ ($c_{12}$ is a locally constant function) with values $0$ on
$V_+$ and $d$ on $V_-$, for some $d\in \Z$.  Since $m_\Lambda \not=0$ we have
$d\not=0$.

For $\varepsilon>0$ we set $U_i^\varepsilon = U_i \times \R \times
\mo]0,\varepsilon[$ and $\Lambda_i^\varepsilon = T^*U_i^\varepsilon \cap (
q_d\opb{q_\pi}(\Lambda) \cup r_d\opb{r_\pi}( \Lambda))$ for $i=1,2$, and
$V_\pm^\varepsilon = V_\pm \times \R \times \mo]0,\varepsilon[$.  By
Proposition~\ref{prop:quant_ouvert1} there exist $\varepsilon>0$ and $F_i \in
\Derb(\cor_{U_i^\varepsilon})$ for $i=1,2$, such that $\dot\SSi(F_i) =
\Lambda_i^\varepsilon$, $F_i$ is simple along $\Lambda_i^\varepsilon$ and we
have isomorphisms $\varphi_+ \cl F_1|_{V_+^\varepsilon} \isoto
F_2|_{V_+^\varepsilon}$ and $\varphi_- \cl F_1|_{V_-^\varepsilon} \isoto
F_2|_{V_-^\varepsilon}[d]$.

\medskip\noindent
(iii) We set $M_0^\varepsilon = M_0\times \R \times \mo]0,\varepsilon[$.  The
action of $\Z$ on $M_0$ induces an action on $M_0^\varepsilon$. For $n\in \Z$ we
denote by $\psi_n \cl M_0^\varepsilon \to M_0^\varepsilon$ the action of $n$.

We set $U'_i = \opb{f}(U_i^\varepsilon)$, $i=1,2$, and we decompose $U'_i =
\bigsqcup_{n\in \Z} U_i^n$ in such a way that $f$ identifies each $U_i^n$ with
$U_i \times \R$ and $U_i^n = \psi_n(U_i^0)$.  We can also assume that $f$
identifies $U_1^0\cap U_2^0$ with $V_+$ and $U_1^0\cap U_2^1$ with $V_-$.

We define $G_i$ on $U'_i$ by $G_i|_{U_i^n} = (\opb{f}(F_i))|_{U_i^n}[nd]$.  Then
the isomorphisms $\varphi_\pm$ of~(ii) induce an isomorphism $\varphi \cl
G_1|_{U'_{12}} \isoto G_2|_{U'_{12}}$.  We let $j_i, j_{12}$ be the inclusions
of $U'_i$, $U'_{12}$ in $M_0^\varepsilon$. We define $G' \in
\Der(\cor_{M_0^\varepsilon})$ by the distinguished triangle
$$
G' \to \roim{j_1} G_1 \oplus \roim{j_2} G_2 
\to[(\varphi,-\id)] \roim{j_{12}} (G_2|_{U'_{12}})  \to[+1] .
$$
Then $G'$ is simple along $q_d\opb{q_\pi}(\Lambda_0) \cup
r_d\opb{r_\pi}(\Lambda_0)$ and we have isomorphisms $\opb{\psi_n}(G') \simeq
G'[nd]$, for each $n\in \Z$. Since $F_1$ and $F_2$ are bounded, $G'$ is locally
bounded.

\medskip\noindent
(iv) We apply Lemma~\ref{lem:isotopy_transl} to $\Lambda$ and we obtain a
homogeneous Hamiltonian isotopy $\phi$ of $\dT^*(M\times\R)$ which keeps
$\Lambda$ fixed and translates $T'_1(\Lambda)$ vertically.  Lifting $\phi$ to
$\dT^*(M_0\times\R)$ we see that $\Lambda_0$ satisfies the conclusions of
Lemma~\ref{lem:isotopy_transl}.  Hence we can apply
Corollary~\ref{cor:restr_it_equiv} to $\Lambda_0$ and we deduce from the object
$G' \in \Derlb(\cor_{M_0^\varepsilon})$ defined in~(iii) an object $G \in
\Derlb(\cor_{M_0\times\R})$ which is simple along $\Lambda_0$ and such that we
have isomorphisms $\opb{\psi_n}(G) \simeq G[nd]$, for each $n\in \Z$.

We define $G_+ = G|_{M_0 \times \{t_0\}}$ for $t_0\gg 0$ as in
Definition~\ref{def:DerlbLambdaplus}.  Theorem~\ref{thm:restr-infini-ff} gives
$\RHom(G,G) \isoto \RHom(G_+,G_+)$.  Since $G\not\simeq 0$ it follows that $G_+
\not\simeq 0$.  We also know that $G_+$ has locally constant cohomology sheaves
and is locally bounded (since $G$ is). Since $M_0$ is connected we deduce that
$G_+$ is bounded.

The isomorphisms $\opb{\psi_n}(G) \simeq G[nd]$ give $\opb{\psi_n}(G_+) \simeq
G_+[nd]$, for all $n\in \Z$.  Since $G_+$ is bounded and non zero, we obtain
$d=0$ but this contradicts the hypothesis $m_\Lambda \not=0$.
\end{proof}

\section{Behaviour at infinity}

We recall the notation $\Derb_{\Lambda,+}(\cor_{M\times\R})$ of
Definition~\ref{def:DerlbLambdaplus} and $F_+ = F|_{M \times \{t_0\}}$ for
$t_0\gg 0$, for $F \in \Derb_{\Lambda,+}(\cor_{M\times\R})$.

\begin{proposition}\label{prop:restr-inf-corps}
We assume that $\cor = \Z$ or $\cor$ is a finite field.  Let $F \in
\Derb_{\Lambda,+}(\cor_{M\times\R})$.  We assume that $F$ is simple along
$\Lambda$.  Then $F_+$ is concentrated in one degree, say $i$, and $H^iF_+$ is a
local system with stalks isomorphic to $\cor$.
\end{proposition}
\begin{proof}
(i) We first assume that $\cor$ is a finite field. 
Let us prove that $F_+$ is concentrated in one degree.  Let $a\leq b$ be
respectively the minimal and maximal integers $i$ such that $H^iF_+ \not\simeq
0$. By Lemma~\ref{lem:stalk_simple_sh} the local systems $H^iF_+$ are of finite
rank.  Since $\cor$ is finite we can find a finite cover $r\cl M' \to M$ such
that $\opb{r}( H^iF_+)$ are trivial, for $i=a,b$.  We set $F' = \opb{(r\times
  \id_\R)} F$ and $\Lambda' = \opb{d(r\times \id_\R)}(\Lambda)$. Then $\opb{r}(
H^iF_+) \simeq H^iF'_+$, $F'$ is simple along $\Lambda'$ and we have $\mu
hom(F',F') \simeq \cor_{\Lambda'}$. Since $\Lambda'/\rspos$ is compact,
Theorem~\ref{thm:muhom=hom} gives
\begin{equation}\label{eq:restr-inf-corps1}
\RHom(F'_+,F'_+) \simeq \rsect(\Lambda'; \cor_{\Lambda'}).
\end{equation}
On the other hand the complex $G=\rhom(F'_+,F'_+)$ is concentrated in degrees
greater than $a-b$ and $H^{a-b}G \simeq \hom(H^aF'_+,H^bF'_+)$ is a non zero
constant sheaf.  Hence $H^{a-b}\RHom(F'_+,F'_+)$ is non zero.
By~\eqref{eq:restr-inf-corps1} we deduce that $H^{a-b}\rsect(\Lambda';
\cor_{\Lambda'})$ also is non zero, which implies $a-b\geq 0$.  Hence $a=b$ and
$F_+$ is concentrated in a single degree.

\medskip\noindent
(ii) Now we prove that $H^aF_+$ is of rank one, that is, $H^aF'_+ \simeq
\cor_{M'}$. There exists $d\geq 1$ such that $H^aF'_+ \simeq \cor^d_{M'}$. The
isomorphism~\ref{eq:restr-inf-corps1} gives in degree $0$:
\begin{equation}\label{eq:restr-inf-corps2}
\Hom(\cor^d,\cor^d) \simeq H^0(\Lambda'; \cor_{\Lambda'}) .
\end{equation}
By Remark~\ref{rem:comp_muhom=hom-infini} this isomorphism is compatible with
the algebra structures of both terms.  Let $I$ be the set of connected
components of $\Lambda'$. We obtain $|I| = d^2$.  The natural decomposition
$H^0(\Lambda'; \cor_{\Lambda'}) \simeq \bigoplus_{i\in I} H^0(\Lambda'_i;
\cor_{\Lambda'_i})$ gives an expression of the unit as a sum of orthogonal
idempotents, $1 = \sum_{i\in I} e_i$, where $e_i$ is the projection
$$
e_i \cl H^0(\Lambda'; \cor_{\Lambda'}) 
 \to H^0(\Lambda'_i; \cor_{\Lambda'_i}), \qquad i\in I.
$$
We let $m_i \in \Hom(\cor^d,\cor^d)$ be the image of $e_i$
by~\eqref{eq:restr-inf-corps2}.  The relation $1 = \sum_{i\in I} e_i$ gives a
decomposition of the identity matrix $I_d = \sum_{i\in I} m_i$ as a sum of $|I|$
non-zero orthogonal projections, that is, $m_i^2 = m_i$ and $m_im_j=0$, for
$i\not= j$. We deduce that $|I| \leq d$, that is, $d^2 \leq d$. Hence $d=1$, as
claimed.

\medskip\noindent
(iii) Now we assume that $\cor = \Z$.
By Lemma~\ref{lem:stalk_simple_sh}, for each $i\in \Z$, there exists $d_i\in \N$
such that the stalks of the local system $H^iF_+$ are isomorphic to $\Z^{d_i}$.
We set $G = F \ltens_\Z \Z/2\Z$. Then $G$ is simple along $\Lambda$ and $G_+
\simeq F_+ \tens_\Z \Z/2\Z$. In particular the stalks of $H^iG_+$ are
isomorphic to $(\Z/2\Z)^{d_i}$. By~(i) and~(ii) we deduce that there exists
$a\in \Z$ such that $d_i=0$ for all $i\not= a$ and $d_a=1$, as claimed.
\end{proof}

\begin{corollary}\label{cor:cohomology_Lambda}
We assume that $\cor = \Z$ or $\cor$ is a finite field and that
$\kss(\cor_\Lambda)$ has at least one global simple object. Then the projection
$\Lambda \to M$ induces an isomorphism $\rsect(M;\cor_M) \isoto
\rsect(\Lambda;\cor_\Lambda)$.
\end{corollary}
\begin{proof}
We choose a simple object $\shf \in \kss(\cor_\Lambda)$. By
Theorem~\ref{thm:quant} there exists $F \in \Derb_{\Lambda,+}(\cor_{M\times\R})$
such that $\kssfunc_{\Lambda}(F) \simeq \shf$.
By Proposition~\ref{prop:restr-inf-corps} $F_+$ is concentrated in one degree,
say $i$, and $H^iF_+$ is a local system with stalks isomorphic to $\cor$.  Hence
$\rhom(F_+,F_+) \simeq \cor_M$ and $\RHom(F_+,F_+) \simeq \rsect(M;\cor_M)$.
Since $F$ is simple we also have $\mu hom(F,F)|_\Lambda \simeq \cor_\Lambda$.
By Theorem~\ref{thm:muhom=hom} we deduce an isomorphism
\begin{equation}\label{eq:cohomology_Lambda1}
\rsect(M;\cor_M) \simeq \rsect(\Lambda;\cor_\Lambda) .
\end{equation}
By construction~\eqref{eq:cohomology_Lambda1} is given by taking the global
sections in the bottom morphism of the commutative diagram:
$$
\def\objectstyle{\scriptstyle}
\xymatrix{
\cor_{M\times\R} \ar[rr]^a \ar[d]^b
 &&  \roim{(\dot\pi_{M\times \R})}(\cor_\lambda) \ar[d]^c_\wr  \\
\rhom(F,F) \ar[r]^-\sim & \roim{(\pi_{M\times \R})} \mu hom(F,F) \ar[r]
& \roim{(\dot\pi_{M\times \R})} (\mu hom(F,F)|_\Lambda) , }
$$
where $b$ and $c$ map the sections $1$ to the identity morphisms.  When taking
global sections, $b$ and $c$ induce isomorphisms and~$a$ induces the natural
morphism $\rsect(M;\cor_M) \to \rsect(\Lambda;\cor_\Lambda)$ given by the
projection of $\Lambda$ to the base $M$. The bottom horizontal arrow
induces~\eqref{eq:cohomology_Lambda1}.  This shows
that~\eqref{eq:cohomology_Lambda1} is indeed induced by the projection to the
base.
\end{proof}

\begin{remark}\label{rem:cohom_Z2Z}
We have seen in Corollary~\ref{cor:objet_global_KSstack} that
$\kss(\cor_\Lambda)$ has at least one global simple object if the Maslov class
of $\Lambda$ and the image of its relative Stiefel-Whitney class in
$H^2(\Lambda;\cor^\times)$ is zero.
By Theorem~\ref{thm:Maslovzero} the Maslov class vanishes in our case.  Hence,
when $\cor = \Z/2\Z$ the stack $\kss(\cor_\Lambda)$ has a global simple object
and Corollary~\ref{cor:cohomology_Lambda} gives: the projection $\Lambda \to M$
induces an isomorphism
$$
\rsect(M;\Z/2\Z_M) \isoto \rsect(\Lambda;\Z/2\Z_\Lambda).
$$
\end{remark}

\section{Vanishing of the Stiefel-Whitney class}

We have introduced a class $rw_2(\lambda_0,\lambda_\Lambda) \in
H^2(\Lambda;\Z/2\Z_\Lambda)$ in Definition~\ref{def:relStieWhit} and
Corollary~\ref{cor:objet_global_KSstack}, where $\lambda_0$ is the fiber bundle
of vertical directions (with respect to the projection $\pi_{M\times\R}$) and
$\lambda_\Lambda$ the tangent vector bundle of $\Lambda$.

Here we prove that $rw_2(\lambda_0,\lambda_\Lambda)$ vanishes.  For this we will
use Theorem~\ref{thm:quant} in the framework of twisted sheaves.  Let $c\in
H^2(M;\Z/2\Z)$ be given and let $\check c = \{c_{ijk}\}$, $i,j,k \in $, be a \v
Cech cocycle representing $c$ with respect to a finite covering $\{U_i\}_{i\in
  I}$ of $M$. We view $\Z/2\Z$ as the multiplicative group $\{\pm 1\}$ and
$c_{ijk} = \pm 1$, for all $i,j,k$.
\begin{definition}\label{def:twisted-sheaves}
A $\check c$-twisted sheaf $F$ on $M$ is the data of sheaves $F_i \in
\Mod(\cor_{U_i})$ and isomorphisms $\varphi_{ij} \cl F_j|_{U_{ij}} \isoto
F_i|_{U_{ij}}$ satisfying the condition
$$
\varphi_{ij} \circ \varphi_{jk} = c_{ijk} \, \varphi_{ik} .
$$
The $\check c$-twisted sheaves form an abelian category that we denote by
$\Mod(\cor^{\check c}_M)$. We denote by $\Derb(\cor^{\check c}_M)$ its derived
category.
\end{definition}
The prestack $U \mapsto \Mod(\cor^{\check c|_U}_U)$ is a stack which is locally
equivalent to the stack of sheaves.  The usual operations on sheaves extend to
twisted sheaves. In particular if $\check c$, $\check d$ are \v Cech cocycles on
$M$ and $F \in \Derb(\cor^{\check c}_M)$, $F' \in \Derb(\cor^{\check d}_M)$, we
have a tensor product $F \ltens F' \in \Der(\cor^{\check c + \check d}_M)$ and a
homomorphism sheaf $\rhom(F,F')\in \Der(\cor^{\check d - \check c}_M)$. If $f
\cl M \to N$ is a morphism of manifolds and $\check d$ is a \v Cech cocycle on
$N$ with values in $\{\pm 1\}$, we have inverse images $\opb{f}, \epb{f} \cl
\Derb(\cor^{\check d}_N) \to \Derb(\cor^{f^* \check d}_M)$ and direct images
$\roim{f}, \reim{f} \cl \Derb(\cor^{f^* \check d}_M) \to \Derb(\cor^{\check
  d}_N)$ with the usual adjunction properties.  The notion of microsupport also
generalizes to the twisted case (since this is a local notion and twisted
sheaves are locally equivalent to sheaves) with the same behaviour with respect
to the sheaves operations.

We can define a Kashiwara-Schapira stack $\kss(\cor^{\check c}_{\Lambda})$ and
formulate a version of Theorem~\ref{thm:quant} in this framework: for $\shf \in
\kss(\cor^{\check c}_{\Lambda})$ there exists $F \in \Derb(\cor^{\check
  c}_{M\times\R})$ such that $\dot\SSi(F) = \Lambda$, $F|_{M\times \{t\}} \simeq
0$ for $t\ll 0$ and $\kssfunc^{\check c}_{\Lambda}(F) \simeq \shf$.

\begin{proposition}\label{prop:vanishingSW}
The class $rw_2(\lambda_0,\lambda_\Lambda) \in H^2(\Lambda;\Z/2\Z_\Lambda)$ is
zero.
\end{proposition}
\begin{proof}
(i) We choose $\cor=\Z$.
By Corollary~\ref{cor:cohomology_Lambda} and Remark~\ref{rem:cohom_Z2Z} we have
$H^2(M;\Z/2\Z_M) \isoto H^2(\Lambda;\Z/2\Z_\Lambda)$. We let $c\in
H^2(M;\Z/2\Z_M)$ be the inverse image of $rw_2(\lambda_0,\lambda_\Lambda)$ by
this isomorphism and we choose a \v Cech cocycle $\check c$ representing $c$.
Then the twisted Kashiwara-Schapira stack $\kss(\cor^{\check c}_{\Lambda})$ has
a simple global object and the twisted version of Theorem~\ref{thm:quant} gives
$F \in \Derb_{\Lambda,+}(\cor^{\check c}_{M\times\R})$ which is simple along
$\Lambda$.  By Proposition~\ref{prop:restr-inf-corps} we have $F_+ \simeq L[d]$
where $L \in \Mod(\cor^{\check c}_M)$ is a twisted locally constant sheaf with
stalks isomorphic to $\Z$ and $d$ is some integer.

\medskip\noindent
(ii) Now we prove that the existence of a locally constant $L \in
\Mod(\cor^{\check c}_M)$ as in~(i) implies that $\check c$ is a boundary, that
is, $rw_2(\lambda_0,\lambda_\Lambda)=0$.

The cocycle $\check c$ is associated with a covering $\{U_i\}_{i\in I}$ of $M$.
The object $L \in \Mod(\cor^{\check c}_M)$ is given by sheaves $L_i \in
\Mod(\cor_{U_i})$ and isomorphisms $\varphi_{ij} \cl L_j|_{U_{ij}} \isoto
L_i|_{U_{ij}}$, for any $i,j\in I$, such that $\varphi_{ij} \circ \varphi_{jk} =
c_{ijk} \, \varphi_{ik}$ for all $i,j,k \in I$.  We can assume that $U_i$ is
contractible and that $U_{ij}$ is connected for any $i,j\in I$.  Since $L$ is
locally constant, we can choose an isomorphism $\varphi_i \cl L|_{U_i} \simeq
\Z_{U_i}$ for each $i\in I$. Then the composition $b_{ij} = \varphi_i
\varphi_{ij} \varphi_j^{-1}$ is an isomorphism $\Z\isoto \Z$, that is, $b_{ij} =
\pm 1$.  We let $\check b$ be the $1$-cochain defined by $\{b_{ij}\}_{i,j\in
  I}$.  Then the equality $\varphi_{ij} \circ \varphi_{jk} = c_{ijk} \,
\varphi_{ik}$ says that $\check c$ is the boundary of $\check b$, as required.
\end{proof}

\section{A canonical quantization}

Summing up the results of this part we obtain a canonical quantization for any
closed conic connected Lagrangian submanifold $\Lambda \subset
T^*_{\tau>0}(M\times\R)$ which satisfies the conditions of
Lemma~\ref{lem:cond_Lambda_exact}.

\begin{theorem}\label{thm:quant_canon}
Let $\cor$ be a ring with finite global dimension.
\begin{itemize}
\item [(i)] There exists $F \in\Derb_{\Lambda,+}(\cor_{M\times\R})$ such that
  $F_- \simeq 0$ and $F_+ \simeq \cor_M$.
\item [(ii)] The object $F$ in~{\rm(i)} is unique up to a unique isomorphism:
  for another such $F' \in \Derb_{\Lambda,+}(\cor_{M\times\R})$ we have a
  canonical isomorphism $\Hom(F,F') \simeq \Hom(F_+,F'_+) \simeq \cor$.
\item [(iii)] The projection $\Lambda \to M$ yields an isomorphism
  $\rsect(M;\cor_M) \isoto \rsect(\Lambda;\cor_\Lambda)$.
\end{itemize}
\end{theorem}
\begin{proof}
(i) We first assume that $\cor=\Z$.  By Theorem~\ref{thm:Maslovzero} and
Proposition~\ref{prop:vanishingSW} we know that $m_\Lambda = 0$ and
$rw_2(\lambda_0,\lambda_\Lambda) = 0$.  By Theorem~\ref{thm:quant} there exists
$F^0 \in \Derb_{\Lambda,+}(\cor_{M\times\R})$ which is simple along $\Lambda$.
By Proposition~\ref{prop:restr-inf-corps} we have $F^0_+ \simeq L[d]$ where $L
\in \Mod(\cor_M)$ is locally constant with stalks isomorphic to $\Z$ and $d$ is
some integer. Let $p\cl M\times \R \to M$ be the projection. Then $F^1 = F^0
\tens \opb{p}L^{\otimes-1}[-d]$ satisfies the required properties.  For a
general ring $\cor$ we set $F = F^1 \ltens_{\Z_{M\times\R}} \cor_{M\times\R}$.

\noindent
(ii) is given by Theorem~\ref{thm:restr-infini-ff}.

\noindent
(iii) is given by Corollary~\ref{cor:cohomology_Lambda}.
\end{proof}

In~\cite{A12} Abouzaid gives a result more precise than
Theorem~\ref{thm:quant_canon}: the projection $\pi_\Lambda \cl \Lambda \to M$
induces an isomorphism of the fundamental groups. Since we already have an
isomorphism between the cohomology groups, it is enough to show that
$\pi_1(\Lambda) \to \pi_1(M)$ is an isomorphism.  It is equivalent to show that
the inverse image by $\pi_\Lambda$ induces an equivalence of categories
$\loc(\cor_M) \isoto \loc(\cor_\Lambda)$, for some field $\cor$.

\begin{proposition}\label{prop:equiv_Poinc}
Let $\cor$ be a field. Let $\pi_\Lambda \cl \Lambda \to M$ be the
projection. Then the inverse image functor $\opb{\pi_\Lambda} \cl \loc(\cor_M)
\to \loc(\cor_\Lambda)$ is an equivalence of categories.
\end{proposition}
\begin{proof}
(i) We first prove that $\opb{\pi_\Lambda}$ is fully faithful. Let $F \in
\Derb_{\Lambda,+}(\cor_{M\times\R})$ be the simple object given by
Theorem~\ref{thm:quant_canon}.  Since $F$ is simple we have $\mu
hom(F,F)|_\Lambda \simeq \cor_\Lambda$ and we deduce, for $L, L' \in
\loc(\cor_M)$,
\begin{equation}\label{prop:equiv_Poinc1}
\mu hom(F\tens \opb{p}L,F\tens \opb{p}L') \simeq 
\hom(\opb{\pi_\Lambda} L, \opb{\pi_\Lambda} L') ,
\end{equation}
where $p\cl M\times\R \to M$ is the projection.  We have $(F\tens \opb{p}L)_+
\simeq L$ and~\eqref{eq:muhom=hom-infini1} together
with~\eqref{prop:equiv_Poinc1} imply
\begin{align*}
\Hom(L,L') & \simeq H^0(\Lambda; \mu hom(F\tens \opb{p}L,F\tens \opb{p}L')) \\
& \simeq  \Hom(\opb{\pi_\Lambda} L, \opb{\pi_\Lambda} L') ,
\end{align*}
which means that $\opb{\pi_\Lambda}$ is fully faithful.

\medskip\noindent
(ii) We prove that $\opb{\pi_\Lambda}$ is essentially surjective.
Let $L_1 \in \loc(\cor_\Lambda)$ be given. We recall that the functor $\mu
hom(F,\cdot)$ induces an equivalence $\kss(\cor_\Lambda) \isoto
\loc(\cor_\Lambda)$ (see Proposition~\ref{prop:KSstack=Dloc}, where the induced
functor is denoted $\ol{\mu hom}(F,\cdot)$).  Hence there exists $\shl_1\in
\kss(\cor_\Lambda)$ such that $\ol{\mu hom}(F,\shl_1) \simeq \shl$.  By
Theorem~\ref{thm:quant} there exists $F_1 \in \Derb(\cor_{M\times\R})$ such that
$\kssfunc_\Lambda(F_1) \simeq L_1$. Then we have $\mu hom(F,F_1)|_\Lambda \simeq
L_1$.

We set $L = (F_1)_+ \in \Derb(\cor_M)$.  Then $\dot\SSi(L) = \emptyset$ and,
since $F_+ \simeq \cor_M$, we also have $L \simeq (F\tens \opb{p}L)_+$.  Hence
$(F_1)_+ \simeq (F\tens \opb{p}L)_+$ and Theorem~\ref{thm:restr-infini-ff} gives
$F_1 \simeq F\tens \opb{p}L$.  We deduce
$$
\mu hom(F,F_1)|_\Lambda \simeq \mu hom(F, F\tens \opb{p}L)|_\Lambda 
 \simeq \opb{\pi_\Lambda} L .
$$
Hence $L_1 \simeq \opb{\pi_\Lambda} L$. This proves that $L$ is concentrated in
degree $0$.  Since $\dot\SSi(L) = \emptyset$ we have $L \in \loc(\cor_M)$ and
$L_1 \in \opb{\pi_\Lambda}(\loc(\cor_M))$.
\end{proof}

As already remarked, Theorem~\ref{thm:quant_canon} and
Proposition~\ref{prop:equiv_Poinc} imply

\begin{corollary}\label{cor:equiv_homot}
The projection $\Lambda \to M$ is a homotopy equivalence.
\end{corollary}

\providecommand{\bysame}{\leavevmode\hbox to3em{\hrulefill}\thinspace}

\vspace*{1cm}
\noindent
\parbox[t]{21em}
{\scriptsize{
\noindent
St{\'e}phane Guillermou\\
Institut Fourier, Universit{\'e} Grenoble I, \\
email: Stephane.Guillermou@ujf-grenoble.fr\\
}}

\end{document}